\documentclass[reqno,10pt,centertags]{amsart}
\usepackage{amsmath,amsthm,amscd,amssymb,latexsym,esint,upref,stmaryrd,
enumerate,color,verbatim,yfonts,mathtools,amsfonts,euscript}
\usepackage{hyperref}

\newcommand*{\mailto}[1]{\href{mailto:#1}{\nolinkurl{#1}}}
\newcommand{\arxiv}[1]{\href{http://arxiv.org/abs/#1}{arXiv:#1}}

\newcommand{\bbC}{{\mathbb{C}}}

\newcommand{\bbN}{{\mathbb{N}}}

\newcommand{\bbR}{{\mathbb{R}}}

\newcommand{\bbZ}{{\mathbb{Z}}}

\newcommand{\cB}{{\mathcal B}}

\newcommand{\cF}{{\mathcal F}}

\newcommand{\cH}{{\mathcal H}}

\newcommand{\cK}{{\mathcal K}}

\newcommand{\cQ}{{\mathcal Q}}

\newcommand{\cU}{{\mathcal U}}
\newcommand{\cV}{{\mathcal V}}
\newcommand{\cW}{{\mathcal W}}
\newcommand{\cX}{{\mathcal X}}

\newcommand{\identity}{I}

\renewcommand{\k}{\varkappa}

\DeclareMathOperator{\vol}{vol}

\DeclareMathOperator{\supp}{supp}

\DeclareMathOperator{\ran}{ran}
\DeclareMathOperator{\dom}{dom}

\DeclareMathOperator{\tr}{tr}

\DeclareMathOperator*{\nlim}{n-lim}
\DeclareMathOperator*{\slim}{s-lim}
\DeclareMathOperator*{\wlim}{w-lim}

\DeclareMathOperator*{\sgn}{sgn}

\renewcommand{\Re}{\text{\rm Re}}
\renewcommand{\Im}{\text{\rm Im}}
\renewcommand{\ln}{\text{\rm ln}}

\renewcommand{\hat}{\widehat}
\renewcommand{\tilde}{\widetilde}
\renewcommand*{\epsilon}{\varepsilon}
\renewcommand*{\theta}{\vartheta}
\renewcommand*{\rho}{\varrho}

\newcommand{\ind}{\operatorname{ind}}
\newcommand{\no}{\notag}
\newcommand{\lb}{\label}
\newcommand{\f}{\frac}

\newcommand{\ol}{\overline}

\newcommand{\hatt}{\widehat}

\renewcommand{\ge}{\geqslant}

\let\geq\geqslant
\let\leq\leqslant

\makeatletter
\def\theequation{\@arabic\c@equation}

\allowdisplaybreaks 
\numberwithin{equation}{section}

\newtheorem{theorem}{Theorem}[section]
\newtheorem{proposition}[theorem]{Proposition}
\newtheorem{lemma}[theorem]{Lemma}
\newtheorem{corollary}[theorem]{Corollary}
\newtheorem{definition}[theorem]{Definition}
\newtheorem{hypothesis}[theorem]{Hypothesis}
\newtheorem{example}[theorem]{Example}

\theoremstyle{remark}
\newtheorem{remark}[theorem]{Remark}

\begin{document}

\numberwithin{equation}{section}
\allowdisplaybreaks

\title[The Callias Index Formula Revisited]{The Callias Index Formula Revisited}
  
\author[F.\ Gesztesy]{Fritz Gesztesy}  
\address{Department of Mathematics,
University of Missouri, Columbia, MO 65211, USA}
\email{\mailto{gesztesyf@missouri.edu}}
\urladdr{\url{https://www.math.missouri.edu/people/gesztesy}}

\author[M.\ Waurick]{Marcus Waurick}
\address{Institut f\"ur Analysis, Fachrichtung Mathematik, Technische Universit\"at Dresden, 
01062 Dresden, Germany}
\email{\mailto{marcus.waurick@tu-dresden.de}}
\urladdr{\url{http://www.math.tu-dresden.de/~waurick/}}


\date{\today}
\thanks{ To appear in {\it Springer Lecture Notes in Mathematics}.}

\begin{abstract} 
We revisit the Callias index formula in connection with supersymmetric Dirac-type operators 
$H$ of the form 
$$
H = \begin{pmatrix} 0 & L^* \\ L & 0 \end{pmatrix}
$$ 
in odd space dimensions $n$, originally derived in 1978, and prove that 
\begin{align}
 \ind (L) 
 &=\left(\frac{i}{8\pi}\right)^{(n-1)/2}\frac{1}{\left[(n-1)/2\right]!} 
 \lim_{\Lambda \to\infty}\frac{1}{2 \Lambda }\sum_{i_{1},\ldots,i_{n} = 1}^n \epsilon_{i_{1}\ldots i_{n}}      \lb{0.1} \\ 
 & \quad \; \times \int_{\Lambda S^{n-1}}\tr_{\bbC^d}(U(x)(\partial_{i_{1}}U)(x)\ldots 
 (\partial_{i_{n-1}}U)(x)) x_{i_{n}}\, d^{n-1} \sigma(x),    \no 
\end{align}
where
\[
U(x) := |\Phi(x)|^{-1} \Phi(x) = \sgn(\Phi(x)), \quad x \in \bbR^n.  
\]
Here the closed operator $L$ in $L^{2}(\mathbb{R}^{n})^{2^{\hat n}d}$ 
is of the form 
\[
L= \cQ + \Phi,
\]
where 
\[
\cQ\coloneqq Q \otimes I_d = \bigg(\sum_{j=1}^{n}\gamma_{j,n}\partial_{j}\bigg) I_d, 
\]
with $\gamma_{j,n}$, $j\in\{1,\ldots,n\}$, elements of the Euclidean Dirac algebra,   
such that $n=2{\hatt n}$ or $n=2{\hatt n}+1$. Here $\Phi$ is identified with $I \otimes \Phi$, 
satisfying  
\begin{align*} 
& \Phi\in C_{b}^{2}\big(\mathbb{R}^{n};\mathbb{C}^{d\times d}\big), \quad d \in \bbN,   \\
& \Phi(x)=\Phi(x)^{*},  \quad x\in\mathbb{R}^{n},     
\end{align*} 
there exists $c>0$, $R\geq0$ such that   
\[
|\Phi(x)|\geq c I_d, \quad x\in\mathbb{R}^{n}\backslash B(0,R),     
\]
and there exists $\varepsilon> 1/2$
such that for all $\alpha\in\mathbb{N}_{0}^{n}$, $\left|\alpha\right|<3$,
there is $\kappa>0$ such that  
\[
\|(\partial^{\alpha}\Phi)(x)\|\leq \begin{cases}
\kappa (1+|x|)^{-1}, & |\alpha|=1,\\[1mm] 
\kappa (1+ |x|)^{-1-\epsilon}, & |\alpha|=2, 
\end{cases}\quad x\in\mathbb{R}^{n}.        
\]
These conditions on $\Phi$ render $L$ a Fredholm operator, and to the best of our 
knowledge they represent the most general conditions known to date for which Callias' 
index formula \eqref{0.1} has been derived. 

We also consider a generalization of the index formula \eqref{0.1} to certain classes of 
non-Fredholm operators $L$ for which \eqref{0.1} represents its (generalized) Witten 
index (based on a resolvent regularization scheme). 
\end{abstract}

\maketitle

\newpage 

{\scriptsize{\tableofcontents}}

\newpage 

\section{Introduction}\label{sec:Introduction}

If pressed to  describe the contents of this manuscript in a nutshell, one could say we embarked on an  attempt to settle the Callias index formula, first presented by Callias \cite{Ca78} in 1978, with the help of functional analytic methods. While we tried at first to follow the path originally envisaged by Callias, we soon had to deviate sharply from his strategy of proof as we intended to derive his index formula under more general conditions on the potential $\Phi$ in the underlying closed operator $L$ (see 
\eqref{1.4}), but also since several of the claims made in \cite{Ca78} can be disproved.

Before describing the need to reconsider Callias' original arguments, and before entering a brief  discussion of new developments in the field since 1978, it may be best to set the stage for the remarkable index formula that now carries his name.  

For a given spatial dimension $n\in\mathbb{N}$, we denote the elements of the 
\emph{Euclidean Dirac algebra} (cf.\ Appendix \ref{sec:Appendix:-the-Construction} for precise details) by $\gamma_{j,n}$, $j\in\{1,\ldots,n\}$. One recalls in this context that for $n=2 \hat n$
or $n=2 \hat n +1$ for some $\hat n \in\mathbb{N}$, $\gamma_{j,n}$ satisfy  
\begin{equation}
\gamma_{j,n}^{*}=\gamma_{j,n}\in\mathbb{C}^{2^{\hat n}\times2^{\hat n}},   \quad 
\gamma_{j,n}\gamma_{k,n}+\gamma_{k,n}\gamma_{j,n}=2\delta_{jk}I_{2^{\hatt n}}, \quad j,k\in\{1,\ldots,n\}.     \lb{1.1} 
\end{equation}
With the elements $\gamma_{j,n}$ in place, one then introduces the constant coefficient, first-order differential operator $Q$ in $L^{2}(\mathbb{R}^{n})^{2^{\hat n}}$ by 
\begin{equation}
Q\coloneqq\sum_{j=1}^{n}\gamma_{j,n}\partial_{j}, \quad 
\dom(Q) = H^{1}(\mathbb{R}^{n})^{2^{\hat n}},      \label{1.2}
\end{equation}
with $H^m(\bbR^n)$, $m\in\bbN$, the standard Sobolev spaces. One notes in passing that 
\begin{equation}
Q^{2}=\Delta I_{2^{\hatt n}}, \quad \dom(Q^2) = H^{2}(\mathbb{R}^{n})^{2^{\hat n}}.  \lb{1.3}      
\end{equation}
Next, let $d \in\mathbb{N}$ and assume that $\Phi\colon\mathbb{R}^{n}\to\mathbb{C}^{d\times d}$ 
is a $d \times d$ self-adjoint matrix with entries given by bounded measurable functions. We introduce the operator $L$ in $L^{2}(\mathbb{R}^{n})^{2^{\hat n}d}$ via 
\begin{align}
L\colon \begin{cases} H^{1}(\mathbb{R}^{n})^{2^{\hat n}d} \subseteq 
L^{2}(\mathbb{R}^{n})^{2^{\hat n}d} \to 
L^{2}(\mathbb{R}^{n})^{2^{\hat n}d},   \\
\psi \otimes \phi \mapsto\left(\sum_{j=1}^{n}\gamma_{j,n}\partial_{j}\psi\right)
\otimes \phi + \left(x\mapsto\psi(x)\otimes \Phi(x)\phi\right). 
\end{cases}            \label{1.4}
\end{align}
Given \eqref{1.2}, we shall abreviate
\begin{equation}
\cQ\coloneqq Q \otimes I_d = \bigg(\sum_{j=1}^{n}\gamma_{j,n}\partial_{j}\bigg) I_d,  \label{1.5}
\end{equation}
and, with a slight abuse of notation, employ the symbol $\Phi$ also in the context of the operation 
\begin{equation}
\Phi \colon \psi\otimes \phi \mapsto\left(x\mapsto\psi(x)
\otimes \Phi(x)\phi\right),    \lb{1.6}
\end{equation} 
(see our notational conventions to suppress tensor products whenever possible, collected in 
Section \ref{s2} and in Remark \ref{rem:differen_mult_op}). Thus, we may write,    
\begin{equation} 
L= \cQ + \Phi.      \lb{1.7}
\end{equation}   
The associated $($self-adjoint\,$)$ supersymmetric Dirac-type operator $H$ in 
$L^{2}(\mathbb{R}^{n})^{2^{\hat n}d} \oplus L^{2}(\mathbb{R}^{n})^{2^{\hat n}d}$ is then of the form 
\begin{equation} 
H = \begin{pmatrix} 0 & L^* \\ L & 0 \end{pmatrix}. 
\end{equation}
We refer to \cite[Ch.~5]{Th92} for a detailed discussion of supersymmetric Dirac-type operators 
and the many explicit examples they represent. 

Next, we strengthen the hypotheses on $\Phi$ to the effect that 
\begin{align} 
& \Phi\in C_{b}^{2}\big(\mathbb{R}^{n};\mathbb{C}^{d\times d}\big), \quad d \in \bbN,  \lb{1.8} \\
& \Phi(x)=\Phi(x)^{*},  \quad x\in\mathbb{R}^{n},     \lb{1.9} 
\end{align} 
there exists $c>0$, $R\geq0$ such that   
\begin{equation} 
\left|\Phi(x)\right|\geq c I_d, \quad x\in\mathbb{R}^{n}\backslash B(0,R),     \lb{1.10} 
\end{equation} 
and there exists $\varepsilon> 1/2$
such that for all $\alpha\in\mathbb{N}_{0}^{n}$, $\left|\alpha\right|<3$,
there is $\kappa>0$ such that  
\begin{equation} 
\|(\partial^{\alpha}\Phi)(x)\|\leq \begin{cases}
\kappa (1+|x|)^{-1}, & |\alpha|=1,\\[1mm] 
\kappa (1+ |x|)^{-1-\epsilon}, & |\alpha|=2, 
\end{cases}\quad x\in\mathbb{R}^{n}.        \lb{1.11} 
\end{equation} 

\begin{theorem} \lb{t1.1} 
Let $n\in\mathbb{N}$ odd, $n \geq 3$. 
Under assumptions \eqref{1.8}--\eqref{1.11} on $\Phi$, the closed operator 
$L\coloneqq \cQ +\Phi$ in $L^{2}(\mathbb{R}^{n})^{2^{\hat n}d}$
is Fredholm with index given by the formula 
\begin{align}
 \ind (L) 
 &=\left(\frac{i}{8\pi}\right)^{(n-1)/2}\frac{1}{[(n-1)/2]!} 
 \lim_{\Lambda \to\infty}\frac{1}{2 \Lambda }\sum_{i_{1},\ldots,i_{n} = 1}^n \epsilon_{i_{1}\ldots i_{n}}     \lb{1.12} \\ 
 & \quad \; \times \int_{\Lambda S^{n-1}}\tr_{\bbC^d}(U(x)(\partial_{i_{1}}U)(x)\ldots 
 (\partial_{i_{n-1}}U)(x)) x_{i_{n}}\, d^{n-1} \sigma(x),    \no 
\end{align}
where
\[
U(x) := |\Phi(x)|^{-1} \Phi(x) = \sgn(\Phi(x)), \quad x \in \bbR^n. 
\]
\end{theorem}

Here $\epsilon_{i_1\cdots i_n}$ denotes the totally anti-symmetric symbol in $n$ coordinates, 
$\tr_{\bbC^d}(\cdot)$ represents the matrix trace in $\bbC^{d \times d}$, $d^{n-1} \sigma(\cdot)$ is the surface measure on the unit sphere $S^{n-1}$ of $\bbR^n$, and we assumed $n\in\mathbb{N}$ to be odd since for algebraic reasons $L$ has vanishing Fredholm index in all even spatial dimensions $n$ 
(cf.\ \eqref{1.19} below). 

Theorem \ref{t1.1} represents the principal result of this manuscript and under these hypotheses 
on $\Phi$ it is new as we suppose no additional asymptotic homogeneity properties on $\Phi$. 
In particular, it extends the original Callias formula for the index of $L$ to the hypotheses \eqref{1.8}--\eqref{1.11} on $\Phi$. We also note that at the end of this manuscript we take some first steps towards computing the Witten index of the operator $L$ under certain conditions on $\Phi$ in which $L$ ceases to be Fredholm, yet its Witten index is still given by a formula analogous to \eqref{1.12}. 

For the topological setting underlying the Callias index formula \eqref{1.12} we refer to the 
discussion by Bott and Seeley \cite{BS78}.  

Next, we succinctly summarize the principal strategy of proof underlying formula \eqref{1.12}. While at first we follow Callias' original strategy of proof, the bulk of our arguments necessarily differ sharply from those in \cite{Ca78} as some of the claims in \cite{Ca78} can clearly be disproved (see our subsequent discussion).  \\[1mm] 
{\bf  Step\,$\boldsymbol{(1)}$:~Computing Fredholm indices abstractly.} ~Let $\cH$ be a separable Hilbert space, $m\in\mathbb{N}$, and 
$T\in \cB\left(\cH^{m},\cH^{m}\right)$. Define the \emph{internal trace}, $\tr_{m} (T) $, \emph{of }$T$ by 
\begin{equation}
\tr_{m} (T) \coloneqq\sum_{j=1}^{m}T_{jj}.       \lb{1.13}
\end{equation}
Next, let $M$ be a densely defined, closed linear operator in $\cH^{m}$, and introduce the 
abbreviation 
\begin{align}
\begin{split} 
B_{M}(z)\coloneqq z\tr_{m}\big((M^{*}M+z)^{-1}-(MM^{*}+z)^{-1}\big),  
\quad z\in\rho(-M^{*}M)\cap\rho(-MM^{*}).   \lb{1.14}
\end{split} 
\end{align}
 
A basic result we employ to compute Fredholm indices then reads as follows:
 
\begin{theorem} \lb{t1.2} 
Assume that $M$ is a densely defined, closed, and linear operator in $\cH^{m}$, and suppose that 
$M$ is Fredholm. In addition, let $\{T_\Lambda\}_{\Lambda\in \mathbb{N}}$, $\{S^*_\Lambda\}_{\Lambda\in \mathbb N}$ be sequences in $\cB (\cH)$, both strongly converging to $I_\cH$ as 
$\Lambda \to \infty$, and introduce 
$S_\Lambda\coloneqq S_\Lambda^{**}$, $\Lambda\in \mathbb{N}$. Assume that for each 
$\Lambda\in \mathbb{N}$, there exists $\delta_\Lambda >0$ with 
$\Omega_\Lambda\coloneqq B(0,\delta_\Lambda)\backslash \{0\}\subseteq \rho(-MM^{*})\cap\rho(-M^{*}M)$ and that the map   
\begin{equation} 
   \Omega_\Lambda \ni z\mapsto T_\Lambda B_M(z)S_\Lambda    \lb{1.15} 
\end{equation} 
takes on values in $\mathcal{B}_1(\mathcal{H})$, such that 
\begin{equation} 
\Omega_{\Lambda} \ni z\mapsto \tr_{\mathcal{H}} (|T_\Lambda B_M(z)S_\Lambda|) 
= \| T_\Lambda B_M(z)S_\Lambda \|_{\cB_1(\cH)}\, \text{ is bounded $($w.r.t.~$z$$)$,}  \lb{1.16} 
\end{equation} 
where 
$\tr_{\cH}(\cdot)$ represents the trace on $\cB_1(\cH)$, the Schatten-von Neumann ideal of trace class operators on $\cH$. Then, 
\begin{equation}
\ind (M) =\lim_{\Lambda\to\infty}\lim_{z\to0} \tr_{\mathcal{H}} (T_\Lambda B_M(z)S_\Lambda).   \lb{1.17}
\end{equation} 
In addition, if $\delta\coloneqq 2^{-1} \inf_{\Lambda \in\mathbb{N}} (\delta_\Lambda ) > 0$ and 
$\Omega\coloneqq B(0,\delta)\ni z\mapsto \tr_\mathcal{H}(T_\Lambda B_M(z)S_\Lambda )$ converges uniformly on $\overline{B(0,\delta)}$ to some function $F(\cdot)$ as $\Lambda  \to \infty$. Then, one can interchange the limits $\Lambda \to \infty$ and $z \to 0$ in \eqref{1.17} and obtains, 
\begin{equation} 
F(0)=\ind (M).   \lb{1.18}
\end{equation} 
\end{theorem}

We emphasize that \eqref{1.17} and \eqref{1.18} represent a subtle, but crucial, deviation from the far simpler strategy employed in \cite[Lemma~1]{Ca78} which entirely dispenses with the additional regularization factors $S_{\Lambda}$ and $T_{\Lambda}$, $\Lambda \in \bbN$. At this point we do not know if \cite[Lemma~1]{Ca78} is valid, however, its proof is clearly invalid and we record a counterexample (kindly communicated to us by H.\ Vogt \cite{Vo14}) to the statement made on line 5 on p.~219 in the proof of \cite[Lemma~1]{Ca78} later in Remark \ref{r2.4}\,$(i)$. After completing this project we became aware of an unpublished preprint by Arai \cite{Ar90} in which it was observed that the index regularization employed in \cite{Ca78} was insufficient. \\[1mm] 
{\bf  Step\,$\boldsymbol{(2)}$:~Applying Step $\boldsymbol{(1)}$ to the 
operator $L$.}~One now identifies $\cH$ and $L^{2}(\mathbb{R}^{n})$, $m$ and $2^{\hat n}d$, 
$M$ and $L$, $T_\Lambda$ and the operator of multiplication by the characteristic function 
of the ball $B(0,\Lambda) \subset \bbR^n$ in $L^2(\bbR^n)$, denoted by $\chi_\Lambda$, 
and chooses $S_\Lambda^*=I_{L^2(\mathbb{R}^n)}$, $\Lambda \in \bbN$.

According to \eqref{1.17} and especially, \eqref{1.18}, we are thus interested in computing the limit for $\Lambda\to\infty$ of $\tr(\chi_\Lambda B_{L}(z))$. Without loss of generality we restrict ourselves in the following to $n \in \bbN$ odd, as a detailed analysis shows that actually
\begin{equation}
B_L(z) = 0 \, \text{ for $n \in \bbN$, $n$ even.}   \lb{1.19} 
\end{equation} 
For $z\in\rho\left(-LL^{*}\right)\cap\rho\left(-L^{*}L\right)$ with $\Re (z) > - 1$, and 
$n\in\mathbb{N}$ odd, $n \geq 3$, one then proceeds to prove that  
$\chi_\Lambda B_{L}(z) \in \cB_1\big(L^2(\bbR^n)\big)$, and that the limit 
\[
f(z)\coloneqq \lim_{\Lambda\to\infty} \tr_{L^2(\bbR^n)} (\chi_\Lambda B_L(z)) 
\]
exists.  \\[1mm] 
{\bf  Step\,$\boldsymbol{(3)}$:~Explicitly compute $f(z)$.}~A careful (and rather lengthy) evaluation of $f(z)$ yields
\begin{align}
f(z) & = (1+z)^{-n/2}\left(\frac{i}{8\pi}\right)^{(n-1)/2}\frac{1}{[(n-1)/2]!} \lim_{\Lambda \to\infty}\frac{1}{2 \Lambda}\sum_{i_{1},\ldots,i_{n} = 1}^n \epsilon_{i_{1}\ldots i_{n}}    \no \\
 & \quad \; \times \int_{\Lambda S^{n-1}}\tr_{\bbC^d}(U(x) (\partial_{i_{1}}U)(x)\ldots 
 (\partial_{i_{n-1}}U)(x)) x_{i_{n}}\, d^{n-1} \sigma(x),    \lb{1.20} \\
& \hspace*{3.5cm} z\in\rho\left(-LL^{*}\right)\cap\rho\left(-L^{*}L\right), \; \Re (z) > - 1. \no
\end{align}

However, at first we are only able to verify \eqref{1.20} for $\Re(z)$ sufficienly large (as a consequence of relying on Neumann series expansions for resolvents). In order to derive 
\eqref{1.20} also for $z$ in a neighborhood of $0$, considerable additional efforts are required. 

Indeed, for achieving the existence of the limit $\Lambda \to \infty$ in \eqref{1.20} for $z$ in a neighborhood of $0$, we employ Montel's theorem and hence need to show that the family of analytic functions 
$\{ z \mapsto \tr (\chi_\Lambda B_L(z))\}_{\Lambda}$ 
constitutes a locally bounded family, that is, one needs to show that for all compact $\Omega\subset \mathbb{C}_{\Re>-1}\cap \rho(-L^*L)\cap \rho(-LL^*)$, 
\[
  \sup_{\Lambda>0}\sup_{z\in \Omega} |\tr (\chi_\Lambda B_L(z))|<\infty.
\]
After proving local boundedness, we use Montel's theorem for deducing that at least for a sequence $\{\Lambda_k\}_{k\in\bbN}$ with $\Lambda_k \underset{k\to\infty}{\longrightarrow} \infty$, the limit $f\coloneqq \lim_{k\to\infty}\tr(\chi_{\Lambda_k} B_L(\cdot))$ exists in the compact open topology 
(i.e., the topology of uniform convergence on compacts). The explicit expression \eqref{1.20} for $f$ then follows by the principle of analytic continuation and so carries over to $z$ in a neighborhood of $0$. In particular, since the limit 
$\lim_{\Lambda\to\infty}\tr(\chi_{\Lambda} B_L(0))$ exists and coincides with the index of $L$, we can then deduce that independently of the sequence $\{\Lambda_k\}_{k\in\bbN}$, the limit $\lim_{\Lambda\to\infty}\tr(\chi_{\Lambda} B_L(\cdot))$ exists in the compact open topology and coincides with $f$ given in \eqref{1.20}. Thus,
\[
f(0) = \ind(L)
\] 
yields formula \eqref{1.12}.

We also emphasize that in connection with  Steps $(1)$--$(3)$, we perform these calculations only in the special case of admissible or $\tau$-admissible potentials $\Phi$ (cf.\ Definitions  \ref{def:phi_admissible} and \ref{d:ta}) and then reduce the general case to $\tau$-admissible potentials.

It is clear from this short outline of our strategy of proof of Callias' index formula \eqref{1.12}, 
that in the end, our proof requires a fair number of additional steps not present in \cite{Ca78}.

Without entering any details at this point, we mention that one needs to distinguish the case $n=3$ from $n \geq 5$ as there are additional regularization steps necessary for $n=3$ due to the lack of regularity of certain integral kernels. In this context we mention that it is unclear to us how
continuity of the integral kernel of $J_{z}^{i}$ on the diagonal, as claimed in 
\cite[p. 224, line 6 from below]{Ca78}, can be proved. Given our detailed approach, the number of  resolvents applied is simply not large enough to conclude continuity (see, in particular, Section \ref{sec:The-Derivation-of-trace-f}).

Perhaps, more drastically, trace class properties of certain integral operators are merely dealt with by checking integrability of the integral kernel on the diagonal, see, for instance, the proof of 
\cite[Lemma~5, p.~225]{Ca78}. 

In addition, the claim that the expression 
\begin{equation}\label{eq:wrong_cancellation}
   \sum_{i_1,\ldots,i_n} \epsilon_{i_1 \ldots i_n} 
   \tr\left((\partial_{i_1}\Phi)(x)\ldots(\partial_{i_n}\Phi)(x)\right)=0,\quad x\in \mathbb{R}^n,  
\end{equation}
vanishes identically, is made on \cite[p.~226]{Ca78}. A simple counter example can 
(locally) be constructed by demanding that $\Phi\colon \mathbb{R}^3\to \mathbb{C}^{2\times 2}$ 
is bounded, $\Phi \in C^\infty\big(\bbR^3; \bbC^{2 \times 2}\big)$, and such that for one particular $x_0 \in \bbR^3 $, 
\[
  (\partial_1 \Phi)(x_0)=\begin{pmatrix}
                        1 & 2 \\ 2 & 1
                     \end{pmatrix},\quad  (\partial_2 \Phi)(x_0)=\begin{pmatrix}
                        1 & 2 \\ 2 & -1
                     \end{pmatrix},\quad  (\partial_1 \Phi)(x_0)=\begin{pmatrix}
                        0 & i \\ -i & 0
                     \end{pmatrix}.
\]
In this case one verifies that 
\[
\sum_{i_1,i_2,i_3} \epsilon_{i_1 i_2 i_3} \tr\left((\partial_{i_1}\Phi)(x_0) 
   (\partial_{i_2}\Phi)(x_0) (\partial_{i_3}\Phi)(x_0)\right)=24i.
\]   

These shortcomings in the arguments presented in \cite{Ca78} not withstanding, Callias' formula 
\eqref{1.12} is remarkable for its simplicity, as has been pointed out before by various authors. 
In particular, it is simpler, yet consistent with the Fedosov--H\"ormander formula \cite{Fe70}, \cite{Fe70a}, \cite{Fe74}, \cite{Fe75}, \cite{Ho79}, \cite[Sect.~19.3]{Ho85} (derived with the help of the pseudo-differential operator calculus), as discussed, for instance, in \cite{An89}, \cite{BS78}, \cite{Sc92}. More precisely, the Fedosov--H\"ormander formula reads as follows, 
\begin{equation}\label{e:fh}
  \ind (L) = -\bigg(\frac{i}{2\pi}\bigg)^{n}\frac{(n-1)!}{(2n-1)!}\int_{\partial B} \tr\Big(\big(\sigma_L^{-1}\textrm{d}\sigma_L\big)^{\wedge (2n-1)}\Big).  
\end{equation}
Here $\sigma_L \colon \mathbb{R}^n\times \mathbb{R}^n\to \mathbb{C}^{2^{\hatt n}d\times 2^{\hatt n}d}$ is the symbol of $L$ given by
\[
   \sigma_L (\xi,x)= \sum_{j=1}^n \gamma_{j,n}i\xi\otimes I_{2^{\hatt n}} +I_d \otimes \Phi(x),\quad \xi,x\in \mathbb{R}^n,
\]
$B\subseteq \mathbb{R}^{2n}$ is a ball of sufficiently large radius centered at the origin such that 
$\sigma_L$ is invertible outside $B$, the orientation of $\bbR^n \times \bbR^n$ is given by 
$dx_1 \wedge d\xi_1 \wedge \cdots \wedge dx_n \wedge d\xi_n > 0$, and 
$\big(\sigma_L^{-1}\textrm{d}\sigma_L\big)^{\wedge (2n-1)}$ 
is evaluated as a matrix product upon replacing ordinary multiplication by the exterior product. 

The Callias index formula properly restated as the 
Fedosov--H\"ormander formula and connections with \emph{half-bounded states} 
were also discussed in \cite{Ca02}. Moreover, with the help of the Cordes--Illner theory (see \cite{Co72,Il77} and the references in \cite{Ra07}), \cite{Ra07} established that the 
Fedosov--H\"ormander formula can also be used for computing the index, if $L$ is considered as an operator from the Sobolev space $W^{1,p}(\mathbb{R}^n)^{2^{\hatt n}d}$ to $L^{p}(\mathbb{R}^n)^{2^{\hatt n}d}$ for some $p\in (1,\infty)$. In addition, \cite{Ra08} (see also \cite{Ra12}) established the validity of the Fedosov--H\"ormander formula assuming the low regularity $\Phi \in C^1$ only (plus vanishing of derivatives at infinity). 

Callias employed Witten's resolvent regularization inherent in \eqref{1.14}, \eqref{1.17}, 
\eqref{1.18}, and we followed this device in this manuscript. For extensions to higher powers of resolvents we refer to \cite{St89}. For connections between supersymmetric quantum mechanics, 
scattering theory and their connections with Witten's resolvent regularized index for Dirac-type operators in various space dimensions, and matrix-valued (resp., operator-valued) coeffcients, we refer, for instance, to \cite{An93a}, \cite{Ar87}, \cite{BGGSS87}, \cite{BMS88}, \cite{Bu92}, \cite{CGPST14}, \cite{CGPST14a}, \cite{Ca04}, \cite{Ma87}, \cite{Ma90}, 
\cite[Chs.~IX, X]{Mu87}, \cite{Mu88}.   

The index problem for Dirac operators defined on complete Riemanniann manifolds has also been studied in \cite{GL83} on the basis of relative index theorems (see also \cite{Ro91}). Based on this approach, \cite{An90} found a generalized version of the Callias index formula, which was further developed and connected with the Atiyah--Singer index theorem in \cite{An93} (see also \cite{Ca01}, \cite{Ca01a}, \cite{ENN96}, \cite{Hi91} in this context). Independently, \cite{Ra94} found an alternative proof for the main result in \cite{An93}, reducing the index problem for the Dirac operator on a non-compact manifold to the compact case, thus making the index theorem in \cite{AB64} applicable. Generalizing results in \cite{An90}, and also using the Atiyah--Singer index theorem, \cite{BM92} (see also \cite{Br92}) derive index formulas on manifolds, containing the Callias index formula as special case.

For further generalizations of the index theorem for the Dirac operator to particular manifolds, we refer to \cite{FH95}. In addition, certain classes of Dirac operators on even-dimensional manifolds are studied in \cite{FH97}, \cite{FGH01}, \cite{FH03}, \cite{FH05} employing $K$ or $KK$-theory. The utility of $KK$-theory in view of the Callias index formula can also be seen in \cite{Ku01}, where a short proof for the main results in \cite{An90} is given. Additional connections between $K$-theory 
and index
 theory for Dirac-type operators have been established, for instance, in \cite{Bu95}, \cite{Ca05}, \cite{CN14}, \cite{Ko11}, \cite{Ko15}. A rather different direction of index theory employing cyclic homology, aimed at even dimensional Dirac-type operators which generally are non-Fredholm, was 
undertaken in \cite{CGK15} (see also \cite{CK14}).
 
The approach to calculating Fredholm indices initiated by Callias \cite{Ca78} also had a profound influence on theoretical physics as is amply demonstrated by the following references \cite{BB84}, 
\cite{CK01}, \cite{EGH80}, \cite{FOW87}, \cite{Hi83}, \cite{Hi86}, \cite{HN89}, \cite{HS87}, \cite{HT82}, \cite{IM84}, \cite{NS86}, \cite{NS86a}, \cite{We79}, \cite{We81}, \cite{WY07}, \cite{Wi14}, and the literature cited therein.   

Returning to Theorem \ref{t1.1}, we emphasize again that our derivation of the Callias index 
formula \eqref{1.12} under conditions \eqref{1.8}--\eqref{1.11} on $\Phi$ is new as the references just mentioned either do not derive an explicit formula for $\ind(L)$ in terms of $\Phi$, or else, 
derive the Fedosov--H\"ormander formula for $\ind(L)$. All previous derivations of \eqref{1.12} made some assumptions on $\Phi$ to the effect that asymptotically, $\Phi$ had to be homogeneous of degree zero. We entirely dispensed with this condition in this manuscript. 

We conclude this introduction with a brief description of the contents of each section. Our notational conventions are summarized in Section \ref{s2}. Section \ref{sec:Functional-Analytic-Preliminarie}
is devoted to computing Fredholm indices employing Witten's resolvent regularization. 
Schatten--von Neumann classes and trace class estimates are treated in Section \ref{sec:Trace-Class-Estimates}. Pointwise bounds for integral kernels are developed in Section 
\ref{sec:ptw_intk}. The operator $L$ underlying this manuscript is presented in Section 
\ref{sec:A-particular-First}. Trace class results, fundamental for deriving formula \eqref{1.20} 
for $f(z)$, are established in Section \ref{sec:The-Derivation-of-trace-f}; estimates for integral 
kernels on the diagonal and the computation of the trace of $\chi_\Lambda B_L(z)$ are discussed 
in Section \ref{sec:der_tra_diag}; the special case $n=3$ is treated in Section \ref{sec:n=3}. In Section \ref{sec:The Index theorem} we formulate our principal result, Theorem \ref{thm:Perturbation} (equivalently, Theorem \ref{t1.1}) and discuss some consequences of formula \eqref{1.12}. Perturbation theory of Helmholtz resolvents (Green's functions) is isolated in 
Section \ref{sec:pert}. The proof of Theorem \ref{thm:Perturbation} for smooth $\Phi$ is presented 
in Section \ref{s12}; the case of general $\Phi$ satisfying \eqref{1.8}--\eqref{1.11} is concluded 
in Section \ref{s13}. The final Section \ref{s14} is devoted to a particular class of non-Fredholm 
operators $L$ and the associated Witten index. Appendix 
\ref{sec:Appendix:-the-Construction} presents a concise construction of the Euclidean Dirac algebra, and Appendix \ref{sB} constructs an explicit counterexample to the trace class assertion in 
\cite[Lemma~5]{Ca78}.

\medskip

\noindent 
{\bf  Acknowledgments.} 
We thank Bernhelm Boo{\ss}-Bavnbek, Alan Carey, Stamatis Dostoglou, Yuri Latushkin, Marius Mitrea, and Andreas Wipf for valuable discussions and/or correspondence. We are especially grateful to Alan Carey for extensive correspondence and hints to the literature, and to Stamatis Dostoglou for a critical reading of a substantial part of our manuscript and for many constructive comments that helped improving our exposition. We also thank Ralph Chill, Martin K\"orber, Rainer Picard, and Sascha Trostorff for 
useful discussions, particularly, concerning the counterexample discussed in Appendix \ref{sB}. In addition, we are indebted to Hendrik Vogt for kindly communicating the example in Remark \ref{r2.4}. 

M.W.\ gratefully acknowledges the hospitality of the Department of Mathematics at the University of Missouri, USA, extended to him during two one-month visits in the spring of 2014 and 2015. 
Moreover, he is particularly indebted to Yuri Tomilov in connection with the support received from the EU grant ``AOS'', FP7-PEOPLE-2012-IRSES, No.~318910. 

\newpage 

\section{Notational Conventions} \lb{s2}

For convenience of the reader we now summarize most of our notational conventions used 
throughout this manuscript. 

We find it convenient to employ the abbreviations, 
$\bbN_{\geq k} := \bbN \cap [k,\infty)$, $k \in \bbN$, $\bbN_0 = \bbN \cup \{0\}$,  
$\mathbb{C}_{\Re> a}\coloneqq \{z\in \mathbb{C} \, | \, \Re (z) > a\}$, $a \in \bbR$.

The identity matrix in $\bbC^r$ will be denoted by $I_r$, $r \in \bbN$. 

Let $\cH$ be a separable complex Hilbert space, 
$(\cdot \, ,\cdot)_{\cH}$ the scalar product in $\cH$ 
(linear in the second factor), and $I_{\cH}$ the identity operator in $\cH$.

Next, let $T$ be a linear operator mapping (a subspace of) a
Banach space into another, with $\dom(T)$, $\ker(T)$, and $\ran(T)$ denoting the
domain, kernel (i.e., null space), and range of $T$. The spectrum and resolvent set of a 
closed linear operator in $\cH$ will be denoted by $\sigma(\cdot)$ and $\rho(\cdot)$. 
For resolvents of closed operators $T$ acting on $\dom(T) \subseteq \cH$, we will 
frequently write $(T - z)^{-1}$ rather than the precise $(T -z I_{\cH})^{-1}$, $z \in\rho(T)$. 

The Banach spaces of bounded and compact linear 
operators in $\cH$ are denoted by $\cB(\cH)$ and $\cB_\infty(\cH)$, 
respectively. The Schatten--von Neumann ideals of compact linear operators
on $\cH$ corresponding to $\ell^p$-summable singular
values will be denoted by $\cB_p(\cH)$ or, if the Hilbert space under consideration
is clear from the context (and, especially, for brevity in connection with proofs) just 
by $\cB_p$, $p\in[1,\infty)$.
The norms on the respective spaces will be noted by $\|T\|_{\cB_p(\cH)}$
for $T\in\cB_p(\cH)$, $p\in[1,\infty)$, and for ease of notation we will occasionally 
identify $\|T\|_{\cB(\cH)}$ with $\|T\|_{\cB_{\infty}(\cH)}$ for $T \in \cB(\cH)$, but caution the 
reader that it is the set of compact operators on $\cH$ that is denoted by $\cB_{\infty}(\cH)$. 
Similarly, $\cB(\cH_1,\cH_2)$ and $\cB_\infty (\cH_1,\cH_2)$ 
will be used for bounded and compact operators between two Hilbert spaces 
$\cH_1$ and $\cH_2$. Moreover, $\cX_1 \hookrightarrow \cX_2$ denotes the 
continuous embedding of the Banach space $\cX_1$ into the Banach space 
$\cX_2$. Throughout this manuscript, if $\cX$ denotes a Banach space, $\cX^*$ 
denotes the {\it adjoint space} of continuous conjugate linear functionals 
on $\cX$, that is, the {\it conjugate dual space} of $\cX$ (rather than the usual 
dual space of continuous linear functionals 
on $\cX$). This avoids the well-known awkward distinction between adjoint 
operators in Banach and Hilbert spaces (cf., e.g., the pertinent discussion 
in \cite[p.\,3--4]{EE89}). In connection with bounded linear functionals on 
$\cX$ we will employ the usual bracket notation 
$\langle \cdot \, , \cdot \rangle_{\cX^*,\cX}$ for pairings between elements of 
$\cX^*$ and $\cX$. 

Whenever estimating the operator norm or a particular trace ideal norm of a finite 
product of operators, $A_1 A_2 \cdots A_N$, with $A_j \in \cB(\cH)$, $j\in \{1,\dots,N\}$, 
$N \in \bbN$, we will simplify notation and write 
\begin{equation}
\prod_{j=1}^N A_j,    \lb{2.prod} 
\end{equation}
disregarding any noncommutativity issues of the operators $A_j$, $j\in\{1,\dots,N\}$. This 
is of course permitted due to standard ideal properties and the associated (noncommutative) 
H\"older-type inequalities (see, e.g., \cite[Sect.~III.7]{GK69}, \cite[Ch.~2]{Si05}). The same 
convention will be applied if operators mapping between several Hilbert spaces are involved.  

We use the \emph{commutator} symbol 
\begin{equation}
[A,B]\coloneqq AB-BA\label{eq:def_commutator}
\end{equation}
for suitable operators $A,B$. For unbounded $A$ and $B$ the natural
domain of $[A,B]$ is the intersection of the respective domains of
$AB$ and $BA$. In particular, $[A,B]$ is not closed in general.
However, in the situations we are confronted with, we shall always be in 
the situation that $[A,B]$ is densely defined and bounded, in particular, 
it is closable with bounded closure. As this
is always the case, we shall -- in order to reduce a clumsy notation
as much as possible -- typically omit the closure bar (i.e., we use $[A,B]$ rather than 
$\ol{[A,B]}$). In fact, most of
the operators under consideration can be extended to suitable distribution
spaces, such that seemingly formal computations can be justified in
the appropriate distribution space.
 
$\wlim$ and $\slim$ denote weak and strong limits in $\cH$ as well as limits in the 
weak and strong operator topology for operators in $\cB(\cH)$, $\nlim$ denotes the 
norm limit of bounded operators operators in $\cH$ (i.e., in the topology of $\cB(\cH)$).

$C_0^{\infty}(\mathbb{R}^{n})$ denotes the space of infinitely often 
differentiable functions with compact support in $\bbR^n$. We typically suppress 
the Lebesgue meausure in $L^p$-spaces, $L^p(\bbR^n) := L^p(\bbR^n; d^n x)$, $\| \cdot\|_{L^p(\bbR^n; d^n x)} := \|\cdot\|_{p}$, and similarly, $L^p(\Omega) := L^p(\Omega; d^n x)$, 
$\Omega \subseteq \bbR^n$, $p \in [1,\infty) \cup \{\infty\}$. To avoid too lengthy expressions, 
we will frequently just write $I$ rather than the precise $I_{L^2(\bbR^n)}$, etc.

Sometimes we use the symbol $\langle \cdot \, , \cdot \rangle_{L^2(\bbR^n)}$ 
(or, for brevity, especially in proofs, simply $\langle \cdot \, , \cdot \rangle$), to indicate the 
fact that the scalar product $(\cdot \, , \cdot)_{L^2(\bbR^n)}$ in $L^2(\bbR^n)$ has been 
continuously extended to the pairing on the entire Sobolev scale, that is, we abreviate 
$\langle \cdot \, , \cdot \rangle_{L^2(\bbR^n)} 
:= \langle \cdot \, , \cdot \rangle_{H^{-s}(\bbR^n), H^{s}(\bbR^n)}$, $s \geq 0$. 

The unit sphere in $\bbR^n$ is denoted by 
$S^{n-1}\coloneqq \{ x\in\mathbb{R}^{n} \, | \, |x|=1\}$, with $d^{n-1} \sigma(\cdot)$ 
representing the surface measure on $S^{n-1}$, $n \in \bbN_{\geq 2}$. The open ball in 
$\bbR^n$ centered at $x_0 \in \bbR^n$ of radius $r_0 >0$ is denoted by $B(x_0,r_0)$. 

Since various matrix structures and tensor products are naturally associated with the Dirac-type operators studied in this manuscript, we had to simplify the notation in several respects to avoid entirely unmanagably long expressions. For example, given $d, \hat n \in \bbN$, spaces such as 
$L^{2}(\bbR^n) \otimes \bbC^d$, $L^{2}(\bbR^n) \otimes \bbC^{2^{\hat n}}$, and 
$L^{2}(\bbR^n) \otimes \bbC^{2^{\hat n}} \otimes \bbC^d$ (and analogously for Sobolev spaces) 
will simply be denoted by $L^{2}(\bbR^n)^d$, $L^{2}(\bbR^n)^{2^{\hat n}}$, and 
$L^{2}(\bbR^n)^{2^{\hat n}d}$, respectively. 

In addition, given a $d \times d$ matrix 
$\Phi\colon\mathbb{R}^{n}\to\mathbb{C}^{d\times d}$ with entries given by bounded measurable functions, and given an element $ \psi\otimes \phi \in L^{2}(\bbR^n)^{2^{\hat n}} \otimes \bbC^d$, we will frequently adhere to a slight abuse of notation and employ the symbol $\Phi$ also in the 
context of the operation 
\begin{equation}
\Phi \colon \psi\otimes \phi \mapsto\left(x\mapsto\psi(x)
\otimes \Phi(x)\phi\right),    \lb{Phi1}
\end{equation} 
and acordingly then shorten this even further to 
\begin{equation}
\Phi \colon \psi \, \phi \mapsto\left(x\mapsto\psi(x) \Phi(x)\phi\right),    \lb{Phi2}
\end{equation} 
Moreover, in connection with constant, invertible $m \times m$ matrices 
$\alpha \in \bbC^{m \times m}$ 
and scalar differential expressions such as $\partial_j$, $\Delta$, etc., we will use the notation 
\begin{equation}
\alpha \partial_j = \partial_j \alpha, \quad \alpha \Delta = \Delta \alpha    \lb{alpha}
\end{equation} 
(with equality of domains) when applying these differential expressions to sufficiently regular functions of the type 
$\eta(\cdot) \otimes c$, $c \in \bbC^m$, abbreviated again by $\eta(\cdot) \, c$. 

In the context of matrix-valued operators we also agree to use the following notational conventions: Given a scalar function $f$ on $\bbR^n$, or a scalar linear operator $R$ in $L^2(\bbR^n)$, we will frequently 
identify $f$ or $R$ with the diagonal matrices $f \, I_{m}$ or $R \, I_m$ in 
$L^2(\bbR^n)^{m \times m}$ for appropriate $m \in \bbN$. 

\begin{remark}
\label{rem:differen_mult_op} We will identify a
function $\Phi$ with its corresponding multiplication operator of
multiplying by this function in a suitable function space. In doing so, for a differential operator $\cQ$, we will distinguish between the expression $\cQ\Phi$ and $\left(\cQ\Phi\right)$ and, similarly, for other differential operators.
Namely, $\cQ\Phi$ denotes the \emph{composition} of the two operators
$\cQ$ and $\Phi$, whereas $\left(\cQ\Phi\right)$ denotes the \emph{multiplication
operator }of multplying by the function $x\mapsto (\cQ\Phi)(x)$. \hfill $\diamond$
\end{remark}

\newpage

\section{Functional Analytic Preliminaries} \label{sec:Functional-Analytic-Preliminarie}

In this section we shall summarize the results obtained by Callias
in \cite[Lemmas 1 and 2]{Ca78}. We emphasize that we only succeeded to prove 
\cite[Lemma 1]{Ca78}
under the stronger condition that the trace norm of the operator under
consideration is bounded on a punctured neighborhood around the origin.
To begin, we recall the setting of \cite[Section II, p.~218]{Ca78}:

\begin{definition}
\label{def:bigtrace_littletrace}Let $\cH$ be a separable Hilbert space, $m\in\mathbb{N}$, and 
$T\in \cB\left(\cH^{m},\cH^{m}\right)$, a bounded linear operator
from $\cH^{m}$ to $\cH^{m}$. Denoting by $\iota_{j}\colon \cH\to \cH^{m}$
the canonical embedding defined by $\iota_{j}h\coloneqq \{\delta_{kj}h\}_{k\in\{1,\ldots,m\}}$,
we introduce $T_{jk}\coloneqq\iota_{j}^{*}T\iota_{k}$ for $j,k\in\{1,\ldots,m\}.$
We define the \emph{internal trace}, $\tr_{m} (T) $, \emph{of }$T$ being the linear operator on $\cH$ given by 
\begin{equation}
\tr_{m} (T) \coloneqq\sum_{j=1}^{m}T_{jj}.       \label{eq:Def_of_int_tr}
\end{equation}
Next, let $M$ be a densely defined, closed linear operator in $\cH^{m}$ and introduce 
Witten's resolvent regularization via 
\begin{align}
\begin{split} 
B_{M}(z)\coloneqq z\tr_{m}\big((M^{*}M+z)^{-1}-(MM^{*}+z)^{-1}\big)\in \cB(\cH),& \\
z\in\rho(-M^{*}M)\cap\rho(-MM^{*}).&
\end{split} 
\label{eq:Def_of_B_L(z)}
\end{align}
We will denote by $\tr_{\cH}(\cdot)$ the trace on $\cB_1(\cH)$,
the Schatten-von Neumann ideal of trace class operators on $\cH$. 
\end{definition}

\begin{remark}\label{rem-tracem}
Let $\mathcal{H}$ be a Hilbert space and let $m\in \mathbb{N}$, 
$T\in \mathcal{B}_1(\mathcal{H}^m)$ and let 
$T_{jk} = \iota_{j}^{*}T\iota_{k}$, $j,k\in\{1,\ldots,m\}$ as in Definition 
\ref{def:bigtrace_littletrace}. Then boundedness of $\iota_{j}^{*}, \iota_{k}$, 
$j,k\in\{1,\ldots,m\}$, and exploiting the ideal property of 
$\cB_1(\cH^m)$ yields
$$
T_{jk} \in \cB_1(\cH), \quad j,k \in \{1,\dots,m\},
$$
in particular, 
$$
\tr_m (T) \in \mathcal{B}_1(\mathcal{H}).
$$  
\hfill $\diamond$
\end{remark}

It should be noted that in general, the internal trace does \emph{not} satisfy the cyclicity 
property in the sense that for $A,B\in \cB(\cH^{m},\cH^{m})$,  
\[
\tr_{m}(AB) \neq \tr_{m}(BA).
\]
However, if one of the operators is actually a matrix with entries in $\bbC$,
then such a result holds:

\begin{proposition}
\label{prop:cyclic property of inner trace} Let $\cH$ be a Hilbert
space, $m\in\mathbb{N}$, $A\in \cB(\cH^{m},\cH^{m})$, $B\in\mathbb{C}^{m\times m}.$
Then
\[
\tr_{m}(AB)=\tr_{m}(BA).
\]
\end{proposition}
\begin{proof}
We have $A=(A_{ij})_{i,j\in\{1,\ldots,m\}}$ and $B=(B_{ij})_{i,j\in\{1,\ldots,m\}}$
with $A_{ij}\in \cB(\cH)$ as in Definition \ref{def:bigtrace_littletrace}
and $B_{ij}\in\mathbb{C}$. Then 
\[
AB=\bigg(\sum_{k\in\{1,\ldots,m\}}A_{ik}B_{kj}\bigg)_{i,j\in\{1,\ldots,m\}} \, \text{ and } \,  
BA=\bigg(\sum_{k\in\{1,\ldots,m\}}B_{ik}A_{kj}\bigg)_{i,j\in\{1,\ldots,m\}}.
\]
Hence,
\begin{align*}
\tr_{m}(AB) & =\sum_{j=1}^m \sum_{k=1}^m A_{jk}B_{kj}\\
 & =\sum_{j=1}^m \sum_{k=1}^m B_{kj}A_{jk}\\
 & =\sum_{k=1}^m \sum_{j=1}^m\ B_{kj}A_{jk}=\tr_{m}(BA).\tag*{{\qedhere}}
\end{align*}
\end{proof}

Next, we need a result of the type of \cite[Lemma 1]{Ca78}, in fact, we need an additional generalization of \cite[Lemma 1]{Ca78} in order to be able to apply it to our situation. We shall briefly recall the notions used in the next result: Given a Hilbert space $\mathcal{K}$, a {\it Fredholm operator} $S\colon \dom(S)\subseteq \cK \to \cK$, denoted by $S \in \Phi(\cK)$, is defined by $S$ being a densely defined, closed, linear operator with finite-dimensional nullspace, 
$\dim(\ker(S)) < \infty$, and closed range, $\ran(S)$, being finite-codimensional, 
$\dim(\ker(S^*)) < \infty$. The {\it Fredholm index}, $\ind(S)$, of a Fredholm operator $S$ is 
then the difference of the dimension of the nullspace and codimension of the range, that is,
\[
  \ind(S) = \dim (\ker(S)) - \dim(\ker(S^*)).
\]
Basic facts on Fredholm operators will be recalled at the end of this section. 
For the next lemma, we shall also use the notion of convergence in the strong operator topology, that is, a sequence $\{T_\Lambda\}_{\Lambda \in \mathbb{N}}$ of bounded linear operators in a Hilbert space $\mathcal{H}$ is said to converge to some $T_\infty\in \mathcal{B}(\mathcal{H})$ in the \emph{strong operator topology}, $\slim_{\Lambda\to\infty} T_\Lambda = T_\infty$, if for all $\phi\in \mathcal{H}$, we have
\[
   \lim_{\Lambda\to\infty} T_\Lambda \phi = T_\infty \phi.
\]
Our (generalized) version of \cite[Lemma 1]{Ca78} then reads as follows.

\begin{theorem}\label{thm:index with Witten} In the situation of Definition \ref{def:bigtrace_littletrace}
assume that $M$ is Fredholm, and that $\{T_\Lambda\}_{\Lambda\in \mathbb{N}}$, $\{S^*_\Lambda\}_{\Lambda\in \mathbb N}$ are sequences in $\cB (\cH)$, both converging to $I_\cH$ in the strong operator topology as $\lambda \to \infty$, and introduce $S_\Lambda\coloneqq S_\Lambda^{**}$, 
$\Lambda\in \mathbb{N}$. Let $B_{M}(\cdot)$ be given by \eqref{eq:Def_of_B_L(z)},
\begin{equation}
B_{M}(z)\coloneqq z\tr_{m}\big((M^{*}M+z)^{-1}-(MM^{*}+z)^{-1}\big), \quad 
z\in\rho(-M^{*}M)\cap\rho(-MM^{*}).   
\label{eq:Def_of_B_L(z)1}
\end{equation}
Assume that for each $\Lambda\in \mathbb{N}$, there exists $\delta_\Lambda >0$ with 
$\Omega_\Lambda\coloneqq B(0,\delta_\Lambda)\backslash \{0\}
\subseteq \rho(-MM^{*})\cap\rho(-M^{*}M)$ and that the map   
\[
   \Omega_\Lambda \ni z\mapsto T_\Lambda B_M(z)S_\Lambda
\]
takes on values in $\mathcal{B}_1(\mathcal{H})$, such that 
\[
\Omega_{\Lambda} \ni z\mapsto \tr_{\mathcal{H}} (|T_\Lambda B_M(z)S_\Lambda|) 
= \| T_\Lambda B_M(z)S_\Lambda \|_{\cB_1(\cH)}
\]
is bounded with respect to $z$. Then
\begin{equation}
\ind (M) =\lim_{\Lambda\to\infty}\lim_{z\to0} \tr_{\mathcal{H}} (T_\Lambda B_M(z)S_\Lambda).   \lb{indL}
\end{equation} 
In addition, if $\delta\coloneqq \frac{1}{2}\inf_{\Lambda \in\mathbb{N}} (\delta_\Lambda )>0$ and 
$\Omega\coloneqq B(0,\delta)\ni z\mapsto \tr_\mathcal{H}(T_\Lambda B_M(z)S_\Lambda )$ converges uniformly on $\overline{B(0,\delta)}$ to some function $F(\cdot)$ as $\Lambda  \to \infty$. Then, one can interchange the limits $\Lambda \to \infty$ and $z \to 0$ in \eqref{indL} and obtains, 
\begin{equation} 
F(0)=\ind (M).   \lb{indF(0)}
\end{equation} 
 \end{theorem}
\begin{proof}
By the Fredholm property of $M$, one deduces that $M^{*}M$ and $MM^{*}$
are Fredholm and, if $0$ lies in the spectrum of either $M^{*}M$
or $MM^{*}$ it is an isolated eigenvalue of finite multiplicity. As $M$ is Fredholm, $\ran(M)$ is closed. Hence, $\ran(M)=\ker(M^{*})^{\bot}$, and since $M$ is closed, $\ker(M^{*}M)=\ker(M)$, as well as, $\ker(MM^{*})=\ker(M^{*}).$
Denote by $P_{\pm}\colon \cH^{m}\to \cH^{m}$ the orthogonal projection
onto $\ker(M)$ and $\ker(M^{*})$, respectively. Since by hypothesis 
$P_{\pm}$ are finite-dimensional operators, so is $\tr_{m} (P_{\pm})$.
Moreover, we have
\[
\tr_{\cH} (\tr_{m} (P_{\pm})) = \tr_{\cH^m} (P_{\pm}) = \dim (\ran(P_{\pm})).
\]
Indeed, the last equality being clear, we only need to show the first
one. Let $\{\phi_{k}\}_{k\in\bbN}$ be an orthonormal basis of $\cH$.
Let $\iota_{j}\colon \cH\to \cH^{m}$ be the canonical embedding given
by $\iota_{j}h\coloneqq\{\delta_{\ell j}h\}_{\ell\in\{1,\ldots,m\}}$ for all
$j\in\{1,\ldots,m\}$. Then it is clear that $\{\iota_{j}\phi_{k}\}_{j\in\{1,\ldots,m\},k\in\bbN}$
constitutes an orthonormal basis for $\cH^{m}$. We have
\begin{align*}
\tr_{\cH^m} (P_{\pm}) & = \sum_{j=1}^m \sum_{k \in \bbN} (\iota_{j}\phi_{k},P_{\pm}\iota_{j}\phi_{k})_{\cH^m} 
 =\sum_{j=1}^m \sum_{k \in \bbN} (\phi_{k},\iota_{j}^{*}P_{\pm}\iota_{j}\phi_{k})_{\cH} \\
 & =\sum_{k\in\bbN}\sum_{j=1}^{m} (\phi_{k},P_{\pm,jj}\phi_{k})_{\cH} 
 =\sum_{k\in\bbN} \bigg(\phi_{k},\sum_{j=1}^{m}P_{\pm,jj}\phi_{k}\bigg)_{\cH} \\
 & =\sum_{k\in\bbN} (\phi_{k},\tr_{m} (P_{\pm})\phi_{k})_{\cH} 
 = \tr_{\cH} (\tr_{m} (P_{\pm})).
\end{align*}
Next, define for $\Lambda\in \mathbb{N},$
\begin{align*}
\Omega_\Lambda \ni z\mapsto\tilde{B}_\Lambda (z) 
& \coloneqq T_\Lambda \big[\tr_{m}\left(z(M^{*}M+z)^{-1}\right)-\tr_{m} (P_{+})
-\tr_{m}\left(z(MM^{*}+z)^{-1}\right) \\
& \quad +\tr_{m} (P_{-})\big]S_\Lambda  \\
 & =T_\Lambda B_{M}(z)S_\Lambda -T_\Lambda \tr_{m} (P_{+})S_\Lambda  + T_\Lambda \tr_{m} (P_{-})S_\Lambda .
\end{align*} 
By \cite[Sect.~III.6.5]{Ka80}, 
\[
z(M^{*}M+z)^{-1}-P_{+} \underset{z \to 0}{\longrightarrow} 0
\]
in operator norm, and similarly for $z\mapsto z(MM^{*}+z)^{-1}-P_{-}$.
We note that $\tilde{B}_\Lambda (z) \in \cB_1(\cH)$, $z \in \Omega_\Lambda $.
Since $\Omega_\Lambda \ni z\mapsto\tr_{\cH}(|T_\Lambda B_{M}(z)S_\Lambda |)$ is bounded,
so is $\Omega_\Lambda \ni z\mapsto\tr_{\cH} \big(\big|\tilde{B}_\Lambda (z)\big|\big)$. For the boundedness of $\Omega_\Lambda \ni z\mapsto\tr_{\cH}
\big(\big|\tilde{B}_\Lambda (z)\big|\big)$ it suffices to observe that 
\begin{align*}
\tr_{\cH} \big(\big|\tilde{B}_\Lambda (z)\big|\big) & = \tr_{\cH}
\big(\big| T_\Lambda B_{M}(z)S_\Lambda -T_\Lambda \tr_{m}
(P_{+})S_\Lambda  + T_\Lambda \tr_{m} (P_{-})S_\Lambda \big|\big)
\\  & = \big\| T_\Lambda B_{M}(z)S_\Lambda -T_\Lambda \tr_{m}
(P_{+})S_\Lambda  + T_\Lambda \tr_{m} (P_{-})S_\Lambda \big\|_{\mathcal{B}_1(\cH)}
\\ & \leq \big\| T_\Lambda B_{M}(z)S_\Lambda \big\|_{\mathcal{B}_1(\cH)}
\\ &\quad + \big\| T_\Lambda \tr_{m}
(P_{+})S_\Lambda\big\|_{\mathcal{B}_1(\cH)}  + \big\|T_\Lambda \tr_{m}
(P_{-})S_\Lambda \big\|_{\mathcal{B}_1(\cH)}
\\ & = \tr_{\cH} \big(\big| T_\Lambda B_{M}(z)S_\Lambda \big|\big)
\\ &\quad + \tr_{\cH} \big(\big| T_\Lambda \tr_{m}
(P_{+})S_\Lambda\big|\big)  + \tr_{\cH} \big(\big|T_\Lambda \tr_{m}
(P_{-})S_\Lambda \big|\big),\quad z\in \Omega_\Lambda,
\end{align*}
and using that the last two summands correspond to traces of
finite-rank operators. Thus,
from the analyticity of $\tr_{\cH}\big(\tilde{B}_\Lambda (\cdot)F\big)$
for every finite-rank operator $F$ on $\cH$, one deduces that $\tilde{B}_\Lambda (\cdot)$
is analytic in the $\cB_1(\cH)$-norm, see, for instance, \cite[Proposition~A.3]{ABHN01}
(or \cite[Theorem A.4.3]{Wa11}). More precisely, in \cite[Theorem A.4.3]{Wa11} 
there is the following characterization of analyticity of Banach space valued functions:
A function $h\colon \mathcal{U}\to \cX$ for some open
$\mathcal{U}\subseteq \mathbb{C}$ and some Banach space $\cX$ is
analytic if and only if $\mathcal{U}\ni z\mapsto \|h(z)\|_{\cX}$ is
bounded on compact subsets of $\mathcal{U}$ and $z\mapsto \langle
h(z),x'\rangle$ is analytic for all $x'\in V$ with $V\subseteq \cX'$
being a norming set for $\cX$. Thus, it suffices to apply
\cite[Theorem A.4.3]{Wa11} to $\cX=\cB_1(\cH)$ as underlying Banach
space, and to observe that the space of finite-rank operators forms a
norming subset of $\cB_1(\cH)$ (cf.\ \cite[Proposition~A.3]{ABHN01}. 

By Riemann's theorem on removable singularities, 
one deduces that $\tilde{B}_\Lambda (\cdot)$ is analytic at $0$ with respect to the 
$\cB_1(\cH)$-norm. As $\tilde{B}_\Lambda (\cdot)$ is also
norm analytic in $\cB(\cH)$, and tends to $0$ as $z \to 0$, one gets 
that $\tilde{B}_\Lambda (z)$ tends to $0$ as $z \to 0$ in $\cB_1(\cH)$-norm. Hence,
\begin{align*}
\lim_{z\to0}\tr_{\mathcal{H}}(T_\Lambda {B_M}(z)S_\Lambda ) & = 
\lim_{z\to0}\Big(\tr_{\cH} \big(\tilde{B}_\Lambda (z)\big) 
+\tr_{\cH} \left(T_\Lambda  \tr_{m} (P_{+}) S_\Lambda \right)   \\
& \qquad  \quad \;\,\, - \tr_{\cH} \left(T_\Lambda  \tr_{m} (P_{-}) S_\Lambda \right)\Big)\\
 & =0+ \tr_{\cH} \left(T_\Lambda  \tr_{m} (P_{+}) S_\Lambda \right) - \tr_{\cH} \left(T_\Lambda  \tr_{m} (P_{-}) S_\Lambda \right).
\end{align*}
Since $\slim_{\Lambda \to\infty} T_\Lambda ,S_\Lambda ^* = I_{\mathcal{H}}$, one obtains 
$T_\Lambda \tr_{m} (P_\pm) S_\Lambda  \underset{\Lambda \to\infty}{\longrightarrow} \tr_m (P_\pm)$ in 
$\mathcal{B}_1(\mathcal{H})$ (see, e.g., \cite[Lemma 6.1.3]{Ya92}). Thus,  
\[\lim_{\Lambda \to\infty} \tr_{\mathcal{H}} (T_\Lambda  \tr_m (P_\pm) S_\Lambda )  
= \tr_{\mathcal{H}} (\tr_m (P_\pm)) = \tr_{\mathcal{H}} (P_\pm).
\] 
Hence,
\begin{align*}
& \lim_{\Lambda \to\infty} \lim_{z\to0}\tr_{\mathcal{H}}(T_\Lambda
{B_M}(z)S_\Lambda )
 \\  & \quad = 0+ \lim_{\Lambda \to\infty}\Big( \tr_{\cH} \left(T_\Lambda
\tr_{m} (P_{+}) S_\Lambda \right) - \tr_{\cH} \left(T_\Lambda
\tr_{m} (P_{-}) S_\Lambda \right) \Big)
 \\  & \quad = \dim(\ran( P_+))-\dim(\ran (P_-))
 \\  & \quad = \dim(\ker( M))-\dim(\ker (M^*))
 \\ & \quad = \ind (M).
\end{align*}
Finally, for the purpose of proving the last statement of the theorem, define 
$F_\Lambda\colon \Omega\ni z\mapsto \tr_\mathcal{H}(T_\Lambda B_M(z)S_\Lambda )$. 
Since $\{F_\Lambda\}_{\Lambda}$ converges uniformly to $F$, one infers that $F$ is 
continuous on $\Omega$. Thus,
\[
 \ind(M)=\lim_{\Lambda\to\infty} \lim_{z\to 0} F_{\Lambda}(z)= \lim_{\Lambda\to\infty} F_{\Lambda}(0) = F(0) = \lim_{z\to 0} F(z)= \lim_{z\to 0} \lim_{\Lambda\to\infty} F_{\Lambda}(z).\qedhere
\] 
\end{proof}

In connection with the last part of Theorem \ref{thm:index with Witten} we note that the (limit of the) map $z\mapsto F(z)$ in $0$ may be regarded as a \emph{generalized} Witten index (see, e.g., \cite{BGGSS87}, \cite{GS88} and the references therein, as well as Section \ref{s14}). 

\begin{remark} \label{r2.4} 
$(i)$ While \cite[p. 218, Lemma 1]{Ca78} might be valid as stated, it remains unclear, 
how the assertion that is stated in line 5 on
page 219 comes about. The author infers the following: Let $\cH$ be
a separable Hilbert space, $\Gamma \subseteq\mathbb{C}$ open with $0\in\partial \Gamma$,
$B\colon \Gamma \to \cB(\cH)$ analytic. Assume that $B(z)\in\cB_1(\cH)$
for all $z\in \Gamma$, $\Gamma\ni z\mapsto\|B(z)\|_{\cB_1(\cH)}$ is bounded
and $\|B(z)\|_{\cB(\cH)}\to0$ as $z\to0$. Then for an orthonormal basis
$\{\phi_{k}\}_{k\in\bbN}$ of $\cH$,  
\begin{equation} 
\textnormal{tr}_{\cH} (B(z)) = \sum_{k=1}^{\infty} (\phi_{k},B(z)\phi_{k})_{\cH} 
\underset{z \to 0}{\longrightarrow} 0. 
\end{equation} 
This statement is invalid as the following example, kindly 
communicated to us by H.\ Vogt \cite{Vo14}, shows: 
For the orthonormal basis $\{\phi_{k}\}_{k\in\mathbb{N}}$ define 
the family of operators $B(\cdot)$ by
\[
B(z)\phi_{k}\coloneqq ze^{-(k-1)z}\phi_{k}, \quad 
z\in \Gamma\coloneqq\{z\in\mathbb{C} \,|\, |\arg (z)|< (\pi/4), \, \left|z\right|<1\}, \quad 
k \in \bbN. 
\]
Then 
\begin{align*} 
\|B(z)\|_{\mathcal{B}_1(\mathcal{H})} & = \tr_{\cH} (|B(z)|) 
=\sum_{k=1}^{\infty} (\phi_{k},\left|B(z)\right|\phi_{k})_{\cH} 
=\sum_{k=1}^{\infty}\left|z\right|e^{-(k-1)\Re (z)} 
\\&=\left|z\right|\sum_{k=0}^{\infty}\left(e^{-\Re (z)}\right)^{k} 
=\frac{\left|z\right|}{1-e^{-\Re (z)}}
\end{align*} 
remains bounded for $z\in \Gamma$. Moreover, $\nlim_{z \to 0} B(z) = 0$ 
in $\cB(\cH)$. However,
\[
\tr_{\cH} (B(z)) =\sum_{k=1}^{\infty}\big(\phi_{k},ze^{-(k-1)z}\phi_{k}\big)_{\cH} 
=z\sum_{k=0}^{\infty}e^{-kz}=z\frac{1}{1-e^{-z}}=\frac{z}{e^{z}-1}e^{z}\underset{z \to 0}{\longrightarrow} 1.
\]
$(ii)$ We shall now elaborate on the fact that an analytic function taking values in the
space of bounded linear operators in a Hilbert space can indeed have different
domains of analyticity if considered as taking values in particular
Schatten--von Neumann ideals. 

Consider an infinite-dimensional Hilbert space $\mathcal{H}$ and pick an orthonormal
basis $\{\phi_{k}\}_{k\in\mathbb{N}}$ in $\cH$. For $z\in\mathbb{C}$ with $\Re (z) \geq 0$,
define
\[
T(z)\phi_{k}\coloneqq e^{-z\ln (k)}\phi_{k}, \quad k \in \bbN.
\]
 Then $T(z)\in \cB(\cH)$ and 
\[
T(z)\psi\coloneqq\sum_{k=1}^{\infty}e^{-z\ln (k)}\left(\phi_{k},\psi\right)\phi_{k}, \quad 
\psi\in\mathcal{H}.
\]
 Moreover, $T(0)=I_{\cH}$, $T(2)\in\mathcal{B}_{1}(\cH)$, $T(1)\in\mathcal{B}_{2}(\cH)$. 
 We note here that for $\Re (z)>1$ the function $z\mapsto T(z)$
is also analytic with values in $\cB_1(\cH)$, however, the trace
norm of $T(\cdot)$ blows up at $z=1$.  \hfill $\diamond$
\end{remark}

We conclude this section with some facts on Fredholm operators. For reasons to be able to handle certain classes of unbounded Fredholm operators in a convenient manner, we now take a slightly more general approach and permit a two-Hilbert space setting as follows: Suppose $\cH_j$, 
$j \in \{1,2\}$, are complex, separable Hilbert spaces. Then 
$S  \colon \dom(S) \subseteq \cH_1 \to \cH_2$, 
$S$ is called a {\it Fredholm operator}, denoted by $S \in \Phi(\cH_1, \cH_2)$, if
\begin{align*}
& \text{$(i)$ $S$ is closed and densely defined in $\cH_1$.} \\
& \text{$(ii)$ $\ran(S)$ is closed in $\cH_2$.} \\
& \text{$(iii)$ $\dim(\ker(S)) + \dim(\ker(S^*)) < \infty$.}
\end{align*}

If $S$ is Fredholm, its 
{\it Fredholm index} is given by 
\begin{equation}
\ind(S) = \dim(\ker(S)) - \dim(\ker(S^*)). 
\end{equation}

If $S\colon \dom(S)\subseteq \cH_1 \to \cH_2$ is densely defined and closed, we associate with 
$\dom(S) \subset \cH_1$ the standard graph Hilbert subspace $\cH_S \subseteq \cH_1$ induced 
by $S$ defined by
\begin{align*}
\cH_S = (\dom(S); (\, \cdot \, , \, \cdot \,)_{\cH_{S}}), \quad 
(f,g)_{\cH_S} = (Sf, Sg)_{\cH_2} + (f,g)_{\cH_1},&   \\ 
\|f\|_{\cH_S} = \big[\|Sf\|_{\cH_2}^2 + \|f\|_{\cH_1}^2\big]^{1/2}, \; f,g \in \dom(S).& 
\end{align*}  

In addition, for $A_0, A_1 \in \Phi(\cH_1, \cH_2) \cap \cB(\cH_1, \cH_2)$, $A_0$ and $A_1$ are 
called {\it homotopic} in $\Phi(\cH_1, \cH_2)$ if there exists $A \colon [0,1] \to \cB(\cH_1, \cH_2)$ 
continuous such that $A(t) \in \Phi(\cH_1, \cH_2)$, $t \in [0,1]$, with $A(0) = A_0$, $A(1) = A_1$. 

Next, following \cite[Chs.~1, 3]{BB13}, \cite[Chs.~XI, XVII]{GGK90}, \cite[Sects.~IV.6, IV.10]{GK92}, \cite[Sect.~I.3]{MP86}, \cite[Ch.~2]{Mu13}, \cite[Chs.~5, 7]{Sc02}, we now summarize a few basic properties of Fredholm operators. 

\begin{theorem} \lb{t3.6}
Let $\cH_j$, $j=1,2,3$, be complex, separable Hilbert spaces,  
then the following items $(i)$--$(vii)$ hold: \\[1mm] 
$(i)$ If $S \in \Phi(\cH_1, \cH_2)$ and $T \in \Phi(\cH_2, \cH_3)$, then 
$TS \in \Phi(\cH_1, \cH_3)$ and 
\begin{equation}
\ind(TS) = \ind(T) + \ind(S). 
\end{equation}
$(ii)$ Assume that $S \in \Phi(\cH_1, \cH_2)$ and $K \in \cB_{\infty}(\cH_1, \cH_2)$, then 
$(S + K) \in \Phi(\cH_1, \cH_2)$ and 
\begin{equation}
\ind(S + K) = \ind(S). 
\end{equation}
$(iii)$ Suppose that $S \in \Phi(\cH_1, \cH_2)$ and $K \in \cB_{\infty}(\cH_S, \cH_2)$, then 
$(S + K) \in \Phi(\cH_1, \cH_2)$ and 
\begin{equation}
\ind(S + K) = \ind(S). 
\end{equation}
$(iv)$ Assume that $S \in \Phi(\cH_1, \cH_2)$. Then there exists $\varepsilon(S) > 0$ such that for any $R \in \cB(\cH_1, \cH_2)$ with  
$\|R\|_{\cB(\cH_1, \cH_2)} < \varepsilon(S)$, one has $(S + R) \in \Phi(\cH_1, \cH_2)$ and 
\begin{equation}
\ind(S + R) = \ind(S), \quad \dim(\ker(S + R)) \leq \dim(\ker(S)). 
\end{equation}
$(v)$ Let $S \in \Phi(\cH_1, \cH_2)$, then $S^* \in \Phi(\cH_2, \cH_1)$ and 
\begin{equation}
\ind(S^*) = - \ind(S). 
\end{equation}
$(vi)$ Assume that $S \in \Phi(\cH_1, \cH_2)$ and that the Hilbert space $\cV_1$ is continuously embedded in $\cH_1$, with 
$\dom(S)$ dense in $\cV_1$. Then $S \in \Phi(\cV_1, \cH_2)$ with $\ker(S)$ and $\ran(S)$ the same whether $S$ is viewed as an operator $S\colon \dom(S)\subseteq \cH_1 \to \cH_2$, or as an operator 
$S\colon \dom(S)\subseteq \cV_1 \to \cH_2$. \\[1mm]
$(vii)$ Assume that the Hilbert space $\cW_1$ is continuously and densely embedded in 
$\cH_1$. If $S \in \Phi(\cW_1, \cH_2)$ then $S \in \Phi(\cH_1, \cH_2)$ with $\ker(S)$ and $\ran(S)$ the same whether $S$ is viewed as an operator $S\colon \dom(S)\subseteq \cH_1 \to \cH_2$, or 
as an operator $S\colon \dom(S)\subseteq \cW_1 \to \cH_2$. \\[1mm]
$(viii)$ Homotopic operators in $\Phi(\cH_1, \cH_2) \cap \cB(\cH_1, \cH_2)$ have equal Fredholm 
index. More precisely, the set $\Phi(\cH_1, \cH_2) \cap \cB(\cH_1, \cH_2)$ is open in 
$\cB(\cH_1, \cH_2)$, hence $\Phi(\cH_1, \cH_2)$ contains at most countably many connected components, on each of which the Fredholm index is constant. Equivalently, 
$\ind \colon \Phi(\cH_1, \cH_2) \to \bbZ$ is locally constant, hence continuous, and homotopy 
invariant. 
\end{theorem} 

A prime candidate for the Hilbert spaces $\cV_1, \cW_1 \subseteq \cH_1$ in 
Theorem \ref{t3.6}\,$(vi)$,\,$(vii)$ (e.g., in applications to differential operators) is the 
graph Hilbert space $\cH_S$ induced by $S$. Moreover, an immediate consequence of 
Theorem \ref{t3.6} we will apply later is the following homotopy invariance of the 
 Fredholm index for a family of Fredholm operators with fixed domain.  
 
\begin{corollary} \lb{c3.7} 
Let $T(s) \in \Phi(\cH_1, \cH_2)$, $s \in I$, where $I \subseteq \bbR$ is a connected interval, with $\dom(T(s)) := \cV_T$ independent of $s \in I$. In addition, assume that 
$\cV_T$ embeds densely and continuously into $\cH_1$ $($for instance, $\cV_T = \cH_{T(s_0)}$ 
for some fixed $s_0 \in I$$)$ and that $T(\cdot)$ is continuous with respect to the norm 
$\|\cdot\|_{\cB(\cV_T, \cH_2)}$. Then
\begin{equation}
\ind(T(s)) \in \bbZ \, \text{ is independent of $s \in I$.} 
\end{equation}
 \end{corollary}

The corresponding case of unbounded operators with varying domains (and $\cH_1 = \cH_2$) is treated in detail in \cite{CL63}. 
 
\newpage

\section{On Schatten--von Neumann Classes and Trace Class Estimates}\label{sec:Trace-Class-Estimates}

This is the first of two technical sections, providing basic results used later
on in our detailed study of Dirac-type operators to be introduced
in Section \ref{sec:A-particular-First}. We also recall results on the Schatten--von Neumann classes and apply these to concrete situations needed in Section \ref{sec:The-Derivation-of-trace-f}.

We start with the following well-known
characterization of Hilbert--Schmidt operators $\cB_2\big(L^{2}(\Omega; d\mu)\big)$ 
in $L^{2}(\Omega; d\mu)$:

\begin{theorem}[see, e.g., {\cite[Theorem 2.11]{Si05}}]
\label{thm:Hilbert-Schmidt-L2} 
Let $(\Omega; \cB; \mu)$ be a separable measure space and 
$k\colon \Omega \times \Omega \to \mathbb{C}$ be $\mu \otimes \mu$ measurable. 
Then the  map  
\begin{align*}
\cU: \begin{cases} L^{2} (\Omega \times \Omega; d\mu \otimes d\mu) \to\cB_2\big(L^{2}(\Omega; d\mu)\big), \\
k \mapsto\left(f\mapsto\left(\Omega \ni x\mapsto\int_{\Omega} k(x,y)f(y) \, d\mu(y)\right)\right), 
\end{cases} 
\end{align*}
is unitary.
\end{theorem}

The H\"older inequality is also valid for trace ideals with $p$-summable singular values.

\begin{theorem} [{H\"older inequality, see, e.g., \cite[Theorem 2.8]{Si05}}]
\label{thm:trace-class-crit} Assume that $\cH$ is a complex, separable Hilbert
space, $m\in\mathbb{N}$, $q_{j}\in[1,\infty]$, $j\in\{1,\ldots,m\}$,
$p\in[1,\infty]$. Assume that
\[
\sum_{j=1}^{m}\frac{1}{q_{j}}=\frac{1}{p}.
\]
Let $T_{j}\in\cB_{q_j}(\cH)$, $j\in\{1,\ldots,m\}$. Then
$T\coloneqq\prod_{j=1}^{m}T_{j}\in\cB_p(\cH)$ and 
\[
\left\Vert T\right\Vert _{\cB_p(\cH)}\leq \prod_{j=1}^{m}\left\Vert T_{j}\right\Vert _{\cB_{q_j}(\cH)}.
\]
\end{theorem}

For $q_1=q_2=m=2$, one obtains a criterion for operators belonging to the trace class 
$\mathcal{B}_1$, which we shall use later on.

\begin{corollary}
\label{cor:Comp-of-trace} Let $(\Omega; \cB, \mu)$ be a separable measure space, 
$k\colon \Omega \times \Omega \to \mathbb{C}$ be $\mu \otimes \mu$ measurable. Moreover, assume that there exists $\ell,m\in L^{2}(\Omega \times \Omega; d\mu \otimes d\mu)$
such that
\[
k(x,y)=\int_{\Omega} \ell(x,w)m(w,y) \, d\mu(w) \, \text{ for $\mu \otimes \mu$~a.e.~$(x,y)
\in\Omega \times \Omega$.}
\]
Then $K$, the associated integral operator with integral kernel $k(\cdot, \cdot)$ in 
$L^2(\Omega \times \Omega; d\mu \otimes d\mu)$, is trace class, 
$K \in \cB_1\big(L^{2}(\Omega; d\mu)\big)$, and  
\[
\textnormal{tr}_{L^2(\Omega; d\mu)} (K) = \int_{\Omega} k(x,x) \, d\mu(x) 
=\int_{\Omega}\int_{\Omega} \ell(x,w)m(w,x) \, d\mu(w)d\mu(x).
\]
\end{corollary}
\begin{proof}
By Theorem \ref{thm:Hilbert-Schmidt-L2} the integral operators $L$
and $M$ associated with $\ell$ and $m$, respectively, are Hilbert--Schmidt
operators. Since $K=LM$, one gets $K \in\ \cB_1\big(L^{2}(\Omega; d\mu)\big)$
by Theorem \ref{thm:trace-class-crit}. Moreover, by Theorem \ref{thm:Hilbert-Schmidt-L2}, 
one concludes that 
\begin{align*}
\tr_{L^2(\Omega; d\mu)} (K) &= \tr_{L^2(\Omega; d\mu)} (LM) 
=\tr_{L^2(\Omega; d\mu)} \big((L^{*})^{*}M\big) \nonumber \\
& = \int_{\Omega}\int_{\Omega}\overline{\overline{\ell(w,x)}}m(x,w)\, d\mu(x)d\mu(w)   \\
& = \int_{\Omega}\int_{\Omega}\ell(x,w)m(w,x) \, d\mu(w) d\mu(x) .\tag*{\qedhere}
\end{align*} 
\end{proof}

In the bulk of this manuscript, Theorem \ref{thm:Hilbert-Schmidt-L2} and 
Corollary \ref{cor:Comp-of-trace} will be applied to the case \[
L^2(\Omega;\mathcal{B}; d\mu)=L^2(\bbR^n; \mathcal{B}(\bbR^n); d^n x)=L^2(\bbR^n)
                                                              \]
We recall $H^1(\bbR^n)$ and $H^2(\bbR^n)$, the spaces of once and twice weakly differentiable $L^2$-functions with derivatives in $L^2$, respectively. Moreover, we shall furthermore consider the differential operator $Q$ in $L^{2}(\mathbb{R}^{n})^{2^{\hat n}}$ by 
\begin{equation}
Q\coloneqq\sum_{j=1}^{n}\gamma_{j,n}\partial_{j}, \quad 
\dom(Q) = H^{1}(\mathbb{R}^{n})^{2^{\hat n}},      \label{eq:Def_of_Q}
\end{equation}
where 
\begin{equation}
\gamma_{j,n}^{*}=\gamma_{j,n}\in\mathbb{C}^{2^{\hat n}\times2^{\hat n}} \, \text{ 
if $n=2 \hat n$ or $n=2 \hat n +1$,} 
\end{equation} 
and $\gamma_{j,n}\gamma_{k,n}+\gamma_{k,n}\gamma_{j,n}=2\delta_{jk}$
for all $j,k\in\{1,\ldots,n\}$, see Definition \ref{def:Euc-D-A}. A first consequence is, 
\begin{equation}
Q^{2}=\Delta I_{2^{\hatt n}}, \quad \dom(Q^2) = H^{2}(\mathbb{R}^{n})^{2^{\hat n}}.       \label{eq:q2=00003DDelta}
\end{equation}
Indeed,
\begin{align*}
QQ & =\sum_{j=1}^{n}\gamma_{j,n}\partial_{j}\sum_{k=1}^{n}\gamma_{k,n}\partial_{k} 
 =\sum_{j=1}^{n}\sum_{k=1}^{n}\gamma_{j,n}\partial_{j}\gamma_{k,n}\partial_{k}\\
 & =\frac{1}{2}\sum_{j=1}^{n}\sum_{k=1}^{n}\gamma_{j,n}\partial_{j}\gamma_{k,n}\partial_{k}+\frac{1}{2}\sum_{j=1}^{n}\sum_{k=1}^{n}\gamma_{k,n}\partial_{k}\gamma_{j,n}\partial_{j}\\
 & =\frac{1}{2}\sum_{j=1}^{n}\sum_{k=1}^{n}\left(\gamma_{j,n}\gamma_{k,n}+\gamma_{k,n}\gamma_{j,n}\right)\partial_{j}\partial_{k} =\Delta I_{2^{\hatt n}}.
\end{align*} 
We study the operator asssociated with the differential expression $Q$ with its properties 
later on in Section \ref{sec:A-particular-First}. More precisely, in Theorem \ref{thm:L_is_closed}
we show that $Q=-Q^{*}$ in $L^{2}(\mathbb{R}^{n})^{2^{\hat n}}$ with
domain $H^{1}(\mathbb{R}^{n})^{2^{\hat n}}$, that is, $Q$ is skew-self-adjoint
in $L^{2}\left(\mathbb{R}^{n}\right)^{2^{\hat n}}$. In particular, we
get for any $\mu\in\mathbb{C}$ with $\Re (\mu) \neq0$ that 
\begin{equation}
\big\Vert (Q+\mu)^{-1}\big\Vert \leq |\Re (\mu)|^{-1}     \label{eq:cont_est_of_Q}
\end{equation}
as an operator from $L^{2}\left(\mathbb{R}^{n}\right)^{2^{\hat n}}$ to
$L^{2}\left(\mathbb{R}^{n}\right)^{2^{\hat n}}$. One notes that by Fourier transform, the operator $Q$ is unitarily equivalent to the Fourier multiplier with symbol $\sum_{j=1}^n \gamma_{j,n} (-i)\xi_j$. Furthermore, by $\gamma_{j,n}=\gamma_{j,n}^*$ and $\gamma_{j,n}^2=I_{2^{\hat n}}$,  the matrix $\gamma_{j,n}$ is unitary. Hence, the symbol of $Q$ may be estimated as follows
\begin{equation}\label{eq:estimate_for_symbol_of_Q}
 \bigg\|\sum_{j=1}^{n}\gamma_{j,n}(-i)\xi_{j}\bigg\|_{\cB(\bbC^{2^{\hat n}})}
 \leq \sum_{j=1}^{n}\left|\xi_{j}\right|\leq \sqrt{n}\bigg(\sum_{j=1}^{n}\left|\xi_{j}\right|^{2}\bigg)^{1/2} 
 = \sqrt{n}\left|\xi\right|,  \quad \xi\in\mathbb{R}^{n}.
\end{equation}
We denote 
\begin{equation}
R_{\mu}\coloneqq\left(-\Delta+\mu\right)^{-1}, \quad 
\mu\in\mathbb{C}\backslash (-\infty, 0].      \label{eq:resolvent_of_laplace}
\end{equation}

We recall our notational conventions collected in Section \ref{s2}. In particular, we recall $[A,B]=AB-BA$, the commutator of two operators $A$ and $B$, see also \eqref{eq:commutator=00003DC}. 

\begin{lemma}
\label{lem:commutator}Let $\mu\in\mathbb{C}$, $\Re (\mu) >0$, and 
$\Psi\in C_{b}^{2}(\mathbb{R}^{n})$.
Then with $Q$ and $R_{\mu}$ given by \eqref{eq:Def_of_Q} and \eqref{eq:resolvent_of_laplace},
respectively, one obtains $($cf.\ Remark \ref{rem:differen_mult_op}$)$, 
\begin{equation}
\left[R_{\mu},\Psi\right]=R_{\mu}\left(Q^{2}\Psi\right)R_{\mu}+2R_{\mu}\left(Q\Psi\right)QR_{\mu}.\label{eq:commutator_of_res_mut}
\end{equation} 
\end{lemma}
\begin{proof}
Recalling Remark \ref{rem:differen_mult_op} concerning multiplication
operators, we compute with the help of \eqref{eq:q2=00003DDelta}
\begin{align*}
\left[R_{\mu},\Psi\right] & =R_{\mu}\Psi-\Psi R_{\mu}\\
 & =R_{\mu}\left(\Psi\left(-\Delta+\mu\right)-\left(-\Delta+\mu\right)\Psi\right)R_{\mu}\\
 & =R_{\mu}\left(\Delta\Psi-\Psi\Delta\right)R_{\mu} 
 =R_{\mu}\left(Q^{2}\Psi-\Psi\Delta\right)R_{\mu}\\
 & =R_{\mu}\left(Q\left(Q\Psi\right)+Q\Psi Q-\Psi\Delta\right)R_{\mu}\\
 & =R_{\mu}\left(\left(Q^{2}\Psi\right)+\left(Q\Psi\right)Q+\left(Q\Psi\right)Q+\Psi Q^{2}-\Psi\Delta\right)R_{\mu}\\
 & =R_{\mu}\left(\left(Q^{2}\Psi\right)+2\left(Q\Psi\right)Q\right)R_{\mu}.\tag*{{\qedhere}}
\end{align*}
\end{proof}

In the course of computing the index of the closed operator $L$ to be introduced later on, we need to establish trace class properties of operators that are products of operators of the form 
discussed in the following lemma. For given $n\in \mathbb{N}_{\geq 3}$, $x\in \mathbb{R}^n$, 
$\mu\in \mathbb{C}_{\Re>0}\coloneqq \{z\in \mathbb{C} \, | \, \Re (z) >0\}$, we let
\begin{equation}\label{e:g_mu}
  g_\mu (x)\coloneqq \frac{1}{\Re (\mu) + |x|^2}, \quad \tilde g_\mu (x)\coloneqq \sqrt{n}\frac{|x|}{\Re (\mu) + |x|^2}.
\end{equation}
One notes that $g_\mu \in L^q(\mathbb{R}^n)$ for all $q>n/2$ and $\tilde g_\mu \in L^q(\mathbb{R}^n)$ for all $q>n$.

\begin{lemma}
\label{lem:Schatten-class-1-operator} Let $\mu\in\mathbb{C}_{\Re>0}$,
$\Psi\in L^{\infty}(\mathbb{R}^{n}),$ $\alpha\in[1,\infty)$, $n\geq3$, 
and recall $R_{\mu}$ from \eqref{eq:resolvent_of_laplace} and $Q$
from \eqref{eq:Def_of_Q}, as well as $g_\mu$ and $\tilde g_\mu$ from \eqref{e:g_mu}. Assume that there exists $\kappa>0$ such
that
\[
\left|\Psi(x)\right|\leq \kappa (1+\left|x\right|)^{- \alpha} \, 
\text{ for a.e.~$x\in\mathbb{R}^{n}$.}
\]
$(i)$ Then for all $q>n$, $R_{\mu}\Psi,\Psi R_{\mu}\in\cB_{q}\left(L^{2}\left(\mathbb{R}^{n}\right)\right)$ and 
\[
\max\big(\left\Vert \Psi R_{\mu}\right\Vert _{\cB_{q}(L^{2}(\mathbb{R}^{n}))}, 
\left\Vert R_{\mu}\Psi\right\Vert _{\cB_{q}(L^{2}(\mathbb{R}^{n}))}\big)
 \leq \left(2\pi\right)^{-n/q} \|\Psi\|_{L^q(\mathbb{R}^n)} \|g_\mu\|_{L^q(\mathbb{R}^n)}<\infty. 
\]
The assertion remains the same, if $R_\mu$ is replaced by $R_\mu Q$ or $Q R_\mu$ and $ \|g_\mu\|_{L^q(\mathbb{R}^n)}$ in the latter estimate is replaced by $ \|\tilde g_\mu\|_{L^q(\mathbb{R}^n)}$.
\\[1mm] 
$(ii)$ Assume, in addition, $\alpha> 3/2$. Then, if $n>3$, there
exists $\theta\in(3/4,1)$ such that 
$R_{\mu}\Psi,\Psi R_{\mu}\in\cB_{2n\theta/3}\big(L^{2}(\mathbb{R}^{n}\big)$. Moreover, 
\begin{align*}
& \max \big(\left\Vert \Psi R_{\mu}\right\Vert _{\cB_{2n\theta/3}(L^{2}(\mathbb{R}^{n}))}, 
\left\Vert R_{\mu}\Psi\right\Vert _{\cB_{2n\theta/3}(L^{2}(\mathbb{R}^{n}))}\big)
\\ 		
& \quad \leq \left(2\pi\right)^{-3/(2\theta)} \|\Psi\|_{L^{2n\theta/3}(\mathbb{R}^n)}
\|g_\mu\|_{L^{2n\theta/3}(\mathbb{R}^n)}<\infty.  
\end{align*}
For $n=3$, $R_{\mu}\Psi,\Psi R_{\mu}\in\cB_2\left(L^{2}\left(\mathbb{R}^{n}\right)\right)$
and
\[
\max\big(\left\Vert \Psi R_{\mu}\right\Vert _{\cB_2(L^{2}(\mathbb{R}^{n}))}, 
\left\Vert R_{\mu}\Psi\right\Vert _{\cB_2(L^{2}(\mathbb{R}^{n}))}\big) \leq \left(2\pi\right)^{-3/2} \|\Psi\|_{L^{2}(\mathbb{R}^n)}\|g_\mu\|_{L^{2}(\mathbb{R}^n)}<\infty.	
\]
$(iii)$ Let $\Theta\in C_{b}^{2}\left(\mathbb{R}^{n}\right)$ with 
\[
\left|\left(Q\Theta\right)(x)\right|+\left|\left(Q^{2}\Theta\right)(x)\right|\leq 
\kappa (1+\left|x\right|)^{- \beta}
\]
for some $\kappa>0$ and $\beta> 3/2$. Then, recalling \eqref{eq:def_commutator},
$\left[R_{\mu},\Theta\right]\in\cB_2\left(L^{2}\left(\mathbb{R}^{n}\right)\right)$ with 
\begin{align*}
& \left\Vert \left[R_{\mu},\Theta\right]\right\Vert _{\cB_2(L^{2}(\mathbb{R}^{n}))} \\ 
& \quad \leq \bigg(\frac{1}{\Re(\mu)}+2\bigg(\frac{1}{\Re(\sqrt{\mu})}+3\frac{\left|\sqrt{\mu}\right|}{\Re(\mu)}\bigg)\bigg)\left(2\pi\right)^{-3/2}        \\
& \qquad \times \max\big(\|Q^2\Theta\|_{L^2(\mathbb{R}^n)},\|Q\Theta\|_{L^2(\mathbb{R}^n)}\big)
\|g_\mu\|_{L^{2}(\mathbb{R}^n)}<\infty
\end{align*}
if $n=3$, and $\left[R_{\mu},\Theta\right]\in\cB_{\left(2 n/3\right)\theta}\left(L^{2}\left(\mathbb{R}^{n}\right)\right)$
with 
\begin{align*}
 \left\Vert \left[R_{\mu},\Theta\right]\right\Vert _{\cB_{2n\theta/3}(L^{2}(\mathbb{R}^{n}))} 
& \leq \bigg(\frac{1}{\Re(\mu)}+2\bigg(\frac{1}{\Re(\sqrt{\mu})}+3\frac{\left|\sqrt{\mu}\right|}{\Re(\mu)}\bigg)\bigg) 
\left(2\pi\right)^{-3/(2\theta)}  \\
& \quad \times \max\big(\|Q^2\Theta\|_{L^{2n\theta/3}(\mathbb{R}^n)},
 \|Q\Theta\|_{L^{2n\theta/3}(\mathbb{R}^n)}\big)    \\
& \quad \times \|g_\mu\|_{L^{2n\theta/3}(\mathbb{R}^n)} < \infty
\end{align*}
for some $\theta\in(0,1)$ with $2n\theta/3 \geq2$, if $n>3$.
\end{lemma}

The proof of Lemma \ref{lem:Schatten-class-1-operator}, is basically contained in the following 
result:

\begin{theorem}[{\cite[Theorem 4.1]{Si05}}]\label{thm:Si} Let $n\in \mathbb{N}$, $p\geq 2$, 
and $\Psi,g\in L^p(\mathbb{R}^n)$. Define $T_{\Psi,g}\in \mathcal{B}(L^2(\mathbb{R}^n))$ as the operator of composition of multiplication by $\Psi$ and $g(i\partial_1,\ldots,i\partial_n)$ as a Fourier multiplier. Then $T_{\Psi,g}\in \mathcal{B}_p(L^2(\mathbb{R}^n))$ and 
\begin{equation}\label{TPhi}
   \| T_{\Psi,g}\|_{\cB_p(L^2(\mathbb{R}^n))}\leq (2\pi)^{-n/p} 
   \|\Psi\|_{L^p(\bbR^n)}\|g\|_{L^p(\bbR^n)}.
\end{equation}
\end{theorem}

The proof of \eqref{TPhi} rests on the observation that there is equality
for $p=2$ and a straight forward estimate for the limiting case $p=\infty$. The general
case follows via complex interpolation.

\begin{proof}[Proof of Lemma \ref{lem:Schatten-class-1-operator}]
Observing $\|T\|_{\cB_p(\cH)}=\|T^{*}\|_{\cB_p(\cH)}$ for
all $T\in\cB_p(\cH)$, we shall only show the respective assertions
for $\Psi R_{\mu}$. For parts $(i)$ and $(ii)$ one uses Theorem \ref{thm:Si}:
One notes that for $\phi_{\mu}(\xi)\coloneqq\big(\left|\xi\right|^{2}+\mu\big)^{-1}$, 
$\phi_{\mu}\left(i\partial_{1},\ldots,i\partial_{n}\right)=R_{\mu}$.
Moreover, one observes that $|\phi_{\mu}|\leq g_\mu \in L^{p}\left(\mathbb{R}^{n}\right)$
for all $p> n/2$. In order to prove item $(i)$ one notes that $\alpha\geq1$
implies that $\Psi\in L^{q}\left(\mathbb{R}^{n}\right)$ for all $q>n$.
Hence, $\Psi R_{\mu}\in \cB_q\big(L^{2}(\mathbb{R}^{n})\big)$. The remaining assertion follows from the fact that the Fourier transform of $QR_\mu$ lies in $L^q$ as it can be estimated by $\tilde g_\mu$, see \eqref{eq:estimate_for_symbol_of_Q}. For part
$(ii)$ one first considers the case $n=3$. Then $2\alpha=\alpha (2/3) 3>3=n$.
Hence, $\Psi\in L^{2}\left(\mathbb{R}^{3}\right)$ and, since $2> 3/2 = n/2$,
one infers that $\Psi R_{\mu}\in\cB_2$. If $n>3$, there exists $\theta\in (3/4,1)$ such that $\alpha\theta> 3/2$. In particular, one has $(2/3) n\theta > (2/3) 4 (3/4)=2$.
Since $\alpha (2/3) \theta n> (3/2) \left(2n/3\right)=n$,
one gets $\Psi\in L^{2n\theta/3}\left(\mathbb{R}^{n}\right)$.
Moreover, since $\left(2n/3\right)\theta> n/2$ as 
$\theta> 3/4$, one concludes that 
$|\phi_{\mu}|\leq g_\mu \in L^{2n\theta/3}\left(\mathbb{R}^{n}\right)$, 
implying $\Psi R_{\mu}\in\cB_{2n\theta/3}$.
In order to show part $(iii)$ one notes that Lemma \ref{lem:commutator} implies 
\[
\left[R_{\mu},\Theta\right]=R_{\mu}\left(Q^{2}\Theta\right)R_{\mu}+2R_{\mu}\left(Q\Theta\right)QR_{\mu}.
\]
Since $QR_{\mu}$ is a bounded linear operator, using \eqref{eq:q2=00003DDelta}
as well as \eqref{eq:resolvent_of_laplace}, one deduces from 
\begin{align*}
QR_{\mu} & =\left(Q+\sqrt{\mu}\right)^{-1}\left(Q+\sqrt{\mu}\right)\left(Q-\sqrt{\mu}\right)R_{\mu}+\sqrt{\mu}R_{\mu}\\
 & =\left(Q+\sqrt{\mu}\right)^{-1}+\sqrt{\mu}R_{\mu}
\end{align*}
its corresponding norm bound $[\Re(\sqrt{\mu})]^{-1} + |\sqrt{\mu}|[\Re(\mu)]^{-1}$,
see \eqref{eq:cont_est_of_Q} for the norm bound of $\left(Q+\sqrt{\mu}\right)^{-1}$.
Thus, the assertion follows from part $(ii)$ and the ideal property of
the Schatten--von Neumann classes.
\end{proof}

Lemma \ref{lem:Schatten-class-1-operator} is decisive for obtaining the following result. We mention here that H.~Vogt \cite{Vo15} subsequently managed to prove the following theorem in a direct way without using Lemma \ref{lem:Schatten-class-1-operator} and thus without the use of Theorem \ref{thm:Si}. In the following theorem (and throughout this manuscript later on) we recall our 
simplifying convention \eqref{2.prod} to abbreviate finite operator products $A_1 A_2 \cdots A_N$ 
by $\prod_{j=1}^N A_j$, regardless of underlying noncommutativity issues, upon relying on ideal properties of the bounded operators $A_j$, $j=1,\dots,N$, $N \in \bbN$.  

\begin{theorem}
\label{thm:Simon_Hilbert-Schmidt} Let $n=2\hatt n+1\in\mathbb{N}_{\geq3}$ odd, $\Psi_{1},\ldots,\Psi_{\hatt n+1}\in C_{b}^{2}\left(\mathbb{R}^{n}\right),$
$\alpha_{1},\ldots,\alpha_{\hatt n+1}\in[1,\infty),$ $\epsilon> 1/2$,
$\kappa>0$, $\mu\in\mathbb{C}_{\Re>0}$, and let $R_{\mu}$,
$Q$, and $\left[\Psi_{\hatt n},R_{\mu}\right]$ be given by \eqref{eq:resolvent_of_laplace},
\eqref{eq:Def_of_Q}, and \eqref{eq:def_commutator}, respectively. \\[1mm] 
$(i)$ Assume that for all $x\in\mathbb{R}^{n}$ and $j\in\{1,\ldots,\hatt n+1\}$, 
\[
\left|\Psi_{j}(x)\right|\leq \kappa (1+\left|x\right|)^{- \alpha_{j}}.
\]
Then 
\[
\prod_{j=1}^{\hatt{n}+1}\Psi_{j}R_{\mu},\prod_{j=1}^{\hatt{n}+1}R_{\mu}\Psi_{j}
\in\cB_2\big(L^2(\bbR^n)\big),
\]
and
\[
\bigg\Vert \prod_{j=1}^{\hatt{n}+1}\Psi_{j}R_{\mu}\bigg\Vert _{\cB_2(L^2(\bbR^n))}
\leq \prod_{j=1}^{\hatt{n}+1}\left\Vert \Psi_{j}R_{\mu}\right\Vert _{\cB_{n+1}(L^2(\bbR^n))}.
\]
$(ii)$ Assume for all $x\in\mathbb{R}^{n}$ and $j\in\{1,\ldots,\hatt{n}-1\}$,
\[
\left|\Psi_{j}(x)\right|\leq \kappa (1+\left|x\right|)^{- \alpha_{j}}, 
\]
and 
\[
\left|\left(Q\Psi_{\hatt{n}}\right)(x)\right|+\left|\left(Q^{2}\Psi_{\hatt{n}}\right)(x)\right|\leq \kappa 
(1+\left|x\right|)^{- \alpha_{\hatt{n}} - \epsilon}.
\]
Then
\[
\prod_{j=1}^{\hatt{n}-1}\Psi_{j}R_{\mu}\left[\Psi_{\hatt{n}},R_{\mu}\right]\in\cB_2\big(L^2(\bbR^n)\big).
\]
\end{theorem}
\begin{proof}
In order to prove parts $(i)$ and $(ii)$, we use Theorem \ref{thm:trace-class-crit}
and Lemma \ref{lem:Schatten-class-1-operator}. For part $(i)$ one observes 
that $\Psi_{j}R_{\mu}\in\cB_{n+1}$ by Lemma \ref{lem:Schatten-class-1-operator}\,$(i)$ 
for all $j\in\{1,\ldots,\hatt{n}+1\}$. Moreover, by $\sum_{j=1}^{\hatt{n}+1}\frac{1}{n+1}=\left(\frac{n-1}{2}+1\right)\frac{1}{n+1}= 1/2$
one concludes with the help of Theorem \ref{thm:trace-class-crit} that
$\prod_{j=1}^{\hatt{n}+1}\Psi_{j}R_{\mu}\in\cB_2$. 

In order to arrive at item $(ii)$, one notes that the case $n=3$ directly follows from
Lemma \ref{lem:Schatten-class-1-operator}\,$(iii)$ since in that case $\hatt{n}-1=0$.
For $n>3$ there exists $\theta\in(3/4,1$) such that 
$\left[\Psi_{\hatt{n}},R_{\mu}\right]\in\cB_{2n \theta/3}$.
The assertion is clear if $2 n \theta/3 \leq 2$.
Thus, we assume that $2n \theta/3 >2$. Let $q\in\mathbb{R}\backslash \{0\}$
be such that 
\begin{equation}\label{e:3.6}
   \left(\hatt{n}-1\right)\frac{1}{q}+\frac{1}{(2 n/3)\theta}=\frac{1}{2}.
\end{equation}
Equation \eqref{e:3.6} with $\hatt{n}-1= (n-3)/2$ reveals 
\[
\frac{1}{q}=\left(\frac{n\theta-3}{n\theta}\right)\frac{1}{n-3}=\left(\frac{n\theta-3\theta}{n\theta}\right)\frac{1}{n-3}+3\left(\frac{\theta-1}{\theta}\right)\frac{1}{n(n-3)}<\frac{1}{n}.
\]
Hence, $q>n$ and, so, from $\Psi_{j}R_{\mu}\in\cB_{q}$, by Lemma
\ref{lem:Schatten-class-1-operator}, the assertion follows from Theorem
\ref{thm:trace-class-crit}. 
\end{proof}

In order to illustrate the latter mechanism and for later purposes, we now discuss an example.

\begin{example}\label{exa:standard_example_0} Let $z > -1$, and 
$\Phi\in C^\infty_b(\mathbb{R}^3;\mathbb{C}^{2\times 2})$ such that for $x\in \mathbb{R}^3$,  with $|x|\geq 1$, 
\[
   \Phi(x)=\sum_{j=1}^3 \frac{x_j}{|x|}\sigma_{j}.
\]
Here $\sigma_{1}\coloneqq\begin{pmatrix} 0 & 1 \\ 1& 0 \end{pmatrix}$, $\sigma_2\coloneqq \begin{pmatrix} 0 & -i \\ i& 0 \end{pmatrix}$, $\sigma_3\coloneqq \begin{pmatrix} 1 & 0 \\ 0& -1 \end{pmatrix}$ denote the Pauli matrices, see also Definition \ref{def:Euc-D-A}. Recalling our convention \eqref{alpha} and $R_{1+z}=(-\Delta+1+z)^{-1}$, we now prove that the operator 
given by
\begin{align*}
   T\coloneqq & \tr_{4} \bigg(\sum_{k=1}^3 \big((R_{1+z}\sigma_k)\otimes (\partial_k\Phi)\big) 
    \\ 
    & \qquad \times 
    \sum_{k=1}^3 \big((R_{1+z}\sigma_k)\otimes (\partial_k\Phi)\big) 
   \sum_{k=1}^3 \big((R_{1+z}\sigma_k)\otimes (\partial_k\Phi)\big) R_{1+z}\bigg) 
\end{align*}
is trace class, $T \in \cB_1\big(L^2(\bbR^3)\big)$.

First of all, with the help of Proposition \ref{prop:comp_of_Dirac_Alge} and introducing the 
fully anti-symmetric symbol in $3$ coordinates, $\epsilon_{jk\ell}$, $j,k,\ell \in \{1,2,3\}$, we may express $T$ as follows $($for notational simplicity, we now drop all tensor product symbols$)$, 
\begin{align*}
   T&=\sum_{1 \leq k_1,k_2,k_3 \leq 3} \tr_4 \big(\sigma_{k_1}\sigma_{k_2}\sigma_{k_3}  
 R_{1+z} I_2 (\partial_{k_1}\Phi) R_{1+z} I_2  (\partial_{k_2}\Phi) R_{1+z} I_2 (\partial_{k_3}\Phi) R_{1+z} I_2\big) \\
    &=\sum_{1 \leq k_1,k_2,k_3 \leq 3} \tr_2 \left(\sigma_{k_1}\sigma_{k_2}\sigma_{k_3}\right)   \\
& \hspace*{1.5cm}   \times  \tr_2\left(R_{1+z} I_2 (\partial_{k_1}\Phi) R_{1+z} I_2  (\partial_{k_2}\Phi) R_{1+z} I_2 (\partial_{k_3}\Phi) R_{1+z} I_2\right) \\
    &=\sum_{1 \leq k_1,k_2,k_3 \leq 3} 2i \epsilon_{k_1k_2k_3}  
 \tr_2\left(R_{1+z} I_2 (\partial_{k_1}\Phi) R_{1+z} I_2 (\partial_{k_2}\Phi) R_{1+z} I_2 (\partial_{k_3}\Phi) R_{1+z} I_2\right)    \\ 
 &=\sum_{1 \leq k_1,k_2,k_3 \leq 3} 2i \epsilon_{k_1k_2k_3} \tr_2\left(R_{1+z}(\partial_{k_1}\Phi) R_{1+z} (\partial_{k_2}\Phi) R_{1+z} (\partial_{k_3}\Phi) R_{1+z}\right) \\
  &=\sum_{1 \leq k_1,k_2,k_3 \leq 3} 2i \epsilon_{k_1k_2k_3} \tr_2\left(R_{1+z}[(\partial_{k_1}\Phi),R_{1+z}] (\partial_{k_2}\Phi) R_{1+z}
  (\partial_{k_3}\Phi) R_{1+z}\right) \\ &\quad+ \sum_{1 \leq k_1,k_2,k_3 \leq 3} 2i \epsilon_{k_1k_2k_3} \tr_2\left(R_{1+z}R_{1+z} (\partial_{k_1}\Phi) (\partial_{k_2}\Phi) 
  R_{1+z} (\partial_{k_3}\Phi) R_{1+z}\right)\\
  &=\sum_{1 \leq k_1,k_2,k_3 \leq 3} 2i \epsilon_{k_1k_2k_3} \tr_2\left(R_{1+z}[(\partial_{k_1}\Phi),R_{1+z}](\partial_{k_2}\Phi) R_{1+z}
  (\partial_{k_3}\Phi) R_{1+z}\right) \\ &\quad+ \sum_{1 \leq k_1,k_2,k_3 \leq 3} 2i \epsilon_{k_1k_2k_3} \tr_2\left(R_{1+z}R_{1+z} (\partial_{k_1}\Phi) 
  (\partial_{k_2}\Phi) [R_{1+z},(\partial_{k_3}\Phi)] R_{1+z}\right)\\
   &\quad+ \sum_{1 \leq k_1,k_2,k_3 \leq 3} 2i \epsilon_{k_1k_2k_3} \tr_2\left(R_{1+z}R_{1+z}(\partial_{k_1}\Phi) (\partial_{k_2}\Phi) (\partial_{k_3}\Phi) R_{1+z} R_{1+z}\right). 
\end{align*}

One computes for $k\in \{1,2,3\}$ and $x\in \mathbb{R}^3$, $|x|\geq 1$,
\[
(\partial_{k}\Phi)(x)= \sum_{j=1}^3 \bigg(\frac{\delta_{kj}}{|x|}-\frac{x_j}{|x|^2}\frac{x_k}{|x|}\bigg)\sigma_j, 
\]
and observes that $\|(\partial_k\Phi)(x)\|\leq  6/|x|$. Moreover, it is easy to see that for all $\beta\in \mathbb{N}_0^3$, with $|\beta|\coloneqq\sum_{j=1}^3\beta_j\geq 2$, there exists 
$\kappa>0$ such that for all $x\in \mathbb{R}^n$, 
\[
\|(\partial^\beta\Phi)(x)\|\leq  \kappa (1+|x|)^{-2}.
\]
The latter estimate together with Theorem \ref{thm:Simon_Hilbert-Schmidt} yields that 
$T\in \cB_1\big(L^2(\bbR^3)\big)$ if and only if
\begin{align*} 
 \tilde T\coloneqq \sum_{1 \leq k_1,k_2,k_3 \leq 3} 2i \epsilon_{k_1k_2k_3} \tr_2\left(R_{1+z}R_{1+z} (\partial_{k_1}\Phi) (\partial_{k_2}\Phi) (\partial_{k_3}\Phi) R_{1+z} R_{1+z}\right)& \\ 
 \in \cB_1\big(L^2(\bbR^3)\big)&.   
\end{align*} 
The latter operator can be rewritten as
\[
  \tilde T =  \sum_{1 \leq k_1,k_2,k_3 \leq 3} 2i \epsilon_{k_1k_2k_3} R_{1+z}^2\tr_2\left((\partial_{k_1}\Phi) (\partial_{k_2}\Phi) (\partial_{k_3}\Phi)\right) R_{1+z}^2.
\]
Next, we inspect the term in the middle in more detail:
\begin{align*}
& \sum_{1 \leq k_1,k_2,k_3 \leq 3} \epsilon_{k_1k_2k_3} \tr_2\left((\partial_{k_1}\Phi) (\partial_{k_2}\Phi) (\partial_{k_3}\Phi)\right) \\
& \quad = \tr_2\left((\partial_{1}\Phi) (\partial_{2}\Phi) (\partial_{3}\Phi)\right) - \tr_2\left((\partial_{1}\Phi) (\partial_{3}\Phi) (\partial_{2}\Phi)\right) 
+ \tr_2\left((\partial_{2}\Phi) (\partial_{3}\Phi) (\partial_{1}\Phi)\right) \\
  & \qquad - \tr_2\left((\partial_{2}\Phi) (\partial_{1}\Phi) (\partial_{3}\Phi)\right) + \tr_2\left((\partial_{3}\Phi) (\partial_{1}\Phi) (\partial_{2}\Phi)\right) 
  - \tr_2\left((\partial_{3}\Phi) (\partial_{2}\Phi) (\partial_{1}\Phi)\right) \\
  & \quad = 3\tr_2\left((\partial_{1}\Phi) (\partial_{2}\Phi) (\partial_{3}\Phi)\right) - 3\tr_2\left((\partial_{1}\Phi) (\partial_{3}\Phi) (\partial_{2}\Phi)\right).
\end{align*}
Employing 
\begin{align*}
& (\partial_{1}\Phi) (\partial_{2}\Phi) (\partial_{3}\Phi)(x)     \\
 & \quad = \sum_{1 \leq j_1,j_2,j_3 \leq 3} \left(\frac{\delta_{1j_1}}{|x|}-\frac{x_{j_1}}{|x|^2}\frac{x_1}{|x|}\right)\sigma_{j_1}\left(\frac{\delta_{2j_2}}{|x|}-\frac{x_{j_2}}{|x|^2}\frac{x_2}{|x|}\right)\sigma_{j_2}\left(\frac{\delta_{3j_3}}{|x|}-\frac{x_{j_3}}{|x|^2}\frac{x_3}{|x|}\right)\sigma_{j_3},
\end{align*} 
one gets 
\begin{align*}
   & \frac{1}{2i} |x|^3\tr_2 \left((\partial_{1}\Phi) (\partial_{2}\Phi) (\partial_{3}\Phi)\right)(x) \\ 
   & \quad = \sum_{1 \leq j_1,j_2,j_3 \leq 3}  \epsilon_{j_1j_2j_3}  \left(\delta_{1j_1}-\frac{x_{j_1}x_1}{|x|^2}\right)\left(\delta_{2j_2}-\frac{x_{j_2}x_2}{|x|^2}\right)\left(\delta_{3j_3}-\frac{x_{j_3}x_3}{|x|^2}\right)\\
   & \quad = \left(1-\frac{x_{1}x_1}{|x|^2}\right)\left(1-\frac{x_{2}x_2}{|x|^2}\right)\left(1-\frac{x_{3}x_3}{|x|^2}\right) \\
   &\qquad-\left(1-\frac{x_{1}x_1}{|x|^2}\right)\left(-\frac{x_{3}x_2}{|x|^2}\right)\left(-\frac{x_{2}x_3}{|x|^2}\right) 
 +\left(-\frac{x_{2}x_1}{|x|^2}\right)\left(-\frac{x_{3}x_2}{|x|^2}\right)\left(-\frac{x_{1}x_3}{|x|^2}\right) \\
   &\qquad-\left(-\frac{x_{2}x_1}{|x|^2}\right)\left(-\frac{x_{1}x_2}{|x|^2}\right)\left(1-\frac{x_{3}x_3}{|x|^2}\right) 
 +\left(-\frac{x_{3}x_1}{|x|^2}\right)\left(-\frac{x_{1}x_2}{|x|^2}\right)\left(-\frac{x_{2}x_3}{|x|^2}\right) \\
   &\qquad-\left(-\frac{x_{3}x_1}{|x|^2}\right)\left(1-\frac{x_{2}x_2}{|x|^2}\right)\left(-\frac{x_{1}x_3}{|x|^2}\right)\\
   & \quad =1-\frac{x_1^2}{|x|^2}-\frac{x_2^2}{|x|^2}-\frac{x_3^2}{|x|^2}\\
   &\qquad +\frac{x_1^2x_2^2}{|x|^4}+\frac{x_1^2x_3^2}{|x|^4}+\frac{x_2^2x_3^2}{|x|^4} - \frac{x_3x_2x_2x_3}{|x|^4}-\frac{x_2x_1x_1x_2}{|x|^4}-\frac{x_3x_1x_1x_3}{|x|^4} \\
   & \qquad -\frac{x_1^2x_2^2x_3^2}{|x|^6}+\frac{x_1^2x_2^2x_3^2}{|x|^6}-\frac{x_1^2x_2^2x_3^2}{|x|^6}+\frac{x_1^2x_2^2x_3^2}{|x|^6}-\frac{x_1^2x_2^2x_3^2}{|x|^6}+\frac{x_1^2x_2^2x_3^2}{|x|^6}\\
    & \quad =0.
\end{align*}
On the other hand,
\begin{align*}
   & \frac{1}{2i} |x|^3\tr_2 \left((\partial_{1}\Phi) (\partial_{3}\Phi) (\partial_{2}\Phi)\right)(x) \\
   & \quad = \sum_{1 \leq j_1,j_2,j_3 \leq 3}  \epsilon_{j_1j_2j_3}  \left(\delta_{1j_1}-\frac{x_{j_1}x_1}{|x|^2}\right)\left(\delta_{3j_2}-\frac{x_{j_2}x_3}{|x|^2}\right)\left(\delta_{2j_3}-\frac{x_{j_3}x_2}{|x|^2}\right)=0.
\end{align*}
We note that in this example, the corresponding formula \eqref{eq:wrong_cancellation} is in fact 
valid, for $|x|\geq 1$. This example is of a similar type as in \cite[Section~IV]{Ca78}. This may 
well be the reason for the erroneous statement in  \cite[p.~226, 2nd highlighted formula]{Ca78}.
\end{example}

We shall also use on occasion the following Hilbert--Schmidt criterion for exactly $\hatt n=(n-1)/2$ factors:

\begin{theorem}\label{thm:HS-criterion for m factors} Let $n=2\hatt{n}+1\in \mathbb{N}_{\geq 3}$ odd, and assume that $\Psi_{1},\ldots,\Psi_{\hatt{n}}\in L^\infty(\mathbb{R}^n)$, $\alpha_1,\ldots,\alpha_{\hatt{n}}\in [1,\infty)$, $\mu\in \mathbb{C}$, $\Re(\mu)>0$. Let $R_\mu$, $Q$ be given by \eqref{eq:resolvent_of_laplace} and \eqref{eq:Def_of_Q}, respectively. Assume that 
$\alpha_{j^*} > 3/2$ for some 
$j^*\in\{1,\ldots,\hatt{n}\}$. Then 
\[
  T\coloneqq  \bigg(\prod_{j=1}^{j^*-1} \Psi_jR_\mu\bigg)\Psi_{j^*}R_\mu 
  \bigg(\prod_{j=j^*+1}^{\hatt{n}} \Psi_jR_\mu\bigg)\in \mathcal{B}_2\left(L^2(\mathbb{R}^n)\right),  
\]
and
\begin{align*} 
& \|T\|_{\mathcal{B}_2(L^2(\bbR^n))}     \\ 
& \quad  \leq  \begin{cases} \big(\prod_{j\in\{1,\ldots,\hatt{n}\}\backslash \{j^*\}} \|\Psi_jR_\mu\|_{\mathcal{B}_q(L^2(\bbR^n))}\big)\|\Psi_j^*R_\mu\|_{\mathcal{B}_{r}(L^2(\bbR^n))}, & \hatt{n}>1,\\
                              \|\Psi_1R_\mu\|_{\mathcal{B}_{2}(L^2(\bbR^n))}, &\hatt{n}=1,
                             \end{cases}
\end{align*} 
where $q= 2(\hatt{n}-1)\theta/(n\theta-3) > n$ and $r=2n\theta/3>2$ for some 
$\theta\in (3/4,1)$, according to Lemma \ref{lem:Schatten-class-1-operator}\,$(ii)$.

The assertion is the same if some of the factors with index $j\in \{1,\ldots,\hatt{n}\}\backslash \{j^*\}$ in the expression for $T$ are replaced by $\Psi_jQR_\mu$. 
\end{theorem}
\begin{proof} By Lemma \ref{lem:Schatten-class-1-operator}\,$(ii)$ one observes that for $\hatt{n}=j^*=1$, $\Psi_{1}R_\mu\in \mathcal{B}_2\big(L^2(\bbR^n)\big)$, and the assertion follows. The rest of the proof is similar to the one of the concluding lines of Theorem \ref{thm:Simon_Hilbert-Schmidt}.
\end{proof}

\newpage

\section{Pointwise Estimates for Integral Kernels}\label{sec:ptw_intk}

The proof of the index theorem relies on (pointwise) estimates of
integral kernels of certain integral operators. These integral operators are
of a form similar to the one in Theorem \ref{thm:Simon_Hilbert-Schmidt}.
In order to guarantee that point-evaluation is a well-defined operation, 
these operators have to possess certain smoothing properties. Before proving the 
corresponding result, we define the
Dirac $\delta$-distribution of point-evaluation at some point $x\in\mathbb{R}^{n}$
of a suitable function $f\colon\mathbb{R}^{n}\to\mathbb{C}$ by 
$$
\delta_{\{x\}}f\coloneqq f(x).
$$
We note that for every $x\in\mathbb{R}^{n}$ one has $\delta_{\{x\}}\in H^{-(n/2)-\epsilon}(\mathbb{R}^{n})$ for all $\epsilon>0$, by the Sobolev embedding theorem 
(see, e.g., \cite[Theorem 7.34$(c)$]{AF03}), and recall that 
\begin{equation}
H^{s}(\mathbb{R}^{n})\coloneqq\big\{f \in L^2(\bbR^n) \,\big| \, (1+|\cdot|^{2})^{s/2}(\cF f)
\in L^{2}(\mathbb{R}^{n})\big\}, \quad s\in\mathbb{R},   \label{eq:def_Hs}
\end{equation}
with norm denoted by $\| \cdot \|_{H^s(\bbR^n)}$, where $\mathcal{F}$ denotes the (distributional) Fourier transform being an extension of 
\begin{equation}
(\cF \phi)(x)\coloneqq (2\pi)^{-n/2}\int_{\mathbb{R}^{n}}e^{-ixy}\phi(y) \, d^n y, \quad x\in\mathbb{R}^{n}, \; \phi\in L^{1}(\mathbb{R}^{n})\cap L^{2}(\mathbb{R}^{n}).\label{eq:def_Fourier_transf}
\end{equation}
For $f\in H^{\frac{n}{2}+\epsilon}(\mathbb{R}^{n})$, we will find it convenient to write
\begin{equation}
\left\langle \delta_{\{x\}},f\right\rangle _{L^{2}(\mathbb{R}^{n})}\coloneqq\delta_{\{x\}}f=f(x),\label{eq:def_dirac_distr}
\end{equation}
where $\langle \cdot \, ,\cdot\rangle_{L^{2}(\mathbb{R}^{n})}$ is understood
as the continuous extension of the scalar product on $L^{2}(\mathbb{R}^{n})$
to the pairing on the entire fractional-order Sobolev scale, 
$\langle\cdot \, ,\cdot\rangle_{L^{2}(\mathbb{R}^{n})} 
:= \langle\cdot \, ,\cdot\rangle_{H^{-s}(\bbR^n), H^s(\bbR^n)}$, $s \geq 0$. 

For convenience of the reader we now prove the following known result: 

\begin{theorem}
\label{thm:continuity of integral kernel} Let $n\in\mathbb{N}$,
$\epsilon>0$, $T\colon H^{-(n/2)-\epsilon}(\mathbb{R}^{n})\to 
H^{(n/2)+\epsilon}(\mathbb{R}^{n})$ linear and bounded $($cf.\ \eqref{eq:def_Hs} and \eqref{eq:def_dirac_distr}$)$. 
Then the map  
\[
t\colon\mathbb{R}^{n}\times\mathbb{R}^{n}\ni(x,y)\mapsto 
t(x,y) = \left\langle \delta_{\{x\}},T\delta_{\{y\}}\right\rangle _{L^{2}(\mathbb{R}^{n})}\in\mathbb{C}
\]
is well-defined, continuous, and bounded. Moreover, if 
\begin{equation} 
T\in \cB\big(H^{-(n/2)-1-\epsilon}(\mathbb{R}^{n}),H^{(n/2)+\epsilon}(\mathbb{R}^{n})\big)\cap \cB\big(H^{-(n/2)-\epsilon}(\mathbb{R}^{n}),H^{(n/2)+1+\epsilon}(\mathbb{R}^{n})\big),
\end{equation} 
then $t$ is bounded and continuously differentiable with bounded derivatives $\partial_j t$, 
$j \in \{1,\dots, 2n\}$.
\end{theorem}

\begin{remark}
\label{rem:Convenient:Green} We note that with the
maps and assumptions introduced in Theorem \ref{thm:continuity of integral kernel}, 
$t(\cdot,\cdot)$ is in fact the \emph{integral kernel} of $T$, that is, $t$ satisfies  
\[
(Tf)(x)=\int_{\mathbb{R}^{n}}t(x,y)f(y) \, d^n y, \quad x\in\mathbb{R}^{n}, 
\]
for all $f\in C_0^{\infty}(\mathbb{R}^{n})$. Indeed, let $x\in\mathbb{R}^{n}$
and $f\in C_0^{\infty}(\mathbb{R}^{n})$. Then one computes 
\begin{align*}
(Tf)(x) & =\left\langle \delta_{\{x\}},Tf\right\rangle _{L^{2}\left(\mathbb{R}^{n}\right)}=\left\langle T^{*}\delta_{\{x\}},f\right\rangle _{L^{2}\left(\mathbb{R}^{n}\right)} \\
 & =\int_{\mathbb{R}^{n}}\overline{T^{*}\delta_{\{x\}}(y)}f(y) \, d^n y 
 =\int_{\mathbb{R}^{n}}\overline{\left\langle \delta_{\{y\}},T^{*}\delta_{\{x\}}\right\rangle _{L^{2}\left(\mathbb{R}^{n}\right)}}f(y) \, d^n y \\
 & 
 =\int_{\mathbb{R}^{n}}\overline{\left\langle T\delta_{\{y\}},\delta_{\{x\}}\right\rangle _{L^{2}\left(\mathbb{R}^{n}\right)}}f(y) \, d^n y 
=\int_{\mathbb{R}^{n}}\left\langle \delta_{\{x\}},T\delta_{\{y\}}\right\rangle _{L^{2}\left(\mathbb{R}^{n}\right)}f(y) \, d^n y  \\
 & =\int_{\mathbb{R}^{n}}t(x,y)f(y) \, d^n y.\qquad\qquad\qquad\quad  
\end{align*}
\hfill $\diamond$ 
\end{remark}

\begin{proof}[Proof of Theorem \ref{thm:continuity of integral kernel}]
First, one observes that for any $y\in\mathbb{R}^{n}$, 
$\delta_{\{y\}}\in H^{-\frac{n}{2}-\epsilon}(\mathbb{R}^{n})$ and 
$\mathcal{F}\delta_{\{y\}}(x)= (2 \pi)^{-n/2}e^{ixy}$
for all $x,y\in\mathbb{R}^{n}$. Consequently, 
$T\delta_{\{y\}}\in H^{\frac{n}{2}+\epsilon}(\mathbb{R}^{n})$,
and hence $\langle\delta_{\{x\}},T\delta_{\{y\}}\rangle_{L^{2}(\mathbb{R}^{n})}$
is well-defined for all $x,y\in\mathbb{R}^{n}$. Moreover, from 
\begin{align*} 
\left|\langle\delta_{\{x\}},T\delta_{\{y\}}\rangle_{L^{2}(\mathbb{R}^{n})}\right| & \leq \left|\delta_{\{x\}}\right|_{-\frac{n}{2}-\epsilon}\left|T\delta_{\{y\}}\right|_{\frac{n}{2}+\epsilon}\leq \left|\delta_{\{0\}}\right|_{-\frac{n}{2}-\epsilon}\left|\delta_{\{0\}}\right|_{-\frac{n}{2}-\epsilon}  \\
& \quad \times 
\left\Vert T\right\Vert_{\cB\big(H^{-\frac{n}{2}-\epsilon}(\mathbb{R}^{n}), H^{\frac{n}{2}+\epsilon}(\mathbb{R}^{n})\big)}, \quad x,y\in\mathbb{R}^{n},
\end{align*} 
one concludes the boundedness of $t$. Next, we show sequential continuity
of $t$. One observes that by the Sobolev embedding theorem,
the map $T\delta_{\{y\}}$ is continuous for all $y\in\mathbb{R}^{n}$
(One recalls that $\mathcal{F}T\delta_{\{y\}}\in L^{1}(\mathbb{R}^{n})$
with the Fourier transform given by \eqref{eq:def_Fourier_transf}.) 
Let $\{(x_{k},y_{k})\}_{k\in\bbN}$ be a convergent sequence in 
$\mathbb{R}^{n}\times\mathbb{R}^{n}$
and denote its limit as $k \to \infty$ by $(x,y).$ One notes that 
$\left|\delta_{\{y\}}-\delta_{\{y_{k}\}}\right|_{-\frac{n}{2}-\epsilon}\to0,$
as $k\to\infty$. Indeed, one gets by Lebesgue's dominated convergence
theorem that
\begin{align*}
\left|\delta_{\{y\}}-\delta_{\{y_{k}\}}\right|_{-\frac{n}{2}-\epsilon}^{2} & =\left(2\pi\right)^{-n}\int_{\mathbb{R}^{n}}\left|\left(e^{ixy}-e^{ixy_{k}}\right)\right|^{2} 
\big[1+\left|x\right|^{2}\big]^{-(n/2) - \epsilon}\, d^n x \underset{k\to \infty}{\longrightarrow} 0.
\end{align*}
Moreover, one observes that $\{\delta_{\{x_{k}\}}\}_{k\in\bbN}$ is uniformly bounded
in $H^{-\frac{n}{2}-\epsilon}(\mathbb{R}^{n})$ by some constant $M$.
Next, let $\eta>0$ and choose $k_{0}\in\mathbb{N}$ such that for $k\geq k_{0}$
one has $\left|\delta_{\{y\}}-\delta_{\{y_{k}\}}\right|_{-\frac{n}{2}-\epsilon}\leq\eta$
and $\left|T\delta_{\{y\}}(x_{k})-T\delta_{\{y\}}(x)\right|\leq\eta$.  
Then one estimates for $k\in\mathbb{N}$, 
\begin{align*}
 & \left|\langle\delta_{\{x_{k}\}},T\delta_{\{y_{k}\}}\rangle_{L^{2}(\mathbb{R}^{n})}-\langle\delta_{\{x\}},T\delta_{\{y\}}\rangle_{L^{2}(\mathbb{R}^{n})}\right|\\
 & \quad \leq \left|\langle\delta_{\{x_{k}\}},T\delta_{\{y_{k}\}}\rangle_{L^{2}(\mathbb{R}^{n})}-\langle\delta_{\{x_{k}\}},T\delta_{\{y\}}\rangle_{L^{2}(\mathbb{R}^{n})}\right|   \\
 & \qquad + \left|\langle\delta_{\{x_{k}\}},T\delta_{\{y\}}\rangle_{L^{2}(\mathbb{R}^{n})}-\langle\delta_{\{x\}},T\delta_{\{y\}}\rangle_{L^{2}(\mathbb{R}^{n})}\right|\\
 & \quad \leq \left|\delta_{\{x_{k}\}}\right|_{-\frac{n}{2}-\epsilon}\left|T\delta_{\{y_{k}\}}-T\delta_{\{y\}}\right|_{\frac{n}{2}+\epsilon}+\left|T\delta_{\{y\}}(x_{k})-T\delta_{\{y\}}(x)\right|\\
 & \quad \leq  M\left\Vert T\right\Vert \left|\delta_{\{y\}}-\delta_{\{y_{k}\}}\right|_{-\frac{n}{2}-\epsilon}+\eta\leq \left(M\left\Vert T\right\Vert +1\right)\eta.
\end{align*}
Next, we turn to the second part of the theorem. Since $H^{\frac{n}{2}+\epsilon+1}(\mathbb{R}^{n})\hookrightarrow H^{\frac{n}{2}+\epsilon}(\mathbb{R}^{n})$,
the map $t$ is continuous by the first part of the theorem. To prove 
differentiability, it suffices to observe that $t$ has continuous
(weak) partial derivatives. Since $T\partial_{j}\in \cB\big(H^{-\frac{n}{2}-\epsilon}(\mathbb{R}^{n}),H^{\frac{n}{2}+\epsilon}(\mathbb{R}^{n})\big)$
and $\partial_{j}T\in \cB\big(H^{-\frac{n}{2}-\epsilon}(\mathbb{R}^{n}),H^{\frac{n}{2}+\epsilon}(\mathbb{R}^{n})\big)$ for all $j\in\{1,\ldots,n\}$, the assertion also follows from the
first part as one observes that 
\begin{align*} 
& (\partial_{j}t)(x,y)=\big\langle \left(\partial_{j}\delta\right)_{\{x\}},T\delta_{\{y\}}
\big\rangle_{L^2(\bbR^n)} 
=-\left\langle \delta_{\{x\}},\partial_{j}T\delta_{\{y\}}\right\rangle_{L^2(\bbR^n)},     \\
& (\partial_{j+n}t)(x,y)=\langle\delta_{\{x\}},T\partial_j\delta_{\{y\}}\rangle, \quad 
j\in\{1,\ldots,n\}, \; (x,y)\in\mathbb{R}^{n}\times\mathbb{R}^{n}. \qedhere
\end{align*} 
 \end{proof}
 
 In the applications discussed later on, we shall be confronted with operators being already defined (or being extendable) to the whole fractional Sobolev scale. So the standard situation in which we will apply Theorem \ref{thm:continuity of integral kernel} is summarized in the following corollary, with some examples in the succeeding proposition. 
 
 \begin{corollary} 
 \label{cor:integral kernels regularity}Let $n\in\mathbb{N}$, $k>n$,
$S,T\colon\bigcup_{\ell\in\mathbb{Z}}H^{\ell}
(\mathbb{R}^{n})\to\bigcup_{\ell\in\mathbb{Z}}H^{\ell}(\mathbb{R}^{n})$.
Assume that for all $\ell\in\mathbb{R}$, 
$T\in \cB\big(H^{\ell}(\mathbb{R}^{n}),H^{\ell+k}(\mathbb{R}^{n})\big)$
and $S\in \cB\big(H^{\ell}(\mathbb{R}^{n}),H^{\ell+k+1}(\mathbb{R}^{n})\big)$,  
and introduce the maps
\begin{align*}
& t\colon\mathbb{R}^{n}\times\mathbb{R}^{n}\ni(x,y)\mapsto\left\langle 
\delta_{\{x\}},T\delta_{\{y\}}\right\rangle_{L^2(\bbR^n)},     \\
& s\colon\mathbb{R}^{n}\times\mathbb{R}^{n}\ni(x,y)\mapsto\left\langle 
\delta_{\{x\}},S\delta_{\{y\}}\right\rangle_{L^2(\bbR^n)}. 
\end{align*} 
Then $t$ is bounded and continuous, and the map $s$ is bounded and 
continuously differentiable with bounded derivatives. 
$($See \eqref{eq:def_Hs} and \eqref{eq:def_dirac_distr}
for $H^{s}\left(\mathbb{R}^{n}\right)$ and $\delta_{\{x\}}$, 
$x\in\mathbb{R}^{n}$,
$s\in\mathbb{R}$.$)$
\end{corollary}
\begin{proof}
This is a direct consequence of Theorem \ref{thm:continuity of integral kernel}.
\end{proof}

\begin{proposition}
\label{prop:concrete case} Let $\mu\in\mathbb{C},$ $\Re(\mu)>0$,
$\ell\in\mathbb{R}$, and $\Phi\in C_{b}^{\infty}(\mathbb{R}^{n}).$  
Then $R_{\mu}=\left(-\Delta+\mu\right)^{-1}\in \cB\big(H^{\ell}(\mathbb{R}^{n}),H^{\ell+2}(\mathbb{R}^{n})\big)$
and%
\footnote{We recall Remark \ref{rem:differen_mult_op}: The symbol $\Phi$ is interpreted as
the operator of multiplication by the function $\Phi$. If $\ell<0$,
this should read as multiplication in the distributional sense. %
} $\Phi\in \cB\big(H^{\ell}(\mathbb{R}^{n})\big)$.
\end{proposition}
\begin{proof}
For $\ell\in\mathbb{Z}$, the first assertion follows easily with
the help of the Fourier transform, the second assertion is a
straightforward induction argument for $\ell\in\mathbb{N}_{0}$; for
$\ell\in-\mathbb{N}$ the result follows by duality. The results for
$\ell\in\mathbb{R}$ follow by interpolation, see \cite[Theorem 2.4.2]{Tr78}.
\end{proof}

The main issue of the considerations in this section are estimates of continuous integral kernels on the respective diagonals. An elementary estimate can be shown for integral operators which are induced by commutators with multiplication operators as the following result confirms.

\begin{proposition}
\label{prop:commutator with phi vanishes}Let $n\in\mathbb{N}$, $\epsilon>0$,
$m\in\mathbb{N}$, and assume that $T\colon H^{-(n/2)-\epsilon}(\mathbb{R}^{n})\to H^{(n/2)+\epsilon}(\mathbb{R}^{n})$ is linear and continuous, and that 
$\Phi\in C_{b}^{\infty}(\mathbb{R}^{n})$. Then the map 
\[
t_{\Phi}\colon\mathbb{R}^{n}\times\mathbb{R}^{n}\ni(x,y)\mapsto\left\langle \delta_{\{x\}},\left[\Phi,T\right]\delta_{\{y\}}\right\rangle _{L^{2}(\mathbb{R}^{n})}\in\mathbb{C},
\]
where $\left[\Phi,T\right]$ is given by \eqref{eq:def_commutator},
is well-defined, continuous, bounded, and satisfies $t_{\Phi}(x,x)=0$, 
$x\in\mathbb{R}^{n}$.
\end{proposition}
\begin{proof}
By Proposition \ref{prop:concrete case} and Theorem \ref{thm:continuity of integral kernel},
one gets that $t_{\Phi}$ is well-defined, continuous, and bounded. For
$x\in\mathbb{R}^{n}$ one then computes
\begin{align*}
\left\langle \delta_{\{x\}},\left[\Phi,T\right]\delta_{\{y\}}\right\rangle  & =\left\langle \delta_{\{x\}},\left(\Phi T-T\Phi\right)\delta_{\{x\}}\right\rangle \\
 & =\left\langle \delta_{\{x\}},\Phi T\delta_{\{x\}}\right\rangle -\left\langle \delta_{\{x\}},\left(T\Phi\right)\delta_{\{x\}}\right\rangle \\
 & =\left\langle \Phi^{*}\delta_{\{x\}},T\delta_{\{x\}}\right\rangle -\big\langle \delta_{\{x\}},T\left(\Phi\delta\right)_{\{x\}}\big\rangle \\
 & =\big\langle \overline{\Phi(x)}\delta_{\{x\}},T\delta_{\{x\}}\big\rangle 
 - \left\langle \delta_{\{x\}},T\Phi(x)\delta_{\{x\}}\right\rangle \\
 & =\left\langle \delta_{\{x\}},\Phi(x)T\delta_{\{x\}}\right\rangle -\left\langle \delta_{\{x\}},T\Phi(x)\delta_{\{x\}}\right\rangle =0.\tag*{{\qedhere}}
\end{align*}
\end{proof}

The next lemma also discusses properties of the integral kernel of a commutator, however, in the following situation, we shall address the commutator with differentiation. 

\begin{lemma}  \label{lem:comm is sum of der} 
Let $T \in \cB\big(L^{2}(\mathbb{R}^{n})\big)$ be induced by the continuously 
differentiable integral kernel
$t\colon\mathbb{R}^{n}\times\mathbb{R}^{n}\to\mathbb{C}$, $j\in\{1,\ldots,n\}$.
If $[\partial_{j},T]  \in \cB\big(L^{2}(\mathbb{R}^{n})\big)$, then $[\partial_{j},T]$ defined
by \eqref{eq:def_commutator} is an operator induced by the integral kernel
$\partial_{j}t+\partial_{j+n}t$. 
\end{lemma}
\begin{proof}
Let $x\in\mathbb{R}^{n}$ and $f\in C_0^{\infty}(\mathbb{R}^{n})$.
One computes 
\begin{align*}
([\partial_{j},T]f)(x) & = (\partial_{j}Tf)(x) - (T\partial_{j}f)(x)   \\
 & =\partial_{j}\int_{\mathbb{R}^{n}}t(x,y)f(y) \, d^n y 
 - \int_{\mathbb{R}^{n}}t(x,y) (\partial_{j}f)(y) \, d^n y\\
 & =\int_{\mathbb{R}^{n}}  (\partial_{j}t)(x,y)f(y) \, d^n y 
 +\int_{\mathbb{R}^{n}} (\partial_{j+n}t)(x,y)f(y) \, d^n y,
\end{align*}
using an integration by parts to arrive at the last equality.
\end{proof}

\begin{remark} \lb{rT}
We elaborate on an important consequence of Lemma \ref{lem:comm is sum of der} as follows: 
For $j\in \{1,\ldots,n\}$, let $T_j\in \cB(L^2(\mathbb{R}^n))$ be induced by the continuously differentiable integral kernel $t_j\colon \mathbb{R}^n\times \mathbb{R}^n\to\mathbb{C}$. Assume that $[\partial_j,T_j]\in \cB(L^2(\mathbb{R}^n))$, $j\in \{1,\ldots,n\}$, and consider the operator
\[
  T\coloneqq \sum_{j=1}^n [\partial_j,T_j]. 
\]
By Lemma \ref{lem:comm is sum of der} one infers that the integral kernel $t$ for $T$ may be computed as follows, 
\[
 t(x,y) = \sum_{j=1}^n (\partial_jt_j+\partial_{j+n}t_j)(x,y),\quad x,y\in \mathbb{R}^n.
\]
Moreover, for $g\coloneqq \{ x\mapsto t_j(x,x)\}_{j\in\{1,\ldots,n\}}$, 
\begin{align} \label{e:div_theo}
  t(x,x) &= \sum_{j=1}^n (\partial_j(y\mapsto t_j(y,y)))(x)\\
         &=\textrm{div} \, (g(x)), \quad x\in \mathbb{R}^n.    \no
\end{align}
This observation will turn out to be useful when computing the trace of certain operators.
\hfill $\diamond$ 
\end{remark}

The remaining section is devoted to obtaining pointwise estimates
of various integral operators on the diagonal. For convenience, we recall the 
$\Gamma$-function (cf.\ \cite[Sect.\ 6.1]{AS72}), given by
\[
   \Gamma(z)\coloneqq \int_0^\infty t^{z-1}e^{-t}\, dt, \quad z\in \mathbb{C}_{\Re>0},
\]
as well as the $n-1$-dimensional volume of the unit sphere $S^{n-1}\subseteq \mathbb{R}^n$, 
\begin{equation}\label{e:om}
   \omega_{n-1} = \frac{2\pi^{(n/2)}}{\Gamma\big(n/2\big)}.
\end{equation}

\begin{proposition}
\label{prop:estimate of the pointwise resolvent at 0}Let $n,m\in\mathbb{N}$,
$m> (n+1)/2$, $\mu\in\mathbb{C}$, $\Re(\mu)>0$, and $R_{\mu}$, 
$\delta_{\{0\}}$, and $Q$ be given by \eqref{eq:resolvent_of_laplace},
\eqref{eq:def_dirac_distr}, and \eqref{eq:Def_of_Q}, respectively.
Then
\begin{align} 
\left|R_{\mu}^{m}\delta_{\{0\}}(0)\right| & \leq \left(\frac{1}{\Re(\mu)}\right)^{m}\left(\sqrt{\Re(\mu)}\right)^{n}c,   \lb{Rd} \\
\left|QR_{\mu}^{m}\delta_{\{0\}}(0)\right| & \leq \left(\frac{1}{\Re(\mu)}\right)^{m}\left(\sqrt{\Re(\mu)}\right)^{n+1}c',   \lb{QRd}
\end{align} 
with 
\begin{equation}
c=\left(2\pi\right)^{-n}\omega_{n-1}\int_{0}^{\infty}r^{n-1}
\big[r^{2}+1\big]^{- (n+3)/2}dr,\label{eq:def_c}
\end{equation}
 and 
\begin{equation}
c'=\left(2\pi\right)^{-n}\sqrt{n}\omega_{n-1}\int_{0}^{\infty}r^{n} 
\big[r^{2}+1\big]^{- (n+3)/2}dr,\label{eq:def_cprime}
\end{equation}
where $\omega_{n-1}$ is given by \eqref{e:om}.
\end{proposition}
\begin{proof}
We estimate $R_{\mu}^{m}\delta_{\{0\}}(0)$ with the help of the Fourier transform as follows.  
\begin{align*}
\left|R_{\mu}^{m}\delta_{\{0\}}(0)\right| & =\left|\left\langle R_{\mu}^{m}\delta_{\{0\}},\delta_{\{0\}}\right\rangle \right| 
=\left(2\pi\right)^{-n}\bigg|\int_{\mathbb{R}^{n}}\frac{1}{\big(\left|\xi\right|^{2}+\mu\big)^{m}}
 \, d^n \xi\bigg|\\
 & \leq \left(2\pi\right)^{-n}\int_{\mathbb{R}^{n}}\frac{1}{\big(\left|\xi\right|^{2}+\Re(\mu)\big)^{m}}
 \, d^n \xi\\
 & =\left(2\pi\right)^{-n}\omega_{n-1}\int_{0}^{\infty}\frac{r^{n-1}}{\left(r^{2}+\Re(\mu)\right)^{m}}
 \, dr\\
 & =\left(2\pi\right)^{-n}\omega_{n-1}\left(\frac{1}{\Re(\mu)}\right)^{m}\int_{0}^{\infty}\frac{r^{n-1}}{\Big(\Big(\frac{r}{\sqrt{\Re(\mu)}}\Big)^{2}+1\Big)^{m}} \, dr\\
 & =\left(2\pi\right)^{-n}\omega_{n-1}\left(\frac{1}{\Re(\mu)}\right)^{m}\int_{0}^{\infty}\frac{\big(t\sqrt{\Re(\mu)}\big)^{n-1}}{\left(t^{2}+1\right)^{m}}\sqrt{\Re(\mu)} \, dt\\
 & =\left(2\pi\right)^{-n}\omega_{n-1}\left(\frac{1}{\Re(\mu)}\right)^{m}\big(\sqrt{\Re(\mu)}\big)^{n}\int_{0}^{\infty}\frac{t^{n-1}}{\left(t^{2}+1\right)^{m}} \, dt.
\end{align*}
In a similar fashion, one estimates \eqref{QRd}, however, first we 
recall from \eqref{eq:estimate_for_symbol_of_Q} that $\left|\sum_{j=1}^{n}\gamma_{j,n}(-i)\xi_{j}\right|\leq \sqrt{n}\left|\xi\right|$, $\xi\in\mathbb{R}^{n}$. Hence, one arrives at 
\begin{align*}
\left|QR_{\mu}^{m}\delta_{\{0\}}(0)\right| 
& \leq \left(2\pi\right)^{-n}\sqrt{n}\int_{\mathbb{R}^{n}}\frac{\left|\xi\right|}{\big(\left|\xi\right|^{2}+\Re(\mu)\big)^{m}} \, d^n\xi\\
 & =\left(2\pi\right)^{-n}\sqrt{n}\omega_{n-1}\int_{0}^{\infty}\frac{r^{n}}{\left(r^{2}+\Re(\mu)\right)^{m}}\, dr\\
 & =\left(2\pi\right)^{-n}\sqrt{n}\omega_{n-1}\left(\frac{1}{\Re(\mu)}\right)^{m}\int_{0}^{\infty}\frac{r^{n}}{\Big(\Big(\frac{r}{\sqrt{\Re(\mu)}}\Big)^{2}+1\Big)^{m}} \, dr\\
 & =\left(2\pi\right)^{-n}\sqrt{n}\omega_{n-1}\left(\frac{1}{\Re(\mu)}\right)^{m}
 \big(\sqrt{\Re(\mu)}\big)^{n+1}\int_{0}^{\infty}\frac{t^{n}}{\left(t^{2}+1\right)^{m}} \, dt.\tag*{\qedhere}
\end{align*}
\end{proof}

The main observation in this subsection, Lemma \ref{lem:asymptotics on diagonal},
needs some preparations which deal with the fundamental solution of the Helmholtz equation 
on $\mathbb{R}^n$ for $n\geq 3$ odd, to be introduced in \eqref{C.1}.

\begin{lemma}
\label{lem:teleskop-max}Let $n,N\in\mathbb{N}$, and $x_{1},\ldots,x_{N}\in\mathbb{R}^{n}$.
Then
\[
\left|x_{1}\right|+\sum_{j=1}^{N-1}\left|x_{j+1}-x_{j}\right|\geq\max_{1\leq k \leq  N}\left|x_{k}\right|.
\]
\end{lemma}
\begin{proof}
We proceed by induction. The case $N=1$ is clear. For $N\in\mathbb{N}$,
one has
\[
\left|x_{1}\right|+\sum_{j=1}^{N}\left|x_{j+1}-x_{j}\right|\geq\left|x_{1}\right|
+ \sum_{j=1}^{N} \big[\left|x_{j+1}\right|-\left|x_{j}\right|\big] = \left|x_{N+1}\right|.
\]
Thus, employing the induction hypothesis, one gets that
\[
\left|x_{1}\right|+\sum_{j=1}^{N}\left|x_{j+1}-x_{j}\right| \geq \max_{1\leq k \leq  N}\left|x_{k}\right|\lor\left|x_{N+1}\right|=\max_{1\leq \ell \leq N+1}\left|x_{\ell}\right|.\tag*{{\qedhere}}
\]
\end{proof}

\begin{lemma}
\label{lem:expo is faster than pol}Let $\alpha>0$, $\beta>0$. Then the map 
\[
\phi\colon[0,\infty)\ni r\mapsto\left(1+\frac{1}{2}r\right)^{\alpha}e^{-\beta r}
\]
satisfies 
$$
|\phi(r)| \leq \begin{cases} 
[\alpha/(2\beta)]^{\alpha}e^{-\alpha+2\beta}, & \alpha>2\beta, \\
1, &\alpha\leq 2\beta,
\end{cases}  \quad r \geq 0.
$$
\end{lemma}
\begin{proof}
From 
\begin{align*}
\phi'(r) & =\bigg(\frac{1}{2}\alpha\left(1+\frac{1}{2}r\right)^{\alpha-1}-\beta\left(1+\frac{1}{2}r\right)^{\alpha}\bigg)e^{-\beta r}\\
 & =\left(\frac{1}{2}\alpha-\beta\left(1+\frac{1}{2}r\right)\right)\left(1+\frac{1}{2}r\right)^{\alpha-1}e^{-\beta r}
\end{align*}
one gets with $r^{*}\coloneqq (\alpha/\beta)-2$ that $\phi'(r^{*})=0$
if $r^{*}>0$. Thus, by $\phi(0)=1$ and $\phi(r)\to0$ as $r\to\infty$,
one obtains the assertion.
\end{proof}

Next, we shall concentrate on pointwise estimates for the fundamental solution of the Helmholtz equation. We denote the integral kernel (i.e., the Helmholtz Green's function) associated with 
$(-\Delta - z)^{-1}$ by $E_n(z;x,y)$, $x, y \in \bbR^n$, $x \neq y$, $n\in\bbN$, $n\ge 2$, 
$z \in \bbC$. Then, 
\begin{align}
& E_n(z;x, y)  \no \\ 
& \quad = \begin{cases} (i/4) \big(2\pi z^{-1/2} |x-y|\big)^{(2-n)/2} 
H^{(1)}_{(n-2)/2}\big(z^{1/2}|x-y|\big), & n\ge 2, \; z\in\bbC\backslash\{0\}, \\
- (2\pi)^{-1} \ln(|x-y|), & n=2, \; z=0, \\
[(n-2) \omega_{n-1}]^{-1} |x-y|^{2-n}, & n \ge 3, \; z=0, 
\end{cases}   \no \\
& \hspace*{6.08cm}   \Im\big(z^{1/2}\big)\geq 0, \; x, y \in\bbR^n, \; x \neq y,   \label{C.1} 
\end{align}
where $H^{(1)}_{\nu}(\cdot)$ denotes the Hankel function of the first kind 
with index $\nu\geq 0$ (cf.\ \cite[Sect.\ 9.1]{AS72}), and 
$\omega_{n-1}$ is given by \eqref{e:om}. We will directly work with the explicit formula 
\eqref{C.1}, even though one could also employ the Laplace transform connection between the 
resolvent and the semigroup of $- \Delta$ which manifests itself in the formula,
\begin{equation}
E_n(z;x,y) = \int_{[0,\infty)} (4 \pi t)^{-n/2} e^{- |x-y|^2/(4t)} e^{z t} \, dt, \quad \Re(z) < 0, 
\; x, y \in \bbR^n, \; x \neq y.
\end{equation} 

Later on, we need the following reformulation of the Helmholtz Green's function in odd space 
dimensions. We will use an explicit expression for the Hankel function of the first kind.

\begin{lemma}\label{lem:reform_of_en} Let $n=2{\hatt n}+1 \in \mathbb{N}_{\geq3}$ odd, 
$\mu\in \mathbb{C}_{\Re>0}$. We denote
\[
   \mathcal{E}_n(-z,r)\coloneqq E_{n}(z;x,y), \quad r>0, \; z\in \mathbb{C}\backslash \{0\}, 
\]
where $x,y\in \mathbb{R}^n$ are such that $|x-y|=r$. Then the following formula holds
\[
   \mathcal{E}_n(\mu,r)= \left(\frac{\sqrt{\mu}}{2}\right)^{{\hatt n}-1} \left(2\pi r\right)^{-{\hatt n}}e^{-\sqrt{\mu}r}\sum_{k=0}^{{\hatt n}-1} \frac{({\hatt n}+k-1)!}{k!({\hatt n}-k-1)!}\left(\frac{1}{2\sqrt{\mu}r}\right)^k. 
\] 
\end{lemma}
\begin{proof}
Our branch of $\sqrt{\cdot}$ is chosen such that $\sqrt{-z}=i\sqrt{z}$ for all $z\in \mathbb{C}$ with 
$\Re (z)>0$. We use the following representation of the Hankel function of the first kind taken from \cite[8.466.1]{GR80}, 
\[
   H_{{\hatt n}-\frac{1}{2}}^{(1)}(z)=\sqrt{\frac{2}{\pi z}}i^{-{\hatt n}}e^{iz}\sum_{k=0}^{{\hatt n}-1}(-1)^k\frac{({\hatt n}+k-1)!}{k!({\hatt n}-k-1)!}\frac{1}{(2iz)^k}, \quad {\hatt n} \in \bbN.
\]
Hence,
\begin{align*}
 \mathcal{E}_n(\mu,r) &=  \frac{i}{4} \bigg(\frac{2\pi r}{i\mu^{1/2}}\bigg)^{\frac{1}{2}-{\hatt n}} 
H^{(1)}_{{\hatt n}-\frac{1}{2}}\big(i\mu^{1/2}r\big) \\
                      &=  \frac{i}{4} \bigg(\frac{2\pi r}{i\mu^{1/2}}\bigg)^{\frac{1}{2}-{\hatt n}} 
\sqrt{\frac{2}{\pi i\mu^{1/2}r}}i^{-{\hatt n}}e^{ii\mu^{1/2}r}\sum_{k=0}^{{\hatt n}-1}(-1)^k\frac{({\hatt n}+k-1)!}{k!({\hatt n}-k-1)!}\frac{1}{(2ii\mu^{1/2}r)^k} \\
                      &=  \frac{i}{4} \bigg(\frac{2\pi r}{i\mu^{1/2}}\bigg)^{\frac{1}{2}-{\hatt n}} 
\sqrt{\frac{2}{\pi i\mu^{1/2}r}}i^{-{\hatt n}}e^{-\mu^{1/2}r}\sum_{k=0}^{{\hatt n}-1}(-1)^k\frac{({\hatt n}+k-1)!}{k!({\hatt n}-k-1)!}\frac{1}{(-2\mu^{1/2}r)^k} \\
& =  \left(\frac{\sqrt{\mu}}{2}\right)^{{\hatt n}-1} \left(2\pi r\right)^{-{\hatt n}}e^{-\sqrt{\mu}r}\sum_{k=0}^{{\hatt n}-1} \frac{({\hatt n}+k-1)!}{k!({\hatt n}-k-1)!}\left(\frac{1}{2\sqrt{\mu}r}\right)^k.\tag*{\qedhere}
\end{align*}
\end{proof}

As a first corollary to be drawn from the explicit formula in the latter result, we now derive some estimates of the Helmholtz Green's function. 

\begin{lemma} \label{lem:Real-part-Bessel_and_other}
Let $n=2{\hatt n}+1\in \mathbb{N}_{\geq 3}$ odd,  $\mu\in \mathbb{C}_{\Re>0}$. We denote  
\begin{equation} 
\mathcal{E}_n(\mu,r)\coloneqq E_n(-\mu;x,y), \quad r>0,    \lb{Enmur}                                                                                                  \end{equation} 
with $x,y\in \mathbb{R}^n$ such that $|x-y|=r$. Then the following assertions $(i)$--$(iii)$ 
hold: \\[1mm]  
$(i)$ Assume that $\mu>0$, then for all $r>0$, 
\begin{equation} 
\mathcal{E}_n(\mu,r)>0.    \label{bess_a} 
\end{equation} 
$(ii)$ For all $r>0$, 
\begin{equation} 
 |\mathcal{E}_n(\mu,r)|\leq  \left(\sqrt{\cos(\arg(\mu))}\right)^{1- {\hatt n}}\mathcal{E}_n(\Re(\mu),r).  \label{bess_b} 
 \end{equation} 
 $(iii)$ Assume that $\mu>0$, then for all $r>0$, 
 \begin{equation}
\exp\left(\sqrt\mu r / 2\right)\mathcal{E}_n(\mu,r)\leq  2^{{\hatt n}-1} \mathcal{E}_n(\mu/4,r).     \label{bess_c}
\end{equation}
\end{lemma}
\begin{proof}
Assertion \eqref{bess_a} is clear due to the fact that $\mathcal{E}_n$ is the fundamental solution of the positive, self-adjoint operator $(-\Delta+\mu)$. (Alternatively, one can also use the explicit representation of $\mathcal{E}_n(\cdot , \cdot)$ in Lemma \ref{lem:reform_of_en}.)

In view of Lemma \ref{lem:reform_of_en}, in order to prove \eqref{bess_b}, it suffices to prove the following two facts,
\begin{equation}
  \frac{|\sqrt{\mu}|}{\sqrt{\Re (\mu)}}= \frac{1}{\sqrt{\cos(\arg(\mu))}}  \, \text{ and } \,  
  |\Re (\sqrt{\mu})|\geq \sqrt{\Re (\mu)}.   \lb{muRe}
\end{equation}
To show these assertions, let $\rho>0$ and $\theta\in \left(- \pi/2,\pi/2\right)$ such that $\mu=\rho e^{i\theta}$. Then 
$\sqrt{\mu}=\sqrt{\rho}e^{i \theta/2}=\sqrt\rho \cos(\theta/2)+i\sqrt\rho\sin(\theta/2)$ as well 
as $\sqrt{\Re (\mu)}=\sqrt{\rho}\sqrt{\cos(\theta)}$. From
\[
   \sqrt{\cos(\theta)}=\sqrt{\left(\cos(\theta/2)\right)^2-\left(\sin(\theta/2)\right)^2} 
   \leq  \sqrt{\left(\cos(\theta/2)\right)^2}=\cos(\theta/2),
\]
and 
\[
  \frac{|\sqrt{\mu}|}{\sqrt{\Re (\mu)}}=\frac{|\sqrt{\rho}e^{i \theta/2}|}{\sqrt{\rho}
  \sqrt{\cos(\arg(\mu))}}, 
\]
assertion \eqref{muRe} follows. 

Finally, we turn to the proof of \eqref{bess_c}. Given the representation of $\mathcal{E}_n(\cdot , \cdot)$ in Lemma \ref{lem:reform_of_en}, one concludes that
\begin{align*}
& \exp\left(\frac{\sqrt{\mu}}{2}r\right)\mathcal{E}_n(\mu,r) = \exp\left(\frac{\sqrt{\mu}}{2}r\right)\left(\frac{\sqrt{\mu}}{2}\right)^{{\hatt n}-1} \left(2\pi r\right)^{-{\hatt n}}e^{-\sqrt{\mu}r}    \\
& \hspace*{3.6cm} \times \sum_{k=0}^{{\hatt n}-1} \frac{({\hatt n}+k-1)!}{k!({\hatt n}-k-1)!}\left(\frac{1}{2\sqrt{\mu}r}\right)^k\\
& \quad = \left(\frac{\sqrt{\mu}}{2}\right)^{{\hatt n}-1} \left(2\pi r\right)^{-{\hatt n}}e^{-\frac{\sqrt{\mu}}2r}\sum_{k=0}^{{\hatt n}-1} \frac{({\hatt n}+k-1)!}{k!({\hatt n}-k-1)!}\left(\frac{1}{2\sqrt{\mu}r}\right)^k\\
& \quad \leq  2^{{\hatt n}-1} \bigg(\frac{\sqrt{\mu/4}}{2}\bigg)^{{\hatt n}-1} \left(2\pi r\right)^{-{\hatt n}}e^{-\sqrt{\frac{\mu}{4}}r}\sum_{k=0}^{{\hatt n}-1} \frac{({\hatt n}+k-1)!}{k!({\hatt n}-k-1)!}\bigg(\frac{1}{2\sqrt{\frac{\mu}{4}}r}\bigg)^k.\tag*{\qedhere}
\end{align*}
\end{proof}

Next, we obtain similar results for the derivative of the fundamental solution. 

\begin{lemma} \label{lem:real_partBessel_derivative}
Let $n=2{\hatt n}+1\in \mathbb{N}_{\geq3}$ odd. Then, for all $\mu\in \mathbb{C}_{\Re>0}$,  there exists $q_\mu \colon \mathbb{R}_{\geq 0}\to \mathbb{R}_{\geq 0}$ with $q_\mu (|\cdot|)\in L^1(\mathbb{R}^n)$, such that the following properties $(i)$--$(iii)$ hold: \\[1mm]
$(i)$ For all $j\in \{1,\ldots,n\}$ and $\mu\in \mathbb{C}$, $\Re(\mu)>0$, and for all 
$x,y\in \mathbb{R}^n$, $x\neq y$, 
\begin{equation} 
 |\partial_j\left(\xi\mapsto {E}_n(-\mu;\xi,y)\right)(x)|\leq  q_\mu(|x-y|).       \label{der_a}
\end{equation} 
$(ii)$ For all $\mu\in \mathbb{C}$, $\Re(\mu)>0$, 
\begin{equation} 
   q_\mu(r)\leq  \left(1/\sqrt{\cos(\arg(\mu))}\right)^{{\hatt n}} q_{\Re (\mu)}(r), \quad r>0.  
\label{der_b} 
\end{equation} 
$(iii)$ For all $\mu>0$, 
\begin{equation} 
  \exp\left(\sqrt\mu r / 2\right)q_\mu(r)\leq  2^{{\hatt n}} q_{\mu/4}(r), \quad r>0.
  \label{der_c} 
\end{equation} 
\end{lemma}
\begin{proof}
 For $r>0$, with $\mathcal{E}_n(\mu,r)$ as in 
 Lemma \ref{lem:Real-part-Bessel_and_other} (and with the help of 
 Lemma \ref{lem:reform_of_en}), one obtains the following derivative of 
 $\mathcal{E}_n(\cdot,\cdot)$ with respect to the second variable, 
\begin{align*}
(\partial_r \mathcal{E}_n)(\mu,r) &= - \left(\frac{\sqrt{\mu}}{2}\right)^{{\hatt n}-1}(2\pi)^{-{\hatt n}}e^{-\sqrt{\mu}r} \\
& \quad \times \sum_{k=0}^{{\hatt n}-1}\frac{({\hatt n}+k-1)!}{k!({\hatt n}-k-1)!}\left(\frac{1}{2\sqrt{\mu}}\right)^{k}r^{-{\hatt n}-k-1}\left(\sqrt{\mu}r+({\hatt n}+k)\right).
\end{align*}
Define for $r>0$, 
\begin{align} 
\begin{split} 
 q_\mu (r)\coloneqq & \bigg|\left(\frac{\sqrt{\mu}}{2}\right)^{{\hatt n}-1}(2\pi)^{-{\hatt n}}e^{-\sqrt{\mu}r} \\
 & \, \times \sum_{k=0}^{{\hatt n}-1}\frac{({\hatt n}+k-1)!}{k!({\hatt n}-k-1)!}\left(\frac{1}{2\sqrt{\mu}}\right)^{k}r^{-{\hatt n}-k-1}\left(\sqrt{\mu}r+({\hatt n}+k)\right)\bigg|.    \lb{qmu} 
 \end{split}
\end{align} 
Then $q_\mu(|\cdot|) \in L^1(\mathbb{R}^n)$. Indeed, due to the presence of the 
$e^{-\sqrt{\mu}r}$-term, only integrability at $x=0$ is an issue here. Since the order of the 
singularity of $q_{\mu}(|x|)$ at $x=0$ is at most $|x|^{-n}$ since ${\hatt n}+({\hatt n}-1)+1=2{\hatt n}<2{\hatt n}+1=n$, also integrability of  $q_\mu(|\cdot|)$ at $x=0$ is ensured. 

To prove \eqref{der_a}, one observes that for fixed $y\in \mathbb{R}^n$ and 
$y\neq x\in \mathbb{R}^n$,  
\[
   \left|\partial_j\left(\xi \mapsto \mathcal{E}_n(\mu,|\xi-y|)\right)(x)\right|
   = q_\mu(|x-y|)\left|\frac{x_j-y_j}{|x-y|}\right|\leq  q_\mu(|x-y|).  
\]
The assertion in \eqref{der_b} follows analogously to that of \eqref{bess_b} with an explicit representation of $\sqrt{\mu}$, together with the observation that 
$1/\sqrt{\cos(\arg(\mu))} \geq 1$. The same arguments apply to the proof of \eqref{der_c}.
\end{proof}

Having established the preparations for estimating the integral kernel of products of resolvents 
$R_\mu$ of the Laplace operator and a multiplication operator $\Psi_j$, we finally come to the fundamental estimates 
\eqref{ttk} in Lemma \ref{lem:asymptotics on diagonal}. All the following results will be of a similar type. Namely, consider a product
\begin{equation}\label{e:pt}
   \Psi_1 R_\mu \Psi_2 R_\mu \cdots \Psi_m R_\mu,
\end{equation}
of smooth, bounded functions $\Psi_j$, $j\in\{1,\dots,m\}$, identified as multiplication operators in $L^2(\mathbb{R}^n)$ and $R_\mu=(-\Delta+\mu)^{-1}$ for some $\mu\in \mathbb{C}_{\Re>0}$. If $m$ is sufficiently large (depending on the space dimension $n$), the operator introduced in \eqref{e:pt} has a continuous integral kernel $t\colon \mathbb{R}^n\times \mathbb{R}^n \to \bbC$. Roughly speaking, we will show that the behavior of $x\mapsto t(x,x)$ is determined by the (algebraic) decay properties of all the functions $\Psi_j$, $j\in\{1,\ldots,n\}$, that is, if $\Psi_j$ decays like $|x|^{-\alpha_j}$ for large $|x|$ for some $\alpha_j\geq 0$, $j\in\{1,\ldots,m\}$, then $x\mapsto t(x,x)$ decays as $|x|^{-\sum_{j=1}^m\alpha_j}$. We will, however, need a more precise estimate. Namely, we also need to establish at the same time the overall constant of this decay behavior as a function of $\mu$. That is why we needed to establish Proposition \ref{prop:estimate of the pointwise resolvent at 0}, see also Remark \ref{rem:on the quantitative version of kappa prime} below. 

The precise statement regarding the estimate of the diagonal of such a  continuous integral kernel reads as follows: 

\begin{lemma} \label{lem:asymptotics on diagonal} 
Let $n=2{\hatt n}+1\in\mathbb{N}_{\geq 3}$, $m\geq\hatt n+1$, and assume that 
$\Psi_{1},\ldots,\Psi_{m+1}\in C_{b}^{\infty}(\mathbb{R}^{n})$, and 
$\mu\in\mathbb{C}$, $\Re(\mu)>0$. Assume that there exists 
$\alpha_{j},\kappa_{j}\in\mathbb{R}_{\geq0}$, 
$j\in\{1,\ldots,m+1\}$, such that  
\[
\left|\Psi_{j}(x)\right|\leq \kappa_{j} (1+|x|)^{-\alpha_j}, 
\quad x\in\mathbb{R}^{n}, \; j\in\{1,\ldots,m+1\}.
\]
Consider the integral kernels $t$ and $t_{k}$ of $T\coloneqq\prod_{j=1}^{m}\Psi_{j}R_{\mu}$
and $T_{k}\coloneqq\prod_{j=1}^{m+1}\Psi_{j,k}R_{\mu},$ respectively $($cf.\ 
Remark \ref{rem:differen_mult_op}$)$, where 
$\Psi_{j,k}=\Psi_{j}(1-\delta_{jk})+\delta_{jk}\Psi_{j}Q$,
$k\in\{1,\ldots,n\}$, with $R_{\mu}$ and $Q$ given by \eqref{eq:Def_of_Q}
and \eqref{eq:resolvent_of_laplace}, respectively. Then $t$ and
$t_{k}$ are continuous and there exists $\kappa' > 0$ such that 
\begin{equation} 
\left|t(x,x)\right|\leq \kappa' [1+(|x|/2)]^{-\sum_{j=1}^{m}\alpha_j}, \quad \left|t_{k}(x,x)\right|\leq \kappa' [1+(|x|/2)]^{-\sum_{j=1}^{m+1}\alpha_j}, \quad x\in\mathbb{R}^{n}.   \lb{ttk}
\end{equation} 
For $\sqrt{\Re(\mu)}>2\sum_{j=1}^{m+1}\alpha_{j}$,
one can choose, with $c_{\arg{(\mu)}}\coloneqq \cos(\arg(\mu))^{-1/2}$,
\[
\kappa'= \big[(2c_{\arg{(\mu)}})^{{\hatt n}-1}\big]^m\kappa_{1}\cdots\kappa_{m}\left(\frac{4}{\Re(\mu)}\right)^{m}\bigg(\sqrt{\frac{\Re(\mu)}{4}}\bigg)^{n}c
\]
 in the first estimate in \eqref{ttk}, and
\[
\kappa'=c_{\arg{(\mu)}}\big[(c_{\arg{(\mu)}})^{{\hatt n}-1}\big]^{m}(2^{{\hatt n}})^m\kappa_{1}\cdots\kappa_{m+1}\left(\frac{4}{\Re(\mu)}\right)^{m+1}\bigg(\sqrt{\frac{\Re(\mu)}{4}}\bigg)^{n+1}c'
\]
 in the second, with $c$ and $c'$ given by \eqref{eq:def_c} and
\eqref{eq:def_cprime}, respectively. 
\end{lemma}
\begin{proof}
We shall only prove the assertion for $T$. The other assertions follow
from the fact that the integral kernel of $QR_{\mu}$ can be bounded
by $x\mapsto c_{\arg(\mu)}^{\hatt n}\mu q_{\Re (\mu)}(|x|)$, see Lemma \ref{lem:real_partBessel_derivative}. Moreover, we shall exploit that the exponential estimates \eqref{bess_c} and \eqref{der_b} in Lemmas \ref{lem:Real-part-Bessel_and_other} and \ref{lem:real_partBessel_derivative}, respectively, are essentially the same.

The stated continuity of the integral kernels follows from $2m\geq 2\hatt n+2>n$, 
Corollary \ref{cor:integral kernels regularity}, and 
Proposition \ref{prop:concrete case}. Indeed, Proposition
\ref{prop:concrete case} implies that 
\[
T\in L\left(H^{\ell}\left(\mathbb{R}^{n}\right),H^{\ell+2m}\left(\mathbb{R}^{n}\right)\right), 
\quad \ell\in\mathbb{R}. 
\]
Thus, by Corollary \ref{cor:integral kernels regularity},
\[
t\colon\mathbb{R}^{n}\times\mathbb{R}^{n}\colon(x,y)\mapsto\left\langle \delta_{\{x\}},T\delta_{\{y\}}\right\rangle 
\]
is continuous as $2m>n$, with $\delta_{\{x\}}$, $x\in\mathbb{R}^{n}$, defined in \eqref{eq:def_dirac_distr}. 

We denote the integral kernel of $R_\mu=(-\Delta+\mu)^{-1}$ by $r_\mu$. Then one notes that 
\[
   r_\mu(x-y)=E_n(-\mu;x,y)=\mathcal{E}_n(\mu;|x-y|).
\]
For simplicity, we now assume that $\mu$ is real (one recalls the estimate 
$\left|r_{\mu}\right|\leq  c_{\arg(\mu)}r_{\Re(\mu)}$ with a positive real number $c_{\arg(\mu)}$ depending on $\arg(\mu)$, see \eqref{bess_b}). One observes that  
\begin{align*}
\frac{1}{1+\left|x_{1}+x\right|}\leq \frac{1}{1+\left|\left|x_{1}\right|-\left|x\right|\right|}=\frac{1}{1+\left|x\right|-\left|x_{1}\right|}\leq \frac{1}{1+\frac{1}{2}\left|x\right|},&  \\
x,x_{1}\in\mathbb{R}^{n}, \; \left|x\right|\text{\ensuremath{\geq}}2\left|x_{1}\right|.&
\end{align*} 
 On the other hand, one obviousy also has 
\[
\frac{1}{1+\left|x_{1}+x\right|}\leq 1, \quad \left|x\right| \leq 2\left|x_{1}\right|.
\]
Introducing the sets,  
\begin{align*} 
& B(R)\coloneqq \big\{ (x_{1},\ldots.x_{m-1})\in\left(\mathbb{R}^{n}\right)^{m-1} \, \big| \, 
\max_{1\leq  j\leq  m-1}\left|x_{j}\right|\leq  R\big\}, \\ 
& \complement B(R)\coloneqq\left(\mathbb{R}^{n}\right)^{m-1}\backslash B(R), \quad 
R\geq0, 
\end{align*} 
one computes for $x\in\mathbb{R}^{n}$, with 
$\tilde{\kappa}\coloneqq\kappa_{1}\cdots\kappa_{m+1}$, 
\begin{align*}
 & \left|t(x,x)\right| =\bigg|\Psi_{1}(x)\int_{\left(\mathbb{R}^{n}\right)^{m-1}}r_{\mu}(x-x_{1})\Psi_{2}(x_{1})r_{\mu}(x_{1}-x_{2})\cdots\Psi_{m}(x_{m-1})    \\ 
 & \hspace*{4cm} \times r_{\mu}(x_{m-1}-x) \, d^n x_1 \cdots d^n x_{m-1}\bigg|\\
 & \quad \leq \int_{\left(\mathbb{R}^{n}\right)^{m-1}}\left|\Psi_{1}(x)\right|r_{\mu}(x-x_{1})\cdots\left|\Psi_{m}(x_{m-1})\right|r_{\mu}(x_{m-1}-x) \\  
& \hspace*{2.2cm} \times d^n x_1 \cdots d^n x_{m-1}\\
 & \quad =\int_{\left(\mathbb{R}^{n}\right)^{m-1}}\left|\Psi_{1}(x)\right|r_{\mu}(x_{1})\cdots\left|\Psi_{m}(x_{m-1}+x)\right|r_{\mu}(x_{m-1}) \, d^n x_1 \cdots d^n x_{m-1}\\
 & \quad \leq \tilde{\kappa}\int_{\left(\mathbb{R}^{n}\right)^{m-1}}\left(\frac{1}{1+\left|x\right|}\right)^{\alpha_{1}}r_{\mu}(x_{1})\cdots\left(\frac{1}{1+\left|x_{m-1}+x\right|}\right)^{\alpha_{m}} \\ 
 & \hspace*{2.5cm} \times r_{\mu}(x_{m-1}) \, d^n x_1 \cdots d^n x_{m-1}\\
 & \quad =\tilde{\kappa}\int_{B\left(\left|x\right|/2\right)}\left(\frac{1}{1+\left|x\right|}\right)^{\alpha_{1}}r_{\mu}(x_{1})\cdots\left(\frac{1}{1+\left|x_{m-1}+x\right|}\right)^{\alpha_{m}} \\ 
 & \hspace*{2.4cm} \times r_{\mu}(x_{m-1}) \, d^n x_1 \cdots d^n x_{m-1}\\
 & \qquad+\tilde{\kappa}\int_{\complement B\left(\left|x\right|/2\right)}\left(\frac{1}{1+\left|x\right|}\right)^{\alpha_{1}}r_{\mu}(x_{1})\cdots\left(\frac{1}{1+\left|x_{m-1}+x\right|}\right)^{\alpha_{m}}    \\
 & \hspace*{2.8cm} \times r_{\mu}(x_{m-1}) \, d^n x_1 \cdots d^n x_{m-1}\\
 & \quad \leq \frac{\tilde{\kappa}}{\left(1+\frac{1}{2}\left|x\right|\right)^{\sum_{j=1}^{m}\alpha_{j}}}\int_{B\left(\left|x\right|/2\right)}r_{\mu}(x_{1})r_{\mu}(x_{1}-x_{2})\cdots r_{\mu}(x_{m-1}) \\ 
 & \hspace*{4.9cm} \times d^n x_1 \cdots d^n x_{m-1}\\
 & \qquad+\tilde{\kappa}\int_{\complement B\left(\left|x\right|/2\right)}r_{\mu}(x_{1})r_{\mu}(x_{1}-x_{2})\cdots r_{\mu}(x_{m-1}) \, d^n x_1 \cdots d^n x_{m-1}\\
 & \quad \leq \frac{\tilde{\kappa}}{\left(1+\frac{1}{2}\left|x\right|\right)^{\sum_{j=1}^{m}\alpha_{j}}}\int_{\left(\mathbb{R}^{n}\right)^{m-1}}r_{\mu}(x_{1})r_{\mu}(x_{1}-x_{2})\cdots r_{\mu}(x_{m-1}) \\ 
&  \hspace*{4.95cm} \times d^n x_1 \cdots d^n x_{m-1}\\
 & \qquad+\tilde{\kappa}\int_{\complement B\left(\left|x\right|/2\right)}r_{\mu}(x_{1})r_{\mu}(x_{1}-x_{2})\cdots r_{\mu}(x_{m-1}) \, d^n x_1 \cdots d^n x_{m-1}.
\end{align*}
By Lemma \ref{lem:teleskop-max} one recalls that for $x_{1},\ldots,x_{m-1}\in\mathbb{R}^{n}$, 
\[
\left|x_{1}\right|+\sum_{j=1}^{m-2}\left|x_{j+1}-x_{j}\right|+\left|x_{m-1}\right|\geq\max_{1\leq  j\leq  m-1}\left|x_{j}\right|.
\]
With the latter observation one estimates, using \eqref{bess_c}, 
\begin{align*}
 & \int_{\complement B\left(\left|x\right|/2\right)}r_{\mu}(x_{1})r_{\mu}(x_{1}-x_{2})\cdots r_{\mu}(x_{m-1}) \, d^n x_1 \cdots d^n x_{m-1}\\
 & \quad =\int_{\complement B\left(\left|x\right|/2\right)}e^{-\frac{\sqrt{\mu}}{2}\left(\left|x_{1}\right|+\sum_{j=1}^{m-2}\left|x_{j+1}-x_{j}\right|+\left|x_{m-1}\right|\right)}r_{\frac{\mu}{4}}(x_{1})r_{\frac{\mu}{4}}(x_{1}-x_{2})\cdots r_{\frac{\mu}{4}}(x_{m-1}) \\ 
 & \hspace*{2.2cm} \times d^n x_1 \cdots d^n x_{m-1}\\
 & \quad \leq  (2^{{\hatt n}-1})^m\int_{\complement B\left(\left|x\right|/2\right)}e^{-\frac{\sqrt{\mu}}{2}\left(\max_{j=1}^{m-1}\left|x_{j}\right|\right)}r_{\frac{\mu}{4}}(x_{1})r_{\frac{\mu}{4}}(x_{1}-x_{2})\cdots r_{\frac{\mu}{4}}(x_{m-1}) \\
 & \hspace*{3.5cm} \times d^n x_1 \cdots d^n x_{m-1}\\
 & \quad \leq (2^{{\hatt n}-1})^m e^{-\frac{\sqrt{\mu}}{4}\left|x\right|} 
 \int_{\complement B\left(\left|x\right|/2\right)}r_{\frac{\mu}{4}}(x_{1})r_{\frac{\mu}{4}}(x_{1}-x_{2})\cdots r_{\frac{\mu}{4}}(x_{m-1}) \\
 & \hspace*{4.65cm} \times d^n x_1 \cdots d^n x_{m-1}\\
 & \quad \leq  (2^{{\hatt n}-1})^m e^{-\frac{\sqrt{\mu}}{4}\left|x\right|}\int_{\left(\mathbb{R}^{n}\right)^{m-1}}r_{\frac{\mu}{4}}(x_{1})r_{\frac{\mu}{4}}(x_{1}-x_{2})\cdots r_{\frac{\mu}{4}}(x_{m-1}) \\ 
 & \hspace*{4.55cm} \times d^n x_1 \cdots d^n x_{m-1}.
\end{align*}
The latter expression decays faster than any power of $(1+ \left|x\right|)^{-1}$.
In fact, given Lemma \ref{lem:expo is faster than pol}, for 
$\sqrt{\mu}>2\sum_{j=1}^{m+1}\alpha_{j}$,
one obtains $e^{-\frac{\sqrt{\mu}}{4}\left|x\right|}[1+ (|x|/2)]^{\sum_{j=1}^{m+1}\alpha_{j}}\leq 1$. 
Hence, for some $\kappa'>0$, 
\[
\left|t(x,x)\right|\leq  (2^{{\hatt n}-1})^m 
\kappa' [1+(|x|/2)]^{-\sum_{j=1}^{m}\alpha_j}, \quad 
x\in\mathbb{R}^{n}.
\]
For the more precise estimate, one observes that
\[
\int_{\left(\mathbb{R}^{n}\right)^{m-1}}r_{\frac{\mu}{4}}(x_{1})r_{\frac{\mu}{4}}(x_{1}-x_{2})\cdots r_{\frac{\mu}{4}}(x_{m-1}) \, d^n x_1 \cdots d^n x_{m-1} 
= R_{\frac{\mu}{4}}^{m}\delta_{\{0\}}(0)
\]
 and then applies Proposition \ref{prop:estimate of the pointwise resolvent at 0}
to estimate the latter expression. 
\end{proof}

\begin{remark} \label{rem:on the quantitative version of kappa prime} 
A further inspection of Lemma \ref{lem:asymptotics on diagonal} reveals 
that if $\sqrt{\Re(\mu)}>2\ell$ for some $\ell\leq \sum_{j=1}^{m+1}\alpha_{j}$, one obtains the
estimates 
\begin{align*} 
\left|t(x,x)\right| & \leq \big[(2c_{\arg{\mu}})^{{\hatt n}-1}\big]^m \kappa_{1}\cdots\kappa_{m}\left(\frac{4}{\Re(\mu)}\right)^{m}\bigg(\sqrt{\frac{\Re(\mu)}{4}}\bigg)^{n}c 
[1+(\left|x\right|/2)]^{- \ell}, \\
& \hspace*{9.4cm}   x \in \bbR^n,   \\
\left|t_{k}(x,x)\right| & \leq c_{\arg{\mu}}\big[(c_{\arg{\mu}})^{{\hatt n}-1}\big]^{m} 2^{{\hatt n} \, m} \kappa_{1}\cdots\kappa_{m+1}\times
\\ &\quad \times\left(\frac{4}{\Re(\mu)}\right)^{m+1}\bigg(\sqrt{\frac{\Re(\mu)}{4}}\bigg)^{n+1}c' 
[1+(|x|/2)]^{-\ell}, \quad x \in \bbR^n, 
\end{align*} 
with $c$ and $c'$ given by \eqref{eq:def_c} and \eqref{eq:def_cprime}, respectively. 

To illustrate the importance of this result, envisage a product as in \eqref{e:pt} with $m$ factors, all of them decaying like $|x|^{-1}$ as $|x| \to \infty$. Later on, we shall see that in certain integrals it suffices to estimate the diagonal decaying like $|x|^{-n}$ as $|x| \to \infty$. Thus, if $m$ is fairly 
large compared to $n$, and hence Lemma \ref{lem:asymptotics on diagonal} yields a decay like 
$|x|^{-m}$, we have to choose the real-part of $\mu$ rather large as the explicit constant is only valid for $\sqrt{\Re(\mu)}>m$. But, if we are only interested in an estimate of the type $|x|^{-n}$, we may choose $\sqrt{\Re (\mu)}$ \emph{a apriori} just larger than $n$. 
\hfill $\diamond$
\end{remark}

A readily applicable version of Lemma \ref{lem:asymptotics on diagonal} reads as follows.

\begin{lemma}
\label{lem:asymptotics in the diagonal with commutator}Let $n=2\hatt n+1\in\mathbb{N}_{\geq3}$ odd, $m\geq\hatt n+1$, assume that 
$\Psi_{1},\ldots,\Psi_{m}\in C_{b}^{\infty}(\mathbb{R}^{n})$, and suppose that  
$\mu\in\mathbb{C}_{\Re>0}$. Let $R_{\mu}$ and $Q$ be given by \eqref{eq:resolvent_of_laplace} and \eqref{eq:Def_of_Q}, respectively.
Assume that there exists $\alpha_{j},\kappa_{j}\in\mathbb{R}_{\geq0}$, 
$j\in\{1,\ldots,m\}$, such that 
\[
\left|\Psi_{j}(x)\right|\leq \kappa_{j} (1+|x|)^{- \alpha_j}, 
\quad x\in\mathbb{R}^{n}, \; j\in\{1,\ldots,m\}.
\]
Let $\ell\in\{2,\ldots,m\}$ and assume that there exists $\epsilon\geq0$
such that 
\[
\left|\left(Q\Psi_{\ell}\right)(x)\right|+\left|\left(Q^{2}\Psi_{\ell}\right)(x)\right|\leq \kappa_{\ell} 
(1+|x|)^{-\alpha_{\ell}-\epsilon}, \quad x\in\mathbb{R}^{n}.
\]
If $t$ denotes the integral kernel of 
\[
T\coloneqq\prod_{j=1}^{\ell-2}\left(\Psi_{j}R_{\mu}\right)\Psi_{\ell-1}\left[R_{\mu},\Psi_{\ell}\right]R_{\mu}\prod_{j=\ell+1}^{m}\Psi_{j}R_{\mu}
\]
$($cf.\ Remark \ref{rem:differen_mult_op} and \eqref{eq:def_commutator}$)$,
then $t$ is continuous on the diagonal and there exists $\kappa'>0$ such that  
\[
\left|t(x,x)\right|\leq \kappa' [1+(|x|/2)]^{- \epsilon 
- \sum_{j=1}^{m}\alpha_j},   \quad x\in\mathbb{R}^{n}.
\]
If $\sqrt{\Re(\mu)}>2\sum_{j=1}^{m}\alpha_{j}+2\epsilon$, then a possible
choice for $\kappa'$ is 
\begin{equation}
\kappa'=\kappa_{1}\cdots\kappa_{m}\left(\frac{4}{\Re(\mu)}\right)^{m}\bigg(\sqrt{\frac{\Re(\mu)}{4}}\bigg)^{n}d,     \label{eq:precise constant for com}
\end{equation}
where $d\coloneqq  \big[(2c_{\arg{\mu}})^{{\hatt n}-1}\big]^m c+2 c_{\arg{\mu}}^{{\hatt n}}\big[(c_{\arg{\mu}})^{{\hatt n}-1}\big]^{m-1} 2^{{\hatt n} \, m}c'$, with $c$ and $c'$ given by
\eqref{eq:def_c} and \eqref{eq:def_cprime}, respectively. 
\end{lemma}
\begin{proof}
One recalls from Lemma \ref{lem:commutator} (see also Remark \ref{rem:differen_mult_op})
that 
\[
\left[R_{\mu},\Psi_{\ell}\right]=R_{\mu}\left(Q^{2}\Psi_{\ell}\right)R_{\mu}+2R_{\mu}\left(Q\Psi_{\ell}\right)QR_{\mu}.
\]
Let $t_{1}$ be the associated integral kernel of 
\[
\Psi_{1}R_{\mu}\cdots\Psi_{\ell-1}R_{\mu}\left(Q^{2}\Psi_{\ell}\right)R_{\mu}R_{\mu}\Psi_{\ell+1}R_{\mu}\cdots\Psi_{m}R_{\mu}
\]
and $t_{2}$ the one of 
\[
\Psi_{1}R_{\mu}\cdots\Psi_{\ell-1}R_{\mu}\left(Q\Psi_{\ell}\right)QR_{\mu}R_{\mu}\Psi_{\ell+1}R_{\mu}\cdots\Psi_{m}R_{\mu}.
\]
By hypothesis and by Lemma \ref{lem:asymptotics on diagonal}, for some 
constant $\kappa'>0$,  
\begin{equation}
\left|t_{1}(x,x)\right|+\left|t_{2}(x,x)\right|\leq \kappa' [1+(|x|/2)]^{- \epsilon 
- \sum_{j=1}^{m}\alpha_j}, 
\quad x\in\mathbb{R}^{n}. \label{eq:L 3.23}
\end{equation}
The quantitative version of this assertion (i.e., the fact that $\kappa'$ given by   
\eqref{eq:precise constant for com} is a possible choice in the estimate 
\eqref{eq:L 3.23}), also follows from Lemma \ref{lem:asymptotics on diagonal}. 
\end{proof}

Finally, we state one more variant of Lemma \ref{lem:asymptotics on diagonal}.

\begin{lemma}
\label{lem:asymptotics on the diagonal and products of psis} Let $n=2\hatt n+1\in \mathbb{N}_{\geq3}$ odd,
$m\geq\hatt n+1$, assume that $\Psi_{1},\ldots,\Psi_{m}\in C_{b}^{\infty}(\mathbb{R}^{n})$, and 
$\mu\in\mathbb{C}$, $\Re(\mu)>0$. Let $R_{\mu}$ be given by \eqref{eq:resolvent_of_laplace},
and $Q$ by \eqref{eq:Def_of_Q}. Let $\epsilon,\alpha_{j},\kappa_{j}\in\mathbb{R}_{\geq0}$, 
and assume that for all $j\in\{1,\ldots,m\}$ and $\ell\in\{2,\ldots,m\},$  
\begin{align*} 
& \left|\Psi_{j}(x)\right|\leq \kappa_{j} (1+|x|)^{-\alpha_j},\,\left|\left(Q\Psi_{\ell}\right)(x)\right|+\left|\left(Q^{2}\Psi_{\ell}\right)(x)\right|\leq \kappa_{\ell} [1+(|x|/2)]^{-\alpha_{\ell}-\epsilon}, \\
& \hspace*{10.65cm} x \in \bbR^n.
\end{align*} 
Then for $\ell\in\{1,\ldots,m\}$, the associated integral kernels
$h_{\ell}$ and $\tilde{h}_{\ell}$ of 
\[
\bigg(\prod_{j=1}^{\ell}\Psi_{j}R_{\mu}\bigg)
\bigg(\prod_{j=\ell+1}^{m}\Psi_{j}\bigg)R_{\mu}^{m-\ell} \, \text{ and } \, 
\bigg(\prod_{j=1}^{\ell-1}\Psi_{j}R_{\mu}\bigg)
\bigg(\prod_{j=\ell}^{m}\Psi_{j}\bigg)R_{\mu}^{m-\ell+1},
\]
respectively $($cf.\ Remark \ref{rem:differen_mult_op}$)$, satisfy for some $\kappa'>0$, 
\[
\big|h_{\ell}(x,x)-\tilde{h}_{\ell}(x,x)\big|\leq \kappa' [1+(|x|/2)]^{- \epsilon 
- \sum_{j=1}^{m}\alpha_{j}}, \quad x\in\mathbb{R}^{n}. 
\]
In addition, if 
$\sqrt{\Re(\mu)}>2\sum_{j=1}^{m}\alpha_{j}+2\epsilon$, then a possible
choice for $\kappa'$ is given by \eqref{eq:precise constant for com}. 
\end{lemma}
\begin{proof}
We will exploit Lemma \ref{lem:asymptotics in the diagonal with commutator} and note that $h_{\ell}-\tilde{h}_{\ell}$ is the associated integral kernel of the operator
\begin{align*}
 & \bigg(\prod_{j=1}^{\ell}\Psi_{j}R_{\mu}\bigg)\bigg(\prod_{j=\ell+1}^{m}\Psi_{j}\bigg)
 R_{\mu}^{m-\ell}-\bigg(\prod_{j=1}^{\ell-1}\Psi_{j}R_{\mu}\bigg)\bigg(\prod_{j=\ell}^{m}\Psi_{j}\bigg)R_{\mu}^{m-\ell+1}\\
 & \quad =\bigg(\bigg(\prod_{j=1}^{\ell-1}\Psi_{j}R_{\mu}\bigg)\Psi_{\ell}\bigg)\bigg(R_{\mu}\bigg(\prod_{j=\ell+1}^{m}\Psi_{j}\bigg)-\bigg(\prod_{j=\ell+1}^{m}\Psi_{j}\bigg)R_{\mu}\bigg)
 R_{\mu}^{m-\ell}\\
 & \quad =\bigg(\bigg(\prod_{j=1}^{\ell-1}\Psi_{j}R_{\mu}\bigg)\Psi_{\ell}\bigg)
 \bigg(\bigg[R_{\mu},\bigg(\prod_{j=\ell+1}^{m}\Psi_{j}\bigg)\bigg]\bigg)R_{\mu}^{m-\ell}.
\end{align*}
By hypothesis,  
\[
\bigg|\bigg(\prod_{j=\ell+1}^{m}\Psi_{j}\bigg)(x)\bigg|
\leq \kappa_{\ell+1}\cdots\kappa_{m} (1+|x|)^{-\sum_{j=\ell+1}^{m}\alpha_{j}}, 
\quad x\in\mathbb{R}^{n}.
\]
Moreover, by the product rule one concludes that  
\[
\bigg|\bigg(Q\prod_{j=\ell+1}^{m}\Psi_{j}\bigg)(x)\bigg|
\leq \kappa_{\ell+1}\cdots\kappa_{m} (1+\left|x\right|)^{ -\epsilon - 
\sum_{j=\ell+1}^{m}\alpha_{j}},  \quad x\in\mathbb{R}^{n}, 
\]
and thus also that 
\[
\bigg|\bigg(Q^{2}\prod_{j=\ell+1}^{m}\Psi_{j}\bigg)(x)\bigg| 
\leq \kappa_{\ell+1}\cdots\kappa_{m} (1+\left|x\right|)^{-\epsilon 
- \sum_{j=\ell+1}^{m}\alpha_{j}}, \quad x\in\mathbb{R}^{n}.
\]
Hence, the assertion indeed follows from Lemma \ref{lem:asymptotics in the diagonal with commutator}. 
\end{proof}

\begin{remark} \label{rem:green's kernels of the first and the last} 
Iterated application Lemma \ref{lem:asymptotics on the diagonal and products of psis} shows that under the same assumptions, the integral kernels $h_{m}$ and $\tilde{h}_{1}$ of 
\[
\bigg(\prod_{j=1}^{m}\Psi_{j}R_{\mu}\bigg) \, \text{ and } \,
\bigg(\prod_{j=1}^{m}\Psi_{j}\bigg)R_{\mu}^{m},
\]
respectively, satisfy for some $\kappa'>0$, 
\[
\big|h_{m}(x,x)-\tilde{h}_{1}(x,x)\big|\leq \kappa' [1+ (|x|/2)]^{- \epsilon 
- \sum_{j=1}^{m}\alpha_{j}}, \quad x\in\mathbb{R}^{n}. 
\]
\hfill $\diamond$ 
\end{remark}

\newpage

\section{Dirac-Type Operators}\label{sec:A-particular-First}

In this section, we discuss the operator $L$ with a bounded smooth potential, studied 
by Callias \cite{Ca78} in $L^{2}(\mathbb{R}^{n})^{p}$ for a suitable $p\in\mathbb{N}$.
We compute its domain, its adjoint and give conditions for the Fredholm property 
of this operator. 

Let 
\[
H_{j}^{1}(\mathbb{R}^{n})\coloneqq\left\{ f\in L^{2}(\mathbb{R}^{n}) \,\big|\, 
\partial_{j}f\in L^{2}(\mathbb{R}^{n})\right\}, \quad j\in\{1,\ldots,n\}, 
\]
where $\partial_{j}f$ denotes the distributional partial derivative
of $f\in L^{2}(\mathbb{R}^{n})$ with respect to the $j$th variable.
One notes that (see also \eqref{eq:def_Hs}) 
\[
H^{1}(\mathbb{R}^{n})=\bigcap_{j\in\{1,\ldots,n\}}H_{j}^{1}\left(\mathbb{R}^{n}\right).
\]

\begin{remark}
\label{rem:Eucl-Dirac-Algebar}In the following, we make use of the
so-called \emph{Euclidean Dirac algebra}, see Appendix 
\ref{sec:Appendix:-the-Construction}
and Definition \ref{def:Euc-D-A} for the construction and some basic properties.
For dimension $n\in\mathbb{N}$ we denote the elements of this algebra
by $\gamma_{j,n}$, $j\in\{1,\ldots,n\}$. One recalls that for $n=2 \hat n$
or $n=2 \hat n +1$ for some $\hat n \in\mathbb{N}$ one has  
\begin{equation}
\gamma_{j,n}^{*}=\gamma_{j,n}\in\mathbb{C}^{2^{\hat n}\times2^{\hat n}},   \quad 
\gamma_{j,n}\gamma_{k,n}+\gamma_{k,n}\gamma_{j,n}=2\delta_{jk}I_{2^{\hatt n}}, \quad j,k\in\{1,\ldots,n\}.
\end{equation}
${}$ \hfill $\diamond$
\end{remark}

We are now in the position to properly define the operator $L$ (and the underlying supersymmetric 
Dirac-type operator $H$) to be studied in the rest of this manuscript.

\begin{definition}
\label{Def:L} Let $d \in\mathbb{N}$ and suppose that 
$\Phi\colon\mathbb{R}^{n}\to\mathbb{C}^{d\times d}$ 
is a bounded measurable function assuming values in the space of $d \times d$ self-adjoint 
matrices. We recall our convention $H^{1}(\mathbb{R}^{n})^{2^{\hat n}d}=H^{1}(\mathbb{R}^{n})^{2^{\hat n}}\otimes \mathbb{C}^d$. With this in mind, we introduce the  
$($closed\,$)$ operator $L$ in $L^{2}(\mathbb{R}^{n})^{2^{\hat n}d}$ via 
\begin{align}
L\colon \begin{cases} H^{1}(\mathbb{R}^{n})^{2^{\hat n}d} \subseteq 
L^{2}(\mathbb{R}^{n})^{2^{\hat n}d} \to 
L^{2}(\mathbb{R}^{n})^{2^{\hat n}d},   \\
\psi \otimes \phi \mapsto\left(\sum_{j=1}^{n}\gamma_{j,n}\partial_{j}\psi\right)
\otimes \phi + \left(x\mapsto\psi(x)\otimes \Phi(x)\phi\right). 
\end{cases}            \label{eq:def_of_L}
\end{align} 
Henceforth, recalling \eqref{eq:Def_of_Q}, we shall abreviate
\begin{equation}
\cQ\coloneqq Q \otimes I_d = \bigg(\sum_{j=1}^{n}\gamma_{j,n}\partial_{j}\bigg) I_d,        \label{eq:def_of_Q2}
\end{equation}
and, with a slight abuse of notation, employ the symbol $\Phi$ also in the context of the operation 
\begin{equation}
\Phi \colon \psi\otimes \phi \mapsto\left(x\mapsto\psi(x)
\otimes \Phi(x)\phi\right),    \lb{Phi}
\end{equation} 
see also our notational conventions to suppress tensor products whenever possible, collected in Section \ref{s2} and in Remark \ref{rem:differen_mult_op}. Thus,   
\begin{equation}
L= \cQ + \Phi.
\end{equation} 
Finally, the underlying $($self-adjoint\,$)$ supersymmetric Dirac-type operator $H$ in 
$L^{2}(\mathbb{R}^{n})^{2^{\hat n}d} \oplus L^{2}(\mathbb{R}^{n})^{2^{\hat n}d}$ is of the form 
\begin{equation} 
H = \begin{pmatrix} 0 & L^* \\ L & 0 \end{pmatrix}. 
\end{equation}
\end{definition}

A detailed treatment of supersymmetric Dirac-type operators can be found in \cite[Ch.~5]{Th92}.

For clarity, we kept the tensor product notation in \eqref{eq:def_of_L}--\eqref{Phi}, but from 
now on we will typically dispense with tensor products to simplify notation. 

In this section, we shall prove the following theorem:

\begin{theorem}[{\cite[Corollary on p. 217]{Ca78}}]
 \label{thm:Fredholm_property} Consider the operator $L$ given
by \eqref{eq:def_of_L}. Assume, in addition, that 
$\Phi\in C_{b}^{\infty}\big(\mathbb{R}^{n};\mathbb{C}^{d \times d}\big)$,
$\Phi(x)=\Phi(x)^{*}$, $x\in\mathbb{R}^{n}$, and assume there exists $c>0$ such that 
$|\Phi(x)| \geq c I_d$, $x\in\mathbb{R}^{n}$, as well as 
$\left(Q\Phi\right)(x)\to0$ as $\left|x\right|\to\infty$. Then the operator 
\begin{equation} 
L=\cQ+\Phi, \quad \dom(L) = \dom(\cQ) 
= H^{1}(\mathbb{R}^{n})^{2^{\hat n}d},    \lb{domL}
\end{equation} 
is closed and Fredholm in $L^{2}(\mathbb{R}^{n})^{2^{\hat n}d}$. 
Consequently, the supersymmetric Dirac-type operator 
\begin{equation} 
H = \begin{pmatrix} 0 & L^* \\ L & 0 \end{pmatrix}, \quad \dom(H) 
= H^{1}(\mathbb{R}^{n})^{2^{\hat n}d} \oplus H^{1}(\mathbb{R}^{n})^{2^{\hat n}d},
\end{equation} 
is self-adjoint and Fredholm in 
$L^{2}(\mathbb{R}^{n})^{2^{\hat n}d} \oplus L^{2}(\mathbb{R}^{n})^{2^{\hat n}d}$. 
\end{theorem}

In order to deduce Theorem \ref{thm:Fredholm_property}, we need some preparations.
The first result is concerned with the operator $Q$ in 
$L^{2}\left(\mathbb{R}^{n}\right)^{2^{\hat n}}$ given by \eqref{eq:Def_of_Q}. Moreover,
we will show that $Q$ is skew-self-adjoint and thus verify the estimate
asserted in \eqref{eq:cont_est_of_Q}.

\begin{theorem} \label{thm:L_is_closed} Let $n\in\mathbb{N}_{\geq2}$ and hence 
$\gamma_{j,n}\in\mathbb{C}^{2^{\hat n} \times 2^{\hat n}}$, $j \in \{1,\dots,n\}$, with $n=2\hatt n$ or $n=2\hatt n+1$, see Remark \ref{rem:Eucl-Dirac-Algebar}.  Denote 
\begin{align*}
\partial_{j}\colon H^{1}_{j}(\mathbb{R}^{n})^{2^{\hat n}}\subseteq L^{2}(\mathbb{R}^{n})^{2^{\hat n}} \to L^{2}(\mathbb{R}^{n})^{2^{\hat n}}, \quad f \mapsto\partial_{j}f, \; j \in \{1,\dots,n\}.
\end{align*}
Then the following assertions $(i)$--$(iii)$ hold: \\[1mm] 
$(i)$ $\partial_{j}$ is a skew-self-adjoint operator, $j\in\{1,\dots,n\}$. \\[1mm]
$(ii)$  $\gamma_{j,n}\partial_{j}=\partial_{j}\gamma_{j,n}$ is skew-self-adjoint, $j\in\{1,\dots,n\}$. 
\\[1mm] 
$(iii)$ $Q=\sum_{j=1}^{n}\gamma_{j,n}\partial_{j}$, 
$\dom (Q) = H^{1}(\mathbb{R}^{n})^{2^{\hat n}}$ is skew-self-adjoint $($and 
thus closed\,$)$ \hspace*{6.5mm} in $L^{2}(\mathbb{R}^{n})^{2^{\hat n}}$.   
\end{theorem}
\begin{proof}
By Fourier transform (see \eqref{eq:def_Fourier_transf}), the
operator $\partial_{j}$ is unitarily equivalent to the operator given
by multiplication by the function $x\mapsto ix_{j}$. The latter is
a multiplication operator taking values on the imaginary axis;
thus, it is skew-self-adjoint. Hence, so is $\partial_{j}$, proving assertion $(i)$. 

Let $j\in \{1,\ldots,n\}$. Assertion $(ii)$ follows from observing that $\gamma_{j,n}$ defines
an isomorphism from $L^{2}(\mathbb{R}^{n})^{2^{\hat n}}$ into itself. Indeed,
this follows from the fact that $\gamma_{j,n}^{2}= I_{2^{\hat n}}$. 
Moreover, since $\gamma_{j,n}$ is a constant coefficient matrix it
leaves the space $C_0^{\infty}(\mathbb{R}^{n})^{2^{\hatt n}}$ invariant. The equality
$\gamma_{j,n}\partial_{j}\phi=\partial_{j}\gamma_{j,n}\phi$ is clear
for $\phi\in C_0^{\infty}(\mathbb{R}^{n})^{2^{\hatt n}}$. Hence, $\partial_{j}\gamma_{j,n}\supseteq\gamma_{j,n}\partial_{j}$ and it remains to 
show that $C_0^{\infty}(\mathbb{R}^{n})^{2^{\hatt n}}$ is a
core for $\partial_{j}\gamma_{j,n}$. Let $\psi\in \dom(\partial_{j}\gamma_{j,n})$.
Then there exists a sequence $\{\psi_{k}\}_{k\in\bbN}$ in $C_0^{\infty}(\mathbb{R}^{n})^{2^{\hatt n}}$
such that $\psi_{k}\to\gamma_{j,n}\psi$ as $k\to\infty$ in $D_{\partial_{j}},$
the domain of $\partial_{j}$ endowed with the graph norm. Defining
$\phi_{k}\coloneqq\gamma_{j,n}^{-1}\psi_{k}=\gamma_{j,n}\psi_{k}\in C_0^{\infty}(\mathbb{R}^{n})^{2^{\hatt n}}$, $k\in\mathbb{N}$, one sees that 
$\phi_{k}\to\psi$ in $L^{2}(\mathbb{R}^{n})^{2^{\hatt n}}$
and $\partial_{j}\gamma_{j,n}\phi_{k}=\partial_{j}\psi_{k}\to\partial_{j}\gamma_{j,n}\psi$
in $L^{2}(\mathbb{R}^{n})^{2^{\hatt n}}$ as $k\to\infty.$ Thus, $\left(\gamma_{j,n}\partial_{j}\right)^{*}
=-\partial_{j}\gamma_{j,n}^{*}=-\partial_{j}\gamma_{j,n}=-\gamma_{j,n}\partial_{j}$.

Assertion $(iii)$ is a bit more involved. We shall prove it in the
next two steps.
\end{proof}

We recall the following well-known fact in the theory of normal operators. 

\begin{theorem}[{see, e.g., \cite[p. 347]{GLMST11}}]
\label{thm:resolvent_normal} Let $\cH$ be a complex separable Hilbert
space and let $A$ and $B$ be self-adjoint, resolvent commuting operators
acting on $\cH$. Then
\[
A+iB
\]
is closed, densely defined, and
\[
\left(A+iB\right)^{*}=A-iB.
\]
\end{theorem}

At this point we are ready to conclude the proof of Theorem \ref{thm:L_is_closed}:

\begin{lemma}
Let $\cH$ be a complex, separable Hilbert space, $A_{1},\ldots,A_{n}$ be resolvent commuting
skew-self-adjoint operators in $\cH$. Let $\{\gamma_{k}\}_{k\in\{1,\ldots,n\}}$
be a family of bounded linear self-adjoint operators in $\cH$, all commuting
with $A_{j}$, $j\in\{1,\ldots,n\}$, in the sense that 
$\gamma_{k}A_{j}=A_{j}\gamma_{k}$, $j,k\in\{1,\ldots,n\}$. Assume that the following equation holds, 
\[
\gamma_{k}\gamma_{k'}+\gamma_{k'}\gamma_{k}=2\delta_{kk'}, 
\quad k,k'\in\{1,\ldots,n\}. 
\]
Then $\sum_{k=1}^{n}\gamma_{k}A_{k}$ is closed on its natural domain 
$\bigcap_{k=1}^n \dom(A_k)$, and 
\begin{equation} 
\left(\sum_{k=1}^{n}\gamma_{k}A_{k}\right)^{*}=-\sum_{k=1}^{n}\gamma_{k}A_{k}.    \lb{dA} 
\end{equation} 
\end{lemma}
\begin{proof}
We prove \eqref{dA} by induction on $n$. The case $n=1$ follows from $\left(\gamma_{1}A_{1}\right)^{*}=A_{1}^{*}\gamma_{1}^{*}=-A_{1}\gamma_{1}=-\gamma_{1}A_{1}$.

Next, assume the assertion holds for $n\in \bbN$ and consider the sum 
\[
A\coloneqq\sum_{k=1}^{n+1}\gamma_{k}A_{k}=\gamma_{1}A_{1}+\sum_{k=2}^{n+1}\gamma_{k}A_{k}, 
\]
with its natural domain $\bigcap_{k=1}^n \dom(A_k)$. Since $\gamma_{1}^{2}=I_{\cH}$, $\gamma_{1}$
defines an isomorphism from $\cH$ into itself. Hence, $A$ is closed
if and only if $\text{\ensuremath{\gamma}}_{1}A$ is closed. One notes,  
\begin{equation}
\left(\gamma_{1}\sum_{k=2}^{n+1}\gamma_{k}A_{k}\right)^{*} =-\sum_{k=2}^{n+1}\gamma_{k}A_{k}\gamma_{1} 
 =-\sum_{k=2}^{n+1}\gamma_{k}\gamma_{1}A_{k}
  = \sum_{k=2}^{n+1}\gamma_{1}\gamma_{k}A_{k} 
=\gamma_{1}\sum_{k=2}^{n+1}\gamma_{k}A_{k},
\end{equation}
in addition, $\gamma_{1}\gamma_{1}A_{1}=A_{1}$ is skew-self-adjoint.
With the help of Theorem \ref{thm:resolvent_normal} it remains to
check whether the resolvents of $A_{1}$ and $\gamma_{1}\sum_{k=2}^{n+1}\gamma_{k}A_{k}$
commute. One observes that for $z\in\rho(A_{1})$, 
\begin{align}
\left(z-A_{1}\right)^{-1}\gamma_{1}\sum_{k=2}^{n+1}\gamma_{k}A_{k} & =\gamma_{1}\left(z-A_{1}\right)^{-1}\sum_{k=2}^{n+1}\gamma_{k}A_{k} 
 =\gamma_{1}\sum_{k=2}^{n+1}\gamma_{k}\left(z-A_{1}\right)^{-1}A_{k}   \no \\
 & \subseteq\gamma_{1}\sum_{k=2}^{n+1}\gamma_{k}A_{k}\left(z-A_{1}\right)^{-1}.      \lb{4.7}
\end{align}
Adding $-z'\left(z-A_{1}\right)^{-1}$ for some $z'\in\rho\left(\gamma_{1}\sum_{k=2}^{n+1}\gamma_{k}A_{k}\right)$
to both sides of inclusion \eqref{4.7}, one obtains 
\begin{align*}
 & -z'\left(z-A_{1}\right)^{-1}+\left(z-A_{1}\right)^{-1}\gamma_{1}\sum_{k=2}^{n+1}\gamma_{k}A_{k}\\
 & \quad \subseteq-z'\left(z-A_{1}\right)^{-1}+\gamma_{1}\sum_{k=2}^{n+1}\gamma_{k}A_{k}\left(z-A_{1}\right)^{-1}. 
\end{align*}
Thus,
\[
\left(z-A_{1}\right)^{-1}\left(z'-\gamma_{1}\sum_{k=2}^{n+1}\gamma_{k}A_{k}\right)\subseteq\left(z'-\gamma_{1}\sum_{k=2}^{n+1}\gamma_{k}A_{k}\right)\left(z-A_{1}\right)^{-1},
\]
implying  
\[
\left(z'-\gamma_{1}\sum_{k=2}^{n+1}\gamma_{k}A_{k}\right)^{-1}\left(z-A_{1}\right)^{-1}\subseteq\left(z-A_{1}\right)^{-1}\left(z'-\gamma_{1}\sum_{k=2}^{n+1}\gamma_{k}A_{k}\right)^{-1},
\]
proving assertion \eqref{dA} since the domain of the operator on the 
left-hand side is all of $\cH$.
\end{proof}

For proving the Fredholm property of $L=\mathcal{Q}+\Phi$, we will employ stability of the Fredholm property under relatively compact perturbations, or, in other words, that the essential spectrum is invariant under additive relatively compact perturbations. Thus, we need a compactness criterion 
and hence we recall the following compactness result for multiplication operators, a consequence of the Rellich--Kondrachov theorem, see \cite[Theorem 6.3]{AF03} (cf. and \eqref{eq:def_Hs} for the definition of $H^{1}(\mathbb{R}^{n})$). 

\begin{theorem}
\label{thm:Newton_pot_is_h1-bdd}Let $n\in\mathbb{N}$ and 
$\phi\ \in L^{\infty}(\bbR^n)$ such that for all $\epsilon>0$ there exists
$\Lambda > 0$ such that for all $x\in\mathbb{R}^{n}\backslash B(0,\Lambda)$, $\left|\phi(x)\right|\leq \epsilon$. 
 Then 
 $$\phi\colon \begin{cases} H^{1}(\mathbb{R}^{n}) \to L^{2}(\mathbb{R}^{n}), \\ 
 f\mapsto\phi(\cdot)f(\cdot), \end{cases} \, \text{ is compact.}  
 $$
\end{theorem}
\begin{proof}
As $H^{1}(\mathbb{R}^{n})$ is a Hilbert space, it suffices to prove that weakly convergent sequences are mapped to norm-convergent sequences: Suppose that  $\{f_{k}\}_{k\in \bbN}$ weakly converges to some $f$ in $H^{1}(\mathbb{R}^{n})$ and denote $M\coloneqq\sup_{k\in\mathbb{N}}\|f_{k}\|_{H^{1}(\bbR^n)}$, which is finite by the uniform boundedness principle.   In particular, $\{f_{k}\}_{k\in\bbN}$ converges weakly in $H^{1}(B(0,\Lambda))$ for
every $\Lambda>0$. Hence, by the Rellich--Kondrachov theorem, for all
$\Lambda>0$ the sequence $\{f_{k}\}_{k\in\bbN}$ converges in $L^{2}(B(0,\Lambda))$.
Next, let $\epsilon>0$. As $f\in L^{2}(\mathbb{R}^{n})$, there exists $\Lambda_{0}>0$
such that $\| f\chi_{B(0,\Lambda_{0})}-f\|_{L^{2}}\leq\epsilon$, where we denoted by $\chi_{B(0,\Lambda_0)}$ the cut-off function being $1$ on the ball $B(0,\Lambda_0)$ and $0$ elswhere.
One can find $\Lambda\geq \Lambda_0$ such that $\left|\phi(x)\right|\leq \epsilon$
for $\left|x\right|\geq \Lambda$, and $k_{0}\in\mathbb{N}$ such that for
all $k\geq k_{0}$, one has $\|f_{k}\chi_{B(0,\Lambda)}-f\chi_{B(0,\Lambda)}\|_{L^{2}}\leq\epsilon.$
Thus, for $k\geq k_{0}$ one arrives at
\begin{align}
\left\Vert \phi f_{k}-\phi f\right\Vert ^{2} & =\left\Vert \phi f_{k}\chi_{B(0,\Lambda )}-\phi f\chi_{B(0,\Lambda )}\right\Vert ^{2}_{L^2}
+\left\Vert \phi f_{k}\chi_{\mathbb{R}^{n}\backslash B(0,\Lambda )}
-\phi f\chi_{\mathbb{R}^{n}\backslash B(0,\Lambda )}\right\Vert^{2}_{L^2}   \no \\
 & \leq \|\phi\|_{L^{\infty}}^{2}\epsilon^{2}+\epsilon^{2}\left(2M\right)^{2}.  
\end{align}
\end{proof}

\begin{remark}
  The latter theorem has the following easy but important corollary: In the situation of Theorem \ref{thm:Newton_pot_is_h1-bdd}, let $\mathcal{H}$ be a Hilbert space continuously embedded into $H^1(\mathbb{R}^n)$, for instance, $\mathcal{H}=H^2(\mathbb{R}^n)$ (cf.~\eqref{eq:def_Hs}), then the operator $\phi_{\mathcal H\to L^2}$ of multiplying by $\phi$ considered from $\mathcal{H}$ to $L^2(\mathbb{R}^n)$ is compact. Denoting by $\iota \colon \mathcal H\to H^1(\mathbb{R}^n)$ the continuous embedding, which exists by hypothesis, one observes that 
  \[
     \phi_{\mathcal{H}\to L^2(\mathbb{R}^n)} = \phi_{H^1(\mathbb{R}^n)\to L^2(\mathbb{R}^n)}\circ \iota,
  \]
with $\phi_{H^1(\mathbb{R}^n)\to L^2(\mathbb{R}^n)}$ being the operator discussed in Theorem \ref{thm:Newton_pot_is_h1-bdd}. Hence, the operator $\phi_{\mathcal{H}\to L^2(\mathbb{R}^n)}$ is compact as a composition of a continuous and a compact operator. \hfill $\diamond$
\end{remark}

The proof of Theorem \ref{thm:Fredholm_property} will rest on the observation that $L$ is Fredholm if and only if the essential spectra of $L^*L$ and $LL^*$ have strictly positive lower bounds. Thus, we formulate two propositions describing the opertors $L^*L$ and $LL^*$ in bit more detail:

\begin{proposition}
\label{prop:Computation of adjoint}The operator $L$ given by \eqref{eq:def_of_L}
is closed and densely defined in $L^{2}(\mathbb{R}^{n})^{2^{\hat n}d}$ and 
\begin{equation} 
L^{*}=-\cQ+\Phi, \quad \dom\left(L^{*}\right)=\dom\left(L\right) = H^{1}(\mathbb{R}^{n})^{2^{\hat n}d}.
\end{equation}  
\end{proposition}
\begin{proof}
Since the operator of multiplication with the function $\Phi$ is
bounded and self-adjoint, the assertion is immediate from Theorem \ref{thm:L_is_closed}. 
\end{proof}

\begin{proposition}
\label{prop:compu_lstartl}Assume that 
$\Phi\in C_{b}^{\infty}\big(\mathbb{R}^{n};\mathbb{C}^{d\times d}\big)$ is pointwise self-adjoint, that is, $\Phi(\cdot) = \Phi(\cdot)^*$. For $L=\cQ+\Phi$ given by \eqref{eq:def_of_L},
one then has $($cf.\ \eqref{Phi}$)$, 
\begin{equation}
L^{*}L=-\Delta I_{2^{\hatt n}d} - C + \Phi^{2} \, \text{ and } \, 
LL^{*}=-\Delta I_{2^{\hatt n}d} + C + \Phi^{2},\label{eq:LstarLandLLstar}
\end{equation}
where
\begin{equation}
C=\sum_{j=1}^{n}\gamma_{j,n} (\partial_{j}\Phi)=(\mathcal{Q}\Phi),\label{eq:commutator=00003DC}
\end{equation}
see also Remark \ref{rem:differen_mult_op}. Moreover, 
\begin{equation} 
\dom(L^{*}L) = \dom(LL^{*}) = H^{2}(\mathbb{R}^{n})^{2^{\hat n}d} 
\end{equation} 
 $($see \eqref{eq:def_Hs}
for a definition of the latter\,$)$. 
\end{proposition}
\begin{proof}
At first one observes that if 
$\psi\in H^{k}(\mathbb{R}^{n})^{2^{\hat n}d}$ 
and $L\psi\in H^{k}(\mathbb{R}^{n})^{2^{\hat n}d}$, 
then $\psi\in H^{k+1}(\mathbb{R}^{n})^{2^{\hat n}d}$. Indeed, from 
$L\psi=\cQ\psi+\Phi\psi$, one infers 
$\psi+\cQ\psi=\psi+L\psi+\Phi\psi\in H^{k}(\mathbb{R}^{n})^{2^{\hat n}d}$ by the 
differentiablity of $\Phi$. By Theorem \ref{thm:L_is_closed},
the operator $\cQ$ is skew-self-adjoint and therefore $-1\in\rho(\cQ)$.
Hence, $\psi=(\cQ + I)^{-1}(\cQ + I)\psi\in H^{k+1}(\mathbb{R}^{n})^{2^{\hat n}d}$. Therefore,
if $\psi\in \dom(L)=H^{1}(\mathbb{R}^{n})^{2^{\hat n}d}$ 
with $L\psi\in \dom(L^{*})=H^{1}(\mathbb{R}^{n})^{2^{\hat n}d}$,
then $\psi\in H^{2}(\mathbb{R}^{n})^{2^{\hat n}d}$. On the other hand, 
if $\psi\in H^{2}(\mathbb{R}^{n})^{2^{\hat n}d}$, 
then also $\psi\in \dom(L^{*}L)$. The same reasoning applies to $LL^{*}$.

Next, we compute $L^{*}L$. With Proposition \ref{prop:Computation of adjoint}
one obtains 
\begin{align*}
L^{*}L & =(- \cQ+\Phi)(\cQ+\Phi) = - \cQ\cQ+\Phi \cQ-\cQ\Phi+\Phi^{2}
\end{align*}
and
\begin{align*}
LL^{*} & =(\cQ+\Phi)(- \cQ+\Phi) = - \cQ\cQ-\Phi \cQ+ \cQ\Phi+\Phi^{2}.
\end{align*}
Recalling $- \cQ\cQ=-\Delta I_{2^{\hatt n}d}$ from \eqref{eq:q2=00003DDelta}, one concludes the proof with the observation 
$\Phi \cQ- \cQ\Phi=\Phi \cQ-\Phi \cQ+C=C$, applying the product rule. 
\end{proof}

We may now come to the proof of the Fredholm property of $L=\mathcal{Q}+\Phi$ with smooth potential $\Phi$ satisfying for some $c>0$, $|\Phi(x)|\geq c I_d$, $x\in \mathbb{R}^n$, as well as satisfying $C(x)=(\mathcal{Q}\Phi)(x) \to 0$ as $|x|\to \infty$:

\begin{proof}[Proof of Theorem \ref{thm:Fredholm_property}]
By hypothesis,
$\left(\Phi(x)\right)^{2}=\left|\Phi(x)\right|^{2}\geq c^2 I_d$, $x\in\mathbb{R}^{n}$. From
\[
-\Delta I_d +\Phi^{2}\geq c^2 I_d,
\]
one deduces that the spectrum of $-\Delta I_d + \Phi^{2}$ is contained in
$[c^2,\infty)$. In particular, one concludes that the essential spectrum
$\sigma_{\textnormal{ess}}(-\Delta I_d+\Phi^{2})$
of $-\Delta I_d + \Phi^{2}$ is also contained in $[c^2,\infty)$. Since
$x\mapsto C(x)=\left(Q\Phi\right)(x)$ satisfies the condition imposed
on $\Phi$ in Theorem \ref{thm:Newton_pot_is_h1-bdd}, one infers that
$C$ is $-\Delta I_d+\Phi^{2}$-compact, since the domain of the latter
(closed) operator coincides with $H^{2}(\mathbb{R}^{n})^d$, which is continuously
embedded into $H^1(\mathbb{R}^n)^d$. Recalling
Proposition \ref{prop:compu_lstartl}, that is,
\[
L^{*}L=-\Delta I_{2^{\hatt n}d} - C+\Phi^{2},
\]
one obtains $\sigma_{\textnormal{ess}}(L^{*}L)=\sigma_{\textnormal{ess}}(-\Delta I_d + \Phi^{2})\subseteq[c^2,\infty)$, as the essential spectrum is invariant under additive relatively compact
perturbations (see, e.g., \cite[Theorem 5.35]{Ka80}). In particular,
$0\notin\sigma_{\text{ess}}(L^{*}L)$ implying that 
$L^{*}L$ is Fredholm. By a similar argument applied to $LL^{*}$, one deduces the
Fredholm property of $L$ (using that $\ker(L)=\ker(L^{*}L)$ and $\ker(L^{*})=\ker(LL^{*})$).
\end{proof}

In the following sections, we are interested in a particular subclass of potentials $\Phi$. In particular, we focus on potentials for which we may apply Theorem \ref{thm:index with Witten}. A first main focus is set on potentials satisfying the properties stated in Definition \ref{def:phi_admissible}, the so-called \emph{admissible} potentials. The reader is referred to Section \ref{sec:The Index theorem} and beyond for possible generalizations. It should be noted, however, that for more general potentials the derivations and arguments are more involved than for the ones mentioned in Definition \ref{def:phi_admissible}. In fact, the main reason being assumption $(ii)$ on the invertibility of $\Phi$ everywhere. It is known (see the end of Section \ref{sec:The Index theorem}) that the operator $L=\mathcal{Q}+\Phi$ has index $0$ for $\Phi$ satisfying Definition \ref{def:phi_admissible}. Later on, we shall see that the study of potentials being invertible on complements of large balls around $0$ can be reduced to the study of potentials being invertible everywhere except on a sufficiently small ball around $0$. 
The arguments for the latter case, in turn, rest on the perturbation theory for the Helmholtz equation, see Section \ref{sec:pert}. Hence, the derivation for the index formula for potentials being invertible everywhere except on a sufficiently small ball can be regarded as a perturbed version of the arguments given for admissible potentials. Therefore, we chose to present the core arguments for the by far simpler case of admissible potentials first. 

The precise notion of what we call \emph{admissible potentials} reads as follows.

\begin{definition} \label{def:phi_admissible} Let $\Phi\colon\mathbb{R}^{n}\to\mathbb{C}^{d \times d}$
for some $d,n\in\mathbb{N}$. We call $\Phi$ \emph{admissible}, if
the following conditions $(i)$--$(iii)$ hold: \\[1mm] 
$(i)$ $($smoothness\,$)$  
$\Phi\in C_{b}^{\infty}\big(\mathbb{R}^{n};\mathbb{C}^{d \times d}\big)$. \\[1mm] 
$(ii)$ $($invertibility and self-adjointness\,$)$ for all $x\in\mathbb{R}^{n}$, 
$\Phi(x)^{*}=\Phi(x)=\Phi(x)^{-1}$. \\[1mm] 
$(iii)$ $($asymptotics of the derivatives\,$)$ for all $\alpha\in\mathbb{N}_{0}^{n}$, 
there exists $\kappa>0$ and $\epsilon> 1/2$ such that 
\[
\|(\partial^{\alpha}\Phi)(x)\|\leq \begin{cases}
\kappa (1+|x|)^{-1}, & \left|\alpha\right|=1, \\[1mm]
\kappa (1+ |x|)^{-1-\epsilon}, & \left|\alpha\right|\geq 2,
\end{cases} \quad x\in\mathbb{R}^{n},
\] 
where we employed multi-index notation and used the convention $\left|\alpha\right|=\sum_{j=1}^{n}\alpha_{j}$.
\end{definition}

\begin{remark}
If $\Psi\in C_{b}^{\infty}\big(\mathbb{R}^{n}\backslash B(0,1);\mathbb{C}^{d \times d}\big)$ 
 is \emph{homogeneous of order $0$}, that is, 
for all $x\in\mathbb{R}^{n}\backslash \{0\}$, $\Psi(x)=\Psi (x/|x|)$, 
then $\Psi$ satisfies Definition \ref{def:phi_admissible}\,$(iii)$. Indeed, one computes for $x=\{x_{j}\}_{j\in\{1,\ldots,n\}}\in\mathbb{R}^{n}\backslash \{0\}$
and $j\in\{1,\ldots,n\}$,  
\[
\partial_{j}\left(\frac{\cdot}{\left|\cdot\right|}\right)(x)=\frac{1}{\left|x\right|}\left(\begin{array}{c}
0\\
\vdots\\
1\\
\vdots\\
0
\end{array}\right)-\frac{x_{j}}{\left|x\right|^{3}}\left(\begin{array}{c}
x_{1}\\
\vdots\\
x_{j}\\
\vdots\\
x_{n}
\end{array}\right), 
\] 
and 
\begin{align*}
\left(\partial_{j}\Psi\right)(x) & =\partial_{j}\left(\Psi\circ\left(\frac{\cdot}{\left|\cdot\right|}\right)\right)(x)\\
 & =\left(\begin{array}{ccccc}
\left(\partial_{1}\Psi\right)\left(x/\left|x\right|\right) & \cdots 
& \left(\partial_{j}\Psi\right)\left(x/\left|x\right|\right) 
& \cdots & \left(\partial_{n}\Psi\right)\left(x/\left|x\right|\right)\end{array}\right)   \\
& \quad \times \left(\frac{1}{\left|x\right|}\left(\begin{array}{c}
0\\
\vdots\\
1\\
\vdots\\
0
\end{array}\right)-\frac{x_{j}}{\left|x\right|^{3}}\left(\begin{array}{c}
x_{1}\\
\vdots\\
x_{j}\\
\vdots\\
x_{n}
\end{array}\right)\right)\\
 & =\frac{1}{\left|x\right|}\sum_{k=1}^{n}\left(\partial_{k}\Psi\right)\left(\frac{x}{\left|x\right|}\right)\left(\delta_{kj}-\frac{x_{k}x_{j}}{\left|x\right|^{2}}\right),
\end{align*}
establishing the assertion. We note that Callias \cite{Ca78} assumes that the potential ``approaches a homogeneous function of order $0$ as $|x|\to\infty$'' such that 
Definition \ref{def:phi_admissible}\,$(iii)$ is satisfied.  \hfill $\diamond$
\end{remark}

\newpage

\section{Derivation of the Trace Formula -- The Trace Class Result}\label{sec:The-Derivation-of-trace-f}

In this section, we shall prove the applicability of Theorem \ref{thm:index with Witten} for the operator 
\begin{equation}
L= \cQ+\Phi\label{eq:def_of_L(2)}
\end{equation}
in $L^{2}(\mathbb{R}^{n})^{2^{\hat n}d}$ as introduced in \eqref{eq:def_of_L}
with 
\[
\cQ=\sum_{j=1}^{n}\gamma_{j,n}\partial_{j}
\]
given by \eqref{eq:def_of_Q2} (or \eqref{eq:Def_of_Q}) and an admissible
potential $\Phi$, see Definition \ref{def:phi_admissible}. More
precisely, we seek to establish that the operator 
\begin{equation}
\chi_\Lambda B_{L}(z)=z\chi_\Lambda\tr_{2^{\hat n}d}\big(\left(L^{*}L+z\right)^{-1}-\left(LL^{*}+z\right)^{-1}\big), \quad z\in\rho(-LL^{*})\cap\rho(-L^{*}L),    \label{eq:def_of_BL(z)2}
\end{equation} belongs to the trace class $\mathcal{B}_1\big(L^2(\mathbb{R}^n)\big)$, 
where $\tr_{2^{\hat n}d}$ is given in \eqref{eq:Def_of_int_tr} and $\chi_\Lambda$ is the multiplication operator of multiplying with the characteristic function of the ball centered at $0$ with radius 
$\Lambda>0$, that is,
\begin{equation}\label{eq:def_of_chi}
   \chi_\Lambda(x)\coloneqq \begin{cases}
                          1, & x\in B(0,\Lambda),\\
                          0, & x\in \mathbb{R}^n\backslash  B(0,\Lambda).
                         \end{cases}
\end{equation}
Regarding Theorem \ref{thm:index with Witten} (with $T_\Lambda = \chi_\Lambda$ and $S_\Lambda^*=I_{L^2(\mathbb{R}^n)}$), we are then interested in computing the limit for $\Lambda\to\infty$ of $\tr_{L^2(\bbR^n)}(\chi_\Lambda B_{L}(z))$.  This requires 
showing that $\chi_\Lambda B_{L}(z)$ is indeed trace class for all $\Lambda>0$. The limit $z\to0$ of $\lim_{\Lambda\to\infty} \tr_{L^2(\bbR^n)} (\chi_\Lambda B_{L}(z))$  (provided it exists in an appropriate way, see \eqref{indF(0)} in Theorem \ref{thm:index with Witten}) then corresponds to the index of $L$. It turns out that to compute the limit of $z\to0$ in the expression 
$\lim_{\Lambda\to\infty} \tr_{L^2(\bbR^n)} (\chi_\Lambda B_{L}(z))$ is rather straightforward (see also Theorem \ref{thm:Fredholm-index}), once the respective formula is established\footnote{From now on, we shall \emph{only} furnish the internal trace, introduced in Definition \ref{def:bigtrace_littletrace}, of operators
living on an orthogonal sum of a Hilbert space, with an additional
subscript. The operator $\tr$ without subscript will always refer to 
the trace of a trace class operator acting in some fixed underlying Hilbert space. In particular,
for $A\in\mathbb{C}^{d \times d}$, the expression $\tr (A)$ denotes 
the sum of the diagonal entries.}.  
The main theorem, which we shall prove in the next two sections, reads as follows.

\begin{theorem}
\label{thm:Witten_reg_n5} Let $z\in\rho\left(-LL^{*}\right)\cap\rho\left(-L^{*}L\right)$
with $\Re (z) > - 1$ and $n\in\mathbb{N}_{\geq3}$ odd. Suppose that  $\Phi$ is admissible $($see Definition \ref{def:phi_admissible}$)$.Then the operator
$\chi_\Lambda B_{L}(z)$ with $B_{L}(z)$ and $\chi_\Lambda$ given by \eqref{eq:def_of_BL(z)2} and \eqref{eq:def_of_chi}, respectively, is trace class, the limit 
$f(z)\coloneqq \lim_{\Lambda\to\infty} \tr (\chi_\Lambda B_L(z))$ exists and is given by  
\begin{align}
\begin{split} 
f(z) & = (1+z)^{-n/2}\left(\frac{i}{8\pi}\right)^{(n-1)/2}\frac{1}{[(n-1)/2]!} 
\lim_{\Lambda \to\infty}\frac{1}{2 \Lambda}\sum_{j,i_{1},\ldots,i_{n-1} = 1}^n 
\epsilon_{ji_{1}\ldots i_{n-1}} \\
 & \quad \; \times \int_{\Lambda S^{n-1}}\tr(\Phi(x)(\partial_{i_{1}} \Phi)(x)\ldots 
 (\partial_{i_{n-1}} \Phi)(x))x_{j}\, d^{n-1} \sigma(x),    \lb{f(z)} 
 \end{split} 
\end{align}
where $\epsilon_{ji_{1}\ldots i_{n-1}}$ denotes the $\epsilon$-symbol
as in Proposition \ref{prop:comp_of_Dirac_Alge}.
\end{theorem}

In order to deduce the latter theorem, we shall have a deeper look
into the inner structure of $B_{L}(z)$. A first step toward our goal is the following
result. 

\begin{lemma}
\label{lem:afirstcomm} Let $L$ and $B_{L}(z)$ be given by \eqref{eq:def_of_L(2)}
and \eqref{eq:def_of_BL(z)2}, respectively. Then for all $z\in\rho(-LL^{*})\cap\rho(-L^{*}L)$, 
\[
2B_{L}(z)=\tr_{2^{\hat n}d} \big(\big[L,L^{*}\left(LL^{*}
+ z\right)^{-1}\big]\big)-\tr_{2^{\hat n}d} \big(\big[L^{*},L\left(L^{*}L+z\right)^{-1}\big]\big) 
\]
$($where $\left[\cdot,\cdot\right]$ represents the commutator symbol, 
cf.\ \eqref{eq:def_commutator}$)$. 
\end{lemma}
\begin{proof}
Let $z\in\rho\left(-LL^{*}\right)\cap\rho\left(-L^{*}L\right)$. One computes 
\begin{align*}
[L,L^{*}\left(LL^{*}+z\right)^{-1}] & =LL^{*}\left(LL^{*}+z\right)^{-1}-L^{*}\left(LL^{*}+z\right)^{-1}L\\
 & =\left(LL^{*}+z\right)\left(LL^{*}+z\right)^{-1}-z\left(LL^{*}+z\right)^{-1}-\left(L^{*}L+z\right)^{-1}L^{*}L\\
 & =1-z\left(LL^{*}+z\right)^{-1}-\left(L^{*}L+z\right)^{-1}\left(L^{*}L+z\right)+\left(L^{*}L+z\right)^{-1}z\\
 & =1-z\left(LL^{*}+z\right)^{-1}-1+\left(L^{*}L+z\right)^{-1}z\\
 & =z\left(L^{*}L+z\right)^{-1}-z\left(LL^{*}+z\right)^{-1},
\end{align*}
 and, interchanging the roles of $L$ and $L^{*}$, one concludes 
\[
[L^{*},L\left(L^{*}L+z\right)^{-1}]=z\left(LL^{*}+z\right)^{-1}-z\left(L^{*}L+z\right)^{-1}. 
\tag*{{\qedhere}}
\]
\end{proof}

The forthcoming Proposition \ref{prop:Reformulation of B} gives a more detailed description of the commutators describing $B_L(z)$ just derived in Lemma \ref{lem:afirstcomm}. First, we need a prerequisit of a more general nature.

\begin{lemma}\label{lem:com_with_Q} Let $n\in\mathbb{N}$, 
$B\in \mathcal{B}\big(L^2(\mathbb{R}^n)^{2^{\hat n}d},L^2(\mathbb{R}^n)^{2^{\hat n}d}\big)$ 
and let $\cQ$ and $\gamma_{j,n}$, $j\in\{1,\ldots,n\}$, as in \eqref{eq:def_of_Q2} and in Remark \ref{rem:Eucl-Dirac-Algebar}, respectively. Then, on the common natural domain of the operator sums involved, one has
\[
   \tr_{2^{\hat n}d} ([\cQ,B]) =\sum_{j=1}^n \tr_{2^{\hat n}d} ([\partial_j,\gamma_{j,n}B])  
   =\sum_{j=1}^n \tr_{2^{\hat n}d} ([\partial_j,B\gamma_{j,n}]).
\] 
\end{lemma}
\begin{proof}
One computes with the help of Proposition \ref{prop:cyclic property of inner trace} and the fact 
$\gamma_{j,n}\partial_j=\partial_j\gamma_{j,n}$, 
\begin{align*}
   \tr_{2^{\hat n}d} (\cQ B-B \cQ) &= \sum_{j=1}^n \tr_{2^{\hat n}d}\left(\gamma_{j,n}\partial_jB - B\gamma_{j,n}\partial_j\right)\\
                    &= \sum_{j=1}^n \left[\tr_{2^{\hat n}d}\left(\gamma_{j,n}\partial_jB\right) - \tr_{2^{\hat n}d}\left((B\gamma_{j,n}\right)\partial_j)\right]  \\
                    &= \sum_{j=1}^n \left[\tr_{2^{\hat n}d}\left(\partial_j\gamma_{j,n}B\right) - \tr_{2^{\hat n}d}(\left(\gamma_{j,n}B\right)\partial_j)\right]\\
                    &= \sum_{j=1}^n \tr_{2^{\hat n}d}(\left[\partial_j,\gamma_{j,n}B\right]).
\end{align*}
The second equality can be shown similarly.
\end{proof}

The following proposition represents the core of the derivation of the index formula. Once it is proven that $\chi_\Lambda B_L(z)$ is trace class, with the trace being computed as the integral over the diagonal of the respective integral kernel, equation \eqref{eq:B_L=00003DJJJ+ALz} will be the key for computing the trace. More precisely, the first summand is a sum of commutators of certain operators with partial derivatives. For the respective integral kernels, this will give us an expression as in Lemma \ref{lem:comm is sum of der} (see also \eqref{e:div_theo}), which will enable us to use Gauss' divergence theorem, explaining the surface integral in \eqref{f(z)}. Furthermore, the second summand in equation \eqref{eq:B_L=00003DJJJ+ALz} as can be seen in equation \eqref{eq:ALz} is basically a commutator of an integral operator and a multiplication operator. The integral kernels of this type of operators have been shown to vanish on the diagonal in Proposition \ref{prop:commutator with phi vanishes}, thus, \eqref{eq:ALz} will give a vanishing contribution to the trace of $B_L(z)$.

\begin{proposition}[{\cite[Proposition~1, p.~219]{Ca78}}]
\label{prop:Reformulation of B} Let $L$ be given by \eqref{eq:def_of_L(2)}
and $z\in\rho(-L^{*}L)\cap\rho(-LL^{*})$. 
Then $B_{L}(z)$ given by \eqref{eq:def_of_BL(z)2} satisfies
\begin{equation}
2B_{L}(z)=\sum_{j=1}^n\big[\partial_{j},J_{L}^{j}(z)\big]+A_{L}(z), 
\label{eq:B_L=00003DJJJ+ALz}
\end{equation} 
where 
\begin{equation}
J_{L}^{j}(z)=\tr_{2^{\hat n}d} \big(L\left(L^{*}L+z\right)^{-1}\gamma_{j,n}\big) 
+\tr_{2^{\hat n}d} \big(L^{*}\left(LL^{*}+z\right)^{-1}\gamma_{j,n}\big),  \quad j\in\{1,\ldots,n\}, 
\label{eq:JJJ}
\end{equation}
and 
\begin{equation}
A_{L}(z)=\tr_{2^{\hat n}d} \big(\big[\Phi,L^{*}(LL^{*}+z)^{-1}\big]\big)  
-\tr_{2^{\hat n}d} \big(\big[\Phi,L(L^{*}L+z)^{-1}\big]\big),   \label{eq:ALz}
\end{equation}
with $\gamma_{j,n}$ as in Remark \ref{rem:Eucl-Dirac-Algebar}
or Appendix \ref{sec:Appendix:-the-Construction}. 
\end{proposition} 
\begin{proof}
One recalls that $L^{*}=- \cQ+\Phi$ from Proposition \ref{prop:Computation of adjoint}.
From Lemma \ref{lem:afirstcomm}, one infers that
\begin{align*}
2B_{L}(z) & =\tr_{2^{\hat n}d} \big(\big[L,L^{*} (LL^{*}+z)^{-1}\big]\big) 
- \tr_{2^{\hat n}d} \big(\big[L^{*},L(L^{*}L+z)^{-1}\big]\big)  \\
 & =\tr_{2^{\hat n}d} \big(\big[ \cQ +\Phi,L^{*}(LL^{*}+z)^{-1}\big]\big) 
 - \tr_{2^{\hat n}d} \big(\big[- \cQ +\Phi,L(L^{*}L+z)^{-1}\big]\big)   \\
 & =\tr_{2^{\hat n}d} \big(\big[ \cQ,L^{*}(LL^{*}+z)^{-1}\big]\big)  
 + \tr_{2^{\hat n}d} \big(\big[ \cQ,L(L^{*}L+z)^{-1}\big]\big)  \\
 & \quad+\tr_{2^{\hat n}d} \big(\big[\Phi,L^{*}(LL^{*}+z)^{-1}\big]\big)   
 - \tr_{2^{\hat n}d} \big(\big[\Phi,L(L^{*}L+z)^{-1}\big]\big).
\end{align*}
The equations 
\[
\tr_{2^{\hat n}d} \big(\big[ \cQ,L^{*}(LL^{*}+z)^{-1}\big]\big)  
 =\sum_{j=1}^{n}\tr_{2^{\hat n}d} \big(\big[\partial_{j},L^{*}(LL^{*}
 + z)^{-1}\gamma_{j,n}\big]\big),
\]
and
\begin{align*}
\tr_{2^{\hat n}d} \big(\big[ \cQ,L(L^{*}L+z)^{-1}\big]\big) 
& =\sum_{j=1}^{n}\tr_{2^{\hat n}d} \big(\big[\partial_{j},L(L^{*}L+z)^{-1}\gamma_{j,n}
\big]\big)
\end{align*}
follow from Lemma \ref{lem:com_with_Q}.
\end{proof}

Next, we show that (a modification in the sense of Theorem \ref{thm:index with Witten} of) $B_{L}(z)$ gives rise to trace class operators. Before doing so in Theorem \ref{thm:trisbounded}, we need a different representation of $B_{L}(z)$ in terms of powers of the resolvent
of the (free) Laplacian. One notes that for $z\in \mathbb{C}$, 
with $\Re (z) > \sup_{x\in\mathbb{R}^{n}}\max_{j}\|\partial_{j}\Phi(x)\|-1$, 
one has $\|CR_{1+z}\|<1$, with $C$ given by \eqref{eq:commutator=00003DC}. Hence, by Proposition \ref{prop:compu_lstartl}, equation \eqref{eq:LstarLandLLstar}, one obtains  
\begin{align}
\left(L^{*}L+z\right)^{-1} & =\left(-\Delta I_{2^{\hatt n}d} -C + (1+z)\right)^{-1} \notag \\
 & =\left(\left(-\Delta I_{2^{\hatt n}d} + (1+z)\right)\left(I_{2^{\hatt n}d} - R_{1+z}C\right)\right)^{-1}\notag \\
 & =\left(I_{2^{\hatt n}d} - R_{1+z}C\right)^{-1}R_{1+z} \notag \\
 & =\sum_{k=0}^{\infty}\left(R_{1+z}C\right)^{k}R_{1+z},   \label{eq:Neumann_lstarl},
\end{align}
and, similarly,
\begin{equation}\label{eq:Neumann_llstar}
\left(LL^{*}+z\right)^{-1} =\left(-\Delta I_{2^{\hatt n}d} +C+ (1+z)\right)^{-1} 
= \sum_{k=0}^{\infty}\left(-R_{1+z}C\right)^{k}R_{1+z}.
\end{equation}
Consequently, by analytic continuation, one obtains for $z\in\rho(-L^{*}L)\cap\rho(-LL^{*})$ with $\Re (z) > -1$, 
\begin{equation}\label{eq:almost_Neumann_lstarl}
  \left(L^{*}L+z\right)^{-1}= \sum_{k=0}^{N}\left(R_{1+z}C\right)^{k}R_{1+z}+\left(R_{1+z}C\right)^{N+1}\left(L^{*}L+z\right)^{-1}, 
\end{equation}
and
\begin{equation}\label{eq:almost_Neumann_llstar}
  \left(LL^*+z\right)^{-1}= \sum_{k=0}^{N}\left(-R_{1+z}C\right)^{k}R_{1+z}+\left(-R_{1+z}C\right)^{N+1}\left(LL^*+z\right)^{-1}, 
\end{equation}
for all $N\in \mathbb{N}$. Focussing on resolvent differences, one gets the following proposition:

\begin{proposition}
\label{prop:resolvent formulas} Let $z\in\mathbb{C}_{\Re> - 1}$.
One recalls $L= \cQ+\Phi$ as in \eqref{eq:def_of_L(2)}, $C=\left( \cQ\Phi\right)$
from \eqref{eq:commutator=00003DC}, and $R_{1+z}$ in \eqref{eq:resolvent_of_laplace}. \\[1mm] 
$(i)$ If $\Re (z) > \sup_{x\in\mathbb{R}^{n}}\max_{j}\|(\partial_{j}\Phi)(x)\|-1$, 
then $z\in\rho\left(-L^{*}L\right)\cap\left(-LL^{*}\right)$ and 
\[
\left(L^{*}L+z\right)^{-1}-\left(LL^{*}+z\right)^{-1}=2\sum_{k=0}^{\infty}\left(R_{1+z}C\right)^{2k+1}R_{1+z}=2\sum_{k=0}^{\infty}R_{1+z}\left(CR_{1+z}\right)^{2k+1},
\]
as well as 
\[
\left(L^{*}L+z\right)^{-1}+\left(LL^{*}+z\right)^{-1}=2\sum_{k=0}^{\infty}\left(R_{1+z}C\right)^{2k}R_{1+z}=2\sum_{k=0}^{\infty}R_{1+z}\left(CR_{1+z}\right)^{2k}.
\]
$(ii)$ If $z\in\rho(-L^{*}L)\cap\rho(-LL^{*})$ and $\Re (z) > -1$, then
for all $N\in\mathbb{N}$, 
\begin{align*}
 & \left(L^{*}L+z\right)^{-1}-\left(LL^{*}+z\right)^{-1}\\
 & \quad =2\sum_{k=0}^{N}R_{1+z}\left(CR_{1+z}\right)^{2k+1}+\big(\left(L^{*}L+z\right)^{-1}-\left(LL^{*}+z\right)^{-1}\big)\left(CR_{1+z}\right)^{2N+2}\\
 & \quad =2\sum_{k=0}^{N}R_{1+z}\left(CR_{1+z}\right)^{2k+1}+\big(\left(L^{*}L+z\right)^{-1}+\left(LL^{*}+z\right)^{-1}\big)\left(CR_{1+z}\right)^{2N+3}, 
\end{align*}
and
\begin{align*}
 & \left(L^{*}L+z\right)^{-1}+\left(LL^{*}+z\right)^{-1}\\
 & \quad =2\sum_{k=0}^{N}R_{1+z}\left(CR_{1+z}\right)^{2k} 
 + \big(\left(L^{*}L+z\right)^{-1}+\left(LL^{*}+z\right)^{-1}\big)\left(CR_{1+z}\right)^{2N+2}.
\end{align*}
\end{proposition}
\begin{proof}
$(i)$ This is a direct consequence of equations \eqref{eq:Neumann_lstarl} and \eqref{eq:Neumann_llstar}. \\ 
\noindent 
$(ii)$ For $z$ as in part $(i)$ one computes, similarly to \eqref{eq:almost_Neumann_lstarl} and \eqref{eq:almost_Neumann_llstar}, with the help of item $(i)$ for $N\in\mathbb{N}$, 
\begin{align*}
 & \left(L^{*}L+z\right)^{-1}-\left(LL^{*}+z\right)^{-1}\\
 & \quad =2\sum_{k=0}^{N}R_{1+z}\left(CR_{1+z}\right)^{2k+1}+2\sum_{k=N+1}^{\infty}R_{1+z}\left(CR_{1+z}\right)^{2k+1}\\
 & \quad =2\sum_{k=0}^{N}R_{1+z}\left(CR_{1+z}\right)^{2k+1}+2\sum_{k=0}^{\infty}R_{1+z}\left(CR_{1+z}\right)^{2k+2N+2+1}\\
 & \quad =2\sum_{k=0}^{N}R_{1+z}\left(CR_{1+z}\right)^{2k+1}+2\sum_{k=0}^{\infty}R_{1+z}\left(CR_{1+z}\right)^{2k+1}\left(CR_{1+z}\right)^{2N+2}\\
 & \quad =2\sum_{k=0}^{N}R_{1+z}\left(CR_{1+z}\right)^{2k+1}+2\sum_{k=0}^{\infty}R_{1+z}\left(CR_{1+z}\right)^{2k}\left(CR_{1+z}\right)^{2N+3}.
\end{align*}
Hence, 
\begin{align*}
 & \left(L^{*}L+z\right)^{-1}-\left(LL^{*}+z\right)^{-1}\\
 & \quad =2\sum_{k=0}^{N}R_{1+z}\left(CR_{1+z}\right)^{2k+1}+\big(\left(L^{*}L+z\right)^{-1}-\left(LL^{*}+z\right)^{-1}\big)\left(CR_{1+z}\right)^{2N+2}\\
 & \quad =2\sum_{k=0}^{N}R_{1+z}\left(CR_{1+z}\right)^{2k+1}+\big(\left(L^{*}L+z\right)^{-1}+\left(LL^{*}+z\right)^{-1}\big)\left(CR_{1+z}\right)^{2N+3},
\end{align*}
again by part $(i)$. Analytic continuation implies the asserted equalities. (The second term in 
item $(ii)$ is treated analogously).
\end{proof}

Before starting the proof that $\chi_\Lambda B_{L}(z)$, with $B_L(z)$ given by \eqref{eq:def_of_BL(z)2}, is trace class, and then prove the trace formula in Theorem \ref{thm:Witten_reg_n5} for this operator, a closer inspection of the operators occuring
in Proposition \ref{prop:Reformulation of B} with the help of Proposition
\ref{prop:resolvent formulas} is in order. In particular, the principal aim of Lemma \ref{lem:almost_Neumann_series}, is twofold: 
on one hand, we will prove that the power series representation of $B_L(z)$,
basically derived in Proposition \ref{prop:resolvent formulas}, starts with an operator
essentially of the form
\[
  R_{1+z}\left(CR_{1+z}\right)^{2k+1}
\]
for some $k \in \bbN_0$. For this kind of operators we have a trace class criterion at hand, 
Theorem \ref{thm:Simon_Hilbert-Schmidt} together with Corollary \ref{cor:Comp-of-trace}. On the other hand, we also prove representation formulas for the operators in \eqref{eq:JJJ} and \eqref{eq:ALz}. These formulas also start with operators involving high powers of $R_{1+z}$. This leads to continuity and differentiability properties for the corresponding integral kernels enabling the application of 
Proposition \ref{prop:commutator with phi vanishes} and 
Lemma \ref{lem:comm is sum of der}.

The key idea for proving Lemma \ref{lem:almost_Neumann_series}, contained in Lemma \ref{l:can}, 
is to use the cancellation properties of the Euclidean Dirac algebra under the trace sign. For the Euclidean Dirac algebra we refer to Definition \ref{def:Euc-D-A}; moreover, we refer to Proposition
\ref{prop:comp_of_Dirac_Alge} for the cancellation properties.

\begin{lemma}\label{l:can} Let $L= \cQ+\Phi$ be given by \eqref{eq:def_of_L(2)}.
Let $z\in\mathbb{C}$ with $\Re (z) > -1$ and $z\in\rho\left(-L^{*}L\right)\cap\rho\left(-LL^{*}\right)$ and recall $C=[Q,\Phi]$, $k\in \mathbb{N}$ odd.
If either $k<n$ or $n$ is even, then
\[
  \tr_{2^{\hat n}d}\Big(R_{1+z} \big(C R_{1+z}\big)^{k} \Big) =0.
\]
\end{lemma}
\begin{proof}
  One observes using the fact that $\gamma_{j,n}$, $j\in\{1,\ldots,n\}$, commutes with both 
  $R_{1+z}$ and $(\partial_\ell\Phi)$, $\ell\in \{1,\ldots,n\}$ (cf.\ 
  Remark \ref{rem:Eucl-Dirac-Algebar}), that
  \begin{align*}
    R_{1+z} \big(C R_{1+z}\big)^{k}
    &= R_{1+z} \big(\sum_{\ell=1}^n \gamma_{\ell,n}(\partial_\ell \Phi) R_{1+z}\big)^{k}
    \\ & = R_{1+z} \sum_{\ell_1,\cdots,\ell_k=1}^n \gamma_{\ell_1,n}(\partial_{\ell_1} \Phi) R_{1+z}\cdots \gamma_{\ell_k,n}(\partial_{\ell_k} \Phi) R_{1+z}
    \\ & = \sum_{\ell_1,\ldots,\ell_k=1}^n \gamma_{\ell_1,n}\cdots \gamma_{\ell_k,n} R_{1+z}(\partial_{\ell_1} \Phi) R_{1+z}\cdots (\partial_{\ell_k} \Phi) R_{1+z}.
  \end{align*}
Next, employing 
\begin{align*}
  &  \tr_{2^{\hat n}d} \Big(\gamma_{\ell_1,n}\cdots \gamma_{\ell_k,n} R_{1+z}(\partial_{\ell_1} \Phi) R_{1+z}\cdots (\partial_{\ell_k} \Phi) R_{1+z}\Big)
  \\ & = \tr_{2^{\hat n}} \Big(\gamma_{\ell_1,n}\cdots \gamma_{\ell_k,n}\Big) \tr_{d}\Big( R_{1+z}(\partial_{\ell_1} \Phi) R_{1+z}\cdots (\partial_{\ell_k} \Phi) R_{1+z}\Big)
\end{align*}
for all $i_1,\ldots,i_k \in \{1,\ldots,n\}$, one concludes that
\[
   \tr_{2^{\hat n}d} \Big(R_{1+z} \big(C R_{1+z}\big)^{k}\Big) =0,
\]
by Proposition \ref{prop:comp_of_Dirac_Alge}.
\end{proof}

\begin{lemma}
\label{lem:almost_Neumann_series} Let $L= \cQ+\Phi$ be given by \eqref{eq:def_of_L(2)}.
Let $z\in\mathbb{C}$ with $\Re (z) > -1$ and $z\in\rho\left(-L^{*}L\right)\cap\rho\left(-LL^{*}\right)$.
One recalls $B_{L}(z)$, $J_{L}^{j}(z)$, and $A_{L}(z)$ given by \eqref{eq:def_of_BL(z)2},
\eqref{eq:JJJ}, and \eqref{eq:ALz}, respectively, as well as $R_{1+z}$
given by \eqref{eq:resolvent_of_laplace}. Then the following assertions hold: \\[1mm]
$(i)$ For all odd $n\in\mathbb{N}_{\geq3}$,
\begin{align}
2 B_{L}(z) &=\sum_{j=1}^n \big[\partial_{j},J_{L}^{j}(z)\big]+A_{L}(z)     \label{BL} \\
&=z\tr_{2^{\hat n}d}\big(2(R_{1+z}C)^{n}R_{1+z}+\big((L^{*}L+z)^{-1}-(LL^{*}+z)^{-1}\big)
(CR_{1+z})^{n+1}\big),  \no
\end{align}
and, for all $j\in\{1,\ldots,n\}$, 
\begin{align*}
J_{L}^{j}(z) & =2\tr_{2^{\hat n}d}\big(\gamma_{j,n} \cQ (R_{1+z}C)^{n-2}R_{1+z}\big)
+ 2\tr_{2^{\hat n}d}\big(\gamma_{j,n}\Phi (R_{1+z}C)^{n-1}R_{1+z}\big)    \\
& \quad+\tr_{2^{\hat n}d}\big(\gamma_{j,n} \cQ \big((L^{*}L+z)^{-1}+(LL^{*}+z)^{-1}\big)
(CR_{1+z})^{n}\big)   \\
& \quad + \tr_{2^{\hat n}d} \big(\gamma_{j,n} \Phi \big((L^{*}L+z)^{-1}
+(LL^{*}+z)^{-1}\big) (CR_{1+z})^{n}\big),
\end{align*}
and
\begin{align*}
A_{L}(z) &=\tr_{2^{\hat n}d} \big(\big[\Phi,\Phi\big(2(R_{1+z}C)^{n}R_{1+z}+\big((L^{*}L+z)^{-1}-(LL^{*}+z)^{-1}\big)  \\
& \hspace*{1.4cm} \times (CR_{1+z})^{n+1}\big)\big]\big)    \\
& \quad-\tr_{2^{\hat n}d} \big(\big[\Phi, \cQ \big(2 (R_{1+z}C)^{n-1}R_{1+z}+\big((L^{*}L+z)^{-1}-(LL^{*}+z)^{-1}\big)    \\
& \hspace*{1.7cm} \times (CR_{1+z})^{n}\big)\big]\big).
\end{align*}
$(ii)$ For all even $n\in\mathbb{N}$,
\begin{equation}
B_{L}(z)=0.    \label{BLeven}
\end{equation}
\end{lemma}
\begin{proof}
From Proposition \ref{prop:resolvent formulas}, one has for $\Re (z) > -1$
and all $N\in\mathbb{N}$,
\begin{align*}
& \left(L^{*}L+z\right)^{-1}-\left(LL^{*}+z\right)^{-1}\\
& \quad =2\sum_{k=0}^{N}R_{1+z} (CR_{1+z})^{2k+1}+\big((L^{*}L+z)^{-1}-(LL^{*}+z)^{-1}\big)(CR_{1+z})^{2N+2}.
\end{align*}
In addition, using Lemma \ref{l:can}, one deduces that for $n$ even,
\[
\tr_{2^{\hat n}d}\big((L^{*}L+z)^{-1}-(LL^{*}+z)^{-1}\big)=0,
\]
and, for $n$ odd,
\begin{align*}
& \tr_{2^{\hat n}d}\big((L^{*}L+z)^{-1}-(LL^{*}+z)^{-1}\big)\\
& \quad =\tr_{2^{\hat n}d} \big(2(R_{1+z}C)^{n}R_{1+z}+\big((L^{*}L+z)^{-1}-(LL^{*}+z)^{-1}\big) (CR_{1+z})^{n+1}\big).
\end{align*}
This proves \eqref{BL}.

In a similar fashion, using again Proposition \ref{prop:resolvent formulas}
and the ``cyclicity property'' of $\tr_{2^{\hat n}d}$ (see Proposition \ref{prop:cyclic property
of inner trace}), one obtains
\begin{align*}
& \tr_{2^{\hat n}d}\big(L\left(L^{*}L+z\right)^{-1}\gamma_{j,n}\big)
+ \tr_{2^{\hat n}d}\big(L^{*}\left(LL^{*}+z\right)^{-1}\gamma_{j,n}\big)  \\
& \quad =\tr_{2^{\hat n}d}\big(L\left(L^{*}L+z\right)^{-1}\gamma_{j,n}+L^{*}\left(LL^{*}+z\right)^{-1}\gamma_{j,n}\big)\\
& \quad =\tr_{2^{\hat n}d}\big(\left( \cQ+\Phi\right)\left(L^{*}L+z\right)^{-1}\gamma_{j,n}
+\left(- \cQ+\Phi\right)\left(LL^{*}+z\right)^{-1}\gamma_{j,n}\big)\\
& \quad =\tr_{2^{\hat n}d}\big( \cQ\left(L^{*}L+z\right)^{-1}\gamma_{j,n} - \cQ\left(LL^{*}+z\right)^{-1}\gamma_{j,n}\big)\\
& \qquad+\tr_{2^{\hat n}d}\big(\Phi\big(\left(L^{*}L+z\right)^{-1}\gamma_{j,n}+\left(LL^{*}+z\right)^{-1}\gamma_{j,n}\big)\big)\\
& \quad =\tr_{2^{\hat n}d}\big(\gamma_{j,n} \cQ\big(\left(L^{*}L+z\right)^{-1}-\left(LL^{*}+z\right)^{-1}\big)\big)\\
& \qquad+\tr_{2^{\hat n}d}\big(\gamma_{j,n}\Phi\big(\left(L^{*}L+z\right)^{-1}+\left(LL^{*}+z\right)^{-1}\big)\big)\\
& \quad =2\tr_{2^{\hat n}d}\big(\gamma_{j,n} \cQ\left(R_{1+z}C\right)^{n-2}R_{1+z}\big)+2\tr_{2^{\hat n}d}\big(\gamma_{j,n}\Phi \left(R_{1+z}C\right)^{n-1}R_{1+z}\big)\\
& \qquad+\tr_{2^{\hat n}d}\big(\gamma_{j,n} \cQ \big(\left(L^{*}L+z\right)^{-1}+\left(LL^{*}+z\right)^{-1}\big)\left(CR_{1+z}\right)^{n}\big)\\
& \qquad+\tr_{2^{\hat n}d}\big(\gamma_{j,n}\Phi\big(\left(L^{*}L+z\right)^{-1}+\left(LL^{*}+z\right)^{-1}\big)\left(CR_{1+z}\right)^{n}\big),
\end{align*}
and
\begin{align*}
A_{L}(z) & =\tr_{2^{\hat n}d} \big(\big[\Phi,L^{*}\left(LL^{*} + z\right)^{-1}\big]\big)
- \tr_{2^{\hat n}d} \big(\big[\Phi,L\left(L^{*}L+z\right)^{-1}\big]\big)   \\
& =\tr_{2^{\hat n}d} \big(\big[\Phi,L^{*}\left(LL^{*}+z\right)^{-1}-L\left(L^{*}L+z\right)^{-1}\big]\big)\\
& =\tr_{2^{\hat n}d}\big(\big[\Phi,\Phi\left(LL^{*}+z\right)^{-1}-\Phi\left(L^{*}L+z\right)^{-1}\big]\big)
\\
& \quad -\tr_{2^{\hat n}d} \big(\big[\Phi, \cQ\left(L^{*}L+z\right)^{-1}
+ \cQ\left(LL^{*}+z\right)^{-1}\big]\big)    \\
& =\tr_{2^{\hat n}d} \big(\big[\Phi,\Phi\big(\left(LL^{*}+z\right)^{-1}-\left(L^{*}L+z\right)^{-1}\big)\big]\big) \\
& \quad -\tr_{2^{\hat n}d} \big(\big[\Phi, \cQ \big(\left(L^{*}L+z\right)^{-1}+\left(LL^{*}+z\right)^{-1}\big)\big]\big) \\
& =\tr_{2^{\hat n}d}\big(\big[\Phi,\Phi\big(2\left(R_{1+z}C\right)^{n}R_{1+z}+\big(\left(L^{*}L+z\right)^{-1}-\left(LL^{*}+z\right)^{-1}\big)  \\
& \hspace*{1.4cm} \times \left(CR_{1+z}\right)^{n+1}\big)\big]\big)\\
& \quad-\tr_{2^{\hat n}d} \big(\big[\Phi, \cQ\big(2\left(R_{1+z}C\right)^{n-1}R_{1+z}+\big(\left(L^{*}L+z\right)^{-1}+\left(LL^{*}+z\right)^{-1}\big)    \\
& \hspace*{1.7cm} \times \left(CR_{1+z}\right)^{n}\big)\big]\big).     \tag*{{\qedhere}}
\end{align*}
\end{proof}

One important upshot of Lemma \ref{lem:almost_Neumann_series} is the fact \eqref{BLeven}, 
implying that only odd dimensions are of interest when computing the index of $L$. Thus, we will 
focus on the case $n$ odd, only.

The next theorem concludes this section and asserts that the trace class assumptions on $B_{L}(z)$
in Theorem \ref{thm:index with Witten} are satisfied for $B_{L}(z)$
given by \eqref{eq:def_of_BL(z)2}. As the sequence $\{T_\Lambda\}_\Lambda$ we shall use $\{\chi_\Lambda\}_\Lambda$ the sequence of multiplication operators induced by multiplying with the cut-off (characteristic) function $\chi_\Lambda$. The sequence $\{S_\Lambda^*\}_\Lambda$ is set to be the constant sequence $S_\Lambda=I_{L^2(\mathbb{R}^n)}$ for all $\Lambda$. Clearly,  
$\chi_\Lambda\in L^{n+1}(\mathbb{R}^n)$ for all $\Lambda>0$. 

\begin{theorem}
\label{thm:trisbounded} Let $n\in\mathbb{N}_{\geq 3}$ odd, $L= \cQ+\Phi$
given by \eqref{eq:def_of_L(2)}. Then there exists $\delta>0$ such that for all $z\in\rho\left(-L^{*}L\right)\cap\rho\left(-LL^{*}\right)$ and $\Lambda>0$, 
the operator $\chi_\Lambda B_{L}(z)$ with $B_{L}(z)$ given by \eqref{eq:def_of_BL(z)2} is trace
class with $z\mapsto\tr (|\chi_\Lambda B_{L}(z)|)$ bounded on $B(0,\delta)\backslash \{0\}$. 
\end{theorem}
\begin{proof}
We start by showing that $z\mapsto\chi_\Lambda R_{1+z}\big((CR_{1+z})^{n}\big)$ is
trace class with trace class norm being bounded around a neighborhood
of $0$, where $C=\left( \cQ\Phi\right)=\sum_{j=1}^{n}\gamma_{j,n}\left(\partial_{j}\Phi\right)$
is given by \eqref{eq:commutator=00003DC}, see also Remark \ref{rem:differen_mult_op},
and $R_{1+z}$ is given by \eqref{eq:resolvent_of_laplace}. Using $n=2\hat n+1$ we write
\[
 \chi_\Lambda R_{1+z}\big(CR_{1+z})^{n}\big)=\left(\chi_\Lambda R_{1+z} \left(CR_{1+z}\right)^{\hat n}\right) \left( \left(CR_{1+z}\right)^{\hat n +1}\right). 
\]
By Theorem \ref{thm:Simon_Hilbert-Schmidt} the operators
\[
\left(\chi_\Lambda R_{1+z} \left(CR_{1+z}\right)^{\hat n}\right) \, \text{ and } \,  
\left( \left(CR_{1+z}\right)^{\hat n +1}\right)
\]
are Hilbert--Schmidt by the admissability of $\Phi$ (in this context, see, in particular, 
Definition \ref{def:phi_admissible}\,$(iii)$). Moreover, the
boundedness of $z\mapsto \chi_\Lambda \tr_{2^{\hat n}d}\left(R_{1+z}C\right)^{n}$ with respect
to the norm in $\cB_1\left(L^{2}\left(\mathbb{R}^{n}\right)\right)$
around a neighborhood of $0$, now follows from Theorem \ref{thm:trace-class-crit}
together with the estimates in Theorem \ref{thm:Simon_Hilbert-Schmidt}
and Lemma \ref{lem:Schatten-class-1-operator} (we note that we apply these statements for $\mu=1+z$ with $z\in \mathbb{C}_{\Re>-1}$).

One recalls (employing the spectral theorem) that for all self-adjoint operators $A$ on a Hilbert space
$\cH$ with $0$ being an isolated eigenvalue, the operator family $z\mapsto z\left(A+z\right)^{-1}$
is uniformly bounded on $B(0,\delta)$ for some $\delta>0$. By Lemma \ref{lem:almost_Neumann_series}, \eqref{BL}, the uniform boundedness of $z\mapsto z\left(A+z\right)^{-1}$ on $B(0,\delta)$ for some $\delta>0$, 
and the ideal property for trace class operators,
it remains to show that $\left(CR_{1+z}\right)^{n+1}$ is trace class,
with trace class norm bounded for $z\in B(0,\delta')$ for some sufficiently small
$\delta'>0$. For $n=2 \hat n +1$, one observes that $\left(CR_{1+z}\right)^{n+1}$
is a sum of operators of the form
\[
\Psi_{1}\cdots R_{1+z}\Psi_{n+1}R_{1+z}=\left(\Psi_{1}\cdots R_{1+z}\Psi_{\hat n +1}R_{1+z}\right)\left(\Psi_{\hat n +2}\cdots R_{1+z}\Psi_{2\hat n +2}R_{1+z}\right),
\]
where $\Psi_{j}$ are multiplication operators with bounded $C^{\infty}$-functions
with the property that for some constant $\kappa>0$, $|\Psi_j(x)|\leq \frac{\kappa}{1+|x|}$, 
$x\in \mathbb{R}^n$. For deriving the trace class property of 
\[
\Psi_{1}\cdots R_{1+z}\Psi_{n+1}R_{1+z}=\left(\Psi_{1}\cdots
R_{1+z}\Psi_{\hat n +1}R_{1+z}\right)\left(\Psi_{\hat n +2}\cdots
R_{1+z}\Psi_{2\hat n +2}R_{1+z}\right),
\] we use Theorem \ref{thm:Simon_Hilbert-Schmidt} and Lemma
\ref{lem:Schatten-class-1-operator}. Let $z_0\in (-1,0)$. By
Theorem \ref{thm:Simon_Hilbert-Schmidt}\,$(i)$ one estimates for
some $\kappa'>0$, depending on $\hat n$, $\kappa$ and $z_0$, and all 
$z\in \mathbb{C}_{\geq z_0}$, 
\[
  \|\left(\Psi_{1}\cdots R_{1+z}\Psi_{\hat n +1}R_{1+z}\right)\|_{\cB_2} \leq \prod_{j=1}^{\hat n +1} \|\Psi_j
R_{1+z}\|
  \leq \kappa' \prod_{j=1}^{\hat n +1}\|\Psi_j\|_{L^{n+1}},
\]
where we used Lemma \ref{lem:Schatten-class-1-operator} in the last
estimate. The same argument applies to
\[
  \left(\Psi_{\hat n +2}\cdots R_{1+z}\Psi_{2\hat n +2}R_{1+z}\right).
\]
This concludes the proof since $\left(CR_{1+z}\right)^{n+1}$ is trace class by 
Theorem \ref{thm:trace-class-crit}. 
\end{proof}

\begin{remark}\label{r:5.8} We note that the method of proof of Theorem \ref{thm:trisbounded} shows that the trace of $\chi_\Lambda B_L(z)$ can be computed as the integral over the diagonal of the respective integral kernel. In fact, we have shown that $\chi_\Lambda B_L(z)$ may be represented as sums of products of two Hilbert--Schmidt operators leading to the trace formula given in Corollary \ref{cor:Comp-of-trace}. \hfill $\diamond$ 
\end{remark}

\newpage

\section{Derivation of the Trace Formula -- Diagonal Estimates}\label{sec:der_tra_diag}

In this section, we shall compute the trace of $\chi_\Lambda B_L(z)$, $\Lambda>0$, $z\in \rho(-L^*L)\cap \rho(-LL^*)\cap \mathbb{C}_{\Re>-1}$, with $B_L$ given by \eqref{eq:Def_of_B_L(z)}. After stating the next lemma (needed to be able to apply Lemma \ref{lem:comm is sum of der} and Proposition \ref{prop:commutator with phi vanishes} to the sum in \eqref{eq:B_L=00003DJJJ+ALz}) we will outline the strategy of the proof. 

We note that for the application of Lemma \ref{lem:comm is sum of der} to the first summand in \eqref{eq:B_L=00003DJJJ+ALz}, one needs to establish continuous differentiability of the integral kernel of \eqref{eq:JJJ}. In this context we emphasize the different regularity of the kernels of \eqref{eq:JJJ} for $n=3$ and $n\geq 5$, necessitating modifications for the case $n=3$ due to the lack of differentiability of \eqref{eq:JJJ}.

\begin{lemma}[{\cite[Lemma 4, p. 224]{Ca78}}]
\label{lem:j,A,Green} Let $n\in\mathbb{N}_{\geq3}$ odd, 
$L= \cQ+\Phi$ be given by \eqref{eq:def_of_L(2)}, and let 
$z\in\rho\left(-LL^{*}\right)\cap\rho\left(-L^{*}L\right)$, 
with $\Re (z) > -1$. Denote the integral kernels of the following operators
\begin{align*}
J_{L}^{j}(z) & =\tr_{2^{\hat n}d}\big(L\left(L^{*}L+z\right)^{-1}\gamma_{j,n}\big)-\tr_{2^{\hat n}d}\big(L^{*}\left(LL^{*}+z\right)^{-1}\gamma_{j,n}\big),  \\
A_{L}(z) & =\tr_{2^{\hat n}d}\big(\big[\Phi,L^{*}\left(LL^{*}+z\right)^{-1}\big]\big) 
- \tr_{2^{\hat n}d}\big(\big[\Phi,L\left(L^{*}L+z\right)^{-1}\big]\big),  
\end{align*}
by $G_{J,j,z}$, $j\in\{1,\ldots,n\}$, and $G_{A,z}$, respectively. 
Then $G_{A,z}$ is continuous and satisfies $G_{A,z}(x,y)\to0$ if
$y\to x$ for all $x\in\mathbb{R}^{n}$. If $n\geq5$, $G_{J,j,z}$
is continuously differentiable on $\mathbb{R}^{n}\times\mathbb{R}^{n}.$ 
\end{lemma}
\begin{proof}
Appealing to Lemma \ref{lem:almost_Neumann_series}, one recalls with
$R_{1+z}$, $\cQ$, and $C$ given by \eqref{eq:resolvent_of_laplace},
\eqref{eq:def_of_Q2}, and \eqref{eq:commutator=00003DC}, respectively,
\begin{align*}
J_{L}^{j}(z) & =2\tr_{2^{\hat n}d}\big(\gamma_{j,n} \cQ\left(R_{1+z}C\right)^{n-2}R_{1+z}\big)+2\tr_{2^{\hat n}d}\big(\gamma_{j,n}\Phi \left(R_{1+z}C\right)^{n-1}R_{1+z}\big)\\
 & \quad+\tr_{2^{\hat n}d}\big(\gamma_{j,n} \cQ \big(\left(L^{*}L+z\right)^{-1}+\left(LL^{*}+z\right)^{-1}\big)\left(CR_{1+z}\right)^{n}\big)\\
 & \quad+\tr_{2^{\hat n}d}\big(\gamma_{j,n}\Phi \big(\left(L^{*}L+z\right)^{-1}+\left(LL^{*}+z\right)^{-1}\big)\left(CR_{1+z}\right)^{n}\big), \quad j\in\{1,\ldots,n\}, 
\end{align*}
and 
\begin{align*}
A_{L}(z) & =\tr_{2^{\hat n}d} \big(\big[\Phi,\Phi\big(2\left(R_{1+z}C\right)^{n}R_{1+z}+\big(\left(L^{*}L+z\right)^{-1}-\left(LL^{*}+z\right)^{-1}\big)    \\
& \hspace*{1.4cm} \times \left(CR_{1+z}\right)^{n+1}\big)\big]\big)\\
 & \quad-\tr_{2^{\hat n}d} \big(\big[\Phi, \cQ \big(2\left(R_{1+z}C\right)^{n-1}R_{1+z}+\big(\left(L^{*}L+z\right)^{-1}+\left(LL^{*}+z\right)^{-1}\big)  \\ 
 & \hspace*{1.7cm} \times \left(CR_{1+z}\right)^{n}\big)\big]\big).
\end{align*}
By Proposition \ref{prop:concrete case} (one recalls that $\Phi$
is admissible and hence $\Phi\in C_{b}^{\infty}\big(\mathbb{R}^{n};\mathbb{C}^{d \times d}\big)$
by Definition \ref{def:phi_admissible}\,$(i)$), one gets for
all $\ell\in\mathbb{R}$, 
\[
Q\gamma_{j,n}\left(R_{1+z}C\right)^{n-2}R_{1+z} \in 
\cB\big(H^{\ell}(\mathbb{R}^{n})^{2^{\hat n}d},H^{\ell+2(n-2)+2-1}(\mathbb{R}^{n})^{2^{\hat n}d}\big).
\]
For $n\geq5$, one obtains from $(2(n-2)+2-1)=n-3>0,$
the continuity of $G_{J,j,z}$ by Corollary \ref{cor:integral kernels regularity}.
Moreover, since $\left(2(n-2)+2-1\right)-n-1=n-4>0,$ Corollary \ref{cor:integral kernels regularity} 
also implies continuous differentiability of $G_{J,j,z}$. Similar arguments
ensure the continuity of the integral kernel of $A_{L}(z)$ (for
$n\geq3$). Moreover, for $n\geq3$, the integral kernel of $A_{L}(z)$ vanishes on the 
diagonal by Proposition \ref{prop:commutator with phi vanishes}. 
\end{proof}

Next, we outline the idea for computing the trace of $\chi_\Lambda B_L(z)$. By 
Theorem \ref{thm:index with Witten} and Theorem \ref{thm:Witten_reg_n5}, we know that the limit $\lim_{\Lambda\to \infty} \tr(\chi_\Lambda B_L(0))$ exists. However, in order to derive the explicit formula asserted in Theorem \ref{thm:Witten_reg_n5} also for $z$ in a neighborhood of $0$, some work is required. As it will turn out, for $z$ with large real part -- at least for a sequence $\{\Lambda_k\}_{k\in\bbN}$ -- we can show that an expression similar to the one in Theorem \ref{thm:Witten_reg_n5} is valid. 

For achieving the existence of the limit (without using sequences) for $z$ in a neighborhood of $0$, we intend to employ Montel's theorem. One recalls that for an open set $U\subseteq \mathbb{C}$, a set $\mathcal{G}\subseteq \mathbb{C}^U\coloneqq \{f \,|\, f\colon U\to \mathbb{C}\}$ is called \emph{locally bounded}, if for all compact $\Omega\subset U$,
\begin{equation}\label{d:normal}
 \sup_{f\in \mathcal{G}} \sup_{z\in \Omega} |f(z)| < \infty.
\end{equation}

\begin{theorem}[Montel's theorem, see, e.g., \cite{Co95}, p.~146--154] \label{t:m} Let $U\subseteq \mathbb{C}$ open, $\{f_\Lambda\}_{\Lambda\in \mathbb{N}}$ a locally bounded family of analytic functions on $U$. Then there exists a subsequence $\{f_{\Lambda_k}\}_{k\in \mathbb{N}}$ and an analytic function $g$ on $U$ such that $f_{\Lambda_k}\to g$ as $k\to\infty$ in the compact open topology $($i.e., for any compact set $\Omega\subset U$, the sequence $\{f_{\Lambda_k}|_{\Omega}\}_{k\in\mathbb{N}}$ converges uniformly to $g|_{\Omega}$$)$.
\end{theorem}

For our particular application of Montel's theorem, we need to show that the family of analytic functions
\[
  \{ z \mapsto \tr (\chi_\Lambda B_L(z))\}_{\Lambda}
\]
constitutes a locally bounded family. Thus, one needs to show that for all compact $\Omega\subset \mathbb{C}_{\Re>-1}\cap \rho(-L^*L)\cap \rho(-LL^*)$, 
\begin{equation}\label{e:normal}
  \sup_{\Lambda>0}\sup_{z\in \Omega} |\tr (\chi_\Lambda B_L(z))|<\infty.
\end{equation}
For this assertion, it is crucial that some integral kernels involved in the computation of the trace vanish on the diagonal, see, for instance, Proposition \ref{prop:commutator with phi vanishes}.  We note that generally, the expression
\begin{equation}\label{e:normal2}
  \sup_{\Lambda>0}\sup_{z\in \Omega} \tr (|(\chi_\Lambda B_L(z))|),
\end{equation}
cannot be finite, as the example constructed in Appendix B demonstrates. In order to prove \eqref{e:normal}, we actually show for all 
$\Omega\subset \mathbb{C}_{\Re>-1}\cap \rho(-L^*L)\cap \rho(-LL^*)$ compact, 
\begin{equation}\label{e:normal3}
  \sup_{\Lambda>0}\sup_{z\in \Omega} |z\tr (\chi_\Lambda B_L(z))|<\infty,  
\end{equation}
and then appeal to the fact that condition \eqref{e:normal3} together with Theorem \ref{thm:trisbounded} implies \eqref{e:normal}, as the next result confirms:

\begin{lemma}\label{lem:Laurent_series_argument} Assume that $\{\phi_k\}_{k\in \bbN}$ is a sequence of 
analytic $($scalar-valued$)$ functions on $B_\mathbb{C}(0,1)$. Assume that $\{z\mapsto z\phi_k(z)\}_{k\in \bbN}$ is locally bounded on $B(0,1)$. Then $\{\phi_k\}_{k\in\bbN}$ is locally bounded on $B(0,1)$.  
\end{lemma}
\begin{proof} Assume that $\{\phi_k\}_{k\in\bbN}$ is not locally bounded on $B(0,1)$. Then there exists a subsequence $\{\phi_{k_\ell }\}_{\ell \in\bbN}$ and a corresponding sequence of complex numbers $\{z_{k_\ell }\}_{\ell \in\bbN}$ with the property that 
$z_{k_\ell }\to 0$ and $|\phi_{k_\ell }(z_{k_\ell })|\to \infty$ as $\ell \to\infty$. Since 
    \[
        \{\psi_{\ell }\}_{\ell \in\bbN} \coloneqq \{z\mapsto z\phi_{k_\ell }(z)\}_{\ell \in\bbN} 
    \]
 is locally bounded on $B(0,1)$ there exists an accumulation point $\psi$ in the compact open topology of analytic functions $\mathcal{H}(B(0,1))$ on $B(0,1)$ by Montel's theorem. Without 
 loss of generality, one can assume that $\psi_\ell \to \psi$ in $\mathcal{H}(B(0,1))$ as $\ell \to\infty$. 
By construction, one has $ \psi_\ell (0)=0$ and for some $r>0$, 
\begin{align*}
   \left|\frac{1}{z}\psi_\ell  (z)-\psi'(0)\right| &\leq  \left|\frac{1}{z}\left(\psi_\ell  (z)-\psi_\ell (0)\right)-\psi_\ell '(0)\right|+|\psi_\ell '(0)-\psi'(0)|\\
                   &\leq  \sup_{z\in B(0,r)} |(\psi'_\ell  (z)-\psi'_\ell (0))| +|\psi_\ell '(0)-\psi'(0)|
\end{align*}
for all $z\in B(0,r)\backslash \{0\}$. Since $\psi'_\ell \to\psi'$ uniformly on compacts, it follows that 
\[
    \limsup_{\ell \to\infty} \sup_{z\in B(0,r)\backslash \{0\}} \left|\frac{1}{z}\psi_\ell  (z)\right|<\infty. 
\]
However, for $\ell $ sufficiently large, one concludes 
\[
   \sup_{z\in B(0,r)\backslash \{0\}} \left|\frac{1}{z}\psi_\ell  (z)\right|\geq  \left|\frac{1}{z_{k_\ell }}\psi_\ell  (z_{k_\ell })\right|=|\phi_{k_\ell }(z_{k_\ell })|\underset{\ell  \to \infty}{\longrightarrow} \infty, 
\]
a contradiction.
\end{proof}

\begin{remark}
It turns out that the analyticity hypothesis in Lemma \ref{lem:Laurent_series_argument} is crucial. Indeed, for every $n\in \mathbb{N}$, there exists a $C^\infty$-function $\psi_n \colon [0,1)\to [0,\infty)$ 
with the properties,   
\[
   \psi_n|_{(0,1/(2n))}=0,\quad 0\leq  \psi_n(x)\leq  \psi_n\left(\frac{1}{n}\right)=n,\quad \psi_n|_{(2/n,1)}=0. 
\]
 Then $\psi_n(0)=0$ and $0\leq  x\psi_n(x)\leq (2/n)n=2$. Considering 
 $\phi_n(x+iy)\coloneqq \psi_n(|x+iy|)$ for $x,y\in \mathbb{R}$, $x+iy\in B(0,1)$, $n\in\mathbb{N}$, one gets that $\phi_n$ is real differentiable and the assumptions of Lemma \ref{lem:Laurent_series_argument}, except for  analyticity, are all satisfied. In addition, $\phi_n(0)=0$, however, 
 $\phi_n\left(1/n\right)=n\to\infty$ as $n\to\infty$. Thus, $\{\phi_n\}_{n\in\bbN}$ is not locally bounded on $B(0,1)$.  \hfill $\diamond$
\end{remark}

The next aim of this section is to establish Theorem \ref{thm:local_boundedness of the trace}, that is, an important step for obtaining \eqref{e:normal}. The terms to be discussed in Theorem \ref{thm:local_boundedness of the trace} split up into a leading order term and the rest. The first term will be studied in Lemma \ref{lem:boundedness_of_leading_new_witten} and the second one in Lemma \ref{lem:boundedness_of_rest_new_witten}. The strategy of proof in these lemmas is the same. It rests on the following observation: Let $U\subseteq \mathbb{C}$ open,  $U\ni z\mapsto T(z) \in \cB \big(L^2(\mathbb{R}^n)\big)$. Assume that for all $z\in U$ we have $T(z)\in \mathcal{B}_1\big(L^2(\mathbb{R}^n)\big)$ and that $z\mapsto \tr (|T(z)|)$ is locally bounded. Then
\[
    \{z\mapsto \tr (\chi_\Lambda T(z))\}_{\Lambda>0}
\]
is locally bounded as well. Indeed, the assertion follows from the boundedness of the family $\{\chi_\Lambda\}_{\Lambda>0}$ as bounded linear (multiplication) operators in $\mathcal{B}(L^2(\mathbb{R}^n))$ and the ideal property of the trace class. In the situations to be considered in the following, the trace class property for $T(z)$ will be shown with the help of the results of Section \ref{sec:Trace-Class-Estimates}.

\begin{lemma}\label{lem:boundedness_of_leading_new_witten}
 Let $L= \cQ+\Phi$ be given by \eqref{eq:def_of_L(2)} and for $z\in\mathbb{C}$ with $\Re (z) > -1$ let $R_{1+z}$ be given
by \eqref{eq:resolvent_of_laplace} and $C$ as in \eqref{eq:commutator=00003DC}, $n\in \mathbb{N}_{>1}$ odd.
For $j\in\{1,\ldots,n\}$, let $\gamma_{j,n}\in\mathbb{C}^{2^{\hat n} \times 2^{\hat n}}$
as in Remark \ref{rem:Eucl-Dirac-Algebar} and $\chi_\Lambda$ as in \eqref{eq:def_of_chi}, $\Lambda>0$. For $z\in \mathbb{C}_{\Re > -1}$ consider
\[
   \psi_\Lambda (z) \coloneqq \chi_\Lambda \tr_{2^{\hat n}d} \big(\left[ \cQ,\Phi\left(CR_{1+z}\right)^{n}\right]\big)   
\]
and 
\[
\tilde\psi_\Lambda(z) \coloneqq \chi_\Lambda \tr_{2^{\hat n}d} \big(\left[ \cQ, \cQ\left(CR_{1+z}\right)^{n}\right]\big). 
\]
Then for all $z\in \mathbb{C}_{\Re>-1}$, the operators $\psi_{\Lambda}(z)$, $\tilde\psi_{\Lambda}(z)$ are trace class and the families 
\[
\{z\mapsto \tr_{L^2(\bbR^n)} (\psi_{\Lambda}(z))\}_{\Lambda>0} \, \text{ and } \, 
\big\{z\mapsto \tr_{L^2(\bbR^n)} \big(\tilde\psi_{\Lambda}(z)\big)\big\}_{\Lambda>0}    \]
 are locally bounded $($cf.\ \eqref{d:normal}$)$.
\end{lemma}
\begin{proof}
First we deal with $\psi_\Lambda (z)$. One computes,  
\[
   \psi_\Lambda (z) = \chi_\Lambda \tr_{2^{\hat n}d} \big(\left[ \cQ,\Phi\left(CR_{1+z}\right)^{n}\right]\big) 
 =\chi_\Lambda \tr_{2^{\hat n}d} \big( \cQ\Phi\left(CR_{1+z}\right)^{n} 
 - \Phi\left(CR_{1+z}\right)^{n} \cQ \big).
\]
Before we discuss the latter operator, we note that
\begin{align*}
 \cQ\Phi\left(CR_{1+z}\right)^{n} &= \Phi \cQ \left(CR_{1+z}\right)^{n}+[ \cQ,\Phi]\left(CR_{1+z}\right)^{n} \\
  &= \Phi \bigg(\sum_{j=1}^n (CR_{1+z})^{j-1}[ \cQ,C]R_{1+z}(CR_{1+z})^{n-j}
  +(CR_{1+z})^n \cQ\bigg) \\ 
  & \quad +[ \cQ,\Phi]\left(CR_{1+z}\right)^{n},  
\end{align*}
where the latter equality follows via an induction argument. Hence, 
\begin{align} 
& \psi_\Lambda (z)=\chi_\Lambda \tr_{2^{\hat n}d} \bigg(\Phi \bigg(\sum_{j=1}^n (CR_{1+z})^{j-1}[ \cQ,C]R_{1+z}(CR_{1+z})^{n-j}\bigg)  \no  \\ 
& \hspace*{2.8cm} +[ \cQ,\Phi](CR_{1+z})^{n}\bigg).\label{e:psi_l}
\end{align} 
Next, with the results of Section \ref{sec:Trace-Class-Estimates}, we will 
deduce that the operator family 
\begin{equation}\label{eq:lem5.11} z\mapsto \bigg(\Phi \bigg(\sum_{j=1}^n 
(CR_{1+z})^{j-1}[ \cQ,C]R_{1+z}(CR_{1+z})^{n-j}\bigg) 
+ [ \cQ,\Phi]\left(CR_{1+z}\right)^{n}\bigg)
\end{equation}
is trace class, which -- together with the estimates in Lemma 
\ref{lem:Schatten-class-1-operator} --  establishes the assertion for 
$\psi_\Lambda$: Indeed, the only difference between \eqref{e:psi_l} and \eqref{eq:lem5.11} is the prefactor $\chi_\Lambda$. So we get the assertion with the help of the reasoning prior to Lemma \ref{lem:boundedness_of_leading_new_witten}. In order to observe that each summand in \eqref{eq:lem5.11} is 
trace class, we proceed as follows. Recall $n=2\hat n + 1$ and let $j\in \{1,
\ldots,\hat n\}$ (the case $n-j\in \{1,\ldots,\hat n\}$ can be dealt with 
similarly). Then, by the admissability of $\Phi$ (see Hypothesis 
\ref{def:phi_admissible}), one infers that $[ \cQ,C]$ is a multiplication 
operator with
\[
 |[ \cQ,C](x)|\leq  \kappa (1+|x|)^{-1-\epsilon}, \quad x\in 
\mathbb{R}^n.
\]
Hence, as $1+\epsilon> 3/2$ by Definition \ref{def:phi_admissible}, Theorem \ref{thm:HS-criterion for m factors} applies and guarentees that 
\[
 (CR_{1+z})^{j-1}[ \cQ,C]R_{1+z}(CR_{1+z})^{\hat n - j} 
\]
 is Hilbert--Schmidt. Using Theorem \ref{thm:Simon_Hilbert-Schmidt}, one deduces that 
 $(CR_{1+z})^{\hat n + 1}$ is also Hilbert--Schmidt and thus 
\[
 (CR_{1+z})^{j-1}[ \cQ,C]R_{1+z}(CR_{1+z})^{\hat n-j}(CR_{1+z})^{\hat n+1}
\]
is trace class, by Theorem \ref{thm:trace-class-crit}.

For $\tilde\psi_\Lambda$ one proceeds similarly. First one notes that
\begin{equation}\label{eq:Q_terms}
  \tilde\psi_\Lambda(z)=\chi_\Lambda \tr_{2^{\hat n}d} \bigg( \cQ \bigg(\sum_{j=1}^n (CR_{1+z})^{j-1}[ \cQ,C]R_{1+z}(CR_{1+z})^{n-j}\bigg) \bigg).
\end{equation}

Applying Theorems \ref{thm:Simon_Hilbert-Schmidt}, \ref{thm:HS-criterion for m factors}, and \ref{thm:trace-class-crit}, one infers the assertion for $\tilde\psi_\Lambda$. However, one has to use the respective assertions, where some of the resolvents of the Laplacian is replaced by $\cQ$ times the resolvents. Indeed, in the sum in \eqref{eq:Q_terms}, the term for $j=1$ yields 
\begin{align*}
& \cQ [ \cQ,C]R_{1+z}(CR_{1+z})^{n-1} = [ \cQ,[ \cQ,C]]R_{1+z}(CR_{1+z})^{n-1}    \\ 
& \hspace*{4.2cm} +[ \cQ,C] \cQ R_{1+z}(CR_{1+z})^{n-1}\\
&\quad = [\cQ,[ \cQ,C]]R_{1+z}(CR_{1+z})^{n-1}+[ \cQ,C]R_{1+z}[ \cQ,C]R_{1+z}(CR_{1+z})^{n-2}   \\
& \qquad+[ \cQ,C]R_{1+z}C \cQ R_{1+z}(CR_{1+z})^{n-2}, 
\end{align*}
and for $j'\in \{2,\ldots,n\}$ one obtains 
\begin{align*}
 & \cQ CR_{1+z}(CR_{1+z})^{j'-2}[ \cQ,C]R_{1+z}(CR_{1+z})^{n-j'}   \\
 &\quad =[ \cQ,C]R_{1+z}(CR_{1+z})^{j'-2}[ \cQ,C]R_{1+z}(CR_{1+z})^{n-j'}\\ 
 &\qquad +C \cQ R_{1+z}(CR_{1+z})^{j'-2}[ \cQ,C]R_{1+z}(CR_{1+z})^{n-j'}.\tag*{\qedhere}
\end{align*}
\end{proof}

The next lemma is the reason, why we have to invoke Lemma \ref{lem:Laurent_series_argument} in our argument. The crucial point is that we can use the Neumann series expressions for the resolvents $(L^*L+z)^{-1}$ and $(LL^*+z)^{-1}$ only for $z$ with large real part. But for $z$ in the vicinity of $0$, we do not have such a representation. Using again the ideal property for trace class operators, we can, however, bound $z (L^*L+z)^{-1}$ for small $z$ in the $\mathcal{B}(L^2(\mathbb{R}^n))$-norm. Introducing the sector $\Sigma_{z_0,\theta} \subset \bbC$ by 
\begin{equation}\label{def:Sigma0}
   \Sigma_{z_0,\theta}\coloneqq \{z\in \mathbb{C} \, | \, \Re (z) > z_0, |\arg(\mu)| < \theta\},
\end{equation}
for some $z_0\in \mathbb{R}$ and $\theta\in (0,\pi/2)$, the result reads as follows:

\begin{lemma}\label{lem:boundedness_of_rest_new_witten}
 Let $L= \cQ+\Phi$ be given by \eqref{eq:def_of_L(2)} and for $z\in\mathbb{C}$ with $\Re (z) > -1$, let $R_{1+z}$ be given
by \eqref{eq:resolvent_of_laplace} and $C$ as in \eqref{eq:commutator=00003DC}, $\chi_\Lambda$ as in \eqref{eq:def_of_chi}, $\Lambda>0$. For $j\in\{1,\ldots,n\}$, let 
$\gamma_{j,n}\in\mathbb{C}^{2^{\hat n} \times 2^{\hat n}}$
as in Remark \ref{rem:Eucl-Dirac-Algebar}. For $z\in \mathbb{C}_{\Re>-1}\cap\rho(-L^*L)\cap\rho(-LL^*)$ consider
\[
\eta_\Lambda (z) \coloneqq \chi_\Lambda \tr_{2^{\hat n}d} \big(\big[ \cQ, \Phi \big(\left(L^{*}L+z\right)^{-1}-\left(LL^{*}+z\right)^{-1}\big)\left(CR_{1+z}\right)^{n+1}\big]\big) 
\]
and
\[ 
\tilde\eta_\Lambda(z) \coloneqq \chi_\Lambda \tr_{2^{\hat n}d} \big(\big[ \cQ,\big(\cQ \big(\left(L^{*}L+z\right)^{-1}-\left(LL^{*}+z\right)^{-1}\big)\left(CR_{1+z}\right)^{n+1}\big)\big]\big).
\] 
Then for all $z\in \mathbb{C}_{\Re>-1}\cap\rho(-L^*L)\cap\rho(-LL^*)$, the operators $\eta_{\Lambda}(z)$, $\tilde\eta_{\Lambda}(z)$ are trace class. There exists $\delta\in (-1,0)$, $\theta\in (0,\pi/2)$ such that the families
\[
\{\Sigma_{\delta,\theta}\cup \mathbb{C}_{\Re>0}\ni z\mapsto z \tr_{L^2(\mathbb{R}^n)} (\eta_{\Lambda}(z))\}_{\Lambda>0} 
\]
and 
\[
\left\{\Sigma_{\delta,\theta}\cup \mathbb{C}_{\Re>0}\ni z\mapsto z \tr_{L^2(\mathbb{R}^n)} 
\big(\tilde\eta_{\Lambda}(z)\big)\right\}_{\Lambda>0}
\]
are locally bounded $($cf.\ \eqref{d:normal}$)$. 
\end{lemma}
\begin{proof}
By the Fredholm property of $L$ there exist $\delta\in (-1,0)$ and $\theta\in (0,\pi/2)$ such that 
\[
 \Sigma_{\delta,\theta}\backslash\{0\}\ni   z\mapsto z\left(L^*L+z\right)^{-1}\text{ and }  \Sigma_{\delta,\theta}\backslash\{0\}\ni   z\mapsto z\left(LL^*+z\right)^{-1}
\]
have analytic extensions to   $\Sigma_{\delta,\theta}$. Let $\Omega\subset \Sigma_{\delta,\theta}\cup \mathbb{C}_{\Re>0}$ be compact. One notes that 
\[\Omega\ni z\mapsto z\left(L^*L+z\right)^{-1} \, \text{ and } \, \Omega \ni z\mapsto z\left(LL^*+z\right)^{-1}
\] 
define bounded families of bounded linear operators from $L^2(\mathbb{R}^n)^{2^{\hat n}d}$ to $H^2(\mathbb{R}^n)^{2^{\hat n}d}$. Indeed, by Proposition \ref{prop:compu_lstartl}, one infers that 
$\phi\mapsto \| (L^*L+1)\phi \|$ and $\phi\mapsto \| (LL^*+1)\phi \|$ define equivalent norms on $H^2(\mathbb{R}^n)^{2^{\hat n}d}$. Hence, for $\phi\in L^2(\mathbb{R}^n)^{2^{\hat n}d}$ and $z\in \Omega\backslash \{0\}$ one computes 
\begin{align*}
   \big\|(L^*L+1) z (L^*L+z)^{-1}\phi\big\|& = |z| \big\|(L^*L+z+1-z) 
   (L^*L+z)^{-1}\phi\big\| \\&
 \leq  |z|\|\phi\| +|z||(1-z)|\frac{1}{|z|}\|\phi\|. 
\end{align*}
Next, consider
\begin{align*}
 \eta_\Lambda (z) &= \chi_\Lambda \tr_{2^{\hat n}d} \big(\big[ \cQ, \Phi \big(\left(L^{*}L+z\right)^{-1}-\left(LL^{*}+z\right)^{-1}\big)\left(CR_{1+z}\right)^{n+1}\big]\big)     \\
&=\chi_\Lambda \tr_{2^{\hat n}d} \big( \cQ\Phi\big(\left(L^{*}L+z\right)^{-1}-\left(LL^{*}+z\right)^{-1}\big)\left(CR_{1+z}\right)^{n+1}   \\  
& \qquad\qquad\quad \;\; - \Phi\big(\left(L^{*}L+z\right)^{-1}-\left(LL^{*}+z\right)^{-1}\big)\left(CR_{1+z}\right)^{n+1} \cQ \big). 
\end{align*}
For the first summand one observes that
\begin{align*}
 & \cQ \Phi\big(\left(L^{*}L+z\right)^{-1}-\left(LL^{*}+z\right)^{-1}\big)\left(CR_{1+z}\right)^{n+1} \notag\\
 & \quad =C\big(\left(L^{*}L+z\right)^{-1}-\left(LL^{*}+z\right)^{-1}\big)\left(CR_{1+z}\right)^{n+1} \label{eq:first_summand}\\
 &\qquad +\Phi \cQ \big(\left(L^{*}L+z\right)^{-1}-\left(LL^{*}+z\right)^{-1}\big)\left(CR_{1+z}\right)^{n+1}.   \notag
\end{align*}
Employing our observation at the beginning of the proof and Theorem \ref{thm:L_is_closed}, one realizes that 
\[
   \Omega \ni z\mapsto \cQ z\big(\left(L^{*}L+z\right)^{-1}-\left(LL^{*}+z\right)^{-1}\big) 
\]
 defines a bounded family of bounded linear operators in $L^2(\mathbb{R}^n)^{2^{\hat n}d}$. Thus, since $\Omega\ni z\mapsto  \left(CR_{1+z}\right)^{n+1}$ is a family of trace class operators, 
\begin{align*}
 \Omega\ni z\mapsto&  z\chi_\Lambda \tr_{2^{\hat n}d} \big( \cQ \Phi \big(\left(L^{*}L+z\right)^{-1}-\left(LL^{*}+z\right)^{-1}\big)\left(CR_{1+z}\right)^{n+1}\big)\\
  & \quad = \tr_{2^{\hat n}d} \big(z\chi_\Lambda \big( \cQ \Phi\big(\left(L^{*}L+z\right)^{-1}-\left(LL^{*}+z\right)^{-1}\big)\left(CR_{1+z}\right)^{n+1}\big)\big)
\end{align*}
is uniformly bounded in $\mathcal{B}_1$, with bound independently of $\Lambda>0$, upon appealing to the ideal property of trace class operators.

The second summand requires the observation that
 \[
 \left(CR_{1+z}\right)^{n+1} \cQ =\left(CR_{1+z}\right)^{n}\left(CR_{1+z}\right) \cQ 
 =\left(CR_{1+z}\right)^{n}\left(C \cQ R_{1+z}\right)
 \]
defines a bounded family of trace class operators for $z\in \Omega$, proving the assertion 
for $\eta_\Lambda$. 

The corresponding assertion for $\tilde\eta_\Lambda$ is conceptually the same. In fact, it follows from  the observation that 
 \begin{align*} 
  \Omega \ni z\mapsto & \cQ \cQ z \big((L^{*}L+z)^{-1}-(LL^{*}+z)^{-1}\big)   \\
  & \quad =\Delta I_{2^{\hatt n}d} z\big((L^{*}L+z)^{-1}-(LL^{*}+z)^{-1}\big)
  \end{align*}
is a bounded family of bounded linear operators by our preliminary observation 
that $\Omega\ni z\to z(L^*L+z)^{-1}$ and $\Omega\ni z\to z 
(LL^*+z)^{-1}$ define uniformly bounded operator families from 
$L^2(\mathbb{R}^n)^{2^{\hat n}d}$ to $H^2(\mathbb{R}^n)^{2^{\hat n}d}$, 
as well as using again the fact that 
\[\Omega \ni z\mapsto  \left(CR_{1+z}\right)^{n+1} \cQ \, \text{ and } \, 
\Omega \ni z\mapsto \left(CR_{1+z}\right)^{n+1}\]
constitute bounded families of trace class operators.   
\end{proof}

Lemmas \ref{lem:boundedness_of_leading_new_witten} and \ref{lem:boundedness_of_rest_new_witten} can be summarized as follows.

\begin{theorem}\label{thm:local_boundedness of the trace} Let $n\in \mathbb{N}_{\geq 3}$ odd, 
let $L= \cQ+\Phi$ be given by \eqref{eq:def_of_L(2)} and for $z\in\mathbb{C}$ with $\Re (z) > -1$ let $R_{1+z}$ be given
by \eqref{eq:resolvent_of_laplace} and $C$ as in \eqref{eq:commutator=00003DC}.
For $j\in\{1,\ldots,n\}$, let $\gamma_{j,n}\in\mathbb{C}^{2^{\hat n} \times 2^{\hat n}}$
as in Remark \ref{rem:Eucl-Dirac-Algebar}. For $z\in \mathbb{C}_{\Re>-1}\cap \rho(-LL^*)\cap\rho(L^*L)$, introduce 
\[
   \phi_\Lambda (z) \coloneqq \chi_\Lambda \tr_{2^{\hat n}d} \big(\big[ \cQ, 
   \Phi\big((L^{*}L+z)^{-1}+(LL^{*}+z)^{-1}\big)(CR_{1+z})^{n}\big]\big) 
\]
 and 
\[
\tilde\phi_\Lambda(z)\coloneqq \chi_\Lambda \tr_{2^{\hat n}d}\big(\big[ \cQ, \big(\cQ \big((L^{*}L+z)^{-1}+(LL^{*}+z)^{-1}\big) (CR_{1+z})^{n}\big)\big]\big).
\]
Then for all $z\in \mathbb{C}_{\Re>-1}\cap \rho(-LL^*)\cap\rho(L^*L)$, the operators $\phi_{\Lambda}(z)$, $\tilde\phi_{\Lambda}(z)$ are trace class. There exists $\delta\in (-1,0)$, $\theta\in (0,\pi/2)$ such that the families
\[
\{\Sigma_{\delta,\theta}\cup \mathbb{C}_{\Re>0}\ni z\mapsto 
\tr_{L^2(\bbR^n)} (z\phi_{\Lambda}(z))\}_{\Lambda>0}  
\]
and 
\[
\big\{\Sigma_{\delta,\theta}\cup \mathbb{C}_{\Re>0}\ni z\mapsto 
\tr_{L^2(\bbR^n)} \big(z\tilde\phi_{\Lambda}(z)\big)\big\}_{\Lambda>0} 
\]
are locally bounded $($cf.\ \eqref{d:normal}$)$. 
\end{theorem}
\begin{proof}
 One recalls from equations \eqref{eq:Neumann_lstarl} and \eqref{eq:Neumann_llstar} the expressions
\begin{align*} 
   \left(L^*L+z\right)^{-1} &= I + \left(L^*L+z\right)^{-1}CR_{1+z},  \\
   \left(LL^*+z\right)^{-1} &= I - \left(LL^*+z\right)^{-1}CR_{1+z}. 
\end{align*} 
Hence, one gets 
\begin{align*}
   \phi_\Lambda (z) &= \chi_\Lambda \tr_{2^{\hat n}d} \big(2\left[ \cQ,\Phi\left(CR_{1+z}\right)^{n}\right]\big)\\ &\quad +
\chi_\Lambda \tr_{2^{\hat n}d} \big(\big[ \cQ,\Phi\big(\left(L^{*}L+z\right)^{-1}-\left(LL^{*}+z\right)^{-1}\big)\left(CR_{1+z}\right)^{n+1}\big]\big) \\ 
&= 2\psi_\Lambda(z)+\eta_\Lambda(z), 
\end{align*}
 and 
\begin{align*}
\tilde\phi_\Lambda(z)&=\chi_\Lambda \tr_{2^{\hat n}d} \big(2\left[ \cQ, \cQ \left(CR_{1+z}\right)^{n}\right]\big)\\
 & \quad +\chi_\Lambda \tr_{2^{\hat n}d} \big(\big[ \cQ, \cQ\big(\left(L^{*}L+z\right)^{-1}-\left(LL^{*}+z\right)^{-1}\big)\left(CR_{1+z}\right)^{n+1}\big]\big)  \\ 
 &=2\tilde\psi_\Lambda(z)+\tilde\eta_\Lambda(z), 
\end{align*}
with the functions introduced in Lemmas \ref{lem:boundedness_of_leading_new_witten} and \ref{lem:boundedness_of_rest_new_witten}. Thus, the assertion on the local boundedness follows from these two lemmas.
\end{proof}

The forthcoming statements are used for showing that for computing the trace the only term that matters is discussed in Proposition \ref{prop:formula for index}. We recall that by Remark \ref{r:5.8}, one can compute the trace of $\chi_\Lambda B_L(z)$ as the integral over the diagonal of its integral kernel. So the estimates on the diagonal derived in Section \ref{sec:ptw_intk} will be used in the following. We shall elaborate on this idea further after having stated the next two auxiliaury results. Both these results serve to show that some integral kernels actually vanish on the diagonal.

\begin{lemma}
\label{lem:better diagonal representation}Let $n\in\mathbb{N}_{\geq3}$
be odd, $z\in\mathbb{C}$ with $\Re (z) > -1$. Let $R_{1+z}$, $\cQ$,
$C$, and $\gamma_{j,n}\in\mathbb{C}^{2^{\hat n} \times 2^{\hat n}}$, $j\in\{1,\ldots,n\}$, 
be given by \eqref{eq:resolvent_of_laplace}, \eqref{eq:def_of_Q2},
\eqref{eq:commutator=00003DC} and as in Remark \ref{rem:Eucl-Dirac-Algebar},
respectively. Let $\Phi\colon\mathbb{R}^{n}\to\mathbb{C}^{d \times d}$
be admissible $($see Definition \ref{def:phi_admissible}$)$. 
Then for all $j\in\{1,\ldots,n\}$,  
\begin{equation} 
\tr_{2^{\hat n}d}\big(\gamma_{j,n} \cQ \left(R_{1+z}C\right)^{n-2}R_{1+z}\big) 
=-\tr_{2^{\hat n}d}\big(\gamma_{j,n} \cQ \left(R_{1+z}\Phi \cQ \right)^{n-2}R_{1+z}\big). 
\lb{7.8}
\end{equation} 
\end{lemma}
\begin{proof}
One has 
\begin{align*}
 \tr_{2^{\hat n}d}\left(\gamma_{j,n} \cQ R_{1+z}CR_{1+z}\right) 
 & =\tr_{2^{\hat n}d}\left(\gamma_{j,n} \cQ R_{1+z}\left( \cQ \Phi-\Phi \cQ \right)R_{1+z}\right),\\
 & =\tr_{2^{\hat n}d}\left(\gamma_{j,n}R_{1+z} \cQ \cQ\Phi R_{1+z}\right)  \\
 & \quad - \tr_{2^{\hat n}d}\left(\gamma_{j,n} \cQ R_{1+z}\Phi \cQ R_{1+z}\right)\\
 & = - \tr_{2^{\hat n}d}\left(\gamma_{j,n} \cQ R_{1+z}\Phi \cQ R_{1+z}\right),
\end{align*}
using Proposition \ref{prop:comp_of_Dirac_Alge} to deduce
that $\tr_{2^{\hat n}d}\left(\gamma_{j,n}R_{1+z} \cQ \cQ\Phi R_{1+z}\right)=0$.
In order to proceed to the proof of \eqref{7.8}, we now show
the following: For all odd $k\in\{3,\ldots,n\}$ and $\ell\in\{0,\ldots,k-2\}$
one has 
\begin{align}
\begin{split} 
& \tr_{2^{\hat n}d}\big(\gamma_{j,n} \cQ \left(R_{1+z}C\right)^{k-2}R_{1+z}\big)    \\ 
& \quad =\left(-1\right)^{\ell}\tr_{2^{\hat n}d}\big(\gamma_{j,n} \cQ \left(R_{1+z}\Phi \cQ \right)^{\ell}
\left(R_{1+z}C\right)^{k-2-\ell}R_{1+z}\big).      \label{eq:better diagonal representation}
\end{split} 
\end{align}
In the beginning of the proof we have dealt with the case
$k=3$. One notes that equation \eqref{eq:better diagonal representation}
always holds for $\ell=0$. Next, we assume that $k\in\{5,\ldots,n\}$
is odd, such that equality \eqref{eq:better diagonal representation}
holds for some $\ell\in\{0,\ldots,k-3\}$. Then one computes
\begin{align*}
 & \tr_{2^{\hat n}d}\big(\gamma_{j,n} \cQ \left(R_{1+z}C\right)^{k-2}R_{1+z}\big)\\
 & \quad =(-1)^{\ell}\tr_{2^{\hat n}d} \big(\gamma_{j,n} \cQ \left(R_{1+z}\Phi 
 \cQ \right)^{\ell}\left(R_{1+z}C\right)^{k-2-\ell}R_{1+z}\big)\\
 & \quad =\left(-1\right)^{\ell}\tr_{2^{\hat n}d}\big(\gamma_{j,n} 
 \cQ \left(R_{1+z}\Phi \cQ \right)^{\ell}R_{1+z}C\left(R_{1+z}C\right)^{k-2-\ell-1}R_{1+z}\big)\\
 & \quad =\left(-1\right)^{\ell}\tr_{2^{\hat n}d}\big(\gamma_{j,n} \cQ 
 \left(R_{1+z}\Phi \cQ \right)^{\ell}R_{1+z}\left( \cQ \Phi-\Phi \cQ \right)
 \left(R_{1+z}C\right)^{k-2-\left(\ell+1\right)}R_{1+z}\big)\\
 & \quad =\left(-1\right)^{\ell}\tr_{2^{\hat n}d}\big(\gamma_{j,n} \cQ \left(R_{1+z} 
 \Phi \cQ \right)^{\ell} \cQ R_{1+z}\Phi\left(R_{1+z}C\right)^{k-2-\left(\ell+1\right)}R_{1+z}\big)\\
 & \qquad+\left(-1\right)^{\ell+1}\tr_{2^{\hat n}d}\big(\gamma_{j,n} \cQ 
 \left(R_{1+z}\Phi \cQ \right)^{\ell+1}\left(R_{1+z}C\right)^{k-2-\left(\ell+1\right)}R_{1+z}\big). 
\end{align*}
By Corollary \ref{cor:comp_Dirac_alg_trace-1}, the first term
on the right-hand side cancels, proving equation (\ref{eq:better diagonal representation}).
Putting $\ell=k-2$ in (\ref{eq:better diagonal representation}) implies the assertion.
\end{proof}

The following result is needed for Lemma \ref{lem:some properties of the first two terms}, however, 
it is also of independent interest. Indeed, we will have occasion to use it rather frequently, when we discuss the case of three spatial dimensions specifically. 
Lemma \ref{lem:pointwise convergence of regularized Green} should be regarded as a regularization method, while preserving self-adjointness properties of the ($L^2$-realization) of the underlying operators:

\begin{lemma}
\label{lem:pointwise convergence of regularized Green} Let $\epsilon>0$,
$n\in\mathbb{N}$, and 
$T\in \cB\big(H^{-(n/2)-\epsilon}(\mathbb{R}^{n}),H^{(n/2)+\epsilon}(\mathbb{R}^{n})\big)$. 
Recalling equation \eqref{eq:def_dirac_distr}, we consider 
\[
t\colon\mathbb{R}^{n}\times\mathbb{R}^{n}\ni(x,y)\mapsto\left\langle \delta_{\{x\}},T\delta_{\{y\}}\right\rangle .
\]
For $\mu>0$ we denote $T_{\mu}\coloneqq\left(1-\mu\Delta\right)^{-1}T\left(1-\mu\Delta\right)^{-1}$
and $t_{\mu}$ correspondingly. Then, for all $(x,y)\in\mathbb{R}^{n}\times\mathbb{R}^{n}$, 
\[
t_{\mu}(x,y) \underset{\mu \downarrow 0}{\longrightarrow} t(x,y).
\]
\end{lemma}
\begin{proof}
It suffices to observe that for all $s\in\mathbb{R}$, 
$\left(1-\mu\Delta\right)^{-1}  \underset{\mu \downarrow 0}{\longrightarrow}I$ strongly
in $H^{s}(\mathbb{R}^{n})$ (see \eqref{eq:def_Hs}).   
\end{proof}

In order to proceed to prove the trace theorem, we need to investigate
the asymptotic behavior of the integral kernel of $J_{L}^{j}(z)$
given by \eqref{eq:JJJ} on the diagonal. By 
Proposition \ref{prop:Reformulation of B} together with Lemma \ref{lem:comm is sum of der}, 
we can use Gauss' divergence theorem for computing the integral over the diagonal (see also \eqref{e:div_theo}). Thus, in the expression for the trace of $\chi_\Lambda B_L(z)$ we will use Gauss' theorem for the ball centered at $0$ with radius $\Lambda$. Having applied the divergence theorem, we integrate over spheres of radius $\Lambda$. The volume element of the surface measure grows with $\Lambda^{n-1}$, so any term decaying faster than that will not contribute to the limit $\Lambda \to \infty$ in \eqref{f(z)}. Consequently, any estimate of integral kernels (or differences of such) to follow with the behavior of $|x|^{n-1+\gamma}$ for some $\gamma>0$ on the diagonal, can be neglected in the limit $\Lambda\to \infty$, when computing the expression $\lim_{\Lambda\to \infty}\tr(\chi_\Lambda B_L(z))$.

\begin{lemma}
\label{lem:some properties of the first two terms}Let $n\in\mathbb{N}$
odd, $j\in\{1,\ldots,n\}$, $z\in\mathbb{C}$, $\Re (z) > -1$ and
$R_{1+z}$ be given by \eqref{eq:resolvent_of_laplace} as well as
$\cQ$, $C$ and $\gamma_{j,n}$ given by \eqref{eq:def_of_Q2}, \eqref{eq:commutator=00003DC}
and as in Remark \ref{rem:Eucl-Dirac-Algebar}. 
Then for $n\geq3$, the integral kernel $h_{2,j}$ of 
\[
2\tr_{2^{\hat n}d}\big(\gamma_{j,n}\Phi \left(R_{1+z}C\right)^{n-1}R_{1+z} \big)
\]
satisfies, 
\[
h_{2,j}(x,x)=h_{3,j}(x,x)+g_{0,j}(x,x),
\]
where $h_{3,j}$ is the integral kernel of 
$2\tr_{2^{\hat n}d}\left(\gamma_{j,n}\Phi C^{n-1}R_{1+z}^{n}\right)$
and $g_{0,j}$ satisfies for some $\kappa>0$, 
\[
\left|g_{0,j}(x,x)\right|\leq \kappa (1+|x|)^{1-n-\epsilon},
\quad x\in\mathbb{R}^{n},
\]
where $\epsilon>1/2$ is given as in Definition \ref{def:phi_admissible}. 
In addition, if $n\geq5$ and $z\in\mathbb{R}$, then the integral kernel $h_{1,j}$ of 
\[
\tr_{2^{\hat n}d}\big(\gamma_{j,n} \cQ \left(R_{1+z}C\right)^{n-2}R_{1+z}\big)
\]
vanishes on the diagonal. 
\end{lemma}
\begin{proof}
We discuss $h_{1,j}$ first and consider the operator
\[
B_{n}\coloneqq\left(\Phi \cQ R_{1+z}\right)^{n-3}\Phi=\Phi\left( \cQ R_{1+z}\Phi\right)^{n-3},
\]
which is self-adjoint for all real $z>-1$. Indeed, this follows from
the self-adjointness of $\Phi$ and the skew-self-adjointness of $\cQ R_{1+z}$.
For $\mu>0$ define $B_{n,\mu}\coloneqq (1-\mu\Delta)^{-1}B_n(1-\mu\Delta)^{-1}$. Then the integral kernel $b_{n,\mu}$ of $B_{n,\mu}$ is continuous. Moreover, for all real $z>-1$, the opertor $B_{n,\mu}$ is self-adjoint, by the self-adjointness of $B_n$ and so $b_{n,\mu}$ is real and satisfies $b_{n,\mu}(x,y)=b_{n,\mu}(y,x)$ for all $x,y\in \mathbb{R}^n$. 
By Lemma \ref{lem:better diagonal representation} one recalls 
\begin{align*}
 & \tr_{2^{\hat n}d}\big(\gamma_{j,n} \cQ \left(R_{1+z}C\right)^{n-2}R_{1+z}\big)\\
 & \quad =-\tr_{2^{\hat n}d}\big(\gamma_{j,n} \cQ \left(R_{1+z}\Phi \cQ \right)^{n-2}R_{1+z}\big)\\
 & \quad =-\tr_{2^{\hat n}d}\big(\gamma_{j,n} \cQ R_{1+z}\left(\Phi \cQ R_{1+z}\right)^{n-2}\big)\\
 & \quad =-\tr_{2^{\hat n}d}\big(\gamma_{j,n} \cQ R_{1+z}\left(\Phi \cQ R_{1+z}\right)^{n-3}\Phi \cQ R_{1+z}\big)\\
 & \quad =-\tr_{2^{\hat n}d}\left(\gamma_{j,n} \cQ R_{1+z}B_{n} \cQ R_{1+z}\right).
\end{align*}
By Fubini's theorem and the symmetry of $B_{n,\mu}$, one has for all $j,k\in\{1,\ldots,n\}$ and $x\in\mathbb{R}^{n}$, $z > - 1$, $\mu > 0$, 
\begin{align*}
\Psi_{x,\mu}(j,k) & \coloneqq\int_{\mathbb{R}^{n}\times\mathbb{R}^{n}}(\partial_{j}r_{1+z})(x_{1}-x)b_{n,\mu}(x_{1},x_{2})(\partial_{k}r_{1+z})(x_{2}-x) \, d^n x_{1} d^n x_{2}\\
 & \, =\int_{\mathbb{R}^{n}\times\mathbb{R}^{n}} (\partial_{j}r_{1+z})(x_{1}-x)b_{n,\mu}(x_{2},x_{1})(\partial_{k}r_{1+z})(x_{2}-x) \, d^n x_{1} d^n x_{2} \\
 & \, =\int_{\mathbb{R}^{n}\times\mathbb{R}^{n}} (\partial_{k}r_{1+z})(x_{1}-x)b_{n,\mu}(x_{1},x_{2})(\partial_{j}r_{1+z})(x_{2}-x) \, d^n x_{1} d^n x_{2}  \\
 & \, = \Psi_{x,\mu}(k,j). 
\end{align*}
By Lemma \ref{lem:pointwise convergence of regularized Green} one has for all 
$x,y\in \mathbb{R}^n$, 
\begin{align*}
	h_{1,j}(x,y)&=-\lim_{\mu\downarrow 0}\tr_{2^{\hat n}d} \bigg(\sum_{i_{2},i_{3} = 1}^n \gamma_{j,n}\gamma_{i_{2},n}\gamma_{i_{3},n}\partial_{i_{2}}\int_{\mathbb{R}^{n}\times\mathbb{R}^{n}}r_{1+z}(x-x_{1})b_{n,\mu}(x_{1},x_{2})   \\
& \hspace*{6.2cm} \times (\partial_{i_{3}} r_{1+z})(x_{2}-y) \, d^n x_{1} d^n x_{2}\bigg)     \\
        & =\lim_{\mu\downarrow 0}\tr_{2^{\hat n}d} \bigg(\sum_{i_{2},i_{3} = 1}^n \gamma_{j,n}\gamma_{i_{2},n}\gamma_{i_{3},n}\int_{\mathbb{R}^{n}\times\mathbb{R}^{n}} (\partial_{i_{2}}r_{1+z})(x_{1}-x)b_{n,\mu}(x_{1},x_{2}) \\
& \hspace*{6.1cm} \times  (\partial_{i_{3}}r_{1+z})(x_{2}-y) \, d^n x_{1} d^n x_{2}\bigg).
\end{align*} 
Thus, it follows from Corollary \ref{cor:comp_Dirac_alg_trace-1}
that
\[
  h_{1,j}(x,x)=\lim_{\mu\downarrow 0}\tr_{2^{\hat n}d} \bigg(\sum_{i_{2},i_{3} = 1}^n \gamma_{j,n}\gamma_{i_{2},n}\gamma_{i_{3},n}\Psi_{x,\mu}(i_{2},i_{3})\bigg)
=0, \quad x\in\mathbb{R}^{n}.
\]
The assertion about $h_{2,j}$ is a direct consequence of Remark \ref{rem:green's kernels of the first and the last} and the asymptotic conditions imposed on $\Phi$.
\end{proof}

For the estimate on the diagonal of the integral kernels of the operators under consideration in the next theorem we need to choose the real part of $z$ large. In fact, we use the Neumann series expression for the resolvents $(L^*L+z)^{-1}$ and $(LL^*+z)^{-1}$ and Remark \ref{rem:on the quantitative version of kappa prime}, both of which making the choice of large $\Re (z)$ necessary. We shall also have an a priori bound on the argument of $z$, recalling the definition \eqref{def:Sigma0} of the sector 
$\Sigma_{z_0,\theta} \subset \bbC$. 

\begin{theorem}
\label{thm:asymptotics of the rest} Let $L= \cQ +\Phi$ be given by \eqref{eq:def_of_L(2)},
and for $z\in\mathbb{C}_{\Re > -1}$, let $R_{1+z}$ be given
by \eqref{eq:resolvent_of_laplace} and $C$ as in \eqref{eq:commutator=00003DC}.
For $j\in\{1,\ldots,n\}$, let $\gamma_{j,n}\in\mathbb{C}^{2^{\hat n} \times 2^{\hat n}}$
$($cf.\ Remark \ref{rem:Eucl-Dirac-Algebar}$)$, and $\theta\in (0,\pi/2)$. Then there exists $z_{0}  > 0$, such that for all $z\in \Sigma_{z_0,\theta}$ $($see \eqref{def:Sigma0}$)$, the integral kernels $g_{1,j}$ and $g_{2,j}$ of the operators 
\[
\tr_{2^{\hat n}d}\big(\gamma_{j,n}\Phi\big(\left(L^{*}L+z\right)^{-1}+\left(LL^{*}+z\right)^{-1}\big)\left(CR_{1+z}\right)^{n}\big)
\]
 and 
\[
\tr_{2^{\hat n}d}\big(\gamma_{j,n}Q\big(\left(L^{*}L+z\right)^{-1}+\left(LL^{*}+z\right)^{-1}\big)\left(CR_{1+z}\right)^{n}\big),
\]
 respectively, satisfy for some $\kappa>0$, 
\[
\big[|g_{1,j}(x,x)|+|g_{2,j}(x,x)|\big] \leq \kappa (1+ |x|)^{-n}, 
\quad x\in\mathbb{R}^{n}.  
\]
\end{theorem}
\begin{proof}
We choose $z_{0}$ such that $\sqrt{z_{0}}>2n$ (one recalls Remark \ref{rem:on the quantitative version of kappa prime})
and that for $M\coloneqq\sup_{x\in\mathbb{R}^{n}}\|\Phi(x)\|\lor\|\left( \cQ \Phi\right)(x)\|$
one has $2M[z_{0} \cos(\theta)]^{-1/2} \leq 1/2$. We treat $g_{1,j}$ first. Let 
$z\in \Sigma_{z_0,\theta}$, then,  
\begin{align*}
 & \gamma_{j,n}\Phi\big(\left(L^{*}L+z\right)^{-1}+\left(LL^{*}+z\right)^{-1}\big)\left(CR_{1+z}\right)^{n}\\
 & \quad =\gamma_{j,n}\Phi2\left(R_{1+z}C\right)^{n}\sum_{k=0}^{\infty}\left(R_{1+z}C\right)^{2k}R_{1+z}\\
 & \quad =\sum_{k=0}^{\infty}\gamma_{j,n}\Phi2\left(R_{1+z}C\right)^{n}\left(R_{1+z}C\right)^{2k}R_{1+z}.
\end{align*}
For $x\in\mathbb{R}^{n}$ one infers (recalling $\delta_{\{x\}}$ in \eqref{eq:def_dirac_distr}),  
\begin{align*}
g_{1,j}(x,x) & =\bigg\langle \delta_{\{x\}},\sum_{k=0}^{\infty}\gamma_{j,n}\Phi2\left(R_{1+z}C\right)^{2k}\left(R_{1+z}C\right)^{n}R_{1+z}\delta_{\{x\}}\bigg\rangle \\
 & =\sum_{k=0}^{\infty}\big\langle \delta_{\{x\}},\gamma_{j,n}\Phi2\left(R_{1+z}C\right)^{2k}\left(R_{1+z}C\right)^{n}R_{1+z}\delta_{\{x\}}\big\rangle .
\end{align*}
Hence, by Lemma \ref{lem:asymptotics on diagonal} together with Remark
\ref{rem:on the quantitative version of kappa prime}, there exists
$c>0$ such that for all $x\in\mathbb{R}^{n}$, 
\[
\big|\big\langle \delta_{\{x\}},\gamma_{j,n}\Phi2\left(R_{1+z}C\right)^{2k}\left(R_{1+z}C\right)^{n}R_{1+z}\delta_{\{x\}}\big\rangle\big| \leq  c\left(\frac{2M}{\sqrt{1+z_{0}}}\right)^{2k}\left(\frac{1}{1+\left|x\right|}\right)^{n}.
\] 
Since $2M[1 + z_{0}]^{-1/2} \leq 1/2$, one concludes that
\begin{align*}
\left|g_{1,j}(x,x)\right| & \leq \sum_{k=0}^{\infty}\left|\langle\delta_{\{x\}},\gamma_{j,n}\Phi2\left(R_{1+z}C\right)^{2k}\left(R_{1+z}C\right)^{n}R_{1+z}\delta_{\{x\}}\rangle\right|\\
 & \leq \sum_{k=0}^{\infty}c\left(\frac{2M}{\sqrt{1+z_{0}}}\right)^{2k}\left(\frac{1}{1+\left|x\right|}\right)^{n}\leq  c\left(\frac{1}{1+\left|x\right|}\right)^{n}.
\end{align*}
The analogous reasoning applies to $g_{2,j}$.
\end{proof}

We conclude the results on estimates of certain integral kernels on the diagonal with the following corollary, which, roughly speaking, says that the diagonal of the integral kernels involved is determined by the integral kernel of the operator to be discussed in Proposition \ref{prop:formula for index}.

\begin{corollary}
\label{cor:final aymptotics} For $z\in\mathbb{C}$, $\Re (z) > -1$, denote
$R_{1+z}$ as in \eqref{eq:resolvent_of_laplace}, let 
$\Phi\colon\mathbb{R}^{n}\to\mathbb{C}^{d \times d}$
be admissible $($see Definition \ref{def:phi_admissible}$)$, and $L = \cQ+\Phi$
as in \eqref{eq:def_of_L(2)}, $\theta\in (0,\pi/2)$. In addition, denote $J_{L}^{j}(z)$ for 
$z\in\rho\left(-L^{*}L\right)\cap\rho\left(-LL^{*}\right)$
as in \eqref{eq:JJJ} for all $j\in\{1,\ldots,n\}$, and $C$ as in
\eqref{eq:commutator=00003DC}. Moreover, let 
$\gamma_{j,n}\in\mathbb{C}^{2^{\hat n} \times 2^{\hat n}}$, $j\in\{1,\ldots,n\}$ 
$($cf.\ Remark \ref{rem:Eucl-Dirac-Algebar}$)$. \\[1mm] 
$(i)$ Let $n\in\mathbb{N}_{\geq5}$, $j\in\{1,\ldots,n\}$. Then there
exists $z_{0} > 0$, such that if $z\in \Sigma_{z_0,\theta}$ 
$($see \eqref{def:Sigma0}$)$, and $h$ and $g$ denote the integral kernel of 
$2\tr_{2^{\hat n}d}\left(\gamma_{j,n}\Phi C^{n-1}R_{1+z}^{n}\right)$
and $J_{L}^{j}(z)$, respectively, then for some $\kappa>0$,  
\[
\left|h(x,x)-g(x,x)\right|\leq \kappa (1+ |x|)^{1-n - \epsilon}, 
\quad x\in\mathbb{R}^{n},  
\]
where $\varepsilon > 1/2$ is given as in Definition \ref{def:phi_admissible}. \\[1mm] 
$(ii)$ The assertion of part $(i)$ also holds for $n=3$, if, in the
above statement, $J_{L}^{j}(z)$ is replaced by $J_{L}^{j}(z)-2\tr_{2d}\left(\gamma_{j,3} 
\cQ R_{1+z}CR_{1+z}\right)$. 
\end{corollary}
\begin{proof}
One recalls from Lemma \ref{lem:almost_Neumann_series},   
\begin{align*}
J_{L}^{j}(z) & =2\tr_{2^{\hat n}d}\big(\gamma_{j,n} \cQ \left(R_{1+z}C\right)^{n-2}R_{1+z}\big)+2\tr_{2^{\hat n}d}\big(\gamma_{j,n}\Phi \left(R_{1+z}C\right)^{n-1}R_{1+z}\big)\\
 & \quad+\tr_{2^{\hat n}d}\big(\gamma_{j,n} \cQ \big(\left(L^{*}L+z\right)^{-1}+\left(LL^{*}+z\right)^{-1}\big)\left(CR_{1+z}\right)^{n}\big)\\
 & \quad+\tr_{2^{\hat n}d}\big(\gamma_{j,n}\Phi \big(\left(L^{*}L+z\right)^{-1}+\left(LL^{*}+z\right)^{-1}\big)\left(CR_{1+z}\right)^{n}\big).  
\end{align*}
With the help of Theorem \ref{thm:asymptotics of the rest} one deduces 
that the integral kernels of the last two terms may be estimated
by $\kappa (1+|x|)^{-n}$ on
the diagonal. The integral kernel of the first term on the right-hand side vanishes on the diagonal, which is asserted in Lemma \ref{lem:some properties of the first two terms}.
Hence, it remains to inspect the second term of the right-hand side.
Thus, the assertion follows from Lemma \ref{lem:some properties of the first two terms}. 
\end{proof}

Having identified the integral kernel $g_j$ of $2\tr_{2^{\hat n}d}\left(\gamma_{j,n}\Phi C^{n-1}R_{1+z}^{n}\right)$ to be the only term determining the trace of $\chi_\Lambda B_L(z)$ as 
$\Lambda\to \infty$, we shall compute the integral over the  diagonal of $g_j$:

\begin{proposition}
\label{prop:formula for index}Let $n\in\mathbb{N}_{\geq 3}$ odd, $C$
as in \eqref{eq:commutator=00003DC}, $z\in\mathbb{C}$, $\Re (z) > -1$, 
with $R_{1+z}$ given by \eqref{eq:resolvent_of_laplace}, 
$\Phi\colon\mathbb{R}^{n}\to\mathbb{C}^{d \times d}$
be admissible $($see Definition \ref{def:phi_admissible}$)$, 
$\gamma_{j,n}\in\mathbb{C}^{2^{\hat n} \times 2^{\hat n}}$, $j\in\{1,\ldots,n\}$, as in Remark \ref{rem:Eucl-Dirac-Algebar}.  
Then for $j\in\{1,\ldots,n\}$, the integral kernel $g_{j}$ of
\[
2\tr_{2^{\hat n}d}\left(\gamma_{j,n}\Phi C^{n-1}R_{1+z}^{n}\right)
\]
 satisfies, 
\begin{align*}
g_{j}(x,x) & = (1+z)^{-n/2}\left(\frac{i}{8\pi}\right)^{(n-1)/2}\frac{1}{\left[(n-1)/2\right]!}    \\ 
& \quad \, \times \sum_{i_{1},\ldots,i_{n-1} = 1}^n \epsilon_{ji_{1}\ldots i_{n-1}} 
\tr\big(\Phi(x) (\partial_{i_{1}}\Phi)(x)\ldots (\partial_{i_{n-1}}\Phi)(x)\big), 
\quad x\in\mathbb{R}^{n},  
\end{align*}
where $\epsilon_{ji_{1}\ldots i_{n-1}}$ denotes the $\epsilon$-symbol
as in Proposition \ref{prop:comp_of_Dirac_Alge}. 
\end{proposition}
\begin{proof}
We recall that $n=2\hat n+1$. With the help
of Proposition \ref{prop:comp_of_Dirac_Alge}, $g_{j}$ is given by
\begin{align*}
(x,y) & \mapsto2(2i)^{\hat n}\sum_{i_{1}\ldots i_{n-1} = 1}^n \epsilon_{ji_{1}\ldots i_{n-1}}\tr\big(\Phi(x)
(\partial_{i_{1}}\Phi)(x)\ldots (\partial_{i_{n-1}}\Phi)(x)\big)\\
& \quad \;\; \times \int_{\left(\mathbb{R}^{n}\right)^{n-1}}r_{1+z}(x-x_{1})r_{1+z}(x_{1}-x_{2})\cdots r_{1+z}(x_{n-1}-y) \, d^n x_1 \cdots d^n x_{n-1}.
\end{align*}
Hence, by substitution in the integral expression and putting $x=y$,
one obtains
\begin{align*}
g_{j}(x,x) & =2(2i)^{\hat n}\sum_{i_{1}\ldots i_{n-1} = 1}^n \epsilon_{ji_{1}\ldots i_{n-1}}\tr\big(\Phi(x)
(\partial_{i_{1}}\Phi)(x)\ldots (\partial_{i_{n-1}}\Phi)(x)\big)\\
& \quad \times \int_{\left(\mathbb{R}^{n}\right)^{n-1}}r_{1+z}(x_{1})r_{1+z}(x_{1}-x_{2})\cdots r_{1+z}(x_{n-1}) \, d^n x_1 \cdots d^n x_{n-1}.
\end{align*}
The last integral can be computed with the help of the Fourier transform
and polar coordinates, as was done in Proposition \ref{prop:estimate of the pointwise resolvent at 0}.
In fact, one gets (see also \cite[3.252.2]{GR80}),  
\begin{align*}
 & \int_{\left(\mathbb{R}^{n}\right)^{n-1}}r_{1+z}(x_{1})r_{1+z}(x_{1}-x_{2})\cdots r_{1+z}(x_{n-1}) \, d^n x_1 \cdots d^n x_{n-1}\\
 & \quad =\left(2\pi\right)^{-n}\frac{2\pi^{n/2}}{\Gamma\left(n/2\right)}\int_{0}^{\infty}r^{n-1}\frac{1}{\left(r^{2}+1+z\right)^{n}} \, dr\\
 & \quad =\left(2\pi\right)^{-n}\frac{2\pi^{n/2}}{\Gamma\left(n/2\right)}
 (1+z)^{-n/2} \frac{2^{-n}\sqrt{\pi}\Gamma\left(n/2\right)}{\left[(n-1)/2\right]!}\\
 & \quad =\frac{1}{2^{2n-1}}\frac{1}{\pi^{(n-1)/2}}\frac{1}{\left[(n-1)/2\right]!}
 (1+z)^{-n/2}, 
\end{align*}
and notes that 
\begin{align*}
& 2(2i)^{(n-1)/2}\frac{1}{2^{2n-1}}\frac{1}{\pi^{(n-1)/2}}\frac{1}{\left[(n-1)/2\right]!}   \no \\
& \quad =\left(\frac{i}{\pi}\right)^{(n-1)/2}\frac{1}{2^{(4n-4-n+1)/2}}\frac{1}{\left[(n-1)/2\right]!}\\
 & \quad =\left(\frac{i}{\pi}\right)^{(n-1)/2}\frac{1}{2^{(3n-3)/2}}\frac{1}{\left[(n-1)/2\right]!}\\
 & \quad =\left(\frac{i}{8\pi}\right)^{(n-1)/2}\frac{1}{\left[(n-1)/2\right]!}.\tag*{\qedhere}
\end{align*}
\end{proof}

Finally, we are ready to prove the (trace) Theorem \ref{thm:Witten_reg_n5}, for $n\geq5$, that is, we consider the operator $L=\mathcal{Q}+\Phi$ with an admissible potential $\Phi$, such that $\Phi$ is smooth and attains values in the self-adjoint, unitary $d\times d$-matrices. In addition, we recall that the first derivatives of $\Phi$ behave like $|x|^{-1}$ for large $x$, whereas higher-order derivatives decay at least with the behavior $|x|^{-1-\epsilon}$ for large $x$ and some $\epsilon> 1/2$. We note that we already established the Fredholm property of $L$ in Theorem \ref{thm:Fredholm_property}. We outline the proof of Theorem \ref{thm:Witten_reg_n5}, for $n\geq5$, as follows. The results in Section \ref{sec:The-Derivation-of-trace-f} yield the applicability of 
Theorem \ref{thm:index with Witten}. More precisely, the operator $\chi_\Lambda B_L(z)$ is trace class with trace computable as the integral over the diagonal of the integral kernel of $\chi_\Lambda B_L(z)$. With Proposition \ref{prop:Reformulation of B} we will deduce that only the term involving $J_L^j(z)$, being analysed in Lemma \ref{lem:almost_Neumann_series}, matters for the computation of the index. Next, we will show that $\{ z\mapsto \tr(\chi_\Lambda B_L(z))\}_{\Lambda>0}$ is locally bounded using Lemma \ref{lem:comm is sum of der} (in particular \eqref{e:div_theo}). The local boundedness result is then obtained via Gauss' divergence theorem and Lemma \ref{lem:some properties of the first two terms} as well as Theorem \ref{thm:local_boundedness of the trace}. Having proved local boundedness, we will use Montel's theorem for deducing that at least for a sequence $\{\Lambda_k\}_{k\in\bbN}$ the limit $f\coloneqq \lim_{k\to\infty}\tr(\chi_{\Lambda_k} B_L(\cdot))$ exists in the compact open topology, that is, the topology of uniform convergence on compacts. With the results from Corollary \ref{cor:final aymptotics} and Proposition \ref{prop:formula for index}, choosing $\Re (z)$ sufficiently large, we get an explicit expression for $f$. The explicit expression for $f$, by the principle of analytic continuation, carries over to $z$ in a neighborhood of $0$. As we know, by 
Theorem \ref{thm:index with Witten}, that the limit $\lim_{\Lambda\to\infty}\tr(\chi_{\Lambda} B_L(0))$ exists and coincides with the index of $L$, we can then deduce that not only for the sequence $\{\Lambda_k\}_{k\in\bbN}$ but, in fact, the limit $\lim_{\Lambda\to\infty}\tr(\chi_{\Lambda} B_L(\cdot))$ exists in the compact open topology and coincides with $f$ given in \eqref{f(z)}. The detailed arguments read as follows.

\begin{proof}[Proof of Theorem \ref{thm:Witten_reg_n5} for $n\geq5$]
 By Theorem \ref{thm:trisbounded}, $\chi_\Lambda B_{L}(z)$ is trace class for every $\Lambda>0$. Moreover,
by Remark \ref{r:5.8}, $\tr (\chi_\Lambda B_{L}(z))$ can be computed
as the integral over the diagonal of the respective integral kernel.
Hence, by Proposition \ref{prop:Reformulation of B}, equation \eqref{eq:B_L=00003DJJJ+ALz}, recalling also Remark  \ref{rem:Convenient:Green},
one obtains 
\begin{align}
2\tr (\chi_\Lambda B_{L}(z)) & = 2\int_{B(0,\Lambda)} \left\langle \delta_{\{x\}},B_{L}(z)\delta_{\{x\}}
\right\rangle_{H^{-(n/2) - \varepsilon}, H^{(n/2)+\varepsilon}} \, d^n x \no \\
 & =\int_{B(0,\Lambda)}\bigg\langle \delta_{\{x\}},\bigg(\sum_{j=1}^{n}\big[\partial_{j},J_{L}^{j}(z)\big]+A_{L}(z)\bigg)\delta_{\{x\}}\bigg\rangle \, d^n x \no \\
 & =\int_{B(0,\Lambda)}\bigg\langle \delta_{\{x\}},\sum_{j=1}^{n}\big[\partial_{j},J_{L}^{j}(z)\big]\delta_{\{x\}}\bigg\rangle \, d^n x, \label{e:p5.1.1}
\end{align}
where we used Lemma \ref{lem:j,A,Green} to deduce that $\left\langle \delta_{\{x\}},A_{L}(z)\delta_{\{x\}}\right\rangle =0$
for all $x\in\mathbb{R}^{n}$.
Next, we prove that $\{z \mapsto \tr (\chi_\Lambda B_L(z))\}_{\Lambda>0}$ is locally bounded. 
One recalls from Lemma \ref{lem:almost_Neumann_series}, 
\begin{align*}
  J_{L}^{j}(z) & =2\tr_{2^{\hat n}d}\big(\gamma_{j,n} \cQ \left(R_{1+z}C\right)^{n-2}R_{1+z}\big)+2\tr_{2^{\hat n}d}\big(\gamma_{j,n}\Phi \left(R_{1+z}C\right)^{n-1}R_{1+z}\big)\\
 & \quad+\tr_{2^{\hat n}d}\big(\gamma_{j,n} \cQ \big(\left(L^{*}L+z\right)^{-1}+\left(LL^{*}+z\right)^{-1}\big)\left(CR_{1+z}\right)^{n}\big)\\
 & \quad+\tr_{2^{\hat n}d}\big(\gamma_{j,n}\Phi\big(\left(L^{*}L+z\right)^{-1}+\left(LL^{*}+z\right)^{-1}\big)\left(CR_{1+z}\right)^{n}\big). 
\end{align*}
Hence, 
\begin{align}
 & \sum_{j=1}^{n}\big[\partial_{j},J_{L}^{j}(z)\big] 
 = \sum_{j=1}^{n}\big[\partial_{j},\big(2\tr_{2^{\hat n}d}
 \big(\gamma_{j,n} \cQ (R_{1+z}C)^{n-2}R_{1+z}\big)  \no  \\ 
 & \hspace*{3.5cm} +2\tr_{2^{\hat n}d}\big(\gamma_{j,n}\Phi (R_{1+z}C)^{n-1}R_{1+z}\big)\big)\big]\no \\
 &\quad +\tr_{2^{\hat n}d}\big(\big[ \cQ, \cQ \big(\left(L^{*}L+z\right)^{-1}+\left(LL^{*}+z\right)^{-1}\big)\left(CR_{1+z}\right)^{n}\big]\big)\no   \\
 & \quad+\tr_{2^{\hat n}d} \big(\big[ \cQ,\Phi \big(\left(L^{*}L+z\right)^{-1}+\left(LL^{*}+z\right)^{-1}\big)\left(CR_{1+z}\right)^{n}\big]\big).\label{e:p5.1.2}
\end{align} 
Denoting by $h_{j}$ the integral kernel of 
$2\textnormal{tr}_{{2^{\hatt n}d}}\left(\gamma_{j,n}\Phi C^{n-1}R_{1+z}^{n}\right)$, one 
observes that for some constant $\kappa>0$, 
$|h_j(x,x)|\leq  \kappa (1+|x|)^{1-n}$, $x\in \mathbb{R}^n$. Hence, for 
any $\Lambda>0$, invoking Lemma \ref{lem:comm is sum of der}, Gauss' theorem implies that 
\begin{align}
 &\bigg|\int_{B(0,\Lambda)} \bigg\langle \delta_{\{x\}},\sum_{j=1}^{n}\left[\partial_{j},2\textnormal{tr}_{2^{\hatt n}d}\left(\gamma_{j,n}\Phi C^{n-1}R_{1+z}^{n}\right)\right]\delta_{\{x\}}\bigg\rangle \, 
 d^nx\bigg| \no
  \\ & \quad =\bigg|\int_{B(0,\Lambda)} \sum_{j=1}^n (\partial_j h_j)(x,x) \, d^n x\bigg| \no
  \\ & \quad =\bigg|\int_{\Lambda S^{n-1}}  \sum_{j=1}^n h_j(x,x)\frac{x_j}{\Lambda} \, d^{n-1} \sigma( x)\bigg| \no
  \\ & \quad \leq  \int_{\Lambda S^{n-1}}  \sum_{j=1}^n \left|h_j(x,x)\right| d^{n-1} \sigma( x) \no
  \\ & \quad \leq  n\kappa (1+\Lambda)^{1-n}\Lambda^{n-1}\omega_{n-1},\label{e:p5.1.3}
\end{align}
(with $\omega_{n-1}$ being the $(n-1)$-dimensional volume
of the unit sphere $S^{n-1}=\left\{ x\in\mathbb{R}^{n} \,|\, \left|x\right|=1\right\} $, see \eqref{e:om}). 
The latter is uniformly bounded with respect to $\Lambda>0$.

Using Lemma \ref{lem:some properties of the first two terms}, the definition of $g_{0,j}$ in that lemma as well as Gauss' theorem, one arrives at
\begin{align}
 & \bigg|\int_{B(0,\Lambda)} \bigg\langle \delta_{\{x\}},\sum_{j=1}^{n}\big[\partial_{j},2\tr_{2^{\hat n}d}\big(\gamma_{j,n} \cQ \left(R_{1+z}C\right)^{n-2}R_{1+z}\big)\big]\delta_{\{x\}}\bigg\rangle \, d^nx \no \\
  &\qquad  + \int_{B(0,\Lambda)} \bigg\langle \delta_{\{x\}},\sum_{j=1}^{n} 
  \big[\partial_{j},\big(2\tr_{2^{\hat n}d}\gamma_{j,n}\Phi \left(R_{1+z}C\right)^{n-1}R_{1+z}\big)\big]\delta_{\{x\}}
  \bigg\rangle \, d^nx \no  \\
  &\qquad  - \int_{B(0,\Lambda)} \bigg\langle \delta_{\{x\}},\sum_{j=1}^{n}\big[\partial_{j},2\textnormal{tr}_{2^{\hat n}d}\left(\gamma_{j,n}\Phi C^{n-1}R_{1+z}^{n}\right)\big]\delta_{\{x\}}\bigg\rangle 
  \, d^nx\bigg| \no \\
  & \quad = \bigg|\int_{B(0,\Lambda)} \sum_{j=1}^n (\partial_jg_{0,j})(x,x) \, d^n x\bigg| \no \\
  & \quad \leq  \bigg|\int_{\Lambda S^{n-1}} \sum_{j=1}^n g_{0,j}(x,x)\frac{x_j}{\Lambda} d^{n-1} \sigma(x)\bigg|\no \\
  & \quad \leq  \int_{\Lambda S^{n-1}} \sum_{j=1}^n\left|g_{0,j}(x,x)\right|d^{n-1} \sigma(x) \no \\
  & \quad \leq  \int_{\Lambda S^{n-1}} n\kappa (1+|x|)^{1-n-\epsilon} d^{n-1} \sigma(x) \no \\
  & \quad \leq  n\kappa (1+\Lambda)^{1-n-\epsilon} \omega_{n-1}\Lambda^{n-1} 
  \underset{\Lambda \to \infty}{\longrightarrow} 0. \label{e:p5.1.4}
\end{align}
Next, Theorem \ref{thm:local_boundedness of the trace} implies that
\begin{align}
&\Big\{z\mapsto z \tr\Big(\chi_\Lambda \Big(\tr_{2^{\hat n}d} \big(\big[ \cQ, \cQ \big((L^{*}L+z)^{-1}+(LL^{*}+z)^{-1}\big) (CR_{1+z})^{n}\big]\big) \no    \\
& \quad +\tr_{2^{\hat n}d} \big(\big[ \cQ,\Phi \big((L^{*}L+z)^{-1}+(LL^{*}+z)^{-1}\big)
(CR_{1+z})^{n}\big]\big) \Big) \Big)\Big\}_{\Lambda>0} \label{e:p5.1.5}
\end{align}
is bounded on any compact neighborhood of $0$ intersected with $B(0,\delta)\cup(\rho(-LL^*)\cap\rho(-L^*L))$ for some $\delta>0$. 
Hence, summarizing equations \eqref{e:p5.1.1} and \eqref{e:p5.1.2}, we get for $z\in \mathbb{C}_{\Re>-1}\cap \rho(-L^*L)\cap \rho(-LL^*)$:
\begin{align*}
  &  z2\tr (\chi_\Lambda B_L(z))
  = z \int_{B(0,\Lambda)}\bigg\langle \delta_{\{x\}},\sum_{j=1}^{n}\big[\partial_{j},J_{L}^{j}(z)\big]\delta_{\{x\}}\bigg\rangle \, d^n x
  \\& \quad = z \int_{B(0,\Lambda)}\bigg\langle \delta_{\{x\}},\sum_{j=1}^{n}\big[\partial_{j},2\tr_{2^{\hat n}d}
 \big(\gamma_{j,n} \cQ (R_{1+z}C)^{n-2}R_{1+z}\big)\big]\delta_{\{x\}}\bigg\rangle \, d^n x,
 \\& \qquad + z \int_{B(0,\Lambda)}\bigg\langle \delta_{\{x\}},\sum_{j=1}^{n}\big[\partial_{j},2\tr_{2^{\hat n}d}\big(\gamma_{j,n}\Phi (R_{1+z}C)^{n-1}R_{1+z}\big)\big]\delta_{\{x\}}\bigg\rangle \, d^n x 
 \\& \qquad + z  \int_{B(0,\Lambda)}\!\!\!	\bigg\langle \delta_{\{x\}},\tr_{2^{\hat n}d}\big(\big[ \cQ, \cQ \big(\left(L^{*}L+z\right)^{-1}    \no \\
 & \hspace*{2.5cm}
 +\left(LL^{*}+z\right)^{-1}\big)\left(CR_{1+z}\right)^{n}\big]\big)\delta_{\{x\}}\bigg\rangle \, d^n x
 \\& \qquad + z  \int_{B(0,\Lambda)}\!\!\!	\bigg\langle \delta_{\{x\}},\tr_{2^{\hat n}d}\big(\big[ \cQ, \Phi \big(\left(L^{*}L+z\right)^{-1}     \no \\ 
 & \hspace*{2.5cm} 
 +\left(LL^{*}+z\right)^{-1}\big)\left(CR_{1+z}\right)^{n}\big]\big)\delta_{\{x\}}\bigg\rangle \, d^n x 
 \\& \quad = z\int_{B(0,\Lambda)} \sum_{j=1}^n (\partial_j g_{0,_j})(x,x) \, d^n x+ z \int_{B(0,\Lambda)} \sum_{j=1}^n (\partial_j h_j)(x,x) \, d^n x + 
 \\ & \qquad +z \tr\Big(\chi_\Lambda \Big(\tr_{2^{\hat n}d} \big(\big[ \cQ, \cQ \big((L^{*}L+z)^{-1}+(LL^{*}+z)^{-1}\big) (CR_{1+z})^{n}\big]\big)     \\
 & \quad \quad\quad\quad\quad\quad +\tr_{2^{\hat n}d} \big(\big[ \cQ,\Phi \big((L^{*}L+z)^{-1}+(LL^{*}+z)^{-1}\big) 
 (CR_{1+z})^{n}\big]\big) \Big) \Big).
\end{align*}
Thus, with the estimates \eqref{e:p5.1.3} and \eqref{e:p5.1.4} together with \eqref{e:p5.1.5}, 
one infers that
 \[
   \{ z\mapsto  2z\tr (\chi_\Lambda B_L(z))\}_{\Lambda>0}
 \]
is locally bounded on $B(0,\delta)\cup \mathbb{C}_{\Re>0}$ for some $\delta>0$. 
By Lemma \ref{lem:Laurent_series_argument} together with Theorem \ref{thm:trisbounded}, one infers that
\[
       \{z\mapsto 2\tr (\chi_\Lambda B_L(z))\}_{\Lambda>0}
      \]
is locally bounded on $B(0,\delta)\cup \mathbb{C}_{\Re>0}$. By Montel's Theorem, there exists a sequence $\{\Lambda_k\}_{k\in\bbN}$ of positive reals tending to infinity such that 
\[
   \{z\mapsto 2 \tr (\chi_{\Lambda_k} B_L(z))\}_{k\in \mathbb{N}}
\]
converges in the compact open topology. We denote by $f$ the respective limit.
Then Lemma \ref{lem:comm is sum of der} implies that for $k\in \mathbb{N}$, 
\[
2\tr (\chi_{\Lambda_k}B_{L}(z)) = \int_{B(0,\Lambda_k)}\sum_{j=1}^{n} (\partial_{j}g_{j})(x,x)\, d^n x. 
\]
and so 
\[
      f(z) = \lim_{k\to\infty} \int_{B(0,\Lambda_k)} \text{div} \,\mathbb{G}_{J,z}(x) \, d^n x. 
\]
Here we denote $\mathbb{G}_{J,z}\coloneqq\{x\mapsto g_{j}(x,x)\}_{j\in\{1,\ldots,n\}}$, 
with $g_{j}$ being the integral kernel of $J_{L}^{j}(z)$ for $j\in\{1,\ldots,n\}$.
Next, let $\theta\in (0,\pi/2)$ and choose $z_0 > 0$ as in Corollary
\ref{cor:final aymptotics}\,$(i)$. Let $z\in \Sigma_{z_0,\theta}$, see \eqref{def:Sigma0}. Recalling that $h_{j}$ is the integral kernel of $2\textnormal{tr}_{{2^{\hat n}d}}\left(\gamma_{j,n}\Phi C^{n-1}R_{1+z}^{n}\right)$, 
we define $\mathbb{H}_{z}\coloneqq\{x\mapsto h_{j}(x,x)\}_{j\in\{1,\ldots,n\}}$.
Due to Corollary \ref{cor:final aymptotics}, one can find $\kappa>0$ such
that for $k\in \mathbb{N}$, 
\begin{align*}
& \left|\int_{\Lambda_k S^{n-1}}\left((\mathbb{G}_{J,z}-\mathbb{H}_{z})(x),\frac{x}{\Lambda_k}
\right)_{\bbR^n} \, d^{n-1} \sigma(x)\right|    \\ 
& \quad \leq \int_{\Lambda_k S^{n-1}} \|(\mathbb{G}_{J,z}-\mathbb{H}_{z})(x)\|_{\bbR^n}\, d^{n-1} \sigma(x)\\
 & \quad \leq \kappa\int_{\Lambda_k S^{n-1}} (1+|x|)^{1-n-\epsilon}\, d^{n-1} \sigma(x)\\
 & \quad =\kappa \Lambda_k^{n-1}\omega_{n-1} (1+\Lambda_k)^{1-n-\epsilon}.
\end{align*}
Consequently, 
\[
\lim_{k\to\infty}\int_{\Lambda_k S^{n-1}}\left((\mathbb{G}_{J,z}-\mathbb{H}_{z})(x),\frac{x}{\Lambda_k}
\right)_{\bbR^n} \, d^{n-1} \sigma(x)=0.
\]
Hence, with the help of Gauss' theorem,
\begin{align}
f(z) & =\lim_{k\to\infty}\int_{B(0,\Lambda_k)}\sum_{j=1}^{n} (\partial_{j}g_{j})(x,x) \, d^n x 
= \int_{\mathbb{R}^{n}}\text{div} \, \mathbb{G}_{J,z}(x) \, \, d^n x\notag \\
 & =\lim_{k\to\infty}\int_{B(0,\Lambda_k)}\text{div} \, \mathbb{G}_{J,z}(x) \, d^n x=\lim_{k\to\infty}\int_{\Lambda_k S^{n-1}}\left( \mathbb{G}_{J,z}(x),\frac{x}{\Lambda_k}\right)_{\bbR^n} \, d^{n-1} \sigma(x) \notag \\
 & =\lim_{k\to\infty}\int_{\Lambda_k S^{n-1}}\left( \mathbb{H}_{z}(x),\frac{x}{\Lambda_k}
 \right)_{\bbR^n} \, d^{n-1} \sigma(x) \notag \\
 & =\left(\frac{i}{8\pi}\right)^{(n-1)/2}\frac{1}{\left[(n-1)/2\right]!} (1+z)^{-n/2} 
 \lim_{k\to\infty}\int_{\Lambda_k S^{n-1}} \notag \\
 & \quad \times \sum_{j=1}^n \bigg(\sum_{i_{1},\ldots,i_{n-1} = 1}^n \epsilon_{ji_{1}\ldots i_{n-1}}\tr\big(\Phi(x) (\partial_{i_{1}}\Phi)(x)\ldots (\partial_{i_{n-1}}\Phi)(x)\big)\bigg)  \no \\ 
 & \qquad \quad \;\;\; \times \left(\frac{x_{j}}{\Lambda_k}\right) \, d^{n-1} \sigma(x),    
  \label{eq:limit_expr_f} 
\end{align} 
where, for the last integral, we used Proposition \ref{prop:formula for index}. By 
Theorem \ref{thm:index with Witten} one has $f(0) = 2 \ind (L)$. In particular, any sequence 
$\{\Lambda_k\}_k$ of positive reals converging to infinity contains a subsequence 
$\{\Lambda_{k_\ell}\}_\ell$ such that for that particular subsequence the limit
\begin{align*} 
& \lim_{\ell\to\infty}\int_{\Lambda_{k_\ell} S^{n-1}} \sum_{j=1}^n 
\bigg(\sum_{i_{1},\ldots,i_{n-1} = 1}^n \epsilon_{ji_{1}\ldots i_{n-1}}\tr\big(\Phi(x) (\partial_{i_{1}}\Phi)(x) \ldots (\partial_{i_{n-1}}\Phi)(x)\big)\bigg)    \\ 
 & \hspace*{2.7cm} \times \left(\frac{x_{j}}{\Lambda_{k_\ell}}\right)\, d^{n-1} \sigma(x) 
\end{align*}  
exists and equals 
\begin{equation}\label{eq:subseq_argum} 
\frac{2 \ind (L) \left[(n-1)/2\right]!} {\left[i/(8\pi)\right]^{(n-1)/2}}.
\end{equation}
Hence, the limit 
\begin{align}\label{eq:limit_exists} 
\begin{split} 
& \lim_{\Lambda\to\infty}\int_{\Lambda S^{n-1}} \sum_{j=1}^n 
\bigg(\sum_{i_{1},\ldots,i_{n-1} = 1}^n \epsilon_{ji_{1}\ldots i_{n-1}}\tr\left(\Phi(x)\left(\partial_{i_{1}}\Phi\right)(x)\ldots\left(\partial_{i_{n-1}}\Phi\right)(x)\right)\bigg) \\ 
& \hspace*{2.5cm} \times \left(\frac{x_{j}}{\Lambda }\right) \, d^{n-1} \sigma(x)  
\end{split} 
\end{align} 
exists and equals the number in \eqref{eq:subseq_argum}. On the other hand, for $z\in \Sigma_{z_0,\theta}$, (see again Corollary \ref{cor:final aymptotics}) the family 
\[
  \{z\mapsto \tr (\chi_{\Lambda} B_L(z))\}_{\Lambda>0}
\]
converges for $\Lambda\to\infty$ on the domain $\Sigma_{z_0,\theta}$ if and only if the limit in \eqref{eq:limit_exists} exists. Indeed, this follows from the explicit expression for the limit in \eqref{eq:limit_expr_f}. Therefore,
\[
  \{z\mapsto \tr (\chi_{\Lambda} B_L(z))\}_{\Lambda>0}
\] 
converges in the compact open topology on $\Sigma_{z_0,\theta}$. By the local boundedness of the latter family on the domain $B(0,\delta)\cup\mathbb{C}_{\Re>0}$, the principle of analytic continuation implies that the latter family actually converges on the domain $B(0,\delta)\cup\mathbb{C}_{\Re>0}$ in the compact open topology. In particular, 
\begin{align*}
&  2 f(z) (1+z)^{n/2} \frac{[(n-1)/2]!}{[i/(8\pi)]^{(n-1)/2}} \\ 
& \quad = \lim_{\Lambda\to\infty}\int_{\Lambda S^{n-1}} \sum_{j=1}^n \bigg(\sum_{i_{1},\ldots,i_{n-1} = 1}^n \epsilon_{ji_{1}\ldots i_{n-1}}\tr \big(\Phi(x) (\partial_{i_{1}}\Phi)(x)\dots (\partial_{i_{n-1}}\Phi)(x)\big) \bigg)    \\
& \hspace*{3.3cm} \times \bigg(\frac{x_{j}}{\Lambda }\bigg) \, d^{n-1} \sigma(x).  \qedhere
\end{align*} 
\end{proof}

\newpage

\section{The Case $n=3$}\label{sec:n=3}

In this section we shall discuss the necessary modifications, such that 
Theorem \ref{thm:Witten_reg_n5} continues to hold also for the case $n=3$. The main issue for the need of extra arguments for this case is the lack of differentiability of the integral kernel of 
$J_L^j(z)$ given in Lemma \ref{lem:j,A,Green}. The main issue being the first summand in the expression for $J_L^j(z)$ derived in Lemma \ref{lem:almost_Neumann_series}, that is, the term
\[
  \tr_{2^{\hat n}d}\big(\gamma_{j,n} \cQ (R_{1+z}C)R_{1+z}\big),
\]
for the integral kernel of which we fail to show differentiability. Indeed, as this opertor increases regularity only by $3$ orders of differentiability, not even continuity of the associated integral kernel is clear. The basic idea to overcome this difficulty and to get the result asserted in Theorem \ref{thm:Witten_reg_n5} also for the case $n=3$ has already been used and is contained in Lemma \ref{lem:pointwise convergence of regularized Green}. So, the operator $B_L(z)$ will be multiplied from the left and from the right by $(1-\mu\Delta)^{-1}$ for some $\mu>0$. The reason for multiplying from both sides is that we wanted to re-use strategies for showing that certain integral kernels vanish on the diagonal. The key for the latter arguments has been the self-adjointness of the operators under consideration, which, in turn, result in symmetry properties for the associated integral kernel. 

An additional fact, enabling the strategy just sketched for the case $n=3$, is the following result.

\begin{proposition}[{See, e.g., \cite{Si05}, p.~28--29, or \cite{Ya92}, Lemma~6.1.3}]
\label{prop:s1-approx} Assume  $\cH$ is a complex, separable Hilbert space,
$B\in\cB_1(\cH),$ and $A \geq 0$ is self-adjoint in $\cH$.
Then for $\mu>0$, $B_{\mu}\coloneqq(1+\mu A)^{-1}B(1+\mu A)^{-1}\in\cB_1(\cH)$
and $B_{\mu}\to B$ in $\cB_1(\cH)$ as $\mu\downarrow 0$. In particular,
$\tr_{\cH} (B_{\mu}) \underset{\mu \downarrow 0}{\longrightarrow} \tr_{\cH} (B)$. 
\end{proposition}

Next, we will give the details for the modifications of the proof of Theorem \ref{thm:Witten_reg_n5} for the case $n=3$. Thus, for $\mu>0$, we introduce the operator 
\begin{equation}
B_{L,\mu}(z)\coloneqq\left(1-\mu\Delta  \right)^{-1}B_{L}(z)\left(1-\mu\Delta  \right)^{-1},\label{eq:B_Lmu}
\end{equation}
where $z\in\rho\left(-L^{*}L\right)\cap\rho\left(-LL^{*}\right)$, $L= \cQ +\Phi$
given by \eqref{eq:def_of_L(2)}, and $B_{L}(z)$ is given by \eqref{eq:def_of_BL(z)2}.
We also introduce
\begin{align}
J_{L,\mu}^{j}(z) & = (1-\mu\Delta )^{-1}\big(\tr_{2^{\hat n}d}\big(L\big(L^{*}L+z)^{-1}\gamma_{j,n}\big)    \no \\ 
& \quad +\tr_{2^{\hat n}d}\big(L^{*}(LL^{*}+z)^{-1}\gamma_{j,n}\big)\big) (1-\mu\Delta )^{-1},    
\label{eq:JJJ_mu}
\end{align}
with $\gamma_{j,n}\in\mathbb{C}^{2^{\hat n} \times 2^{\hat n}}$, $j\in\{1,\ldots,n\}$,  
as in Remark \ref{rem:Eucl-Dirac-Algebar}, and 
\begin{align}
A_{L,\mu}(z) & =(1-\mu\Delta )^{-1}\big(\tr_{2^{\hat n}d}\big(\big[\Phi,L^{*}(LL^{*}+z)^{-1}\big]\big)   \no \\ 
& \quad -\tr_{2^{\hat n}d}\big(\big[\Phi,L(L^{*}L+z)^{-1}\big]\big)\big) (1-\mu\Delta )^{-1}   
\label{eq:ALz_mu}
\end{align}
for the admissible potential $\Phi$ (see Definition \ref{def:phi_admissible}).
By Theorem \ref{thm:trisbounded} and the ideal property of $\cB_1(\cH)$,
the operator $B_{L,\mu}(z)$ is trace class for all $\mu>0$ and $z\in\rho\left(-L^{*}L\right)\cap\rho\left(-LL^{*}\right)$
and $\Re (z) > -1$. As for the case $n\geq5$, we need the following
more detailed description of the operator $B_{L,\mu}(z)$: 

\begin{lemma}
\label{lem:repre of bmu} Let $L= \cQ +\Phi$ as in \eqref{eq:def_of_L(2)}, 
$\mu>0$, $z\in\rho\left(-L^{*}L\right)\cap\rho\left(-LL^{*}\right)$, with 
$\Re (z) > -1$. Then with $J_{L,\mu}^{j}(z)$, $j\in\{1,\ldots,m\}$, 
and $A_{L,\mu}(z)$ given by \eqref{eq:JJJ_mu} and \eqref{eq:ALz_mu}, 
one has 
\[
2B_{L,\mu}(z)=\sum_{j=1}^{n}\big[\partial_{j},J_{L,\mu}^{j}(z)\big]+A_{L,\mu}(z). 
\] 
\end{lemma}
\begin{proof}
The only nontrivial item to be established, invoking Proposition \ref{prop:Reformulation of B}
together with equations \eqref{eq:B_L=00003DJJJ+ALz}, \eqref{eq:JJJ},
and \eqref{eq:ALz}, is to establish that for $j\in\{1,\ldots,n\}$, 
\[
\big[\partial_{j},J_{L,\mu}^{j}(z)\big]=\left(1-\mu\Delta \right)^{-1}\big[\partial_{j},J_{L}^{j}(z)\big]\left(1-\mu\Delta \right)^{-1}.
\]
Recalling $\gamma_{j,n}\in\mathbb{C}^{2^{\hat n} \times 2^{\hat n}}$, $j\in\{1,\ldots,n\}$  
(cf.\ Remark \ref{rem:Eucl-Dirac-Algebar}), one observes that 
\begin{align*}
 & \left(1-\mu\Delta \right)^{-1}\big[\partial_{j},\tr_{2^{\hat n}d}\big(L\left(L^{*}L+z\right)^{-1}\gamma_{j,n}\big)\big]\left(1-\mu\Delta  \right)^{-1}\\
 & \quad =\left(1-\mu\Delta  \right)^{-1}\partial_{j}\tr_{2^{\hat n}d} \big(L\left(L^{*}L+z\right)^{-1}\gamma_{j,n}\big) \\
& \qquad 
 -\tr_{2^{\hat n}d} \big(L\left(L^{*}L+z\right)^{-1}\gamma_{j,n}\big) \partial_{j} 
 \left(1-\mu\Delta \right)^{-1}\\
 & \quad =\partial_{j}\left(1-\mu\Delta  \right)^{-1}\tr_{2^{\hat n}d} \big(L\left(L^{*}L+z\right)^{-1}\gamma_{j,n}\big) \left(1-\mu\Delta  \right)^{-1}\\
 & \qquad-\left(1-\mu\Delta \right)^{-1}\tr_{2^{\hat n}d} \big(L\left(L^{*}L+z\right)^{-1}\gamma_{j,n}\big)\left(1-\mu\Delta \right)^{-1}\partial_{j}\\
 & \quad =\big[\partial_{j},\left(1-\mu\Delta \right)^{-1}\tr_{2^{\hat n}d} 
 \big(L\left(L^{*}L+z\right)^{-1}\gamma_{j,n}\big)\left(1-\mu\Delta \right)^{-1}\big],
\end{align*}
yielding the assertion. 
\end{proof}

In contrast to the operator $J_{L}^j(z)$, the integral kernel for the regularized operator $J_{L,\mu}^j(z)$ satisfies the desired differentiability properties:

\begin{corollary}
\label{cor:diffble of reg green} Let $L= \cQ +\Phi$ be given by \eqref{eq:def_of_L(2)},
$z\in\rho\left(-L^{*}L\right)\cap\rho\left(-LL^{*}\right),$ with $\Re (z) > -1$, and suppose $\mu > 0$.
If $n\in\mathbb{N}$ is odd, then for all $j\in\{1,\ldots,n\}$, the
integral kernel of {$J_{L,\mu}^{j}(z)$ given by \eqref{eq:JJJ_mu}
is continuously differentiable.} 
\end{corollary}
\begin{proof}
We recall $R_{1+z}$ as given by \eqref{eq:resolvent_of_laplace}, $\cQ$
and $C$ given by \eqref{eq:def_of_Q2} and \eqref{eq:commutator=00003DC}, respectively,
as well as $\gamma_{j,n}$ as in Remark \ref{rem:Eucl-Dirac-Algebar}.
According to Proposition \ref{prop:concrete case}
for $\ell\in\mathbb{R}$, it suffices to observe that the operator 
\[
\left(1-\mu\Delta  \right)^{-1}\tr_{2^{\hatt n}d}\gamma_{j,n} \cQ \left(R_{1+z}C\right)^{n-2}R_{1+z}\left(1-\mu\Delta \right)^{-1}
\]
is continuous from $H^{\ell}(\mathbb{R}^{n})$ (see \eqref{eq:def_Hs})
to $H^{\ell+2(n-2)+2+4-1}(\mathbb{R}^{n})=H^{\ell+2n+1}(\mathbb{R}^{n}).$ 
Thus, by Corollary \ref{cor:integral kernels regularity}, the assertion
follows from $2n>n$. 
\end{proof}

Next, we turn to a variant of the first assertion in 
Lemma \ref{lem:some properties of the first two terms}. 

\begin{lemma}
\label{lem:reg vanishes on diagonal}Let $\mu>0$, $z\in\mathbb{C}$,
$\Re (z) > -1$, $R_{1+z}$ given by \eqref{eq:resolvent_of_laplace},
$C$ given by \eqref{eq:commutator=00003DC}, and $\cQ$ given by \eqref{eq:def_of_Q2}. 
Then for all $j\in\{1,2,3\}$, the integral kernel of 
\[
\left(1-\mu\Delta\right)^{-1}\tr_{2d}\left(\gamma_{j,3} \cQ R_{1+z}CR_{1+z}\right)\left(1-\mu\Delta \right)^{-1}
\]
vanishes on the diagonal, where $\gamma_{1,3},\gamma_{2,3},\gamma_{3,3}\in\mathbb{C}^{2\times2}$
are given as in Remark \ref{rem:Eucl-Dirac-Algebar} $($see also Appendix 
\ref{sec:Appendix:-the-Construction}$)$.
\end{lemma}
\begin{proof}
We denote the integral kernel under consideration
by $h_{j}$, $j\in\{1,2,3\}$. From Lemma \ref{lem:better diagonal representation},
one recalls,   
\begin{align*}
 & \left(1-\mu\Delta  \right)^{-1}\tr_{2d}\left(\gamma_{j,3} \cQ R_{1+z}CR_{1+z}\right)\left(1-\mu\Delta  \right)^{-1}\\
 & \quad =-\left(1-\mu\Delta  \right)^{-1}\tr_{2d}\left(\gamma_{j,3} \cQ R_{1+z} 
 \Phi \cQ R_{1+z}\right)\left(1-\mu\Delta  \right)^{-1}.
\end{align*}
With $\left(1-\mu\Delta \right)^{-1}= (1/\mu) 
\left((1/\mu) - \Delta \right)^{-1}$
one computes, 
\begin{align*}
h_{j}(x,x) & =-\frac{1}{\mu^{2}}\int_{\left(\mathbb{R}^{3}\right)^{3}}r_{1/\mu}(x-x_{1}) 
\tr_{2d} \bigg(\gamma_{j,3}\sum_{i_{1}=1}^{3}\gamma_{i_{1},3} (\partial_{i_{1}}r_{1+z})(x_{1}-x_{2})\Phi(x_{2})\\
 & \quad \times \sum_{i_{2}=1}^{3}\gamma_{i_{2},3}(\partial_{i_{2}}r_{1+z})(x_{2}-x_{3})\bigg)
 r_{1/\mu}(x_{3}-x) \, d^n x_1 d^n x_2  d^n x_3\\
 & =\frac{1}{\mu^{2}}\int_{\left(\mathbb{R}^{3}\right)^{3}}r_{1/\mu}(x-x_{1}) 
 \tr_{2d}\bigg( \gamma_{j,3}\sum_{i_{1}=1}^{3}\gamma_{i_{1},3}
 (\partial_{i_{1}}r_{1+z})(x_{2}-x_{1})\Phi(x_{2})\\
 & \quad \times \sum_{i_{2}=1}^{3}\gamma_{i_{2},3} 
 (\partial_{i_{2}}r_{1+z})(x_{2}-x_{3})\bigg) r_{1/\mu}(x_{3}-x) \, d^n x_1 d^n x_2  d^n x_3\\
 & =\frac{1}{\mu^{2}}\int_{\left(\mathbb{R}^{3}\right)^{3}}r_{1/\mu}(x_{1}-x)\tr_{2d}\bigg(\gamma_{j,3}\sum_{i_{1}=1}^{3}\gamma_{i_{1},3} (\partial_{i_{1}}r_{1+z})(x_{2}-x_{1})\Phi(x_{2})\\
 & \quad \times \sum_{i_{2}=1}^{3}\gamma_{i_{2},3} 
 (\partial_{i_{2}}r_{1+z})(x_{2}-x_{3})\bigg)r_{1/\mu}(x_{3}-x) \, d^n x_1 d^n x_2  d^n x_3, 
 \quad x\in\mathbb{R}^{n}. 
\end{align*}
The latter expression is symmetric in $x_{2}$ and $x_{3}$, by Fubini's
theorem. Hence, the assertion follows as in Lemma
\ref{lem:better diagonal representation} with the help of Corollary
\ref{cor:comp_Dirac_alg_trace-1}.
\end{proof}

Now we are in position to prove the trace theorem for dimension $n=3$. Of course the principal strategy for the proof is similar to the one for dimensions $n\geq 5$ and, thus, need not be repeated.

\begin{theorem}
\label{thm:Witten-reg_n3} Let $n=3$, $L= \cQ +\Phi$ given by \eqref{eq:def_of_L(2)}, 
and $\chi_\Lambda$ given by \eqref{eq:def_of_chi}. 
Then for all $z\in\mathbb{C}$ with $z\in\rho\left(-LL^{*}\right)\cap\rho\left(-L^{*}L\right)$, 
and $B_{L}(z)$ given \eqref{eq:def_of_BL(z)2}, $\chi_\Lambda B_L(z)$ is trace class for all $\Lambda>0$. The limit $f(\cdot)\coloneqq \lim_{\Lambda\to\infty} \tr(\chi_\Lambda B_L(\cdot))$ 
exists in the compact open topology and the formula
\begin{align}
\begin{split} 
f(z) & =\frac{i}{16\pi} (1+z)^{-3/2}\lim_{\Lambda\to\infty}\sum_{j,i_{1},i_{2} = 1}^3 \epsilon_{ji_{1}i_{2}}\frac{1}{\Lambda}    \\
& \quad \, \times \int_{\Lambda S^{2}}\tr\big(\Phi(x)\left(\partial_{i_{1}}\Phi\right)(x)\left(\partial_{i_{2}}\Phi\right)(x)\big)x_{j}\, d^{n-1} \sigma(x)     \lb{f(z)-n=3}   \\
\end{split} 
\end{align}
holds, where $\epsilon_{ji_{1}i_{2}}$ denotes the $\epsilon$-symbol as
in Proposition \ref{prop:comp_of_Dirac_Alge}.   
\end{theorem}
\begin{proof} Let $\Lambda>0$, $\mu>0$. Denote the integral kernels of $A_{L}(z)$ and $\sum_{j}\big[\partial_{j},J_{L}^{j}(z)\big]$ by $\mathbb{A}_{L}$ and $\mathbb{J}_{L}$, respectively, and correspondingly
for $A_{L,\mu}(z)$ and $\sum_{j}\big[\partial_{j},J_{L,\mu}^{j}(z)\big]$,
where the respective operators are given by \eqref{eq:ALz}, \eqref{eq:JJJ},
\eqref{eq:ALz_mu}, and \eqref{eq:JJJ_mu}. One notes that by Lemma \ref{lem:pointwise convergence of regularized Green}, 
$\mathbb{A}_{L,\mu}\to\mathbb{A}_{L}$ and $\mathbb{J}_{L,\mu}\to\mathbb{J}_{L}$
pointwise as $\mu\to0$. One recalls from Proposition \ref{prop:Reformulation of B}
and Theorem \ref{thm:trisbounded} together with Proposition \ref{cor:Comp-of-trace}
that (similarly to the case $n=5$), 
\begin{align*}
2\tr (\chi_\Lambda B_{L}(z)) & =\int_{B(0,\Lambda)}\left(\mathbb{A}_{L}+\mathbb{J}_{L}\right)(x,x)\, d^n x,
\end{align*}
and, as $\mathbb{A}_{L}$ and the integral kernel of $B_{L}(z)$
are continuous, so is $\mathbb{J}_{L}$. Hence, by Lemma \ref{lem:pointwise convergence of regularized Green}
and using $\mathbb{A}_{L}(x,x)=0$ (see Lemma \ref{lem:j,A,Green}),
one obtains 
\begin{align*}
2\tr (\chi_\Lambda B_{L}(z)) & =\int_{B(0,\Lambda)}\mathbb{A}_{L}(x,x)+\mathbb{J}_{L}(x,x)\, d^n x\\
 & =\int_{B(0,\Lambda)}\lim_{\mu\to0}\mathbb{J}_{L,\mu}(x,x)\, d^n x\\
 & =\lim_{\mu\to0}\int_{B(0,\Lambda)}\mathbb{J}_{L,\mu}(x,x)\, d^n x,
\end{align*}
where the last equality follows from the fact that the family of integral kernels of 
\[
\{2 B_{L,\mu}(z)-A_{L,\mu}(z)\}_{\mu>0}
\]
is locally uniformly bounded: To prove the latter assertion, we note that 
due to Corollary \ref{cor:integral kernels regularity}, 
$\{2 B_{L,\mu}(z)-A_{L,\mu}(z)\}_{\mu>0}$
defines a uniformly bounded family of continuous linear operators
from $H^{\ell}(\mathbb{R}^{n})$ (see \eqref{eq:def_Hs}) to $H^{\ell+2n-1}(\mathbb{R}^{n})$, 
$\ell\in\mathbb{R}$. Indeed, this follows from the representation
in Lemma \ref{lem:almost_Neumann_series} together with Proposition
\ref{prop:concrete case} and the fact that for all $s\in\mathbb{R}$, 
$\left(1-\mu\Delta\right)^{-1}\to I$ strongly in $H^{s}(\mathbb{R}^{n})$. Next, we 
denote 
\[
\mathbb{K}_{L,\mu}\coloneqq\big\{x\mapsto g_{L,\mu}^{j}(z)(x,x)\big\}_{j\in\{1,2,3\}},
\]
where $g_{L,\mu}^{j}(z)$ is the integral kernel of $J_{L,\mu}^{j}(z)$, 
$j\in\{1,2,3\}$, and $\mathbb{K}_{L}$ that for 
\[
\big\{J_{L}^{j}(z)-2\tr_{2d} (\gamma_{j,3} \cQ R_{1+z}CR_{1+z})\big\}_{j\in\{1,2,3\}}.
\]
Invoking Lemmas \ref{lem:reg vanishes on diagonal} and \ref{lem:pointwise convergence of regularized Green}, and hence the fact that $\{x\mapsto\mathbb{K}_{L,\mu}(x)\}_{\mu>0}$
is locally uniformly bounded, one obtains 
\begin{align*}
\lim_{\mu\to0}\int_{B(0,\Lambda)}\mathbb{J}_{L,\mu}(x,x)\, d^n x & =\lim_{\mu\to0}\int_{\Lambda S^{2}}\left( \mathbb{K}_{L,\mu}(x),\frac{x}{\Lambda}\right) \, d^{n-1} \sigma(x)\\
 & =\int_{\Lambda S^{2}}\lim_{\mu\to0}\left( \mathbb{K}_{L,\mu}(x),\frac{x}{\Lambda}\right) \, d^{n-1} \sigma(x)\\
 & =\int_{\Lambda S^{2}}\left(\mathbb{K}_{L}(x),\frac{x}{\Lambda}\right) \, d^{n-1} \sigma(x). 
\end{align*}
As in the case $n\geq5$, one computes with the help of Corollary
\ref{cor:final aymptotics} that for $z\in\Sigma_{z_0,\theta}$, see \eqref{def:Sigma0}, for some fixed $\theta\in (0,\pi/2)$ and $z_0\in \mathbb{R}$ sufficiently large, the limit $\Lambda\to\infty$ actually coincides with $\mathbb{K}_{L}$ replaced
by the vector of integral kernels of 
$\big\{2\tr_{2d}\big(\gamma_{j,3}\Phi C^{2}R_{1+z}^{3}\big)\big\}_{j\in\{1,2,3\}}$ (employing 
analogous arguments using Lemmas \ref{lem:boundedness_of_leading_new_witten}, \ref{lem:boundedness_of_rest_new_witten}, and \ref{lem:Laurent_series_argument}). 
Hence one can compute this expression explicitly with the
help of Proposition \ref{prop:formula for index}, ending up with \eqref{f(z)-n=3}. 
\end{proof}

\newpage

\section{The Index Theorem and Some Consequences}\label{sec:The Index theorem}

Putting the results of the Sections \ref{sec:Functional-Analytic-Preliminarie}
and \ref{sec:The-Derivation-of-trace-f} together, we arrive at the
following theorem:

\begin{theorem}
\label{thm:Fredholm-index} Let $n\in\mathbb{N}_{\geq3}$ odd, $d\in\mathbb{N}$,
$\Phi\colon\mathbb{R}^{n}\to\mathbb{C}^{d\times d}$ admissible $($see
Definition \ref{def:phi_admissible}$)$. 
Then the operator $L= \cQ +\Phi$ given by \eqref{eq:def_of_L(2)} is 
Fredholm and  
\begin{align}
\begin{split}
 \ind (L) &= \left(\frac{i}{8\pi}\right)^{(n-1)/2}\frac{1}{\left[(n-1)/2\right]!}
 \lim_{\Lambda \to\infty}\frac{1}{2 \Lambda }\sum_{j,i_{1},\ldots,i_{n-1} = 1}^n \epsilon_{ji_{1}\ldots i_{n-1}}  \label{eq:Fredhom_index} \\
 & \quad \; \times \int_{\Lambda S^{n-1}}\tr (\Phi(x) (\partial_{i_{1}} \Phi)(x)\ldots
(\partial_{i_{n-1}} \Phi)(x))x_{j}\, d^{n-1} \sigma(x),   
\end{split} 
\end{align} 
where $\epsilon_{ji_{1}\ldots i_{n-1}}$ denotes the $\epsilon$-symbol
as introduced in Proposition \ref{prop:comp_of_Dirac_Alge}. 
\end{theorem}
\begin{proof}
Appealing to Theorem \ref{thm:index with Witten} and 
Theorem \ref{thm:Witten_reg_n5} (or \ref{thm:Witten-reg_n3}), we have $f(0)=\ind (L)$, with $f$ 
from Theorem \ref{thm:Witten_reg_n5}.
\end{proof}

In Corollary \ref{c:adind0} at  the end of this section we will show that actually, $ \ind (L) = 0$ 
for admissible $\Phi$.

Next, we indicate how the index theorem obtained can be generalized
to potentials $\Phi$ belonging only to $C_{b}^{2}\big(\mathbb{R}^{n};\mathbb{C}^{d\times d}\big)$
satisfying $\left|\Phi(x)\right|\geq c$ for all $x\in\mathbb{R}^{n}\backslash B(0,R)$
for some $R>0$, $c>0$. More precisely, we will prove the following theorem later in 
Section \ref{s12} in the case where $\Phi$ is $C^{\infty}$ and in full generality in 
Section \ref{s13}: 

\begin{theorem}
\label{thm:Perturbation} Let $n\in\mathbb{N}_{\geq3}$ odd, $d\in\mathbb{N}$,
$\Phi\in C_{b}^{2}\big(\mathbb{R}^{n};\mathbb{C}^{d\times d}\big)$. Assume
the following properties
\[
\Phi(x)=\Phi(x)^{*},  \quad x\in\mathbb{R}^{n},
\]
there exists $c>0$, $R\geq0$ such that $\left|\Phi(x)\right|\geq c$, 
$x\in\mathbb{R}^{n}\backslash B(0,R)$, and that there is $\varepsilon> 1/2$
such that for all $\alpha\in\mathbb{N}_{0}^{n}$, $\left|\alpha\right|<3$,
there is $\kappa>0$ such that 
\[
\|(\partial^{\alpha}\Phi)(x)\|\leq \begin{cases}
\kappa (1+|x|)^{-1}, & |\alpha|=1,\\[1mm] 
\kappa (1+ |x|)^{-1-\epsilon}, & |\alpha|=2, 
\end{cases}\quad x\in\mathbb{R}^{n}.
\]
We recall $\mathcal{Q}=\sum_{j=1}^{n}\gamma_{j,n}\partial_{j}$, with 
$\gamma_{j,n}\in\mathbb{C}^{2^{\hat n} \times 2^{\hat n}}$, $j\in\{1,\ldots,n\}$, 
given in \eqref{eq:def_of_Q2} or Theorem \ref{thm:L_is_closed}. 
Then the operator $L\coloneqq \cQ +\Phi$ considered in $L^{2}(\mathbb{R}^{n})^{2^{\hat n}d}$
is a Fredholm operator and  
\begin{align}
 \ind (L) &=\left(\frac{i}{8\pi}\right)^{(n-1)/2}\frac{1}{\left[(n-1)/2\right]!} 
 \lim_{\Lambda \to\infty}\frac{1}{2 \Lambda }\sum_{j,i_{1},\ldots,i_{n-1} = 1}^n \epsilon_{ji_{1}\ldots i_{n-1}}     \lb{iL} \\ 
 & \quad \; \times \int_{\Lambda S^{n-1}}\tr (U(x) (\partial_{i_{1}} U)(x)\ldots 
 (\partial_{i_{n-1}} U)(x))x_{j}\, d^{n-1} \sigma(x),    \no 
\end{align}
where
\[
U(x) = |\Phi(x)|^{-1} \Phi(x) = \sgn(\Phi(x)), \quad x \in \bbR. 
\]
\end{theorem}

While in this manuscript we focus on the functional analytic proof of Callias' index formula \eqref{iL}, we refer to the discussion by Bott and Seeley \cite{BS78} for its underlying topological setting (homotopy invariants, etc.).  

In Theorem \ref{thm:Perturbation}, there are two main difficulties to cope with: on the one hand -- in contrast to the situation in Theorem \ref{thm:Fredholm-index} -- the potential is only assumed to be invertible on the complement of large balls, on the other hand the potential is only $C^2$. We will address the second case later on, and concern ourselves with the invertibility issue first. However, before providing the proof of Theorem \ref{thm:Perturbation} in these more general cases, we give a motivating fact underlining the need for Theorem \ref{thm:Perturbation}. In particular, in Theorem \ref{thm:there is no Cinfty unitary-self-adjoint potential}, we show that a particular class of potentials cannot be treated with the help of Theorem \ref{thm:Fredholm-index}. The main problem preventing the applicability of Theorem \ref{thm:Fredholm-index} is the everywhere invertibility assumed in Definition \ref{def:phi_admissible}. 

We note that the special case $n=3$ in connection with Yang--Mills--Higgs fields and monopoles has been discussed in detail in \cite[Sect.\ II.5]{JT80} and \cite[Sect.\ VIII.4]{Na91}. 

Before turning to Theorem \ref{thm:there is no Cinfty unitary-self-adjoint potential}, we shall provide a result that studies the sign of an operator. This study is needed, as the formula in Theorem \ref{thm:Perturbation} involves $x\mapsto \Phi(x)/|\Phi(x)|=\sgn(\Phi(x))$.

\begin{theorem}
\label{thm:properties of sign function} Let $\cH$ be a Hilbert space,
$A\in \cB(\cH)$, $\Re\left(A^{2}\right)\geq c$ for some $c>0$. 
Then the integral $($see, e.g., \cite[Ch.~5, equation (5.3)]{Hi08}$)$, 
\[
\sgn(A)\coloneqq\frac{2}{\pi}A\int_{0}^{\infty}\left(t^{2}+A^{2}\right)^{-1}dt
\]
converges in $\cB(\cH)$ and $\sgn(\cdot)$ is analytic on 
$B\big(A, (\|A\|^{2}+c)^{1/2} - \|A\|)\big)$.   
Moreover, if in addition, $A=A^{*}$ then 
\[
\sgn(A)=A |A|^{-1},  
\]
 and $\sgn(A)$ is unitary.
 \end{theorem}
\begin{proof}
From $\Re\left(t^{2}+A^{2}\right)\geq t^{2}+c$ it follows that 
$\big\| \left(t^{2}+A^{2}\right)^{-1}\big\|\leq (t^{2}+c)^{-1}$
and, thus, the integral converges in operator norm. In order to prove
analyticity, we first show that given $B\in \cB(\cH)$ with $\Re(B)\geq c$
the function $T\mapsto\int_{0}^{\infty}\left(t^{2}+B+T\right)^{-1}dt$
is analytic at $0$ with convergence radius at least $c$. Hence,
let $B\in \cB(\cH)$ with $\Re (B) \geq c.$ Then 
$\big\Vert \left(t^{2}+B\right)^{-1}\big\Vert \leq (t^{2}+c)^{-1}\leq c^{-1}$
for all $t\in\mathbb{R}$. If $T\in \cB(\cH)$ with $\|T\|\leq \theta c$
for some $0<\theta<1$, then $\left\Vert \left(t^{2}+B\right)^{-1}T\right\Vert \leq \theta$
for all $t\in\mathbb{R}$ and thus, 
\begin{align*}
\int_{0}^{\infty}\left(t^{2}+\left(B+T\right)\right)^{-1}dt & =\int_{0}^{\infty}\left(1+\left(t^{2}+B\right)^{-1}T\right)^{-1}\left(t^{2}+B\right)^{-1}dt\\
 & =\int_{0}^{\infty}\sum_{k=0}^{\infty}\left(-\left(t^{2}+B\right)^{-1}T\right)^{k}\left(t^{2}+B\right)^{-1}dt\\
 & =\sum_{k=0}^{\infty}\int_{0}^{\infty}\left(-\left(t^{2}+B\right)^{-1}T\right)^{k}\left(t^{2}+B\right)^{-1}dt.
\end{align*}
One observes that $c_{k}\colon \cB(\cH)^{k}\to \cB(\cH)$ given by 
\begin{equation*} 
c_{k}(T,\ldots,T)\coloneqq\int_{0}^{\infty}\left(-\left(t^{2}+B\right)^{-1}T\right)^{k}\left(t^{2}+B\right)^{-1}dt, 
\end{equation*} 
is a bounded $k$-linear form with bound $c^{-k}\pi/(2\sqrt{c})$.
Indeed, for the contractions $T_{1},\ldots,T_{k}\in \cB(\cH)$ one estimates 
\begin{align*}
& \bigg\Vert \int_{0}^{\infty}\prod_{j=1}^{k}\big[-(t^{2}+B)^{-1}T_{j}\big](t^{2}+B)^{-1} \, dt\bigg\Vert     \\ 
& \quad \leq \int_{0}^{\infty}\prod_{j=1}^{k}\big\Vert \big[-(t^{2}+B)^{-1}T_{j}\big]
(t^{2}+B)^{-1}\big\Vert \, dt\\
 & \quad \leq \int_{0}^{\infty}\big\Vert (t^{2}+B)^{-k}\big\Vert \big\Vert (t^{2}+B)^{-1}\big\Vert \, dt\\
 & \quad \leq \left(\frac{1}{c}\right)^{k}\int_{0}^{\infty}\frac{1}{t^{2}+c} \, dt=\left(\frac{1}{c}\right)^{k}\frac{\pi}{2\sqrt{c}}.
\end{align*}
In particular, the power series has convergence radius at least $c$.
It follows that $T\mapsto\int_{0}^{\infty}\left(t^{2}+T\right)^{-1}dt$
is analytic about $A^{2}$ with convergence radius $c.$ If $T\in \cB(\cH)$
with $\|T\|< \big((\|A\|^{2}+c)^{1/2}-\|A\|\big)$, then 
\begin{align*}
\left\Vert (A+T)^{2}-A^{2}\right\Vert  & <2\|A\|\|T\|+\|T\|^{2}\\
 & \leq 2\|A\|\big((\|A\|^{2}+c)^{1/2}-\|A\|\big)+\big((\|A\|^{2}+c)^{1/2}-\|A\|\big)^{2}\\
 & =c.
\end{align*}
Hence, the map $T\mapsto\int_{0}^{\infty}\left(t^{2}+T^{2}\right)^{-1}dt$
is analytic about $A$ with convergence radius at least $\big((\|A\|^{2}+c)^{1/2}-\|A\|\big)$. 

The equality and unitarity now follow from the functional calculus for
self-adjoint opertors and the respective equality for numbers.
\end{proof}

The next fact provides a more detailed account on the behavior of $x\mapsto \sgn(\Phi(x))$ for smooth $\Phi$. We note that the following result has been asserted implicitly in a modified form in \cite[last paragraph on p.~226]{Ca78}.

\begin{lemma}\label{lem:almost admissible} Let $n,d\in\mathbb{N}_{\geq 1}$, 
$\Phi\in C_b^\infty\big(\mathbb{R}^n;\mathbb{C}^{d\times d}\big)$ pointwise self-adjoint, 
$c,R\geq0$, $c\neq 0$. Assume that for all $x\in \mathbb{R}^n\backslash  B(0,R)$, $|\Phi(x)|\geq c$. Let $\tau>0$. Then there exists $U \in C^\infty_b\big(\mathbb{R}^n;\mathbb{C}^{d\times d}\big)$ pointwise self-adjoint, and a function 
$u \in C^\infty_b(\mathbb{R}^n;\mathbb{R}_{\geq 0})$ with $0\leq u \leq  1$, $u_{\mathbb{R}^n\backslash  B(0,\tau)}=1$, such that 
\[
  U(x)=\sgn(\Phi(x)), \; x\in \mathbb{R}^n\backslash  B(0,R) \, \text{ and } \, 
  U(x)^2 = u(x)I_d, \; x\in \mathbb{R}^n.
\]
Moreover, for all $\beta\in\mathbb{N}_0^n$, $\beta\neq 0$, there exists $\kappa>0$ such that for all $x\in \mathbb{R}^n\backslash  B(0,R)$, 
\[
  \|\partial^\beta U(x)\|\leq \kappa \sum_{\alpha\in\mathbb{N}_0^n, \gamma\in \mathbb{N}, |\alpha|\gamma=|\beta|}\|\partial^\alpha \Phi(x)\|^{\gamma}. 
\]
\end{lemma}

\begin{remark} \label{rem:functions of assumptions satisfy assumptions}
$(i)$ We note that the function $U$ constructed in Lemma \ref{lem:almost admissible} attains values in the set of unitary matrices (on $\mathbb{R}^n\backslash  B(0,R)$). Indeed, this follows from Theorem \ref{thm:properties of sign function}. \\[1mm]
$(ii)$ In the situation of Lemma \ref{lem:almost admissible}, assume, in addition, that $\Phi$ satisfies the following estimates: For some
$\epsilon> 1/2$ and for $\alpha\in\mathbb{N}_{0}^{n}$, there is a constant $\kappa_1>0$
such that 
\[
\|(\partial^{\alpha}\Phi)(x)\|\leq \begin{cases}
\kappa_1 (1+|x|)^{-1}, & |\alpha|=1,\\[1mm]
\kappa_1 (1+|x|)^{-1-\epsilon}, & |\alpha|\geq 2,
\end{cases} \quad x\in\mathbb{R}^{n}.
\]
 Then $U$ constructed in Lemma \ref{lem:almost admissible} satisfies analogous estimates: For 
 $\alpha\in \mathbb{N}_0^n$ there exists $\kappa_2>0$ such that
\[
\|(\partial^{\alpha} U)(x)\|\leq \begin{cases}
\kappa_2 (1+|x|)^{-1}, & |\alpha|=1,\\[1mm]
\kappa_2 (1+|x|)^{-1-\epsilon}, & |\alpha|\geq 2.
\end{cases}
\]
In particular, if $\Phi$ is admissible (see Definition \ref{def:phi_admissible}), then so is 
$U=\sgn (\Phi)$. \hfill $\diamond$
\end{remark}

\begin{proof}[Proof of Lemma \ref{lem:almost admissible}] One observes that 
$x\mapsto \sgn (\Phi(x))$ is $C^\infty_b$ on $\mathbb{R}^n\backslash  B(0,R')$ for some $0<R'<R$, by Theorem \ref{thm:properties of sign function}. Moreover, for $j\in\{1,\ldots,n\}$, 
\[
(\partial_{j} U)(x)= ({\sgn}'(\Phi(x)))(\partial_{j}\Phi)(x).
\]
Thus, 
\[
(\partial_{k}\partial_{j} U)(x)= {\sgn}''(\Phi(x))(\partial_{k}\Phi)(x) (\partial_{j}\Phi)(x) 
+ {\sgn}'(\Phi(x))(\partial_{k}\partial_{j}\Phi)(x).
\]
Continuing in this manner, we obtain the estimates for the derivatives,
once noticing that $x\mapsto\sgn^{(k)}\left(\Phi(x)\right)$ is bounded
for all $k\in\mathbb{N}$. Indeed, by the boundedness
of $\Phi$ and since $\left|\Phi(x)\right|\geq c$ for all $x\in\mathbb{R}^{n}\backslash  B(0,R)$,
the set $\{\Phi(x) \, | \, x\in\mathbb{R}^{n}\backslash B(0,R)\}\subseteq\mathbb{C}^{d\times d}$ 
is relatively compact and its closure is contained in the domain of
analyticity of $\sgn(\cdot)$. Hence, $x\mapsto\sgn^{(k)}\left(\Phi(x)\right)$
is indeed bounded.   

Next, let $\eta\in C_b^\infty(\mathbb{R}^n)$ with 
\[
   \eta(x)\begin{cases}
            = R',& |x|\leq  R',\\
            \in [R',R],& R'<|x|<R,\\
            =|x|, & |x|\geq R.
           \end{cases}
\]
Then $\alpha \colon \mathbb{R}^n\backslash \{0\} \to \mathbb{R}^n,x\mapsto \eta(|x|)\frac{x}{|x|}$ is $C^\infty$ and $\alpha(x)=x$ for all $|x|\geq R$. Let, in addition, $\phi\colon \mathbb{R}^n \to \mathbb{R}_{\geq 0}$ be a $C^\infty$-function such that $\phi(x)=1$ for $x\in \mathbb{R}^n\backslash  B(0,\tau)$ and with $0\leq  \phi \leq  1$ and $\phi(x)=0$ on $B(0,\tau/2)$. Then a suitable choice for $U$ is
\[
  x\mapsto \phi(x)\sgn(\Phi(\alpha(x))).\tag*{\qedhere}
\] 
\end{proof}

One might wonder, whether the function $u$ vanishing at the origin in Lemma \ref{lem:almost admissible} is really needed. In fact, if it was possible for any arbitrarily differentiable potential $\Phi$ discussed in Theorem \ref{thm:Perturbation}, to choose $u$ in Lemma \ref{lem:almost admissible} being $1$ also at the origin, the only nontrivial assertion of Theorem \ref{thm:Perturbation} would be the differentiability issue. However, the next example indicates that Theorem \ref{thm:Perturbation} has a nontrivial application. 

\begin{theorem}\label{thm:there is no Cinfty unitary-self-adjoint potential} Consider the function
$\Phi \colon \mathbb{R}^3 \to \mathbb{C}^{2\times 2}$ 
such that 
\[
 \Phi(x) = \sum_{j=1}^3 \sigma_j \frac{x_j}{|x|}, \quad  |x|\geq 1,  
\]
as in Example \ref{exa:standard_example_0}. Then there is no $U \in C^\infty\big(\mathbb{R}^3;\mathbb{C}^{2\times 2}\big)$ with the property that $U(x)=\Phi(x)$ for all $x\in \mathbb{R}^n$, $|x|\geq 1$, $U(x)=U(x)^*$ and for some $c>0$, $|U(x)|\geq c$, $x\in \mathbb{R}^n$. 
\end{theorem}
\begin{proof} 
We will proceed by contradiction and assume the existence of such a $U$. By Lemma \ref{lem:almost admissible} (and Remark \ref{rem:functions of assumptions satisfy assumptions}\,$(i)$), we may assume without loss of generality that $U$ assumes values in the self-adjoint unitary operators in $\mathbb{C}^{2\times 2}$. The latter are of the form
\[
   W=\begin{pmatrix}
      a & b+id \\ b-id & c 
   \end{pmatrix}, \quad a,b,c,d\in \mathbb{R}
\]
with $W^*W=I_2$. From the latter equation, one reads off
\begin{align*}
   1 &=a^2+b^2+d^2,  \\
   0 &=(a+c)b,  \\
   0 &=(a+c)d,  \\
   1 &=c^2+b^2+d^2.
\end{align*}
Hence, either $a\neq -c$, which implies $b=d=0$ and $a=c=\pm 1$, or $a=-c$ with $a^2+b^2+d^2=1$. Note that, in the latter case, we have $W=a\sigma_1+b\sigma_2+d\sigma_3$ and $\det (W) = - 1$. Hence, since $U$ is pointwise invertible everywhere, and 
$\det (\pm I_2)=1$, by the intermediate value theorem, one infers that 
\[ 
U[\mathbb{R}^3]\subseteq \{ a\sigma_1+b\sigma_2+c\sigma_3 \, | \,  a,b,c\in \mathbb{R}, \, a^2+b^2+c^2=1\}\eqqcolon \mathcal{U} 
\]
 Identifying $\mathcal{U}$ with $S^2$ and using $U|_{S^{2}}=\identity_{S^2}$, one observes 
 that $U$ is a retraction of $B(0,1)$ for $S^2$, which is a contradiction.
 We provide some details for the latter claim.
Assume there exists a continuous map $f\colon \overline{B(0,1)}\subset \mathbb{R}^3\to S^2$ with the property  $f(x)=x$ for all $x\in S^2$. Denoting the identity on $\overline{B(0,1)}$ by $I_{\overline{B(0,1)}}$, one considers the homotopy $H$ of $f$ and $I_{\overline{B(0,1)}}$ given by
\[
 H(\lambda,x)\coloneqq \lambda f(x)+(1-\lambda)I_{\overline{B(0,1)}}(x),\quad \lambda \in [0,1],
 \; x\in \overline{B(0,1)}.
\] 
In the following, we denote by $\textrm{deg}(g,z_0)$ Brouwer's degree of a function $g\colon \overline{B(0,1)}\to \mathbb{R}^3$ in the point $z_0 \in \mathbb{R}^3\backslash  g[S^2]$. One observes that 
$0\in \mathbb{R}^3\backslash  H(\lambda, S^2)=\mathbb{R}^3\backslash  S^2$ for all $\lambda\in [0,1]$, by the hypotheses on $f$. Hence, by homotopy invariance of Brouwer's degree, one gets, using $0\in I_{\overline{B(0,1)}}[\overline{B(0,1)}]$ and $0\notin f[\overline{B(0,1)}]=S^2$,
\[
   1 = \textrm{deg}(I_{\overline{B(0,1)}},0)=\textrm{deg}(H(0,\cdot),0)=\textrm{deg}(H(1,\cdot))=\textrm{deg}(f,0)=0,
\]
a contradiction. 
\end{proof}

While we decided to provide an explicit proof of 
Theorem \ref{thm:there is no Cinfty unitary-self-adjoint potential}, it should be mentioned 
that is is a special case of ``Brouwer's no retraction theorem'' (see, e.g., \cite[Theorem~3.12]{Cr78}): 
There is no continuous map $f:\ol{B(0,1)} \to S^{n-1}$ that is the identity on $S^{n-1}$. (Here 
$\ol{B(0,1)}$ denotes the closed unit ball in $\bbR^n$, $n \in \bbN$.) 

In the remainder of this section, we study the index formula \eqref{iL} in more detail. More precisely, we will show an invariance principle which will lead to a proof of Corollary \ref{c:adind0}, which shows that for admissible potentials $\Phi$, the index of $\mathcal{Q}+\Phi$ vanishes, reproducing \cite[Theorem~5.2]{Ra08} in our context.

Let $n,d\in\mathbb{N}$, $\cU\subseteq \mathbb{R}^n$ open, 
$\Upsilon\in C^1\big(\cU;\mathbb{C}^{d\times d}\big)$. For $x\in \cU$ we introduce the expression
\begin{equation}\label{e:Mphi}
   M_\Upsilon(x)\coloneqq \sum_{i_1,\ldots,i_n=1}^n \epsilon_{i_1\cdots i_n} \tr \big(\partial_{i_1}\Upsilon(x)\cdots\partial_{i_n}\Upsilon(x)\big),
\end{equation}
where $\epsilon_{i_1\cdots i_n}$ denotes the totally anti-symmetric symbol in $n$ coordinates.

\begin{remark}\label{r:mphi}
 The relationship of the index formula for potentials $\Phi$ as in Theorem \ref{thm:Perturbation} and the function defined in \eqref{e:Mphi} is as follows: Let $U$ be $C^2$-smooth with $U=\sgn(\Phi)$ on the complement of a sufficiently large ball. For $\Lambda>0$, one computes with the help of Gauss' divergence theorem
 \begin{align*}
 &\frac{1}{\Lambda}\sum_{i_{1},\ldots,i_{n} = 1}^n \epsilon_{i_{1}\ldots i_{n}}
\int_{\Lambda S^{n-1}}\tr (U(x) (\partial_{i_{1}} U)(x)\ldots 
(\partial_{i_{n-1}}\ U)(x))x_{i_{n}}\, d^{n-1} \sigma(x)      \\
& \quad =\sum_{i_{1},\ldots,i_{n} = 1}^n \epsilon_{i_{1}\ldots i_{n}}
\int_{B(0,\Lambda)}\tr ((\partial_{i_{1}}\ U)(x)\ldots 
(\partial_{i_{n}} U)(x))\, d^n x     \\ 
& \quad = \int_{B(0,\Lambda)} M_{U}(x)d^n x.
 \end{align*}
 Hence, the index formula for the operator $L=\mathcal{Q}+\Phi$ discussed in Theorem \ref{thm:Perturbation} may be rewritten as follows
 \begin{equation}\label{e:indM}
   \ind(L)=\left(\frac{i}{8\pi}\right)^{(n-1)/2}\frac{1}{\left[(n-1)/2\right]!} \lim_{\Lambda\to\infty} 
   \frac{1}{2} \int_{B(0,\Lambda)} M_{U}(x) \, d^n x.
 \end{equation}
 \hfill $\diamond$
\end{remark}

\begin{definition}[Transformations of constant orientation] \lb{d10.8} Let $n\in \mathbb{R}^n$, 
$\cU\subseteq \mathbb{R}^n$ open, dense. We say that $T\colon \cU\to \mathbb{R}^n$ is a \emph{transformation of constant orientation}, if the following properties $(i)$--$(iii)$ are satisfied: \\[1mm] 
$(i)$ $T$ is continuously differentiable and injective. \\[1mm] 
$(ii)$ $T[\cU]$ is dense in $\mathbb{R}^n$. \\[1mm] 
$(iii)$ The function $\cU\ni x\mapsto \sgn (\det (T'(x)))$ is either identically $1$ or $-1$. We 
define \hspace*{7mm} $\sgn(T)\coloneqq \sgn (\det (T'(x)))$ for some $($and hence for all\,$)$ 
$x\in \cU$. 
\end{definition}

The sought after invariance principle then reads as follows:

\begin{theorem}\label{t:inv} Let $n\in\mathbb{N}_{\geq3}$ odd, $d\in\mathbb{N}$,
$\Phi\in C_{b}^{2}\big(\mathbb{R}^{n};\mathbb{C}^{d\times d}\big)$. Assume
the following properties: 
\[
\Phi(x)=\Phi(x)^{*},  \quad x\in\mathbb{R}^{n},
\]
there exists $c>0$, $R\geq0$ such that $\left|\Phi(x)\right|\geq c$
for all $x\in\mathbb{R}^{n}\backslash B(0,R)$, and that there is $\varepsilon> 1/2$
such that for all $\alpha\in\mathbb{N}_{0}^{n}$, $\left|\alpha\right|<3$,
there is $\kappa>0$ such that
\[
\|(\partial^{\alpha}\Phi)(x)\|\leq \begin{cases}
\kappa (1+|x|)^{-1}, & |\alpha|=1,\\[1mm]
\kappa (1+ |x|)^{-1-\epsilon}, & |\alpha|=2,
\end{cases}\quad x\in\mathbb{R}^{n}.
\]
We recall $\mathcal{Q}=\sum_{j=1}^{n}\gamma_{j,n}\partial_{j}$, with
$\gamma_{j,n}\in\mathbb{C}^{2^{\hat n} \times 2^{\hat n}}$, $j\in\{1,\ldots,n\}$,
given in \eqref{eq:def_of_Q2} or in Theorem \ref{thm:L_is_closed}. In addition, 
let $T\colon \cU\subseteq \mathbb{R}^n\to\mathbb{R}^n$ $($with $\cU$ as in 
Definition \ref{d10.8}$)$ be a transformation of constant orientation. Assume that 
$\Phi_T\coloneqq \overline{\Phi\circ T}$ 
$($the closure of the mapping $\Phi\circ T$$)$ satisfies the assumptions imposed on $\Phi$. 
Then $L_1=\mathcal{Q}+\Phi$ and $L_2=\mathcal{Q}+\Phi_T$ are Fredholm and
\[
  \ind(L_1)=\sgn(T)\ind(L_2).
\]
\end{theorem}

Before proving Theorem \ref{t:inv} we need a chain rule for the function defined in \eqref{e:Mphi}.

\begin{lemma}\label{l:chain} Let $n,d\in \mathbb{N}$, $\cU\subseteq \mathbb{R}^n$ open, 
$\Phi\in C^1\big(\mathbb{R}^n;\mathbb{C}^{d\times d}\big)$, $T\in C^1(\cU; \mathbb{R}^n)$.  Then, 
\[
   M_{\Phi\circ T}(x)= M_\Phi (T(x))\det(T'(x)), \quad x\in \cU.
\]
\end{lemma}
\begin{proof}
 One recalls that for an $n\times n$-matrix $A=(a_{ij})_{i,j\in \{1,\ldots,n\}} \in \mathbb{C}^{n\times n}$, its determinant may be computed as follows
 \[
    \det (A) = \sum_{i_1,\ldots,i_n=1}^n \epsilon_{i_1\cdots i_n} a_{i_11}\cdots a_{i_nn}.
 \]
 Consequently, for $k_1,\ldots,k_n\in \{1,\ldots,n\}$, one gets
 \[
    \epsilon_{k_1\cdots k_n}\det (A)= \sum_{i_1,\ldots,i_n=1}^n \epsilon_{i_1\cdots i_n} a_{i_1k_1}\cdots a_{i_nk_n}.
 \]
 Using the chain rule of differentiation, one obtains for $x\in \cU$,
 \begin{align*}
   & M_{\Phi\circ T}(x)
   = \sum_{i_1,\ldots,i_n=1}^n \epsilon_{i_1\cdots i_n} \tr_d\big(\partial_{i_1}(\Phi\circ T)(x)\cdots \partial_{i_n}(\Phi\circ T)(x)\big)
   \\ & \quad = \sum_{i_1,\ldots,i_n=1}^n \epsilon_{i_1\cdots i_n} \tr_d\big(\sum_{k_1=1}^n\partial_{k_1}\Phi(T(x))\partial_{i_1}T_{k_1}(x)\cdots \sum_{k_n=1}^n\partial_{k_n}\Phi(T(x))\partial_{i_n}T_{k_n}(x)\big)
   \\  & \quad = \sum_{k_1,\ldots,k_n=1}^n\sum_{i_1,\ldots,i_n=1}^n \epsilon_{i_1\cdots i_n}\partial_{i_1}T_{k_1}(x)\cdots \partial_{i_n}T_{k_n}(x) \\ 
   &\hspace*{3.55cm} \times \tr_d\big(\partial_{k_1}\Phi(T(x))\cdots \partial_{k_n}\Phi(T(x))\big)
   \\  & \quad = \sum_{k_1,\ldots,k_n=1}^n\epsilon_{k_1\cdots k_n}\det( T'(x)) \tr_d\big(\partial_{k_1}\Phi(T(x))\cdots \partial_{k_n}\Phi(T(x))\big)
   \\ & \quad = M_\Phi (T(x))\det(T'(x)).\qedhere
 \end{align*}
\end{proof}

\begin{proof}[Proof of Theorem \ref{t:inv}]
Let $U$ be $C^2$-smooth and such that $\sgn(\Phi)=U$ on complements of sufficiently large balls. One observes that $U_T\coloneqq \overline{U\circ T}=\sgn(\Phi_T)$. In particular, $\ind(\mathcal{Q}+\Phi_T)=\ind(\mathcal{Q}+U_T)$. Next, we set
\[
   c_n\coloneqq \frac{1}{2}\left(\frac{i}{8\pi}\right)^{(n-1)/2}\frac{1}{\left[(n-1)/2\right]!}.
\]
By Theorem \ref{thm:Perturbation} together with Remark \ref{r:mphi}, and taking into account the chain rule, Lemma \ref{l:chain}, one computes, 
\begin{align*}
   \ind(\mathcal{Q}+U_T) & = c_n\lim_{\Lambda\to\infty} \int_{B(0,\Lambda)} M_{U_T}(x)\, d^n x
   \\  & = c_n\lim_{\Lambda\to\infty} \int_{B(0,\Lambda)} M_{U\circ T}(x)\, d^n x
   \\  & = c_n\lim_{\Lambda\to\infty} \int_{B(0,\Lambda)} M_{U}(T(x))\det(T'(x))\, d^n x
   \\  & = \sgn(T) c_n\lim_{\Lambda\to\infty} \int_{B(0,\Lambda)} M_{U}(T(x))|\det(T'(x))|\, d^n x
   \\  & = \sgn(T) c_n\lim_{\Lambda\to\infty} \int_{T[B(0,\Lambda)]} M_{U}(x)\, d^n x
   \\  & = \sgn(T) c_n\lim_{\Lambda\to\infty} \int_{B(0,\Lambda)\cap T[B(0,\Lambda)]} M_{U}(x)\, d^n x,
\end{align*}
using the transformation rule for integrals.

To conclude the proof, we are left with showing
\[
c_n\lim_{\Lambda\to\infty} \int_{B(0,\Lambda)\cap T[B(0,\Lambda)]} M_{U}(x)\, d^n x=\ind(\mathcal{Q}+U).
\]
For this purpose one notes that $T$ is continuously invertible, by hypothesis. Hence, the range of $T$ is open. Since the range of $T$ is also dense, $\{\chi_{T[B(0,\Lambda)]}\}_{\Lambda\in \mathbb{N}}$ converges in the strong operator topology of $\mathcal{B}\big(L^2(\mathbb{R}^n)\big)$ to $I_{L^2(\mathbb{R}^n)}$, where $\chi_{T[B(0,\Lambda)]}$ denotes the characteristic function of the set $T[B(0,\Lambda)]$, $\Lambda>0$. Thus, for $L=\mathcal{Q}+U$, one computes
\begin{align*}
\ind(L)  
&=\lim_{\Lambda\to\infty}\lim_{z\to 0_+}z \tr_{L^2(\mathbb{R}^n)}\big( \chi_{T[B(0,\Lambda)]}\chi_{\Lambda}\tr_{2^{\hatt n}d}\big((L^*L+z)^{-1}-(LL^*+z)^{-1}\big)\big)
  \\  
  & = c_n\lim_{\Lambda\to\infty} \int_{B(0,\Lambda)\cap T[B(0,\Lambda)]} M_{U}(x)\, d^n x,
\end{align*}
proving the assertion.
\end{proof}

Finally, we apply Theorem \ref{t:inv} and prove that for admissible potentials $\Phi$, 
$\ind(\mathcal{Q}+\Phi)=0$:

\begin{corollary}\label{c:adind0} Let $n\in \mathbb{N}_{\geq 3}$ odd, $d\in \mathbb{N}$. Let $\Phi$ be admissible, see Definition \ref{def:phi_admissible}. Let $\mathcal{Q}$ be as in \eqref{eq:def_of_Q2} and $L=\mathcal{Q}+\Phi$ as in \eqref{eq:def_of_L(2)}. Then $L$ is Fredholm and $\ind(L)=0$.
\end{corollary}
\begin{proof}
  By invariance of the Fredholm index under relatively compact perturbations (cf.\ 
 Theorem \ref{t3.6}\,$(iii)$), we can assume without loss of generality, that $\Phi$ is constant in a neighborhood of $0$. We consider $T\colon \mathbb{R}^n \backslash \{0\}\to\mathbb{R}^n$ given by
  \[
    T(x) \coloneqq \frac{x}{|x|^2},\quad x\in \mathbb{R}^n \backslash \{0\}.
  \]
 One observes that $T$ is a transformation of constant orientation. Moreover, as $\Phi$ is admissible, so is $\Phi_T\coloneqq \overline{\Phi\circ T}$. In particular, since $\Phi$ is constant in a neighborhood of $0$, we find $\Lambda>0$ such that for all $x\in \mathbb{R}^n$ with $|x|\geq \Lambda$, $(\partial_i \Phi_T)(x)=0$. Hence, $\ind(\mathcal{Q}+\Phi_T)=0$, by Theorem \ref{thm:Fredholm-index} and, thus $\ind(\mathcal{Q}+\Phi)=0$, by Theorem \ref{t:inv}.
\end{proof}

For an entirely diffferent approach to Corollary \ref{c:adind0} we refer again to 
\cite[Theorem~5.2]{Ra08}. 

\begin{remark} \lb{r10.12}
As kindly pointed out to us by one of the referees, Corollary \ref{c:adind0} permits a
more elementary proof as follows. If $\Phi$ is admissible, then 
$x\mapsto \Phi_t(x) \coloneqq \Phi(tx)$, $t > 0$, 
is also admissible and the associated operators
$L_t\colon H^{1}(\mathbb{R}^{n})^{2^{\hat n}d} \to L^{2}(\mathbb{R}^{n})^{2^{\hat n}d}$, 
$t>0$, are all Fredholm. In addition, the map, $(0,\infty) \ni t \mapsto L_t \in 
\cB\big(H^{1}(\mathbb{R}^{n})^{2^{\hat n}d}, L^{2}(\mathbb{R}^{n})^{2^{\hat n}d}\big)$ 
is continuous. Thus (cf.\ Corollary \ref{c3.7}), 
\begin{equation} 
\ind (L_1) = \ind (L_t), \quad t>0. 
\end{equation} 
However, \eqref{eq:LstarLandLLstar}
leads to 
\begin{equation} 
L_t^* L_t = - \Delta I_{2^{\hatt n}d} - C_t + \Phi_t^{2}, \quad  
L_t L_t^{*} =-\Delta I_{2^{\hatt n}d} + C_t + \Phi_t^{2}, 
\end{equation} 
where
\begin{equation} 
C_t = \sum_{j=1}^{n}\gamma_{j,n} (\partial_{j}\Phi_t)=(\mathcal{Q}\Phi_t), \quad t>0. 
\end{equation} 
Hence, for some constant $c>0$, 
$\|C_t\|_{\cB(L^{2}(\mathbb{R}^{n})^{2^{\hat n}d})} \leq c \, t$ for $0<t$ sufficiently small. 
In particular, for $0 < t$ sufficiently small, the operators $L_t^* L_t$ and  $L_t L_t^*$ are 
boundedly invertible and hence $\ind (L_t) = 0$, implying $\ind (L_1) = \ind (L) = 0$. 
\hfill $\diamond$
\end{remark}

\newpage

\section{Perturbation Theory for the Helmholtz Equation}\label{sec:pert}

Before we are in a position to provide a proof of Theorem \ref{thm:Perturbation}, we need some results concerning the perturbation theory of Helmholtz operators. More precisely, we study operators (and their fundamental solutions) of the form
\[
   (-\Delta + \mu +\eta) 
\]
in odd space dimensions $n\geq3$ and $\eta\in L^\infty(\mathbb{R}^n)$ with small support around the origin and $\mu\in \mathbb{C}_{\Re>0}$. For $\mu\in \mathbb{C}_{\Re >0}$, $\eta\in L^\infty(\mathbb{R}^n)$, recalling $R_\mu=(-\Delta+\mu)^{-1}$, one formally computes 
\begin{align}\label{def:R_psi_z}
  R_{\eta+\mu}&\coloneqq \big(-\Delta  + \eta +\mu)\big)^{-1} 
   =\big((-\Delta  + \mu)(1  + R_\mu\eta\big)^{-1}  \no \\ 
   & \, =\sum_{k=0}^\infty (R_{\mu}(-\eta))^kR_{\mu}.
\end{align}
This computation can be made rigorous, if $\|R_{\mu}(\eta)\|_{\mathcal{B}(L^2(\mathbb{R}^n))}<1$. The first aim of this section is to provide a proof of the fact that if $\|\eta\|_{L^\infty}\leq 1$, then 
indeed $\|R_{\mu}(\eta)\|_{\mathcal{B}(L^2(\mathbb{R}^n))}<1$ for ``sufficiently'' many $\mu$, that is, for $\mu$ belonging to the closed sector 
\begin{equation}\label{def:Sigma}
   \ol{\Sigma_{\mu_0,\theta}} = \{z\in \mathbb{C} \, | \, \Re (\mu) \geq \mu_0, |\arg(\mu)|\leq \theta\}
\end{equation}
for some $\mu_0\in \mathbb{R}$, $\theta\in [0,\frac{\pi}{2}]$, provided the support of $\eta$ is 
sufficiently small. 

For $\mu>0$, $x,y\in \mathbb{R}^n$, $x\neq y$, we introduce
\begin{equation}\label{def:s_mu}
   s_\mu(x-y)\coloneqq \frac{e^{-\sqrt{\mu}|x-y|}}{|x-y|^{n-2}}.
\end{equation}

The next lemma shows that the Helmholtz Green's function basically behaves like 
$s_{\mu}$ in \eqref{def:s_mu}. We note that a similar estimate was used in 
\cite[p.~224, formula (c)]{Ca78}. However, we further remark that the factor $\lambda$ introduced in the following result does not occur in \cite[p.~224, formula (c)]{Ca78}, yielding a hidden $z$-dependence of the constant $K$ occuring there.

\begin{lemma}\label{l:7.18} Let $n\in \mathbb{N}_{\geq3}$ odd, $\lambda\in (0,1)$. For $\mu>0$ denote the integral kernel of $R_\mu= (-\Delta+\mu)^{-1}$ in $L^2(\mathbb{R}^n)$ by $r_\mu$, see Lemma \ref{lem:reform_of_en} or \eqref{C.1}, and let $s_\mu$ be as in \eqref{def:s_mu}. Then there exist $c_1,c_2>0$ such that for all $\mu>0$,  
\[
   r_{\mu}(x-y) \leq c_1 s_{\lambda\mu}(x-y), \text{ and }s_{\mu}(x-y)\leq c_2 r_{\mu}(x-y), \quad 
   x,y\in \mathbb{R}^n, \; x\neq y.
\] 
\end{lemma}
\begin{proof}
 For the first inequality, one observes that for $k\in\{0,\ldots,\hatt n-1\}$, with $n=2\hatt n+1$, 
 the function 
 \[
     \mathbb{R}_{\geq 0}\ni \beta\mapsto \beta^k e^{-(1-\sqrt{\lambda})\beta}
 \]
  is bounded by some $d_k>0$. Next, let $x,y\in \mathbb{R}^n$, $x\neq y$ and $r\coloneqq |x-y|$, $\mu>0$. Then, for $k\in \{0,\ldots,\hatt n-1\}$, one estimates
  \[
     \frac{e^{-\sqrt{\mu}|x-y|}}{|x-y|^{n-2-k}} (\sqrt{\mu})^k  = \frac{e^{-\sqrt{\mu}r}}{r^{n-2}} (r\sqrt{\mu})^k
      \leq d_k \frac{e^{-\sqrt{\lambda\mu}r}}{r^{n-2}}.
  \]
  Hence, the first inequality asserted follows from Lemma \ref{lem:reform_of_en}.
  Employing again Lemma \ref{lem:reform_of_en}, the second inequality can be derived easily.
\end{proof}

We can now come to the announced result of bounding the operator norm of $R_\mu\eta$ 
given $\eta$ is supported on a small set. We note that smallness of the support is independent 
of $\mu$, if one assumes $\mu$ to lie in a sector.

\begin{lemma}\label{lem:discussion of Neumann} Let $\mu_0>0$, $\theta\in (0,\pi/2)$, 
$\beta>0$, $n\in\mathbb{N}_{\geq 3}$ odd. Then there exists $\tau>0$ such that for all 
$\mu\in \Sigma_{\mu_0,\theta}$, see \eqref{def:Sigma},   
 \[
    \|R_{\mu}\eta\|_{\mathcal{B}(L^2(\mathbb{R}^n))}\leq\beta  
 \]
for all $\eta\in L^\infty(\mathbb{R}^n)$, $\|\eta\|_{L^\infty}\leq  1$ and $\supp (\eta) \subset  B(0,\tau)$. 
\end{lemma}
\begin{proof}
 Let $\tau>0$ and $\eta\in L^\infty(\mathbb{R}^n)$ such that $\supp (\eta) \subset  B(0,\tau)$ and $\|\eta\|_{L^\infty}\leq  1$. Let $\mu\in \Sigma_{\mu_0,\theta}$ and denote the fundamental solution of $(-\Delta+\mu)$ by $r_\mu$, see also Lemma \ref{lem:Real-part-Bessel_and_other}. By 
 estimate \eqref{bess_b} in Lemma \ref{lem:Real-part-Bessel_and_other}, there exists $c_1\geq1$ such that 
 \[
    |r_\mu(x-y)|\leq c_1 r_{\Re \mu}(x-y), \quad x,y\in \mathbb{R}^n, \; x\neq y, \; \mu\in \Sigma_{\mu_0,\theta}.
 \]
Next, by Lemma \ref{l:7.18}, there exists $c_2>0$ such that for all $\mu \geq \mu_0$, 
 \[
      r_\mu(x-y) \leq  c_2 s_{\frac{1}{2}\mu}(x-y), \quad x,y\in \mathbb{R}^n, \; x\neq y.
 \]
For $\mu \geq\mu_0$, one notes that 
$\|s_{\mu/2}\|_{L^1(\mathbb{R}^n)}\leq \|s_{\mu_0/2}\|_{L^1(\mathbb{R}^n)}<\infty$. Hence, 
for $\mu\in \Sigma_{\mu_0,\theta}$ and $u\in C_0^\infty(\mathbb{R}^n)$, one gets
\begin{align*}
  &\|R_{\mu}(\eta) u\|_{L^2(\bbR^n)}^2 
  = \int_{\mathbb{R}^n} \left| \int_{\mathbb{R}^n} r_\mu(x-y)\eta(y)u(y)d^n y\right|^2 \, d^n x\\
 & \quad =\int_{\mathbb{R}^n} \bigg|\int_{B(0,\tau)} r_\mu(x-y)\eta(y)u(y)d^n y\bigg|^2 \, d^n x\\
                           & \quad \leq c_1^2\int_{\mathbb{R}^n}\bigg( \int_{B(0,\tau)} r_{\Re(\mu)}(x-y)|\eta(y)||u(y)|d^n y\bigg)^2 \, d^n x\\
                           & \quad \leq c_1^2c_2^2\int_{\mathbb{R}^n} \bigg(\int_{B(0,\tau)} s_{\frac{1}{2}\Re(\mu)}(x-y) d^ny\bigg) \int_{B(0,\tau)} s_{\frac{1}{2}\Re(\mu)}(x-y)|u(y)|^2 \, d^n y d^n x\\
                           & \quad \leq c_1^2c_2^2\int_{\mathbb{R}^n} \bigg(\int_{B(0,\tau)} s_{\frac{1}{2}\Re(\mu)}(y) d^ny\bigg) \int_{\mathbb{R}^n} s_{\frac{1}{2}\Re(\mu)}(x-y)|u(y)|^2 \, d^n y d^n x\\
                           & \quad = c_1^2c_2^2\int_{B(0,\tau)} s_{\frac{1}{2}\Re(\mu)}(y) \, d^ny \, 
                           \|s_{\frac{1}{2}\mu_0}\|_{L^1(\mathbb{R}^n)} \|u\|^2_{L^2(\mathbb{R}^n)}.
\end{align*}
One observes that 
\[
 \int_{B(0,\tau)} s_{\frac{1}{2}\Re(\mu)}(y) \, d^ny 
 \leq \omega_{n-1}\int_0^\tau \frac{1}{r^{n-2}}r^{n-1} \,dr=\frac{\tau^2}{2 \omega_{n-1}},
\]
and hence, 
\[
   \|R_\mu\eta\|_{\mathcal{B}(L^2(\mathbb{R}^n))}\leq c_1c_2 
   \sqrt{\frac{\|s_{\mu_0/2}\|_{L^1(\mathbb{R}^n)}}{2\omega_{n-1}}}\tau.\qedhere
\]
\end{proof}

\begin{remark} \label{rem:remark on op norm of Rpsiz} 
$(i)$ Let $\mu_0>0$, and $\theta\in (0,\frac{\pi}{2})$, $\kappa>0$. Then for all 
$\mu \in \Sigma_{\mu_0,\theta}$ (see \eqref{def:Sigma}), there exists $\tau>0$ such that for 
$\eta\in L^\infty(\mathbb{R}^n)$, with $\|\eta\|_{L^\infty}\leq\kappa$ and $\eta=0$ on 
$\mathbb{R}^n\backslash  B(0,\tau)$, the operator $R_{\eta+\mu}=(-\Delta + \eta + \mu)^{-1}$ 
exists as a bounded linear operator in $L^2(\mathbb{R}^n)$ and its norm is arbitarily close 
to $\|R_{\mu}\|$. Indeed, for $\beta<1$ with $\|R_{\mu}\eta\|\leq\beta$ one computes
 \[
    \|R_{\eta + \mu}\|\leq  \sum_{k=0}^\infty \beta^k \|R_{\mu}\| =\frac{1}{1-\beta} \|R_{\mu}\|.
 \]
$(ii)$ In the situation of part $(i)$, we shall now elaborate some more on the properties of 
$R_{\eta+\mu}$ with $\|R_{\mu}\eta\|\leq\beta<1$ for all $\mu\in \Sigma_{\mu_0,\theta}$. Assuming, in addition, $\eta\in C^\infty$, then $R_{\eta+\mu}$ extends by interpolation to a bounded linear operator to the full Sobolev scale $H^s(\mathbb{R}^n)$ (see \eqref{eq:def_Hs} for a definition), $s\in \mathbb{R}$. Moreover, from 
  \[
     (-\Delta+\mu)R_{\eta+\mu}=(-\Delta+\mu)\sum_{k=0}^\infty R_{\mu}(-\eta R_{\mu})^k=\sum_{k=0}^\infty ((-\eta) R_{\mu})^k, 
  \]
 one gets $\|(-\Delta+\mu)R_{\eta+\mu}\|\leq (1-\beta)^{-1}$, yielding $R_{\eta+\mu}\in \cB(H^s(\mathbb{R}^n);H^{s+2}(\mathbb{R}^n))$ for all $s\in \mathbb{R}$. \hfill $\diamond$
\end{remark}

With Lemma \ref{lem:discussion of Neumann} we have an a priori condition on the support of $\eta$ to make the operator $R_{\eta+\mu}$ well-defined. The forthcoming results, the very reason of this entire section, provide estimates for the integral kernels of the perturbed operator in terms of the unperturbed one. Of course, these estimates also rely on a Neumann series type argument. The main step is the following lemma.

\begin{lemma}\label{lem:7.13b} Let $n\in \mathbb{N}_{\geq 3}$ odd, $\mu_0>0$, $\kappa >0$. For any $\lambda\in (0,1)$, there exists $\tau>0$ such that for all $\eta\in L^\infty(\mathbb{R}^n)$, with $\|\eta\|_{L^\infty}\leq \kappa$ and $\supp (\eta) \subset B(0,\tau)$, such that for all $k\in \mathbb{N}_{\geq3}$, $\mu \geq \mu_0$, the integral kernel $\tilde r_k$ of $(R_\mu \eta)^k R_\mu$ satisfies
\[
   |\tilde r_k (x,y)|\leq \lambda^k r_{\mu/4}(x-y), \quad x,y\in \mathbb{R}^n, \; x\neq y,
\]
where $r_\mu$ is the integral kernel of $R_\mu=(-\Delta+\mu)^{-1}$ given by \eqref{C.1}.
\end{lemma}

We postpone the proof of Lemma \ref{lem:7.13b} and show three preparatory results first.

\begin{lemma}\label{lem:7.14} Let $n\in \mathbb{N}_{\geq 3}$ odd, $\mu>0$, $\tau > 0$, $s_\mu$ as in \eqref{def:s_mu}. Then for all $x,z\in \mathbb{R}^n$, $x\neq z$, the inequality,
\[
   \int_{B(0,\tau)} s_\mu(x-y)s_\mu(y-z) \, d^n y\\ \leq 2^{n-3} \omega_{n-1}\tau^2 s_\mu(x-z), 
\]
holds, with $\omega_{n-1}$ the $(n-1)$-dimensional volume of the unit sphere $S^{n-1}\subseteq \mathbb{R}^{n}$ $($see also \eqref{e:om}$)$.
\end{lemma}
\begin{proof} One notes that, by the triangle inequality, 
\[
   e^{-\sqrt{\mu}|x-y|}e^{-\sqrt{\mu}|y-z|}\leq e^{-\sqrt{\mu}|x-z|}, \quad x,y,z\in \mathbb{R}^n.
\]
Hence, one is left with showing
\[
\int_{B(0,\tau)} \frac{1}{|x-y|^{n-2}}\frac{1}{|y-z|^{n-2}} \, d^n y \leq 2^{n-3} \omega_{n-1}\tau^2
   \frac{1}{|x-z|^{n-2}}, \quad x,y,z\in \mathbb{R}^n, \; x\neq z.
 \]  
Let $x,y,z\in \mathbb{R}^n$. Then
\[
   |x-z|^{n-2}\leq (|x-y|+|y-z|)^{n-2} \leq 2^{n-3}(|x-y|^{n-2}+|y-z|^{n-2}).
\]
Hence, 
\begin{align*}
   &\int_{B(0,\tau)} \frac{|x-z|^{n-2}}{|x-y|^{n-2}|y-z|^{n-2}} \, d^n y 
   \\ & \quad \leq  \int_{B(0,\tau)} \frac{2^{n-3}(|x-y|^{n-2}+|y-z|^{n-2})}{|x-y|^{n-2}|y-z|^{n-2}} \, d^n y 
   \\ & \quad = 2^{n-3}\int_{B(0,\tau)}\bigg( \frac{1}{|x-y|^{n-2}}+\frac{1}{|y-z|^{n-2}}\bigg) \, d^n y
   \\ & \quad \leq 2^{n-2}\int_{B(0,\tau)} \frac{1}{|y|^{n-2}} \, d^n y
   \\ & \quad = 2^{n-2}\omega_{n-1}\int_{0}^\tau r \, dr = 2^{n-3} \omega_{n-1}\tau^2. \qedhere
\end{align*}
\end{proof}

\begin{proposition}\label{p:7.7} Let $n\in \mathbb{N}$, $\mu_0>0$, and $q\in L^1(\mathbb{R}^n)\cap C(\mathbb{R}^n\backslash \{0\})$. Assume that $V_q$, the operator defined by convolution with $q$, defines a self-adjoint, nonnegative operator in $L^2(\mathbb{R}^n)$. Then for all $\mu \geq \mu_0$ and $x,y\in \mathbb{R}^n$, $x\neq y$,  
\[
  \mu_0 \int_{\mathbb{R}^n} r_{\mu}(x-x_1)q(x_1-y) \, d^n x_1 \leq q(x-y). 
\]  
\end{proposition}
\begin{proof}
 Let $\mu \geq \mu_0$. As the convolution with $q$ commutes with differentiation, it also commutes with  $(-\Delta+\mu)$, $R_\mu$ or powers thereof. Since $V_q\geq 0$, there exists a unique nonnegative square root $V_q^{1/2}$, which also commutes with $R_\mu$, $R_\mu^{-1}$ and powers thereof. For $\phi\in H^2(\mathbb{R}^n)$ and $\mu \geq \mu_0$,  
 \begin{align*}
   \big( (-\Delta+\mu)\phi,V_q\phi\big)_{L^2}& =\big( (-\Delta+\mu) V_q^{1/2}\phi,V_q^{1/2}\phi\big)_{L^2}
   \\ & \geq \mu_0 \big( V_q^{1/2}\phi,V_q^{1/2}\phi\big)_{L^2}
   \\ & = \mu_0 \big( \phi,V_q \phi\big)_{L^2}.
 \end{align*}
 Putting $\phi\coloneqq (-\Delta+\mu)^{-1/2}\psi=R_\mu^{1/2} \psi$ for some $\psi\in H^2(\mathbb{R}^n)$, one infers 
 \begin{align*}
   \big( \psi, V_q\psi\big)_{L^2}& \geq \mu_0 \big( R_\mu^{1/2}\psi,V_q R_\mu^{1/2}\psi\big)_{L^2}
   \\&\geq \mu_0 \big( \psi,R_\mu V_q\psi\big)_{L^2}.
 \end{align*}
 As $(-\Delta+\mu)^{-1/2}[H^2(\mathbb{R}^n)]$ is dense in $L^2(\mathbb{R}^n)$, it follows that $ V_q- \mu_0R_\mu V_q$ is a nonnegative integral operator, which implies the asserted inequality.
\end{proof}

Applying Proposition \ref{p:7.7} with $q=r_\mu$ twice, one gets the proof of the following result.

\begin{lemma}\label{lem:7.14c} Let $n\in \mathbb{N}$, $\mu_0>0$. Then for all $\mu \geq \mu_0$, 
\begin{equation}\label{eq:7.14.3}
    \mu_0^2 \int_{(\mathbb{R}^n)^2} r_{\mu}(x-x_1)r_{\mu}(x_1-x_2)r_{\mu}(x_2-y) \, 
    d^n x_1 d^nx_2 \leq r_{\mu}(x-y), \quad x,y\in \mathbb{R}^n, \; x\neq y.
 \end{equation} 
\end{lemma}

\begin{proof}[Proof of Lemma \ref{lem:7.13b}] 
 Recalling \eqref{def:s_mu}, 
\[
   s_\mu(x-y)=\frac{e^{-\sqrt{\mu}|x-y|}}{|x-y|^{n-2}}, \quad x,y\in \mathbb{R}^n, 
   \; x\neq y, \; \mu > 0.
\]
 By Lemma \ref{l:7.18}, there exist $c_1,c_2>0$ such that for all $\mu >0$ and $x,y\in \mathbb{R}^n$, $x\neq y$,  
 \begin{equation}\label{eq:7.14.1}
        r_{\mu}(x-y)\leq c_1 s_{\mu/4}(x-y), \quad 
        s_{\mu/4}(x-y) \leq  c_2 r_{\mu/4}(x-y).
 \end{equation}
Next, one recalls, with $n=2\hatt n+1$ for some $\hatt n\in \mathbb{N}$, from Lemma \ref{lem:Real-part-Bessel_and_other}, equation \eqref{bess_c}, that
\[
   r_\mu(x-y)\leq 2^{\hatt n-1} r_{\mu/4}(x-y), \quad x,y\in \mathbb{R}^n, \; x\neq y.
\]
Let $\tau > 0$. We estimate for $x,y\in \mathbb{R}^n$, $x\neq y$, $\eta\in L^\infty(\mathbb{R}^n)$, with $\supp (\eta) \subset B(0,\tau)$, $\mu \geq \mu_0$, using Lemma \ref{lem:7.14} and inequality \eqref{eq:7.14.1}, with $\kappa_\tau\coloneqq 2^{n-3} \omega_{n-1}\tau^2$,
\begin{align*}
   &|\tilde r_k(x,y)| 
   = \bigg|\int_{(\mathbb{R}^n)^{k}} r_\mu (x-x_1)\eta(x_1)r_\mu(x_1-x_2)\cdots \eta(x_{k})r_\mu(x_{k}-y) \, d^n x_1 \cdots d^n x_{k}\bigg|
   \\ & \quad \leq \|\eta\|_{L^\infty}^k \int_{(B(0,\tau))^{k}} r_\mu (x-x_1)r_\mu(x_1-x_2)\cdots r_\mu(x_{k}-y) \, d^n x_1 \cdots d^n x_{k}
   \\ & \quad \leq \|\eta\|_{L^\infty}^kc_1^{k-1}\int_{(B(0,\tau))^{k}} r_\mu (x-x_1) 
   s_{\mu/4}(x_1-x_2) 
   \\ & \hspace*{4cm}  \cdots \times s_{\mu/4}(x_{k-1}-x_{k})r_\mu(x_{k}-y) \, d^n x_1 \cdots d^n x_{k}
   \\ & \quad \leq \|\eta\|_{L^\infty}^k(c_1\kappa_\tau)^{k-1}\int_{(B(0,\tau))^{2}} r_\mu (x-x_1)s_{\mu/4}(x_1-x_{k})r_\mu(x_{k}-y) \, d^n x_1 d^n x_{k}
   \\ & \quad \leq \|\eta\|_{L^\infty}^k(c_1\kappa_\tau)^{k-1}c_2\int_{(B(0,\tau))^{2}} r_\mu (x-x_1) 
   r_{\mu/4}(x_1-x_{k})r_\mu(x_{k}-y) \, d^n x_1 d^n x_{k}
   \\ & \quad \leq \|\eta\|_{L^\infty}^k(c_1\kappa_\tau)^{k-1}2^{n-3}c_2\int_{(\mathbb{R}^n)^{2}} 
   r_{\mu/4} (x-x_1)r_{\mu/4}(x_1-x_{k})   \\
& \hspace*{5.3cm}  \times r_{\mu/4}(x_{k}-y) \, d^n x_1 d^n x_{k}
   \\ & \quad \leq \|\eta\|_{L^\infty}^k\frac{16}{\mu_0^2 }(c_1\kappa_\tau)^{k-1}2^{n-3}c_2 
   r_{\mu/4}(x-y),\end{align*}
   where, in the last estimate, we used Lemma \ref{lem:7.14c}.
\end{proof}

Having proved Lemma \ref{lem:7.13b}, we can now formulate and prove the result for the estimate of the perturbed and the unperturbed integral kernels (Green's functions).

\begin{theorem}\label{thm:7.18} Let $n\in \mathbb{N}_{\geq 3}$ odd, $\mu_0>0$, 
$\theta\in (0,\pi/2)$, $\kappa>0$. Then there exists $c,\tau>0$ such that for all $\mu\in \Sigma_{\mu_0,\theta}$ and $\eta\in C^\infty(\mathbb{R}^n)$ with $\supp(\eta)\subset B(0,\tau)$, $\|\eta\|_{L^\infty}\leq \kappa$, the estimate
\[
   |r_{\eta+\mu}(x,y)|\leq c r_{\Re(\mu)}(x-y), \quad x,y\in \mathbb{R}^n, \; x\neq y, 
\]
holds, where $r_{\eta+\mu}$ and $r_{\Re(\mu)}$ are the integral kernels for the operators 
$R_{\eta+\mu}=(-\Delta+ \eta+\mu)^{-1}$ and $R_{\Re(\mu)}$, respectively. 
\end{theorem}
\begin{proof}
One recalls that $r_\mu$ denotes the integral kernel of $R_\mu$. 
 According to Lemma \ref{lem:Real-part-Bessel_and_other}, \eqref{bess_b}, there exists $c_1\geq 1$ such that for all $\mu\in\Sigma_{\mu_0,\theta}$ one has 
\[
   |r_{\mu}(x-y)|\leq c_1 r_{\Re (\mu)}(x-y), \quad x,y\in \mathbb{R}^n, \; x\neq y.
\]
Next, by Lemma \ref{lem:discussion of Neumann}, one chooses $\tau_1>0$ such that 
$\|R_{\mu}\eta\|\leq 1/2$ for all $\mu\in \Sigma_{\mu_0,\theta}$ and $\eta\in L^\infty(\mathbb{R}^n)$ with $\supp (\eta) \subset B(0,\tau_1)$ and $\|\eta\|_{L^\infty}\leq \kappa$, implying that $R_{\eta+\mu}$ is a well-defined bounded linear operator in $L^2(\mathbb{R}^n)$ (see, e.g., Remark \ref{rem:remark on op norm of Rpsiz}).

Let $\tau_2>0$ be such that for all $k\in \mathbb{N}_{\geq 1}$, the integral kernel 
$\tilde r_{k,\Re(\mu)}$ for the operator $(R_{\Re(\mu)}\eta)^kR_{\Re(\mu)}$ satisfies
  \begin{equation}\label{eq:7.18}
       |\tilde r_{k,\Re(\mu)}(x,y)|\leq c_2 \frac{1}{(2c_1)^k}r_{\Re(\mu)/4}(x-y)
  \end{equation}
  for all $x,y\in \mathbb{R}^n$, $x\neq y$, $\mu\in \Sigma_{\mu_0,\theta}$, $\eta\in L^\infty(\mathbb{R}^n)$ with $\|\eta\|_{L^\infty}\leq \kappa$ and  $\supp (\eta) \subset B(0,\tau_2)$ and some $c_2>0$, which is possible by Lemma \ref{lem:7.13b}. 
  
  Let $\tau\coloneqq \min\{\tau_1,\tau_2\}$. Then, for $x,y\in \mathbb{R}^n$, and $\eta\in C^\infty(\mathbb{R}^n)$, with $\supp (\eta) \subset B(0,\tau)$ and $\|\eta\|_{L^\infty}\leq \kappa$ 
  one gets for $N,M\in \mathbb{N}$, $N>M$, $\mu\in \Sigma_{\mu_0,\theta}$,
  \begin{align}
    \bigg|\sum_{k=M}^N \big(\underbrace{(r_{\mu}*\eta)\cdots (r_{\mu}*\eta)}_{k\text{-times}}\big) r_{\mu}(x-y)\bigg|
    &\leq \sum_{k=M}^N c_1^{k+1}|\tilde r_{k,\Re(\mu)}(x,y)| \notag
   \\ &\leq c_2 c_1\sum_{k=M}^\infty 2^{-k}r_{\Re(\mu)/4}(x-y) \label{e:t.7.8}
   \\ &\leq c_2 c_1 2^{-M+1}r_{\Re(\mu)/4}(x-y). \notag
  \end{align}
   Thus, 
   \[
    \tilde r(x,y)\coloneqq \sum_{k=0}^\infty \big(\underbrace{(r_{\mu}*\eta)\cdots (r_{\mu}*\eta)}_{k\text{-times}}\big) r_{\mu}(x-y), \quad x,y\in \mathbb{R}^n, \; x\neq y, 
   \]
   defines a function, which, by the differentiability of $\eta$, coincides with the fundamental solution of $(-\Delta+\eta+\mu)$. For $x,y\in \mathbb{R}^n$, $x\neq y$, $\mu\in \Sigma_{\mu_0,\theta}$, one thus gets, using \eqref{e:t.7.8} for $M=1$ and $N\to\infty$,
   \begin{align*}
          |r_{\eta+\mu}(x,y)|& \leq |r_{\eta+\mu}(x,y)-r_{\mu}(x-y)|+|r_{\mu}(x-y)|
          \\ & \leq c_2 c_1 r_{\Re(\mu)/4}(x-y) + c_1 r_{\Re(\mu)}(x-y)
          \\ & \leq (c_2 c_1+c_12^{\hatt n-1}) r_{\Re(\mu)/4}(x-y),
   \end{align*}
   where, in the last estimate, we used Lemma \ref{lem:Real-part-Bessel_and_other}, 
   \eqref{bess_c}, with $n=2\hatt n+1$ for some $\hatt n\in \mathbb{N}$. Finally, for $\mu\in \Sigma_{\mu_0,\theta}$, and from
   \[
      R_{\eta+\mu} = \sum_{k=0}^\infty R_\mu ((-\eta)R_\mu)^k= R_\mu - R_\mu \eta R_{\eta+\mu},
   \]
 one reads off, for $x,y\in \mathbb{R}^n$, $x\neq y$,
   \begin{align*}
      |r_{\eta +\mu}(x,y)|&\leq |r_{\mu}(x-y)|+|r_{\mu}*\eta r_{\eta+\mu}(x,y)|
      \\  & \leq c_1 r_{\Re(\mu)}(x-y)+c_1\kappa (c_2 c_1+c_12^{\hatt n-1}) 
      r_{\Re (\mu)}*r_{\Re(\mu)/4}(x-y)
      \\ & \leq  c_1\bigg(1+ \kappa (c_2 c_1+c_12^{\hatt n-1})\frac{4}{\mu_0}\bigg) r_{\Re(\mu)}(x-y),
   \end{align*}
   where we used Proposition \ref{p:7.7} for $q=r_{\Re(\mu)/4}$ for obtaining the last estimate.
\end{proof}

As a first application of Theorem \ref{thm:7.18}, in the spirit of the results derived in Section \ref{sec:ptw_intk}, we can show the following result.

\begin{corollary}\label{cor:7.19} Let $n\in\mathbb{N}_{\geq 3}$ odd, with $n=2\hatt n+1$ for some $\hatt n\in \mathbb{N}$, $\mu_0>0$, $\theta\in (0,\pi/2)$. Then there exists $\tau>0$, such that for $m\in \mathbb{N}_{>\hatt n}$ there exists $c\geq 1$ with the following properties: 
Given $\Psi_1,\ldots,\Psi_m\in C_b^\infty(\mathbb{R}^n)$, with 
\[
   |\Psi_j(x)|\leq \kappa (1+|x|)^{- \alpha_j},  
   \quad x\in \mathbb{R}^n, \; j\in\{1,\ldots,m\},
\]
for some $\alpha_1,\ldots,\alpha_m,\kappa\in [0,\infty)$,
then for all $\eta_j\in C_b^\infty(\mathbb{R}^n)$, $\|\eta_j\|_{L^\infty}\leq 1$, $j\in \{1,\ldots,m\}$, and $\supp (\eta_j) \subset B(0,\tau)$, the integral kernel $t_\mu$ of $\prod_{j\in \{1,\ldots,m\}} R_{\eta_j+\mu}\Psi_j$ satisfies
\[
   |t_\mu(x,x)|\leq \kappa^m c (1+|x|)^{-\sum_{j=1}^m\alpha_j},  
   \quad x\in \mathbb{R}^n, \; \mu\in \Sigma_{\mu_0,\theta}.
\] 
\end{corollary}
\begin{proof}
  Choose $\tau>0$ as the minimum of $\tau$'s according to Theorem \ref{thm:7.18} with $\kappa=1$ and Lemma \ref{lem:discussion of Neumann} with $\beta=\frac{1}{2}$. Let $\Psi_1,\ldots,\Psi_m$, 
  $\eta_1,\ldots,\eta_m$, $m$ as in Corollary \ref{cor:7.19}, and let $\kappa'>\kappa$. Choose $\tilde \Psi_j\in C^\infty_b(\mathbb{R}^n;[0,\infty))$ with 
  \[
   |\Psi_j(x)|\leq \tilde\Psi_j(x) \leq \kappa' (1+|x|)^{-\alpha_j},  
   \quad x\in \mathbb{R}^n, \; j\in\{1,\ldots,m\}.
  \]
  Then, by Theorem \ref{thm:7.18}, there exists $c>0$ with 
  \[
    |r_{\eta_j+\mu}(x,y)|\leq c r_{\Re(\mu)}(x-y), \quad x,y\in \mathbb{R}^n, \; x\neq y,
    \; \mu\in \Sigma_{\mu_0,\theta}. 
  \]
  Hence, for $x\in \mathbb{R}^n$ and $\mu\in \Sigma_{\mu_0,\theta}$ one obtains 
  \[
     |((r_{\eta_1+\mu}*\Psi_1)\cdots (r_{\eta_m+\mu}*\Psi_m))(x,x)|
      \leq c^m \big((r_{\Re(\mu)}*\tilde \Psi_1)\cdots (r_{\Re(\mu)}*\tilde \Psi_m)\big)(x,x).
  \]
 Thus, the assertion follows from Lemma \ref{lem:asymptotics on diagonal}.
\end{proof}

 \begin{remark}\label{r:7.10} A result similar to Corollary \ref{cor:7.19} holds if for some index 
 $j\in\{1,\ldots,m\}$, the operator $R_{\eta_j+\mu}$ is replaced by $\partial_\ell R_{\eta_j+\mu}$ for some $\ell\in \{1,\ldots,n\}$. For obtaining such a result, one needs a version of Lemma \ref{lem:real_partBessel_derivative} where, in this lemma, the fundamental solution for the Helmholtz equation is replaced by the respective one for $(-\Delta+\eta_j+\mu)u=f$. 
 \hfill $\diamond$  
 \end{remark}
 
 In the rest of this section, we shall establish the remaining estimate needed, to obtain a proof for Remark \ref{r:7.10}. More precisely, we aim for a proof of the following result:

 \begin{theorem}\label{t:7.19c} Let $n\in \mathbb{N}_{\geq 3}$ odd, for 
 $\mu\in \mathbb{C}_{\Re>0}$, let $q_\mu$ as in Lemma \ref{lem:real_partBessel_derivative}, $\mu_0>0$, $\theta\in (0,\pi/2)$, $\kappa>0$. Then there exists $c\geq1$ and $\tau>0$ such that for all $j\in \{1,\ldots,n\}$, $\eta\in C_b^\infty(\mathbb{R}^n)$, $\|\eta\|_{L^\infty}\leq\kappa$, with $\supp (\eta) \subset B(0,\tau)$, and $\mu\in \Sigma_{\mu_0,\theta}$, we have for all $x,y\in \mathbb{R}^n$, $x\neq y$,
\[
   |\partial_j (\xi\mapsto r_{\eta+\mu}(\xi,y))(x)|\leq  c \, q_{\Re (\mu)}(|x-y|)
\]
with $r_{\eta+\mu}$ denoting the integral kernel of $R_{\eta+\mu}=(-\Delta+\eta+\mu)^{-1}$, the latter being given by \eqref{def:R_psi_z}.
\end{theorem}

The proof of Theorem \ref{t:7.19c} will follow similar ideas as the one for Theorem \ref{thm:7.18}. We 
start with the following result:

\begin{theorem}\label{thm:7.19ba} Let $n\in\mathbb{N}_{\geq3}$, $k\in \bbN$, $k < n$, $\mu>0$. Then the operator 
\[
   L^2(\mathbb{R}^n)\ni \psi \mapsto \bigg( x\mapsto \int_{\mathbb{R}^n} \frac{e^{-\mu |x-y|}}{|x-y|^k} \psi(y) \, d^ny\bigg)\in L^2(\mathbb{R}^n)
\]
is well-defined, bounded, and positive definite. 
\end{theorem}
\begin{proof}
 The operator is well-defined and bounded by Young's inequality together with the observation that $f\colon x\mapsto e^{-\mu |x|} |x|^{-k}$ is an $L^1(\mathbb{R}^n)$-function. Moreover, for $\epsilon>0$ 
 we set
 \[
    \phi_\epsilon \colon [0,\infty) \to \mathbb{R}, \quad r\mapsto \frac{e^{-\mu r}}{(r+\epsilon)^k}.
 \]
  Then $\phi_\epsilon$ is a completely monotone function, since the maps $r\mapsto e^{-\mu r}$ and $r\mapsto (r+\epsilon)^{-k}$ are completely monotone. Observing that $\phi_\epsilon(r)\to 0$ as $r\to \infty$ and using the criterion on positive definiteness in  \cite[Theorem 2]{Tr89} one infers that $\phi_\epsilon(|\cdot|)*$ is a positive semi-definite operator. Moreover, since $\phi_\epsilon\to\phi$ in $L^1(\mathbb{R}^n)$ as $\epsilon\to 0$, one gets that 
$\phi_\epsilon(|\cdot|) *\to \phi(|\cdot|) *$ in $\mathcal B(L^2(\mathbb{R}^n))$ as $\epsilon\to 0$. Hence, for all $\psi\in L^2(\mathbb{R}^n)$ one infers 
  \[
     0\leq \lim_{\epsilon\to 0} \big( \phi_\epsilon * \psi,\psi \big)_{L^2(\bbR^n)} 
     = \big( \phi * \psi,\psi \big)_{L^2(\bbR^n)}.\qedhere
  \]
\end{proof}

\begin{corollary}\label{c:7.12} Let $n\in \mathbb{N}_{\geq 3}$, $k\in \bbN$, $k < n$, $\mu_0,\mu_1>0$. Denote $q\colon \mathbb{R}^n\backslash \{0\}\ni x\mapsto e^{-\mu_1 |x|} |x|^{-k}$. Then for all 
$\mu \geq \mu_0$, 
 \begin{equation}\label{eq:7.19c.2}
    \mu_0\int_{\mathbb{R}^n} r_{\mu}(x-x_1)q(x_1-x) \, d^nx_1 \leq q(x-y), \quad 
    x,y\in \mathbb{R}^n, \; x\neq y,
 \end{equation}
 where $r_\mu$ is the integral kernel for $(-\Delta+\mu)^{-1}\in \mathcal{B}(L^2(\mathbb{R}^n))$.
\end{corollary}
\begin{proof}
 By Theorem \ref{thm:7.19ba}, $q$ satisfies the assumptions in Proposition \ref{p:7.7}, implying inequality \eqref{eq:7.19c.2}.
\end{proof}

We conclude with the proof of Theorem \ref{t:7.19c}, yielding the proof of Remark \ref{r:7.10}.

\begin{proof}[Proof of Theorem \ref{t:7.19c}]
 Choose $\tau>0$ such that $\|R_{\mu}\eta\|\leq 1/2$ for all $\eta\in L^\infty(\mathbb{R}^n)$, $\|\eta\|_{L^\infty}\leq \kappa$, with $\supp (\eta) \subset B(0,\tau)$, and $\mu\in \Sigma_{\mu_0,\theta}$, as permitted by Lemma \ref{lem:discussion of Neumann}. Next, let $j\in\{1,\ldots,n\}$ and recall
 \begin{align}
    \partial_j R_{\eta+\mu} &= \partial_j \sum_{k=0}^\infty R_{\mu}\big((-\eta)R_{\mu}\big)^k\notag
     \\ & = \partial_j R_{\mu} + \partial_j R_{\mu} (-\eta)R_{\mu} \sum_{k=1}^\infty \big((-\eta)R_{\mu}\big)^{k-1}\notag
    \\ & = \partial_j R_{\mu} + \partial_j R_{\mu} (-\eta) \sum_{k=0}^\infty R_{\mu}\big((-\eta)R_{\mu}\big)^{k}\notag
     \\ & = \partial_j R_{\mu} - \partial_j R_{\mu} (\eta) R_{\eta+\mu}.\label{eq:7.19c.1}
 \end{align}
 Let $q_\mu$ be as in Lemma \ref{lem:real_partBessel_derivative}. Upon appealing to Lemma \ref{lem:real_partBessel_derivative} (see, in particular, inequalities \eqref{der_a} and \eqref{der_b}), one is  left with estimating the integral kernel associated with the second summand in \eqref{eq:7.19c.1}, which we denote by $t$. Using Theorem \ref{thm:7.18} and Lemma \ref{lem:real_partBessel_derivative}, \eqref{der_b}, there exists $c_1\geq1$ such that 
 \[        |r_{\eta+\mu}(x,y)|\leq c_1 r_{\Re (\mu)/4}(x-y) \text{ and }q_{\mu}(|x-y|)\leq 
 c_1 q_{\Re (\mu)}(|x-y|)
 \]                       
  for all $\mu\in \Sigma_{\mu_0,\theta}$ and $x,y\in \mathbb{R}^n$, $x\neq y$. Thus, for all $x,y\in \mathbb{R}^n$, $x\neq y$, $\mu\in \Sigma_{\mu_0,\theta}$, one gets with the help of \eqref{eq:7.19c.2} (using that $q_{\Re(\mu)}(|\cdot|)$ is a nonnegative linear combination of functions discussed in Corollary \ref{c:7.12}), 
\begin{align*}
    |t(x,y)|& = \big| \partial_j r_{\mu}*(\eta)r_{\eta+\mu}(x,y)\big|
     \\     &\leq c_1 \int_{B(0,\tau)} q_{\mu}(|x-x_1|)|\eta(x_1)|r_{\Re (\mu)/4}(x_1-y) \, d^n x_1
     \\  &\leq \|\eta\|_{L^\infty} c_1^2 \int_{\mathbb{R}^n} q_{\Re (\mu)}(|x-x_1|) 
     r_{\Re (\mu)/4}(x_1-y) \, d^n x_1
     \\ &\leq  \|\eta\|_{L^\infty} \frac{4c_1^2}{\mu_0} q_{\Re (\mu)}(|x-y|).\qedhere
\end{align*}
\end{proof}

\newpage

\section{The proof of Theorem \ref{thm:Perturbation}: The Smooth Case} \lb{s12}

In this section, we treat Theorem \ref{thm:Perturbation} for the particular case of $C^\infty$-potentials\footnote{We note that this section may explain the reasoning underlying the last lines on 
\cite[p.~226]{Ca78}. }. Let $n\in \mathbb{N}_{\geq 3}$ odd, 
$\Phi\in \mathbb{C}_b^\infty\big(\mathbb{R}^n;\mathbb{C}^{d\times d}\big)$ for some $d\in \mathbb{N}$. Assume that $\Phi(x)=\Phi(x)^*$ and that for some $c>0$ and $R>0$ one has the strict positive definiteness condition $|\Phi(x)|\geq cI_d$ 
for all $x\in \mathbb{R}^n\backslash  B(0,R)$. With the operator $L=\mathcal{Q}+\Phi$ as in \eqref{eq:def_of_L(2)}, we proceed as follows: At first, we show that if $(\mathcal{Q}\Phi)(x)\to 0$, $|x|\to\infty$, then $L$ is a Fredholm operator (Lemma \ref{lem:1L_Fred_Rge0}). Next, if one defines $U \in C_b^\infty\big(\mathbb{R}^n;\mathbb{C}^{d\times d}\big)$ to coincide with $\sgn (\Phi)$ on $\mathbb{R}^n\backslash  B(0,R)$ as in Lemma \ref{lem:almost admissible}, we show that the operator $\mathcal{Q}+U$ is also a Fredholm operator with the same index (Theorem \ref{thm:3L_Fred_Rge0}). Moreover, in this theorem, we shall also show that changing $U$ to be unitary everywhere but on a small ball around $0$ will not change the index. As this ball may be chosen arbitrarily small, we are in the position to proceed with a similar strategy to derive the index as in Section \ref{sec:The-Derivation-of-trace-f} and use the results from Section \ref{sec:pert}. In that sense, the following may also be considered as a first attempt for a perturbation theory for the generalized Witten index introduced at the end of Theorem \ref{thm:index with Witten}.

We start with the Fredholm property for the operator considered in Theorem \ref{thm:Perturbation} with smooth potentials (see also Theorem \ref{thm:Fredholm_property}).

\begin{lemma}\label{lem:1L_Fred_Rge0} Let $n,d\in \mathbb{N}$, $L= \cQ+\Phi$ as in \eqref{eq:def_of_L(2)}, with $\Phi\in C_b^\infty\big(\mathbb{R}^n; \mathbb{C}^{d\times d}\big)$ and 
$\Phi(x)=\Phi(x)^*$, $x\in \mathbb{R}^n$. Assume that $C(x)\coloneqq\left( \cQ\Phi\right)(x)\to 0$ 
as $|x|\to\infty$ $($see also \eqref{eq:commutator=00003DC}$)$, and that there exist $c>0$ and $R>0$ 
such that with $|\Phi(x)|\geq cI_{d}$ for all $x\in \mathbb{R}^n\backslash  B(0,R)$. Then $L$ is a Fredholm operator. 
\end{lemma}
\begin{proof}
One recalls from Proposition \ref{prop:compu_lstartl} that $L^*L=-\Delta - C+ \Phi^2$ and $LL^*=-\Delta + C+\Phi^2$. The latter two operators are $\Delta$-compact perturbations of $-\Delta+\Phi^2$ due to $C(x)\to 0$ as $|x|\to\infty$ and Theorem \ref{thm:Newton_pot_is_h1-bdd}. 
Next, since $-\Delta+\Phi^2+c^2\chi_{B(0,R)}\geq -\Delta + c^2$, the operator $-\Delta+\Phi^2+c^2\chi_{B(0,R)}$ is continuously invertible. But, $-\Delta+\Phi^2+c^2\chi_{B(0,R)}$ is also a $\Delta$-compact perturbation of $-\Delta+\Phi^2$. Thus, by the invariance of the Fredholm property under relatively compact perturbations, one concludes the Fredholm property for $-\Delta+\Phi^2$ and thus the same for $L^*L$ and $LL^*$.
\end{proof}

As a corollary, we obtain the assertion that one might also consider potentials being pointwise unitary outside large balls. In this context, we refer the reader also to the beginning of Section \ref{sec:The Index theorem}. One notes that also Theorem \ref{thm:Perturbation} hints in the same direction as in the index formula only the sign of the potential occurs.

\begin{theorem}\label{thm:3L_Fred_Rge0}  Let $n,d\in \mathbb{N}$, $L= \cQ+\Phi$ as in \eqref{eq:def_of_L(2)} with $\Phi\in C_b^\infty\big(\mathbb{R}^n; \mathbb{C}^{d\times d}\big)$ 
with $\Phi(x)=\Phi(x)^*$, $x\in \mathbb{R}^n$. Assume that there exist $c>0$ and $R>0$ such that  $|\Phi(x)|\geq cI_{d}$ for all $x\in \mathbb{R}^n\backslash  B(0,R)$ and $(\mathcal{Q}\Phi)(x)\to 0$ as $|x|\to \infty$. If $U \in C_b^\infty\big(\mathbb{R}^n;\mathbb{C}^{d\times d}\big)$ pointwise self-adjoint with $\sgn (\Phi(x)) = U(x)$ for all $x\in \mathbb{R}^n$ with $|x|\geq R'$ for some $R' \geq R$, 
then $\tilde L\coloneqq \cQ + U$ is Fredholm and $\ind (L) = \ind \big(\tilde L\big)$.
\end{theorem}
\begin{proof}
 From Lemma \ref{lem:almost admissible}, one gets $(\mathcal{Q} U)(x)=(\mathcal{Q}\sgn(\Phi))(x)\to 0$ as $|x|\to\infty$. Hence, by Lemma \ref{lem:1L_Fred_Rge0} the operators $L=\cQ+\Phi$ and $\tilde L=\cQ+U$ are Fredholm (for $\tilde L$ one observes that $U(x)^2=I_d$ for all $|x|\geq R'$). Next, the operator family $[0,1]\ni \lambda \mapsto 
 \cQ+(1-\lambda)U + \lambda \min\{c/2, 1/2\} U$ defines a homotopy from $\tilde L$ to 
 $\cQ+\min\{c/2, 1/2\} U$, which is a homotopy of Fredholm operators as 
 $(1-\lambda)+\lambda \min\{c/2,1/2\}=1-\lambda(1-\min\{c/2,1/2\})\geq \min\{c/2,1/2\}>0$ 
 for all $\lambda\in (0,1)$ and hence Lemma \ref{lem:1L_Fred_Rge0} applies. The rest of the proof is concerned with showing that $[0,1]\ni \lambda \mapsto \cQ+(1-\lambda)\Phi+\lambda \min\{c/2,1/2\} U$ defines a homotopy of Fredholm operators. Employing Lemma \ref{lem:1L_Fred_Rge0}, it suffices to show that for some $\tilde c>0$, 
 \[
    \big[(1-\lambda)\Phi(x)+\lambda \min\{c/2,1/2\} U(x)\big]^2\geq \tilde cI_d
 \]
for all $x\in \mathbb{R}^n\backslash  B(0,R')$ and $\lambda\in [0,1]$. By the spectral theorem for symmetric $d\times d$-matrices it suffices to show that for real numbers $\alpha\in \mathbb{R}$ with $\alpha^2 \geq c^2$, one has  for some $\tilde c>0$, 
 \[
    \big[(1-\lambda)\alpha+\lambda \min\{c/2, 1/2\}\sgn(\alpha)\big]^2\geq \tilde c.
 \]
But since 
 \[
 \big[(1-\lambda)\alpha+\lambda \min\{c/2, 1/2\}\sgn(\alpha)\big]^2 
 = \big[(1-\lambda)|\alpha|+\lambda \min\{c/2, 1/2\}\big]^2,
 \]
 it remains to observe that 
 \[
    (1-\lambda)|\alpha|+\lambda \min\{c/2,1/2\}\geq 
    (1-\lambda) c+\lambda \min\{c/2,1/2\}\geq \min\{c/2,1/2\}.\qedhere
 \]
\end{proof}

We remark that the assumptions in Theorem \ref{thm:3L_Fred_Rge0} can be met, using Lemma \ref{lem:almost admissible}. This, and \cite{Ca78} motivates the following notion of ``Callias admissibility''. 

\begin{definition} \label{def:almost_admissible} Let $n,d\in \mathbb{N}$. We say that a map $\Phi\in C_b^\infty\big(\mathbb{R}^n;\mathbb{C}^{d\times d}\big)$ is called \emph{Callias admissible}, if the following conditions $(i)$--$(iii)$ are satisfied: \\[1mm] 
$(i)$ $\Phi(x)=\Phi(x)^*$, $x\in \mathbb{R}^n$. \\[1mm]  
$(ii)$ There exists $R>0$ such that $\Phi(x)$ is unitary for all $x\in \mathbb{R}^n\backslash  B(0,R)$.
\\[1mm] 
$(iii)$ There exists $\epsilon> 1/2$ such that for all $\alpha\in \mathbb{N}_0^n$, there is $\kappa\geq 0$ with
 \[
    \|\partial^\alpha \Phi(x)\|\leq \kappa \begin{cases}
                                              (1+|x|)^{-1}, & |\alpha|=1,\\[1mm] 
                                              (1+|x|)^{- 1 - \epsilon}, & |\alpha|\geq2,
\end{cases}   \quad x\in \mathbb{R}^n.
 \]
\end{definition}

\begin{remark}\label{rem:almo_ad} Let 
$\Phi\in C^\infty_b\big(\mathbb{R}^n; \mathbb{C}^{d\times d}\big)$ be Callias admissible. By 
Theorem \ref{thm:3L_Fred_Rge0}, for any potential 
$U \in C^\infty_b\big(\mathbb{R}^n; \mathbb{C}^{d\times d}\big)$ coinciding with $\sgn(\Phi)$ on large balls, the   operators $L=\mathcal{Q}+\Phi$ and $\tilde L =\mathcal{Q}+U$ are Fredholm with the same index. Moreover, by Theorem \ref{thm:properties of sign function} and the unitarity of $\Phi$ on large balls, one infers that on large balls $U = \Phi$. Next, by Lemma \ref{lem:almost admissible}, we may choose $U$ to be unitary everywhere but on a small ball centered at $0$. In addition, we can choose $U$ such that 
\begin{equation*}
U(x)^2= u(x)I_d, \quad x\in \mathbb{R}^n,  
\end{equation*}
with $u \in C^\infty(\mathbb{R}^n;[0,1])$ and $u = 1$ on $\mathbb{R}^n\backslash  B(0,\tau)$ for every chosen $\tau>0$. For that reason, in order to compute the index for $L=\mathcal{Q}+\Phi$, our main focus only needs to be potentials with the properties of $U$ discussed here, and then one can employ the results of Section \ref{sec:pert}. ${}$ \hfill $\diamond$ 
\end{remark}

Remark \ref{rem:almo_ad} leads to the following definition:
  
\begin{definition}[$\tau$-admissibility]\label{d:ta} Let $n,d\in \mathbb{R}^n$, $\tau>0$, 
  $U \in C^\infty_b\big(\mathbb{R}^n;\mathbb{C}^{d \times d}\big)$. We say that $U$ is \emph{admissible on $\mathbb{R}^n\backslash  B(0,\tau)$}, in short, \emph{$\tau$-admissible}, 
if $U$ is Callias admissible and there exists $u \in C^\infty(\mathbb{R}^n;[0,1])$ satisfying 
\begin{equation}\label{e:r.aa} 
U(x)^2 = u(x)I_d, \quad x\in \mathbb{R}^n,  
\end{equation} 
with the property that $u = 1$ on $\mathbb{R}^n\backslash B(0,\tau)$. 
\end{definition}

To get a first impression of the difference between the notions of admissibility (see Definition \ref{def:phi_admissible}) and $\tau$-admissibility, we will compute the resolvent differences of the resolvents of $L^*L$ and $LL^ *$ in Proposition \ref{prop:resolvent formulas2}, with 
$L=\mathcal{Q}+U$ given as in \eqref{eq:def_of_L(2)} for some $\tau$-admissible $U$ with $u$ as in \eqref{rem:almo_ad}. First, we note that by Proposition \ref{prop:compu_lstartl}, one has
\[
   L^*L = -\Delta I_{2^{\hatt n}d} + C + u I_{2^{\hatt n}d},
\]
with $C=(\mathcal{Q} U)$. In Section \ref{sec:The-Derivation-of-trace-f}, and in particular in Proposition \ref{prop:resolvent formulas}, we discussed the resolvent $(L^*L+z)^{-1}$ in terms of $R_{1+z}=(-\Delta I_{2^{\hatt n}d}+ (1+z))^{-1}$. The latter operator needs to be replaced by the following (see also \eqref{def:R_psi_z})
\begin{align}\label{def:R_psi_z2}
  R_{u+z}&\coloneqq (-\Delta I_{2^{\hatt n}d} + u +z)^{-1}\\ \notag
   &\, =\left(-\Delta I_{2^{\hatt n}d} + (u-1) I_{2^{\hatt n}d} + (z+1)I_{2^{\hatt n}d}\right)^{-1}\\ \notag
   &\, =\sum_{k=0}^\infty (R_{1+z}(u-1) I_{2^{\hatt n}d})^kR_{1+z},
\end{align}
provided the latter series converges. As already discussed in Lemma \ref{lem:almost admissible}, this can be ensured if $\tau$ is chosen small enough. Thus, for this pupose, we shall fix the parameters according to the results in Section \ref{sec:pert}:

 \begin{hypothesis}\label{hyp:parameters} Let $n=2\hatt n+1$, $\hatt n\in \mathbb{N}_{\geq1}$,
 \begin{equation}\label{eq:delta_0}
    \delta_0 \in (-1,0), \quad \theta\in (0,\pi/2).
 \end{equation}
 For $\mu_0\coloneqq \delta_0+1$ let $\tau_{\ref{lem:discussion of Neumann}}$ as in Lemma \ref{lem:discussion of Neumann} for $\beta= 1/2$, $\tau_{\ref{t:7.19c}}$ as in Theorem \ref{t:7.19c} for $\kappa=1$,  $\tau_{\ref{thm:7.18}}$ as in Theorem \ref{thm:7.18} for $\kappa=1$, and $\tau_{\ref{cor:7.19}}$ as in Corollary \ref{cor:7.19}. Define 
 \begin{equation}\label{e:tau}
    \tau\coloneqq \min\{\tau_{\ref{lem:discussion of Neumann}}, \tau_{\ref{thm:7.18}}, \tau_{\ref{cor:7.19}} ,\tau_{\ref{t:7.19c}} \}.
 \end{equation}
\end{hypothesis}
 
 As mentioned already, for $\tau$-admissible potentials, we shall derive the index theorem similarly to the derivation for admissible potentials. More precisely, at first, we will focus on computing the trace of $\chi_\Lambda B_L(z)$, as in Theorem \ref{thm:Witten_reg_n5}. We note that the following parallels the Section \ref{sec:The-Derivation-of-trace-f}. 

To start, we need to state a result similar to Proposition \ref{prop:resolvent formulas}. In fact, using the expressions in \eqref{eq:almost_Neumann_llstar} and \eqref{eq:almost_Neumann_lstarl}, with $R_{1+z}$ replaced by $R_{u+z}$ (see \eqref{def:R_psi_z}), even the proof turns out to be the same. 

\begin{proposition}
\label{prop:resolvent formulas2} Assume Hypothesis \ref{hyp:parameters}, let $z\in \Sigma_{\delta_0,\theta}$, and suppose 
$U \in C^\infty_b\big(\mathbb{R}^n;\mathbb{C}^{d\times d}\big)$ is $\tau$-admissible $($cf.\ Definition \ref{d:ta}$)$, with $u$ as in \eqref{e:r.aa}$)$. We recall that 
$L= \cQ+U$ as in \eqref{eq:def_of_L(2)}, $C=\left( \cQ U\right)$ in  \eqref{eq:commutator=00003DC}, and $R_{u+z}$ in \eqref{def:R_psi_z2}. 
If, in addition, $z\in\rho(-L^{*}L)\cap\rho(-LL^{*})$, then
for all $N\in\mathbb{N}$, 
\begin{align*}
& \left(L^{*}L+z\right)^{-1}-\left(LL^{*}+z\right)^{-1}\\
 & \quad =2\sum_{k=0}^{N}R_{u+z}\left(CR_{u+z}\right)^{2k+1}+\big(\left(L^{*}L+z\right)^{-1}-
 \left(LL^{*}+z\right)^{-1}\big)\left(CR_{u+z}\right)^{2N+2}\\
 & \quad =2\sum_{k=0}^{N}R_{u+z}\left(CR_{u+z}\right)^{2k+1}+\big(\left(L^{*}L+z\right)^{-1}+
 \left(LL^{*}+z\right)^{-1}\big)\left(CR_{u+z}\right)^{2N+3}, 
\end{align*}
and
\begin{align*}
 & \left(L^{*}L+z\right)^{-1}+\left(LL^{*}+z\right)^{-1}\\
 & \quad =2\sum_{k=0}^{N}R_{u+z}\left(CR_{u+z}\right)^{2k} 
 + \big(\left(L^{*}L+z\right)^{-1}+\left(LL^{*}+z\right)^{-1}\big)\left(CR_{u+z}\right)^{2N+2}.
\end{align*}
\end{proposition}

Next, we formulate the variant of Lemma \ref{lem:almost_Neumann_series}:

\begin{lemma}
\label{lem:almost_Neumann_series2}  Assume Hypothesis \ref{hyp:parameters}, $z\in \Sigma_{\delta_0,\theta}$, let  $U \in C^\infty_b\big(\mathbb{R}^n;\mathbb{C}^{d\times d}\big)$ be $\tau$-admissible $($cf.\ Definition \ref{d:ta}$)$, with $u$ as in \eqref{e:r.aa}. Let $L= \cQ+U$ be given by \eqref{eq:def_of_L(2)} and $z\in \Sigma_{\delta_0,\theta}\cap\rho\left(-L^{*}L\right)\cap\rho\left(-LL^{*}\right)$. We recall $B_{L}(z)$, $J_{L}^{j}(z)$, and $A_{L}(z)$ given by \eqref{eq:def_of_BL(z)2},
\eqref{eq:JJJ}, and \eqref{eq:ALz} $($with $\Phi$ replaced by $U$$)$, respectively, as well 
as $R_{u+z}$
given by \eqref{def:R_psi_z2}. Then the following assertions hold:  
\begin{align*} 
2 B_{L}(z) &=\sum_{j=1}^n \big[\partial_{j},J_{L}^{j}(z)\big]+A_{L}(z),  \\
&=z\tr_{2^{\hat n}d}\big(2(R_{u+z}C)^{n}R_{u+z}+\big((L^{*}L+z)^{-1}-(LL^{*}+z)^{-1}\big)
(CR_{u+z})^{n+1}\big),   
\end{align*}
with
\begin{align*}
J_{L}^{j}(z) & =2\tr_{2^{\hat n}d}\big(\gamma_{j,n} \cQ (R_{u+z}C)^{n-2}R_{u+z}\big) 
+ 2\tr_{2^{\hat n}d}\big(\gamma_{j,n} U (R_{u+z}C)^{n-1}R_{u+z}\big)    \\
 & \quad+\tr_{2^{\hat n}d}\big(\gamma_{j,n} \cQ \big((L^{*}L+z)^{-1}+(LL^{*}+z)^{-1}\big) 
 (CR_{u+z})^{n}\big)   \\
& \quad + \tr_{2^{\hat n}d} \big(\gamma_{j,n} U \big((L^{*}L+z)^{-1}
 +(LL^{*}+z)^{-1}\big) (CR_{u+z})^{n}\big),    \\
 & \hspace*{6.8cm}  j\in\{1,\ldots,n\},   
\end{align*}
and 
\begin{align*}
A_{L}(z) &=\tr_{2^{\hat n}d} \big(\big[U, U\big(2(R_{u+z}C)^{n}R_{u+z}+\big((L^{*}L+z)^{-1}-(LL^{*}+z)^{-1}\big)  \\ 
& \hspace*{1.4cm} \times (CR_{u+z})^{n+1}\big)\big]\big)    \\  
& \quad-\tr_{2^{\hat n}d} \big(\big[U, \cQ \big(2 (R_{u+z}C)^{n-1}R_{u+z}+\big((L^{*}L+z)^{-1}-(LL^{*}+z)^{-1}\big)    \\ 
& \hspace*{1.7cm} \times (CR_{u+z})^{n}\big)\big]\big).   
\end{align*} 
\end{lemma}
\begin{proof}
The proof follows line by line those of Lemma \ref{lem:almost_Neumann_series}, observing that $R_{u+z}$ commutes with $\gamma_{j,n}$, $j\in\{1,\ldots,n\}$.
\end{proof}

\begin{remark} For even space dimensions $n$ -- as in Lemma \ref{lem:almost_Neumann_series} -- the corresponding operator $B_L(z)$ also vanishes for all $z\in \rho(-L^*L)\cap \rho(-LL^*)$. That is why we will disregard even space dimensions from now on.  
\hfill $\diamond$
\end{remark}

The proof of the variant of Theorem \ref{thm:trisbounded} is slightly more involved:

\begin{theorem}
\label{thm:trisbounded2}  Assume Hypothesis \ref{hyp:parameters}, $z\in \Sigma_{\delta_0,\theta}$. Let $U \in C^\infty_b\big(\mathbb{R}^n;\mathbb{C}^{d\times d}\big)$ be $\tau$-admissible $($cf.\ Definition \ref{d:ta}$)$, with $u$ as in \eqref{e:r.aa}. Let $L= \cQ+U$ be given by \eqref{eq:def_of_L(2)}. Then there exists $\delta_0 \leq \delta<0$, such that for all $z\in \Sigma_{\delta,\theta}\cap\rho\left(-L^{*}L\right)\cap\rho\left(-LL^{*}\right)$ and $\Lambda>0$, the operator $\chi_\Lambda B_{L}(z)$, with $B_{L}(z)$ given by \eqref{eq:def_of_BL(z)2}, is trace
class with $z\mapsto\tr (|\chi_\Lambda B_{L}(z)|)$ bounded on $B(0,|\delta|)\backslash \{0\}$. Moreover, the trace of $\chi_\Lambda B_{L}(z)$ may be computed as the integral over the diagonal of the corresponding integral kernel.
\end{theorem}
\begin{proof} It suffices to observe that if $\eta\in L^{n+1}(\mathbb{R}^n)$, one has 
$R_{u+z}\eta \in \mathcal{B}_{n+1}(L^2(\mathbb{R}^n))$, with 
\[
   \|R_{u+z}\eta\|_{\mathcal{B}_{n+1}}\leq 2 \|R_{1+z}\eta\|_{\mathcal{B}_{n+1}}.
\]
Indeed, from 
\[
 R_{u+z}\eta=\sum_{k=0}^\infty (R_{1+z}(u-1))^kR_{1+z}\eta,
\] 
the ideal property, Hypothesis \ref{hyp:parameters}, and \eqref{def:R_psi_z}, it follows that
\[
  \|R_{u+z}\eta\|_{\mathcal{B}_{n+1}}=\sum_{k=0}^\infty \|(R_{1+z}(u - 1))^k\|_{\mathcal{B}_{\infty}} \|R_{1+z}\eta\|_{\mathcal{B}_{n+1}}\leq 2\|R_{1+z}\eta\|_{\mathcal{B}_{n+1}}.
\]
The rest of the proof of the trace class property follows literally that of Theorem \ref{thm:trisbounded}. The assertion concerning the computation of the trace rests on 
Remark \ref{r:5.8}, which applies in this context.
\end{proof}

The variant of Lemma \ref{lem:j,A,Green} with $L=\cQ+U$ instead of $L=\cQ+\Phi$ for some $\tau$-admissible $U$ need not be stated again as it only contains a statement about the regularity of the integral kernels of $J_L^j(z)$ and $A_L(z)$ (see Lemma \ref{lem:almost_Neumann_series2}), $j\in \{1,\ldots,n\}$. Its proof, however, varies slightly from that of Lemma \ref{lem:j,A,Green} in the sense that $R_{1+z}$ should be replaced by $R_{u+z}$ and $\Phi$ by $U$. In addition, we recall Remark \ref{rem:remark on op norm of Rpsiz}\,$(ii)$ to the effect that the application of $R_{u+z}$ increases weak differentiability by two units.

For the proof of Lemma \ref{lem:boundedness_of_leading_new_witten}, we extensively used 
that $\cQ$ commutes with $R_{1+z}$. However, on notes that $\cQ$ does not commute with $R_{u+z}$. In fact, one has  
\[
   [R_{u+z},\cQ]=R_{u+z}\cQ-\cQ R_{u+z}=R_{u+z}(Q u)R_{u+z},
\]
recalling our convention to denote the operator of multipliying with the function 
$x\mapsto (Q u)(x)$ by $(Q u)$. Due to this lack of commutativity, the proof of the analog to 
Lemma \ref{lem:boundedness_of_leading_new_witten} is more involved and expanding the resolvent $R_{u+z}$ in the way done in \eqref{def:R_psi_z2}, the terms discussed in Lemma \ref{lem:boundedness_of_leading_new_witten} turn out to be the leading terms in a power series expression:  

\begin{lemma}\label{lem:boundedness_of_leading_new_witten2}
 Assume Hypothesis \ref{hyp:parameters}, let $z\in \Sigma_{\delta_0,\theta}$, and suppose that 
  $U \in C^\infty_b\big(\mathbb{R}^n;\mathbb{C}^{d\times d}\big)$ is $\tau$-admissible $($cf.\ Definition \ref{d:ta}$)$, with $u$ as in \eqref{e:r.aa}, and $C=(\cQ U)$. Let $L= \cQ+U$ be given by \eqref{eq:def_of_L(2)} and $\chi_\Lambda$ as in \eqref{eq:def_of_chi}, $\Lambda>0$. 
 For $z\in \Sigma_{\delta_0,\theta}$, $\Lambda>0$, define
\[
   \xi_\Lambda (z) \coloneqq \chi_\Lambda \tr_{2^{\hat n}d} \big(\left[ \cQ, U\left(CR_{u+z}\right)^{n}\right]\big)
\]
 and 
\[
\tilde\xi_\Lambda(z)\coloneqq \chi_\Lambda \tr_{2^{\hat n}d} \big(\left[ \cQ, \cQ\left(CR_{u+z}\right)^{n}\right]\big). 
\]
Then for all $z\in \Sigma_{\delta_0,\theta}$, the operators $\xi_{\Lambda}(z)$, $\tilde \xi_{\Lambda}(z)$ are trace class and the families 
\[
\{z\mapsto \tr_{L^2(\bbR^n)} (\xi_{\Lambda}(z))\}_{\Lambda>0} \, \text{ and } \, 
\big\{z\mapsto \tr_{L^2(\bbR^n)} \big(\tilde\xi_{\Lambda}(z)\big)\big\}_{\Lambda>0}    \]
 are locally bounded $($cf.\ \eqref{d:normal}$)$.
\end{lemma}
\begin{proof}
As in the proof of Lemma \ref{lem:boundedness_of_leading_new_witten}, we start out with $\xi_\Lambda (z)$ and observe with \eqref{def:R_psi_z}, 
\begin{align*}
\xi_\Lambda (z) &= \chi_\Lambda \tr_{2^{\hat n}d} \big(\left[ \cQ, U \left(CR_{u+z}\right)^{n}\right]\big) \\
 &= \chi_\Lambda \tr_{2^{\hat n}d} \bigg(\bigg[ \cQ, U \bigg(C\sum_{k=0}^\infty (R_{1+z}(u-1))^kR_{1+z}\bigg)^{n}\bigg]\bigg)
 \\ & = \chi_\Lambda \tr_{2^{\hat n}d} \bigg(\bigg[ \cQ, U \sum_{k=0}^\infty \sum_{\substack{0\leq k_1,\ldots,k_n\leq k\\k_1+\ldots+k_n=k}}
\bigg(C(R_{1+z}(u-1))^{k_1}R_{1+z}C    \\
& \hspace*{2.1cm} \times 
(R_{1+z}(u-1))^{k_2} R_{1+z}\cdots C(R_{1+z}(u-1))^{k_n}R_{1+z}\bigg)\bigg]\bigg)
 \\ & =\sum_{k=0}^\infty \sum_{\substack{0\leq k_1,\ldots,k_n\leq k\\k_1+\ldots+k_n=k}}\chi_\Lambda \tr_{2^{\hat n}d} \bigg(\bigg[ \cQ, U 
 \\ &\qquad \quad 
 \times\bigg(C(R_{1+z}(u-1))^{k_1}R_{1+z}\cdots C(R_{1+z}(u-1))^{k_n}R_{1+z}\bigg)\bigg]\bigg).
\end{align*}
In the expression for $\xi_\Lambda(z)$ just derived, we note that the summand for $k=0$ has been discussed in Lemma \ref{lem:boundedness_of_leading_new_witten}, so we are left with showing the trace class property for the summands belonging to $k>0$. Moreover, we need to derive an estimate guaranteeing that the sum in the expression for $\xi_\Lambda(z)$ converges in $\mathcal{B}_1$. Let $k\in\mathbb{N}_{\geq1}$ and $k_1,\ldots,k_n\in \mathbb{N}_{\geq 0}$ such that $k_1+\ldots+k_n=k$, and consider
\begin{align} \notag
  &S_{k_1,\ldots,k_n}
  \\ & \quad \coloneqq \chi_\Lambda \tr_{2^{\hat n}d} \bigg(\bigg[ \cQ, U 
\bigg(C(R_{1+z}(u-1))^{k_1}R_{1+z}\cdots C(R_{1+z}(u-1))^{k_n}R_{1+z}\bigg)\bigg]\bigg) \no
\\ \notag & \, \quad = \chi_\Lambda \tr_{2^{\hat n}d} \bigg( \cQ U
\bigg(C(R_{1+z}(u-1))^{k_1}R_{1+z}\cdots C(R_{1+z}(u-1))^{k_n}R_{1+z}\bigg) 
\\ & \, \qquad - U
\bigg(C(R_{1+z}(u-1))^{k_1}R_{1+z}\cdots C(R_{1+z}(u-1))^{k_n}R_{1+z}\bigg)\cQ \bigg).
\label{eq:lem7.13}
\end{align}
Let $j\in \{1,\ldots,n\}$ be the smallest index for which $k_j\geq 1$. Then the first summand in  \eqref{eq:lem7.13} reads
\begin{align}
 \notag T&\coloneqq  \cQ U 
(CR_{1+z})^{j-1}C(R_{1+z}(u-1))^{k_j}R_{1+z}\cdots C(R_{1+z}(u-1))^{k_n}R_{1+z} 
\\ \notag
& \, = \cQ U
(CR_{1+z})^{j-1}CR_{1+z}((u-1)R_{1+z})^{k_j}\cdots (CR_{1+z})((u-1)R_{1+z})^{k_n}
\\  \label{eq:lem7.13.2}
& \, = \cQ U
(CR_{1+z})^{j}((u-1)R_{1+z})^{k_j}\cdots (CR_{1+z})((u-1)R_{1+z})^{k_n}.
\end{align}
From 
\begin{align*}
 \cQ U \left(CR_{1+z}\right)^{j} &= U \cQ \left(CR_{1+z}\right)^{j}+[ \cQ, U]\left(CR_{1+z}\right)^{j} \\
  & = U \bigg(\sum_{\ell=1}^j (CR_{1+z})^{\ell-1}[ \cQ,C]R_{1+z}(CR_{1+z})^{j-\ell}
  +(CR_{1+z})^j \cQ\bigg) \\ 
  & \quad +[ \cQ, U]\left(CR_{1+z}\right)^{j},  
\end{align*}
one infers
\[
 \cQ U \left(CR_{1+z}\right)^{j}\in \mathcal{B}_{(n+1)/j}, 
\]
by Lemma \ref{lem:Schatten-class-1-operator} and the H\"older-type inequality for the Schatten class operators, Theorem \ref{thm:trace-class-crit}. On the right-hand side of \eqref{eq:lem7.13.2}, apart from $\left(CR_{1+z}\right)^{j}$, there are  $n-j$ factors of the form $CR_{1+z}\in \mathcal{B}_{n+1} $. In addition, there is at least one factor $(u-1)R_{1+z}\in \mathcal{B}_{n+1}$, by 
Lemma \ref{lem:Schatten-class-1-operator} and the fact that $(u-1)\in L^{n+1}(\mathbb{R}^n)$ (as $(u-1)$ is bounded and compactly supported). Hence, by the trace ideal property and the choice of the parameters as in Hypothesis \ref{hyp:parameters}, one gets
\[
  \|T\|_{\mathcal{B}_{1}}\leq \|\cQ U\left(CR_{1+z}\right)^{j}\|_{\mathcal{B}_{n+1)/j}}\|CR_{1+z}\|^{n-j} _{\mathcal{B}_{n+1}}\|(u-1)R_{1+z}\|_{\mathcal{B}_{n+1}} 2^{1-k}. 
\]
The second term under the trace sign in the expression for $S_{k_1,\ldots,k_n}$ (see \eqref{eq:lem7.13}) can be dealt with similarly, so there exists $\kappa>0$ independently of $\Lambda>0$, $z\in \Sigma_{\delta_0,\theta}$, and $k \in \bbN$, such that
\[
   \|S_{k_1,\ldots,k_n}\|_{\mathcal{B}_1}\leq \kappa 2^{1-k}.
\]
Hence, for all $\Lambda>0$ and $z\in \Sigma_{\delta_0,\theta}$ one gets
\begin{align} 
  \|\xi_\Lambda(z)\|_{\mathcal{B}_1} &\leq \|\psi_\Lambda(z)\|_{\mathcal{B}_1}+ \sum_{k=1}^\infty \sum_{\substack{0\leq k_1,\ldots,k_n\leq k\\k_1+\ldots+k_n=k}} \|S_{k_1,\ldots,k_n}\|_{\mathcal{B}_1}     \no \\ 
&\leq \|\psi_\Lambda(z)\|_{\mathcal{B}_1}+  \sum_{k=1}^\infty \kappa (k+1)^n 2^{1-k},    \label{eq:lem7.13.3}
\end{align}
where $\psi_\Lambda(z)$ is defined in Lemma \ref{lem:boundedness_of_leading_new_witten}. Inequality \eqref{eq:lem7.13.3} yields the assertion for $\xi_\Lambda$. 

A similar reasoning -- as in Lemma \ref{lem:boundedness_of_leading_new_witten} for $\tilde\psi_\Lambda(z)$ -- applies to $\tilde \xi_\Lambda(z)$.
\end{proof}

Next, we turn to the proof of a modified version of Lemma \ref{lem:boundedness_of_rest_new_witten}:

\begin{lemma}\label{lem:boundedness_of_rest_new_witten2}
 Assume Hypothesis \ref{hyp:parameters}. Let $z\in \Sigma_{\delta_0,\theta}$, and assume that 
 $U \in C^\infty_b\big(\mathbb{R}^n;\mathbb{C}^{d\times d}\big)$ is $\tau$-admissible $($cf.\ Definition \ref{d:ta}$)$, with $u$ as in \eqref{e:r.aa}, $C=(\cQ U)$. Let $L= \cQ+U$ be given by \eqref{eq:def_of_L(2)} and $\chi_\Lambda$ as in \eqref{eq:def_of_chi}, $\Lambda>0$. 
 For $z\in \Sigma_{\delta_0,\theta}$, $\Lambda>0$, define
\[
   \zeta_\Lambda (z) \coloneqq \chi_\Lambda \tr_{2^{\hat n}d} \big(\big[ \cQ, U \big(\left(L^{*}L+z\right)^{-1}-\left(LL^{*}+z\right)^{-1}\big)\left(CR_{u+z}\right)^{n+1}\big]\big)
\]
 and 
\[
\tilde\zeta_\Lambda(z)\coloneqq \chi_\Lambda \tr_{2^{\hat n}d} \big(\big[ \cQ,\big(\cQ \big(\left(L^{*}L+z\right)^{-1}-\left(LL^{*}+z\right)^{-1}\big)\left(CR_{u+z}\right)^{n+1}\big)\big]\big).
\]
Then for all $z\in \Sigma_{\delta_0,\theta}\cap\rho(-L^*L)\cap\rho(-LL^*)$, the operators $\zeta_{\Lambda}(z)$, $\tilde\zeta_{\Lambda}(z)$ are trace class and there exists $\delta\in (\delta_0,0)$ such that the families 
\[
\{\Sigma_{\delta,\theta}\ni z\mapsto z\tr_{L^2(\mathbb{R}^n)} 
(\zeta_{\Lambda}(z))\}_{\Lambda>0} \, \text{ and } \,
\big\{\Sigma_{\delta,\theta}\ni z\mapsto z\tr_{L^2(\mathbb{R}^n)} 
\big(\tilde\zeta_{\Lambda}(z)\big)\big\}_{\Lambda>0}
\]
are locally bounded $($cf.\ \eqref{d:normal}$)$.
\end{lemma}
\begin{proof}
On can follow the proof of Lemma \ref{lem:boundedness_of_rest_new_witten} line by line upon  replacing $\Phi$ by $U$ and $R_{1+z}$ by $R_{u+z}$. (We recall  $CR_{u+z} \in \mathcal{B}_{n+1}$ with $\|CR_{u+z}\|_{\mathcal{B}_{n+1}}\leq 2\|CR_{1+z}\|_{\mathcal{B}_{n+1}}$).
\end{proof}

As in the derivation of Theorem \ref{thm:Witten_reg_n5} we summarize the results obtained for local boundedness in a theorem (cf.\ Theorem \ref{thm:local_boundedness of the trace}):

\begin{theorem}\label{thm:local_boundedness of the trace2}  Assume Hypothesis \ref{hyp:parameters}, $z\in \Sigma_{\delta_0,\theta}$. Let $U \in C^\infty_b\big(\mathbb{R}^n;\mathbb{C}^{d\times d}\big)$ be $\tau$-admissible $($cf.\ Definition \ref{d:ta}$)$, with $u$ as in \eqref{e:r.aa}, $C=(\cQ U)$. Let $L= \cQ+U$ be given by \eqref{eq:def_of_L(2)} and $\chi_\Lambda$ as in \eqref{eq:def_of_chi}, $\Lambda>0$.  Define for $z\in \Sigma_{\delta_0,\theta}\cap \rho(-LL^*)\cap\rho(-L^*L)$, 
\[
   \iota_\Lambda (z) \coloneqq \chi_\Lambda \tr_{2^{\hat n}d} \big(\big[ \cQ, 
   U \big((L^{*}L+z)^{-1}+(LL^{*}+z)^{-1}\big)(CR_{u+z})^{n}\big]\big),
\]
 and 
\[
\tilde\iota_\Lambda(z)\coloneqq \chi_\Lambda \tr_{2^{\hat n}d}\big(\big[ \cQ, \cQ \big((L^{*}L+z)^{-1}+(LL^{*}+z)^{-1}\big) (CR_{u+z})^{n}\big]\big).
\]
Then for all $z\in \Sigma_{\delta_0,\theta}\cap \rho(-LL^*)\cap\rho(-L^*L)$, the operators $\iota_{\Lambda}(z)$, $\tilde\iota_{\Lambda}(z)$ are trace class 
and there exists $\delta\in (\delta_0,0)$ such that the families 
\[
\{\Sigma_{\delta,\theta} \ni z\mapsto z\tr_{L^2(\bbR^n)} (\iota_{\Lambda}(z))\}_{\Lambda>0} 
\, \text{ and } \, 
\big\{\Sigma_{\delta,\theta} \ni z\mapsto z\tr_{L^2(\bbR^n)} \big(\tilde\iota_{\Lambda}(z)\big)\big\}_{\Lambda>0} 
\]
are locally bounded $($cf.\ \eqref{d:normal}$)$.
\end{theorem}
\begin{proof}
 As in the proof for Theorem \ref{thm:local_boundedness of the trace2}, it suffices to realize that $\iota_\Lambda(z)=2\xi_\Lambda(z)+\zeta_\Lambda(z)$ and $\tilde\iota_\Lambda(z)=2\tilde\xi_\Lambda(z)+\tilde\zeta_\Lambda(z)$ with the functions introduced in Lemmas \ref{lem:boundedness_of_leading_new_witten2} and \ref{lem:boundedness_of_rest_new_witten2}. Thus, the assertion follows from the Lemmas \ref{lem:boundedness_of_leading_new_witten2} and \ref{lem:boundedness_of_rest_new_witten2}.
\end{proof}

The proof of the result analogous to Lemma \ref{lem:some properties of the first two terms} needs some modifications. In particular, one should pay particular attention to the assertion concerning $h_{1,j}$: In Lemma \ref{lem:some properties of the first two terms} we proved that $h_{1,j}$ vanishes on the diagonal. Here, we are only able to give an estimate.

\begin{lemma}
\label{lem:some properties of the first two terms2}  Assume Hypothesis \ref{hyp:parameters}. Let $z\in \Sigma_{\delta_0,\theta}$, and assume that 
$U \in C^\infty_b\big(\mathbb{R}^n;\mathbb{C}^{d\times d}\big)$ is $\tau$-admissible $($cf.\ Definition \ref{d:ta}$)$, with $u$ as in \eqref{e:r.aa}, $C=(\cQ U)$. Let $L= \cQ+U$ be given by \eqref{eq:def_of_L(2)},
$R_{u+z}$ as in \eqref{def:R_psi_z} as well as
$\cQ$, and $\gamma_{j,n}$, $j\in\{1,\ldots,n\}$, given by \eqref{eq:def_of_Q2}, 
and as in Remark \ref{rem:Eucl-Dirac-Algebar}, respectively. 
Then for $n\geq3$, the integral kernel $h_{2,j}(z)$ of 
\[
2\tr_{2^{\hat n}d}\big(\gamma_{j,n} U \left(R_{u+z}C\right)^{n-1}R_{u+z} \big)
\]
satisfies, 
\[
h_{2,j}(z)(x,x)=h_{3,j}(z)(x,x)+g_{0,j}(z)(x,x),
\]
where $h_{3,j}(z)$ is the integral kernel of $2\tr_{2^{\hat n}d}\left(\gamma_{j,n} U C^{n-1}R_{u+z}^{n}\right)$
and $g_{0,j}(z)$ satisfies
\[
 \sup_{z\in\Sigma_{\delta_0,\theta}}\left|g_{0,j}(z)(x,x)\right|\leq \kappa (1+|x|)^{1-n-\epsilon}.
\]
for all $x\in\mathbb{R}^{n}$ and some $\kappa>0$. \\
\noindent 
In addition, if $n\geq5$ and $z\in\mathbb{R}$, then the integral kernel $h_{1,j}(z)$ of 
\[
\tr_{2^{\hat n}d}\big(\gamma_{j,n} \cQ \left(R_{u+z}C\right)^{n-2}R_{u+z}\big)
\]
satisfies
\[
   \sup_{z\in\Sigma_{\delta_0,\theta}}\left|h_{1,j}(z)(x,x)\right|\leq \kappa (1+|x|)^{-n}.
\]
\end{lemma}
\begin{proof} 
We start with $h_{1,j}(z)$. Using the Neumann series expression in \eqref{def:R_psi_z}, one computes,   
\begin{align*}
  H_{1,j}(z) & \coloneqq \tr_{2^{\hat n}d}\big(\gamma_{j,n} \cQ \left(R_{u+z}C\right)^{n-2}R_{u+z}\big) \\
  & \, = \tr_{2^{\hat n}d}\bigg(\gamma_{j,n} \cQ \bigg(\sum_{k=0}^\infty (R_{1+z}(u-1))^kR_{1+z}C\bigg)^{n-2}\sum_{k=0}^\infty (R_{1+z}(u-1))^kR_{1+z}\bigg) \\
  & \, = \tr_{2^{\hat n}d}\big(\gamma_{j,n} \cQ \left(R_{1+z}C\right)^{n-2}R_{1+z}\big)
  \\&\quad  + \tr_{2^{\hat n}d}\bigg(\gamma_{j,n} \cQ \bigg(\sum_{k=1}^\infty (R_{1+z}(u-1))^kR_{1+z}C\bigg)   \\
& \, \qquad \times  \bigg(\sum_{k=0}^\infty (R_{1+z}(u-1))^kR_{1+z}C\bigg)^{n-3}\times
  \sum_{k=0}^\infty (R_{1+z}(u-1))^kR_{1+z}\bigg)   \\
& \, \quad + \cdots +\tr_{2^{\hat n}d}\bigg(\gamma_{j,n} \cQ \bigg(\sum_{k=0}^\infty (R_{1+z}(u-1))^kR_{1+z}C\bigg)^{n-2}   \no \\
& \, \qquad \qquad \qquad \quad \;\; \times \sum_{k=1}^\infty (R_{1+z}(u-1))^kR_{1+z}\bigg) 
  \\ & \, = \tr_{2^{\hat n}d}\big(\gamma_{j,n} \cQ  (R_{1+z}C)^{n-2}R_{1+z}\big)
  \\ & \, \quad + \tr_{2^{\hat n}d}\big(\gamma_{j,n} \cQ (R_{1+z}(u-1)) (R_{u+z}C)^{n-2}R_{u+z}\big)   \\ 
& \, \quad +\cdots + \tr_{2^{\hat n}d}\big(\gamma_{j,n} \cQ (R_{u+z}C)^{n-2}(R_{1+z}(u-1))R_{u+z}\big).
\end{align*}
By Lemma \ref{lem:some properties of the first two terms}, the diagonal of the integral kernel associated with 
\[
\tr_{2^{\hat n}d}\big(\gamma_{j,n} \cQ \left(R_{1+z}C\right)^{n-2}R_{1+z}\big)
\]
 vanishes. Thus, it remains to address the asymptotics of the diagonal of the integral kernel associated with
\begin{align*}
& \tr_{2^{\hat n}d}\big(\gamma_{j,n} \cQ (R_{1+z}(u-1)) (R_{u+z}C)^{n-2}R_{u+z}\big)   \\ 
& \quad + \cdots + \tr_{2^{\hat n}d} \big(\gamma_{j,n} \cQ (R_{u+z}C)^{n-2}
(R_{1+z}(u-1))R_{u+z}\big).
\end{align*}
One observes that the function $(u-1)$ vanishes outside $B(0,\tau)\subset \mathbb{R}^n$ (we recall Hypothesis \ref{hyp:parameters}). Being bounded by $1$, it particularly satisfies the estimate
\[
    |(u-1)(x)|\leq (1+\tau)^n (1+|x|)^{-n}, \quad x\in \mathbb{R}^n.
\]
Realizing that the function $C$ is bounded, the assertion for $h_{1,j}(z)$ follows from Remark \ref{r:7.10}.
 
 The assertion about $h_{2,j}$ can be shown with Remark \ref{rem:green's kernels of the first and the last} (replacing the operators $R_\mu$ in that remark by $R_{u+z}$ and using that the integral kernel of $R_{u+z}$ can be estimated by the respective one for $R_{1+\Re (z)}$, see Theorem \ref{thm:7.18}) and the asymptotic conditions imposed on $U$ (see Definition \ref{d:ta}).
\end{proof}

The analog of Theorem \ref{thm:asymptotics of the rest}, stated below, is now shown in the same way, employing Theorems \ref{t:7.19c} and \ref{thm:7.18}:

\begin{theorem}
\label{thm:asymptotics of the rest2}  Assume Hypothesis \ref{hyp:parameters}, $z\in \Sigma_{\delta_0,\theta}$. Let $U \in C^\infty_b\big(\mathbb{R}^n;\mathbb{C}^{d\times d}\big)$ 
be $\tau$-admissible $($cf.\ Definition \ref{d:ta}$)$, with $u$ as in \eqref{e:r.aa}, $C=(\cQ U)$. Let $L= \cQ+U$ be given by \eqref{eq:def_of_L(2)}, $R_{u+z}$ as in \eqref{def:R_psi_z} as well as
$\cQ$, and $\gamma_{j,n}$, $j\in\{1,\ldots,n\}$, given by \eqref{eq:def_of_Q2}, and as in Remark \ref{rem:Eucl-Dirac-Algebar}, respectively. Then there exists $z_0 > 0$, such that for all $z\in\mathbb{C}$ with $\Re (z) > z_{0}$, the integral kernels $g_{1}$ and 
$g_{2}$ of the operators 
\[
\tr_{2^{\hat n}d}\big(\gamma_{j,n} U \big(\left(L^{*}L+z\right)^{-1}+\left(LL^{*}+z\right)^{-1}\big)\left(CR_{u+z}\right)^{n}\big)
\]
 and 
\[
\tr_{2^{\hat n}d}\big(\gamma_{j,n}Q\big((L^{*}L+z)^{-1} + (LL^{*}+z)^{-1}\big)\left(CR_{u+z}\right)^{n}\big),
\]
 respectively, satisfy for some $\kappa>0$, 
\[
\left(\left|g_{1}(x,x)\right|+\left|g_{2}(x,x)\right|\right)\leq \kappa (1+|x|)^{-n}, 
\quad x\in\mathbb{R}^{n}.  
\]
\end{theorem}

The next result, the analog of Corollary \ref{cor:final aymptotics}, is slightly different compared to the previous analogs since the operator $\tr_{2^{\hat n}d}\left(\gamma_{j,n} U C^{n-1}R_{1+z}^{n}\right)$ is \emph{not} replaced by $\tr_{2^{\hat n}d}\left(\gamma_{j,n} U C^{n-1}R_{u+z}^{n}\right)$. Indeed, Corollary \ref{cor:final aymptotics} was used to show that the only important term for the computation for the index is given by $\tr_{2^{\hat n}d}\left(\gamma_{j,n} U C^{n-1}R_{1+z}^{n}\right)$, for which we computed the integral over the diagonal of the corresponding integral kernel in Proposition \ref{prop:formula for index}, eventually yielding the formula for the index. Since the asserted formulas for admissible and $\tau$-admissible potentials are the same, we need to have a result to the effect that the integral of over the diagonal of the integral kernels of the operators $\tr_{2^{\hat n}d}\left(\gamma_{j,n} U C^{n-1}R_{u+z}^{n}\right)$ and $\tr_{2^{\hat n}d}\left(\gamma_{j,n} U C^{n-1}R_{1+z}^{n}\right)$ should lead to the same results. In fact, this is part of the proof of the following result: 

\begin{corollary}
\label{cor:final aymptotics2}  Assume Hypothesis \ref{hyp:parameters}, $z\in \Sigma_{\delta_0,\theta}$. Let $U \in C^\infty_b\big(\mathbb{R}^n;\mathbb{C}^{d\times d}\big)$ be $\tau$-admissible $($cf.\  Definition \ref{d:ta}$)$, with $u$ as in \eqref{e:r.aa}, $C=(\cQ U)$. Let $L= \cQ+U$ be given by \eqref{eq:def_of_L(2)},
$R_{u+z}$ as in \eqref{def:R_psi_z} as well as
$\cQ$, and $\gamma_{j,n}$, $j\in\{1,\ldots,n\}$, given by \eqref{eq:def_of_Q2}, 
and as in Remark \ref{rem:Eucl-Dirac-Algebar}, respectively.  \\[1mm] 
$(i)$ Let $n\in\mathbb{N}_{\geq5}$, $j\in\{1,\ldots,n\}$. Then there
exists $z_{0} > 0$, such that if $z\in\mathbb{C}$,
$\Re (z) > z_{0}$, and $h$ and $g$ denote the integral kernel of 
$2\tr_{2^{\hat n}d}\left(\gamma_{j,n} U C^{n-1}R_{1+z}^{n}\right)$
and $J_{L}^{j}(z)$, respectively, then for some $\kappa>0$,  
\[
\left|h(x,x)-g(x,x)\right|\leq \kappa (1+|x|)^{1-n-\epsilon}, 
\quad x\in\mathbb{R}^{n}.  
\]
$(ii)$ The assertion of part $(i)$ also holds for $n=3$, if, in the
above statement, $J_{L}^{j}(z)$ is replaced by $J_{L}^{j}(z)-2\tr_{2d}\left(\gamma_{j,3} 
\cQ R_{1+z}CR_{1+z}\right)$. 
\end{corollary}
\begin{proof}
One recalls from Lemma \ref{lem:almost_Neumann_series2},   
\begin{align*}
J_{L}^{j}(z) & =2\tr_{2^{\hat n}d}\big(\gamma_{j,n} \cQ \left(R_{u+z}C\right)^{n-2}R_{u+z} \big)+2\tr_{2^{\hat n}d}\big(\gamma_{j,n} U \left(R_{u+z} C\right)^{n-1}R_{u+z} \big)\\
 & \quad+\tr_{2^{\hat n}d}\big(\gamma_{j,n} \cQ \big(\left(L^{*}L+z\right)^{-1}+\left(LL^{*}+z\right)^{-1}\big)\left(CR_{u+z} \right)^{n}\big)\\
 & \quad+\tr_{2^{\hat n}d}\big(\gamma_{j,n} U \big(\left(L^{*}L+z\right)^{-1}+\left(LL^{*}+z\right)^{-1}\big)\left(CR_{u+z} \right)^{n}\big).  
\end{align*}
With the help of Theorem \ref{thm:asymptotics of the rest2} one deduces 
that the integral kernels of the last two terms may be estimated
by $\kappa (1+ |x|)^{-n}$ on
the diagonal. The integral kernel of the first term is also bounded by $\kappa' (1+ |x|)^{-n}$ for a suitable $\kappa'$ by Lemma \ref{lem:some properties of the first two terms2}.
Hence, it remains to inspect the second term on the right-hand side.
The assertion follows from Lemma \ref{lem:some properties of the first two terms2} once we establish estimates for the respective integral kernels of the differences
\begin{equation}\label{eq:fa5}
\tr_{2^{\hat n}d}\left(\gamma_{j,n} U C^{n-1}R_{u+z}^{n}\right) 
- \tr_{2^{\hat n}d}\left(\gamma_{j,n} U C^{n-1}R_{1+z}^{n}\right)
\end{equation}
and 
\begin{equation}\label{eq:fa3}
\tr_{2d}\left(\gamma_{j,3} 
\cQ R_{u+z}CR_{u+z}\right)-\tr_{2d}\left(\gamma_{j,3} 
\cQ R_{1+z}CR_{1+z}\right).
\end{equation}
In this context we will use the equation
\begin{align*}
   R_{u+z}&=\sum_{k=0}^\infty R_{1+z}((u-1)R_{1+z})^k
   \\&= R_{1+z} + \sum_{k=1}^\infty R_{1+z}((u-1)R_{1+z})^k
   \\ & = R_{1+z} + R_{1+z}(u-1)R_{u+z}.
\end{align*}
Thus, \eqref{eq:fa3} and \eqref{eq:fa5} read
\begin{align}
& \tr_{2d}\left(\gamma_{j,3} 
\cQ R_{1+z}(u-1)R_{u+z}CR_{1+z}\right)+\tr_{2d}\left(\gamma_{j,3} 
\cQ R_{1+z}CR_{1+z}(u-1)R_{u+z}\right)+    \no \\ 
& \quad +\tr_{2d}\left(\gamma_{j,3} 
\cQ R_{1+z}(u-1)R_{u+z}CR_{1+z}(u-1)R_{u+z}\right) \label{eq:fa3.1}
\end{align}
and 
\begin{equation}\label{eq:fa5.1}
    \sum_{k=1}^n \tr_{2^{\hat n}d}\left(\gamma_{j,n} U C^{n-1}R_{u+z}^{k-1} 
    R_{1+z}(u-1)R_{u+z}R_{u+z}^{n-k}\right),
\end{equation}
  respectively. In each summand of \eqref{eq:fa3.1} and \eqref{eq:fa5.1} there is one 
  term $(u-1)$ which is compactly supported and thus clearly satisfies for some $\kappa'>0$, 
  $|(u-1)(x)|\leq \kappa' (1+|x|)^{-n}$, $x\in \mathbb{R}^n$. Hence, the assertion on the asymptotics of the integral kernels associated with the operators in \eqref{eq:fa3} and \eqref{eq:fa5} follows from Corollary \ref{cor:7.19}.
\end{proof}

We are now ready to prove Theorem \ref{thm:Perturbation} for $R>0$ and smooth potentials $\Phi$. 

\begin{theorem}[Theorem \ref{thm:Perturbation} for $R>0$, smooth case]\label{t:7.27} Let $n,d\in \mathbb{N}$, $n\geq 3$ odd, and $\Phi\in C_b^\infty\big(\mathbb{R}^n,\mathbb{C}^{d\times d}\big)$ satisfy the following assumptions: \\[1mm]
$(i)$ $\Phi(x)=\Phi(x)^*$, $x\in \mathbb{R}^n$. \\[1mm]
$(ii)$ There exist $c>0$ and $R>0$ such that $|\Phi(x)|\geq cI_d$ for all $x\in \mathbb{R}^n\backslash  B(0,R)$. \\[1mm] 
$(iii)$ There exists $\epsilon> 1/2$ such that for all $\alpha\in \mathbb{N}_0^n$ there 
is $\kappa>0$ with 
 \[
    \|\partial^\alpha \Phi(x)\|\leq \kappa \begin{cases}
(1+|x|)^{-1},& |\alpha|=1, \\[1mm] 
(1+|x|)^{-1-\epsilon},&|\alpha|\geq2,
\end{cases}  \quad x\in \mathbb{R}^n.
 \]
Let $\tilde L=\mathcal{Q}+\Phi$ as in \eqref{eq:def_of_L(2)}, $\delta_0\in (-1,0), \theta\in (0,\pi/2)$. Then there exists $\tau>0$ such that for all $\tau$-admissible potentials $U$ with 
$U = \sgn (\Phi)$ on sufficiently large balls, and with $L\coloneqq \mathcal Q+U$, the following assertions $(\alpha)$--$(\delta)$ hold: 
\begin{align} 
& \text{$(\alpha)$ There exists $\delta_0\leq \delta<0$ and $0<\theta\leq \theta_0$  such that for all $\Lambda>0$ the family}  \no \\
& \qquad \quad \Sigma_{\delta,\theta}\ni z\mapsto z\chi_\Lambda \tr_{2^{\hatt n}d} ((L^*L+z)^{-1}-(LL^*+z)^{-1})\in \mathcal{B}_1(L^2(\mathbb{R}^n))     \label{i:7.27.1} \\
& \quad \;\;\; \text{is analytic.}   \no \\
& \text{$(\beta)$ The family $\{ f_\Lambda\}_{\Lambda>0}$ of holomorphic functions}   \no \\
& \qquad \quad 
f_\Lambda \colon \Sigma_{\delta,\theta}\ni  z\mapsto \tr \big(z\chi_\Lambda \tr_{2^{\hatt n}d} ((L^*L+z)^{-1}-(LL^*+z)^{-1})\big)     \label{i:7.27.2} \\
& \quad \;\;\; \text{is locally bounded $($see \eqref{d:normal}$)$.}    \no \\
& \text{$(\gamma)$ The limit $f\coloneqq \lim_{\Lambda\to\infty} f_\Lambda$ exists in the compact open topology and satisfies for} \no \\ 
& \quad \;\;\; \text{all $z\in \Sigma_{\delta,\theta}$,}   \no \\
& \qquad \quad  
      f(z) = c_n (1+z)^{-n/2} \lim_{\Lambda\to\infty}\frac{1}{\Lambda}\sum_{j,i_{1},\ldots,i_{n-1} = 1}^n \epsilon_{ji_{1}\ldots i_{n-1}}    \no \\ 
& \qquad \qquad \qquad \times
  \int_{\Lambda S^{n-1}}\tr (U(x) (\partial_{i_{1}} U(x) \ldots 
  (\partial_{i_{n-1}} U)(x)) x_{j}\, d^{n-1} \sigma(x),      \label{i:7.27.3} \\
& \quad \;\;\; \text{where}   \no \\
&  \qquad \quad c_n\coloneqq \frac{1}{2} 
\left(\frac{i}{8\pi}\right)^{(n-1)/2}\frac{1}{\left[(n-1)/2\right]!}.  \no \\
& \text{$(\delta)$ the operators $\tilde L$ and $L$ are Fredholm operators and}   \no \\
& \qquad \quad \ind \big(\tilde L\big) = \ind (L) = f(0)    \no \\
& \qquad \qquad \qquad \, = c_n  \lim_{\Lambda\to\infty}\frac{1}{\Lambda}\sum_{j,i_{1},\ldots,i_{n-1} = 1}^n \epsilon_{ji_{1}\ldots i_{n-1}}    \label{i:7.27.4} \\ 
&\qquad \qquad \qquad \quad \, \times
  \int_{\Lambda S^{n-1}} \tr (U(x) (\partial_{i_{1}} U)(x) \ldots 
  (\partial_{i_{n-1}} U)(x)) x_{j}\, d^{n-1} \sigma(x).    \no 
 \end{align}
\end{theorem}

As mentioned earlier in connection with the proof of Theorem \ref{t:7.27}, we will follow the analogous reasoning used for the proof for Theorem \ref{thm:Witten_reg_n5}. So for the proof of Theorem \ref{t:7.27} we now need to replace the statements Theorem \ref{thm:trisbounded}, Lemma \ref{lem:almost_Neumann_series}, Lemma \ref{lem:some properties of the first two terms}, Theorem \ref{thm:local_boundedness of the trace} and Corollary \ref{cor:final aymptotics} by the respective results Theorem \ref{thm:trisbounded2}, Lemma \ref{lem:almost_Neumann_series2}, Lemma \ref{lem:some properties of the first two terms2}, Theorem \ref{thm:local_boundedness of the trace2} and Corollary \ref{cor:final aymptotics2} obtained in this section. Since large parts of the proof would just be a repetition of arguments used in the proof of Theorem \ref{thm:Witten_reg_n5}, we will not give a detailed proof for the case $n\geq 5$. However, in Section \ref{sec:der_tra_diag}, we only sketched how the result for $n=3$ comes about. As this case is notationally less messy, we will now give the full proof for the case $n=3$. As in Section \ref{sec:n=3}, the core idea is to regularize the expressions involved by multiplying $B_L(z)$ with $(1-\mu\Delta)^{-1}$, $\mu>0$, from either side, which results in 
\begin{equation*}
B_{L,\mu}(z)=\left(1-\mu\Delta  \right)^{-1}B_{L}(z)\left(1-\mu\Delta  \right)^{-1},
\end{equation*}
(compare with \eqref{eq:B_Lmu}), and similarly for $J_{L,\mu}^{j}(z)$ and $A_{L,\mu}(z)$, recalling  \eqref{eq:JJJ_mu} and \eqref{eq:ALz_mu}, respectively. 

\begin{proof}[Proof of Theorem \ref{t:7.27}, $n=3$]
Part $(\alpha)$: This follows from Theorem \ref{thm:trisbounded2}.  \\[1mm]  
Part $(\beta)$: Again by Theorem \ref{thm:trisbounded2}, the expression $\tr (\chi_\Lambda B_{L}(z))$, with $B_L(z)$ as given in \eqref{eq:Def_of_B_L(z)}, can be computed as the integral over the diagonal of its integral kernel. Next, we denote by $\mathbb{A}$ and $\mathbb{J}$ the integral kernels for the operators $A_L(z)$ and $ 2B_L(z)-A_L(z)$, respectively, and correspondingly $\mathbb{A}_\mu$ and $\mathbb{J}_\mu$ for $A_{L,\mu}(z)$ and $2 B_{L,\mu}(z)-A_{L,\mu}(z)$, $\mu>0$. Hence, Proposition \ref{prop:commutator with phi vanishes} applied to 
$\mathbb{A}$ yields,
\begin{align*} 
2 f_\Lambda(z) & = 2\tr (\chi_\Lambda B_{L}(z)) 
   \\ & =  \int_{B(0,\Lambda)} \mathbb{A}(x,x)+\mathbb{J}(x,x) \, d^3 x
   =  \int_{B(0,\Lambda)} \mathbb{J}(x,x) \, d^3 x
   \\ & =\lim_{\mu\to 0}  \int_{B(0,\Lambda)} \mathbb{J}_\mu(x,x) \, d^3 x,
\end{align*}
where in the last equality we used the continuity the integral kernels of $2B_L(z)$ and $A_L(z)$, as well as Lemma \ref{lem:pointwise convergence of regularized Green}.

Next, appealing to the analog result of Lemma \ref{lem:repre of bmu} for $\Phi$ being replaced by the $\tau$-admissible potential $U$, one arrives at
\begin{align*} 
2 f_{\Lambda}(z)& =\lim_{\mu\to 0}  \int_{B(0,\Lambda)} \mathbb{J}_\mu(x,x) \, d^3 x
 \\  & = \lim_{\mu\to 0} \int_{B(0,\Lambda)} \bigg\langle \delta_{\{x\}}, \sum_{j=1}^3 [\partial_j,J_{L,\mu}^j(z)] \delta_{\{x\}} \bigg\rangle \, d^3x.
\end{align*}
Denote	
\[
\mathbb{K}_{L,\mu}\coloneqq\big\{x\mapsto g_{L,\mu}^{j}(z)(x,x)\big\}_{j\in\{1,2,3\}},
\]
where $g_{L,\mu}^{j}(z)$ is the integral kernel of $J_{L,\mu}^{j}(z)$, 
$j\in\{1,2,3\}$, and $\mathbb{K}_{L,z}$ that for 
\[
\big\{J_{L}^{j}(z)-2\tr_{2d} (\gamma_{j,3} \cQ R_{1+z}CR_{1+z})\big\}_{j\in\{1,2,3\}}.
\]
Invoking Lemmas \ref{lem:reg vanishes on diagonal} and \ref{lem:pointwise convergence of regularized Green}, and hence the fact that $\{x\mapsto\mathbb{K}_{L,\mu}(x,x)\}_{\mu>0}$
is locally bounded, one obtains 
\begin{align*}
\lim_{\mu\to0}\int_{B(0,\Lambda)}\mathbb{J}_{L,\mu}(x,x)\, d^3 x & =\lim_{\mu\to0}\int_{\Lambda S^{2}}\left( \mathbb{K}_{L,\mu}(x,x),\frac{x}{\Lambda}\right) \, d^{2} \sigma(x)\\
 & =\int_{\Lambda S^{2}}\lim_{\mu\to0}\left( \mathbb{K}_{L,\mu}(x,x),\frac{x}{\Lambda}\right) \, d^{2} \sigma(x)\\
 & =\int_{\Lambda S^{2}}\left(\mathbb{K}_{L,z}(x,x),\frac{x}{\Lambda}\right) \, d^{2} \sigma(x). 
\end{align*}
Hence, one arrives at
\begin{align}
  2f_{\Lambda}(z)& = \int_{\Lambda S^{2}}\left(\mathbb{K}_{L,z}(x),\frac{x}{\Lambda}\right) \, d^{2} \sigma(x) \label{e:7.27.1}
  \\  & = \int_{B(0,\Lambda)} \bigg\langle \delta_{\{x\}},\sum_{j=1}^3 [\partial_j, \tilde{J}_{L}^j(z)]\delta_{\{x\}} \bigg\rangle \, d^3x,   \no
\end{align} 
with
\[
   \tilde{J}_{L}^{j}(z) \coloneqq J_{L}^{j}(z)-2\tr_{2d} (\gamma_{j,3} \cQ R_{1+z}CR_{1+z}).
\]
For proving that $\{f_\Lambda\}_{\Lambda>0}$ is locally bounded, we recall from Lemma \ref{lem:almost_Neumann_series2} that
\begin{align*}
 \tilde{J}_{L}^{j}(z) & = 2\tr_{2d}\big(\gamma_{j,3} \cQ (R_{u+z}C)R_{u+z}\big) - 2\tr_{2d} (\gamma_{j,3} \cQ R_{1+z}CR_{1+z})
 \\ &\quad + 2\tr_{2d}\big(\gamma_{j,3} U (R_{u+z}C)^{2}R_{u+z}\big)    
 \\  & \quad+\tr_{2d}\big(\gamma_{j,3} \cQ \big((L^{*}L+z)^{-1}+(LL^{*}+z)^{-1}\big) 
 (CR_{u+z})^{3}\big)   \\
& \quad + \tr_{2d} \big(\gamma_{j,3} U \big((L^{*}L+z)^{-1}
 +(LL^{*}+z)^{-1}\big) (CR_{u+z})^{3}\big),
\end{align*}
and, thus,
\begin{align*}
 & \sum_{j=1}^3 \big[\partial_j,\tilde{J}_{L}^{j}(z)\big]
 \\ & \quad =  \sum_{j=1}^3 \big[\partial_j, \big(2\tr_{2d}\big(\gamma_{j,3} \cQ (R_{u+z}C)R_{u+z}\big) - 2\tr_{2d} (\gamma_{j,3} \cQ R_{1+z}CR_{1+z})\big)\big]
 \\ & \qquad +  \sum_{j=1}^3 \big[\partial_j, 2\tr_{2d}\big(\gamma_{j,3} U (R_{u+z}C)^{2}R_{u+z}\big) \big]
 \\  & \qquad + \tr_{2d} \big(\big[ \cQ, 
   U \big((L^{*}L+z)^{-1}+(LL^{*}+z)^{-1}\big)(CR_{u+z})^{3}\big]\big)   \\
& \qquad + \tr_{2d}\big(\big[ \cQ, \big(\cQ \big((L^{*}L+z)^{-1}+(LL^{*}+z)^{-1}\big) (CR_{u+z})^{3}\big)\big]\big).
\end{align*}
Hence, 
\begin{align*}
  2f_{\Lambda}(z) & = \int_{B(0,\Lambda)} \bigg\langle \delta_{\{x\}},\sum_{j=1}^3 \bigg[\partial_j, \big(2\tr_{2d}\big(\gamma_{j,3} \cQ (R_{u+z}C)R_{u+z}\big) \\ &\quad\quad\quad\quad\quad\quad\quad\quad\quad	- 2\tr_{2d} (\gamma_{j,3} \cQ R_{1+z}CR_{1+z})\big)\bigg] \delta_{\{x\}}\bigg\rangle \, d^3 x
  \\ & \quad +\int_{B(0,\Lambda)} \bigg\langle \delta_{\{x\}},\sum_{j=1}^3 \big[\partial_j, 2\tr_{2d}\big(\gamma_{j,3} U (R_{u+z}C)^{2}R_{u+z}\big) \big] \delta_{\{x\}} \bigg\rangle \, d^3 x
  \\& \quad + \tr(\iota_\Lambda(z))+\tr(\tilde\iota_\Lambda(z)),
\end{align*} 
where $\iota_\Lambda(z)$ and $\tilde\iota_\Lambda(z)$ are defined in Theorem \ref{thm:local_boundedness of the trace2}. With the help of part $(\alpha)$, Theorem \ref{thm:local_boundedness of the trace2} and Lemma \ref{lem:Laurent_series_argument}, to prove  part $(\beta)$, it suffices to prove the local boundedness of
\begin{align*}
 & z\mapsto  \int_{B(0,\Lambda)} \bigg\langle \delta_{\{x\}},\sum_{j=1}^3 \big[\partial_j, \big(2\tr_{2d}\big(\gamma_{j,3} \cQ (R_{u+z}C)R_{u+z}\big)
 \\
 & \qquad - 2\tr_{2d} (\gamma_{j,3} \cQ R_{1+z}CR_{1+z})\big)\big]  \delta_{\{x\}}\bigg\rangle \, d^3 x
\end{align*}
and
\begin{equation*}
 z\mapsto \int_{B(0,\Lambda)} \bigg\langle \delta_{\{x\}},\sum_{j=1}^3 \big[\partial_j, 2\tr_{2d}\big(\gamma_{j,3} U (R_{u+z}C)^{2}R_{u+z}\big) \big] \delta_{\{x\}}\bigg\rangle \,d^3 x,  
\end{equation*}
both considered as families of functions indexed by $\Lambda>0$. Appealing to Gauss' divergence theorem, it suffices to show that the integral kernels associated with the operators
\begin{align*}
  & \cQ R_{u+z}CR_{u+z}- \cQ R_{1+z}CR_{1+z} 
  \\ & \quad = \cQ R_{1+z}(u-1)R_{u+z}CR_{u+z}+\cQ R_{u+z}CR_{1+z}(u-1)R_{u+z}
  \\& \qquad +\cQ R_{1+z}(u-1)R_{u+z}CR_{1+z}(u-1)R_{u+z}
\end{align*}
and 
\[ 
U (R_{u+z}C)^{2}R_{u+z}
\]
can be estimated  on the diagonal by $\kappa' (1+|x|)^{-2}$ for some $\kappa' > 0$ for sufficiently large $|x|$. However, this is a consequence of Corollary \ref{cor:7.19}, proving part $(\beta)$. 
\\[1mm]  
Part $(\gamma)$:  By Montel's Theorem, there exists a sequence $\{\Lambda_k\}_{k\in\bbN}$ of positive reals tending to infinity such that $f\coloneqq \lim_{k\to\infty} f_{\Lambda_k}$ exists in the compact open topology. From \eqref{e:7.27.1}, one recalls 
\[
 2f_{\Lambda}(z) = \int_{\Lambda S^{2}}\left(\mathbb{K}_{L,z}(x,x),\frac{x}{\Lambda}\right) \, d^{2} \sigma(x),
\] 
with $\mathbb{K}_{L,z}$ denoting the integral kernel of
\[
   \big\{J_{L}^{j}(z)-2\tr_{2d} (\gamma_{j,3} \cQ R_{1+z}CR_{1+z})\big\}_{j\in\{1,2,3\}}.
\]
Next, we choose $z_0 > 0$ as in Corollary
\ref{cor:final aymptotics2}\,$(ii)$ and let $z\in \Sigma_{z_0,\theta}$. With $h_{j}$, the integral kernel of $2\textnormal{tr}_{2^{}d}\left(\gamma_{j,3}\Psi C^{2}R_{1+z}^{3}\right)$, 
we define $\mathbb{H}_{z}\coloneqq\left(x\mapsto h_{j}(x,x)\right)_{j\in\{1,2,3\}}$.
Due to Corollary \ref{cor:final aymptotics2}\,$(ii)$ one can find $\kappa>0$ such
that for $k\in \mathbb{N}$, 
\begin{align*}
& \left|\int_{\Lambda_k S^{2}}\left((\mathbb{K}_{J,z}-\mathbb{H}_{z})(x),\frac{x}{\Lambda_k}
\right)_{\bbR^n} \, d^{2} \sigma(x)\right|    \\ 
& \quad \leq \int_{\Lambda_k S^{2}} \|(\mathbb{K}_{J,z}-\mathbb{H}_{z})(x)\|_{\bbR^n}\, d^{2} \sigma(x)\\
 & \quad \leq \kappa\int_{\Lambda_k S^{2}} (1+|x|)^{-2-\epsilon}\, d^{2} \sigma(x)\\
 & \quad =\kappa \Lambda_k^{2}\omega_{2} (1+\Lambda_k)^{-2-\epsilon}.
\end{align*}
Hence, 
\[
\lim_{k\to\infty}\int_{\Lambda_k S^{2}}\left((\mathbb{K}_{J,z}-\mathbb{H}_{z})(x),\frac{x}{\Lambda_k}
\right)_{\bbR^n} \, d^{2} \sigma(x)=0, 
\]
and 
\begin{align}
2f(z) & =\lim_{k\to\infty}\int_{\Lambda_k S^{2}}\left( \mathbb{K}_{J,z}(x),\frac{x}{\Lambda_k}\right)_{\bbR^n} \, d^{2} \sigma(x) \notag \\
& =\lim_{k\to\infty}\int_{\Lambda_k S^{2}}\left( \mathbb{H}_{z}(x),\frac{x}{\Lambda_k}
\right)_{\bbR^n} \, d^{2} \sigma(x) \notag \\
& =\left(\frac{i}{8\pi}\right) (1+z)^{-3/2} \lim_{k\to\infty}\int_{\Lambda_k S^{2}} \notag \\
& \quad \times \sum_{j=1}^3 \bigg(\sum_{i_{1},i_{2} = 1}^3 \epsilon_{ji_{1}i_{2}}\tr (U(x) (\partial_{i_{1}} U)(x) (\partial_{i_{2}} U)(x)\big)\bigg)\bigg(\frac{x_{j}}{\Lambda_k}\bigg) \, d^{2} \sigma(x),
\label{eq:limit_expr_f2}
\end{align} 
where, for the last integral, we used Proposition \ref{prop:formula for index}. 
Theorem \ref{thm:index with Witten} implies $f(0)=\ind (L)$. In particular, any sequence $\{\Lambda_k\}_{k\in\bbN}$ of positive reals converging to infinity contains a subsequence $\{\Lambda_{k_\ell}\}_\ell$ such that for that particular subsequence the limit
\begin{align*} 
& \lim_{\ell\to\infty}\int_{\Lambda_{k_\ell} S^{2}} 
\sum_{j,i_{1},i_{2} = 1}^3 \epsilon_{ji_{1}	i_{2}}\tr (U(x) (\partial_{i_{1}} U)(x) 
(\partial_{i_{2}} U)(x)\big) \bigg(\frac{x_{j}}{\Lambda_{k_\ell}}\bigg) \, d^{2} \sigma(x) 
\end{align*}  
exists and equals 
\begin{equation}\label{eq:subseq_argum2}
  \frac{2 \ind (L)} {\left[i/(8\pi)\right] (1+z)^{-3/2}}.
\end{equation}
Hence, the limit
\begin{align}\label{eq:limit_exists2} 
& \lim_{\Lambda\to\infty}\int_{\Lambda S^{2}} 
\sum_{j,i_{1},i_{2} = 1}^3 \epsilon_{ji_{1}\ldots i_{2}}\tr (U(x) (\partial_{i_{1}} U)(x) 
(\partial_{i_{2}} U)(x)) \bigg(\frac{x_{j}}{\Lambda}\bigg) \, d^{2} \sigma(x)  
\end{align}
 exists and equals the number in \eqref{eq:subseq_argum2}. On the other hand, for $\Re(z)>z_0$ with $z_0>0$ sufficiently large (according to Corollary \ref{cor:final aymptotics2}\,$(i)$) the 
family $\{f_\Lambda\}_{\Lambda>0}$ converges for $\Lambda\to\infty$ on the domain $\mathbb{C}_{\Re>z_0}\cap \Sigma_{\delta,\theta}$ if and only if the limit in \eqref{eq:limit_exists2} exists. Indeed, this follows from the explicit expression for the limit in \eqref{eq:limit_expr_f2}. Therefore, $\{f_\Lambda\}_{\Lambda>0}$
converges in the compact open topology on $\mathbb{C}_{\Re>z_0}\cap \Sigma_{\delta,\theta}$. By the local boundedness of the latter family on  $\Sigma_{\delta,\theta}$, the principle of analytic continuation for analytic functions implies that the latter family actually converges on the domain $\Sigma_{\delta,\theta}$ in the compact open topology. In particular, 
\begin{align*}
& \frac{2 f(z) (1+z)^{3/2}}{i/(8\pi)}    \\ 
& \quad = \lim_{\Lambda\to\infty}\int_{\Lambda S^{2}} 
\sum_{j,i_{1},i_{2} = 1}^3 \epsilon_{ji_{1}\ldots i_{2}}\tr (U(x) (\partial_{i_{1}} U)(x) 
(\partial_{i_{2}} U)(x)) \bigg(\frac{x_{j}}{\Lambda}\bigg) \, d^{2} \sigma(x).   
\end{align*} 
Part $(\delta)$: The Fredholm property of $L$ and $\tilde L$ follows from Lemma \ref{lem:2L_Fred_Rge0}, and the equality $\ind \big(\tilde L\big) =\ind (L)$ follows from Theorem \ref{thm:3L_Fred_Rge0} and Remark \ref{rem:almo_ad}. The remaining equality in \eqref{i:7.27.4} follows from part $(\gamma)$ and Theorem \ref{thm:index with Witten}.
\end{proof}

\newpage

\section{The proof of Theorem \ref{thm:Perturbation}: The General Case} \lb{s13} 

The strategy to prove Theorem \ref{thm:Perturbation} for potentials which are only $C^2$ consists in an additional convolution with a suitable mollifier, applying Theorem \ref{thm:Perturbation} (i.e., 
Theorem \ref{t:7.27}) for the $C^\infty$-case, and to use suitable perurbation theorems for the Fredholm index. The next result gathers information on mollified functions.

\begin{proposition}
\label{prop:mollpot} Let $n,d\in\mathbb{N}$, 
$\Phi\in C_{b}^{1}\big(\mathbb{R}^{n};\mathbb{C}^{d\times d}\big)$. Assume $\zeta\in C_0^{\infty}(\mathbb{R}^{n})$ with $\zeta\geq0$,
$\supp (\zeta) \subset B(0,1)$, $\int_{\mathbb{R}^{n}}\zeta(x) \, d^n x = 1$ and
for $\gamma>0$ define $\zeta_{\gamma}\coloneqq \gamma^{-n}\zeta\left(1/\gamma\cdot\right)$
and $\Phi_{\gamma}\coloneqq\zeta_{\gamma}*\Phi$. \\[1mm] 
$(i)$ For all $0<\gamma<1$ and $j\in\{1,\ldots,n\}$ one has
\[
\left\Vert \Phi(x)-\Phi_{\gamma}(x)\right\Vert \leq \gamma \sup_{y\in B(0,\gamma)}\left\Vert \Phi'(x+y)\right\Vert, \quad x\in\mathbb{R}^{n}.
\]
$(ii)$ Assume, in addition, $\Phi\in C^2_b\big(\mathbb{R}^n;\mathbb{C}^{d\times d}\big)$ and for some $\epsilon> 1/2$ and all $\alpha\in \mathbb{N}_0^n$, $|\alpha|=2$, that for some $\kappa>0$,  
\[
   \|\partial^\alpha \Phi(x)\|\leq \kappa (1+|x|)^{-1-\epsilon}, 
   \quad x\in \mathbb{R}^n. 
\]
Then for all $\alpha\in\mathbb{N}_{0}^{n}$ with $\left|\alpha\right|\geq2$, $0<\gamma<1$, 
\[
\|(\partial^{\alpha}\Phi_{\gamma})(x)\|\leq \|\partial^{\alpha-\beta}\zeta\|_{\infty}v_{n} 
\gamma^{2 - |\alpha|}
\kappa (1-\gamma+|x|)^{-1-\epsilon}, \quad x\in\mathbb{R}^{n},
\]
with $v_{n}$ the $n$-dimensional volume of the unit ball
in $\mathbb{R}^{n}$ and $\beta\in\mathbb{N}_{0}^{n}$ such that 
$(\alpha-\beta)\in\mathbb{N}_{0}^{n}$
and $\left|\beta\right|=2$. \\[1mm] 
$(iii)$ If $\Phi(x)=\Phi(x)^*$ for all $x\in\mathbb{R}^n$, then $\Phi_\gamma(x)=\Phi_{\gamma}(x)^*$ for all $x\in \mathbb{R}^n$, $\gamma>0$. \\[1mm] 
$(iv)$ If there exist $c>0$ and $R>0$ such that $|\Phi(x)|\geq cI_d$
for all $x\in\mathbb{R}^{n}\backslash B\left(0,R\right)$, then there exists $\gamma_{0}>0$ such that for all $0<\gamma<\gamma_{0}$
\[
|\Phi_{\gamma}(x)|\geq (c/2) I_d, \quad x\in\mathbb{R}^{n} \backslash B(0,R).
\]
\end{proposition}
\begin{proof}
$(i)$ In order to prove the first inequality, let $x\in\mathbb{R}^{n}$,
$0<\gamma<1$. Then one computes 
\begin{align*}
\|\Phi(x)-\Phi_{\gamma}(x)\| & =\bigg\Vert \int_{\mathbb{R}^{n}}\left(\Phi(x)-\Phi(x-y)\right)\zeta_{\gamma}(y) \, d^n y \bigg\Vert \\
 & \leq \int_{B(0,1)}\left\Vert \left(\Phi(x)-\Phi(x-\gamma y)\right)\right\Vert \zeta(y) \, d^n y \\
 & \leq \bigg(\int_{B(0,1)}\zeta(y) \, d^n y\bigg) \gamma\sup_{y\in B(0,\gamma)}\left\Vert \Phi'(x+y)\right\Vert.
\end{align*}
$(ii)$ Let $0<\gamma<1$ and $\alpha\in\mathbb{N}_{0}^{n}$
with $\left|\alpha\right|\geq2$, and let $\beta\in\mathbb{N}_{0}^{n}$
with $\left|\beta\right|=2$ be such that $(\alpha-\beta)\in\mathbb{N}_{0}^{n}$.
Then for $x\in\mathbb{R}^{n}$,  
\begin{align*}
\left\Vert (\partial^{\alpha}\Phi_{\gamma})(x)\right\Vert  & =\left\Vert (\partial^{\alpha}\left(\zeta_{\gamma}*\Phi)\right)(x)\right\Vert \\
 & =\bigg\Vert \left(\left(\partial^{\alpha-\beta}\zeta_{\gamma}\right)* 
 \left(\partial^{\beta}\Phi\right)\right)(x)\bigg\Vert \\
 & =\bigg\Vert \int_{\mathbb{R}^{n}}\left(\partial^{\alpha-\beta}\zeta_{\gamma}\right)(y) 
 \big(\partial^{\beta}\Phi\big)(x-y) \, d^n y \bigg\Vert \\
 & \leq \int_{\mathbb{R}^{n}} \gamma^{-n-\left|\alpha\right|+2}\big\Vert \big(\partial^{\alpha-\beta}\zeta\big) (y/\gamma) (\partial^{\beta}\Phi)(x-y)\big\Vert  \, d^n y \\
 & \leq \int_{B(0,1)} \gamma^{2 - |\alpha|}\left|\big(\partial^{\alpha-\beta}\zeta\big)(y)\right|\left\Vert (\partial^\beta\Phi)(x-\gamma y)\right\Vert  \, d^n y \\
 & \leq \|\partial^{\alpha-\beta}\zeta\|_{\infty}\int_{B(0,1)} \gamma^{2 - |\alpha|}  
 \kappa (1-\gamma+|x|)^{-1-\epsilon} \, d^n y  \\[1mm]
 & \leq \|\partial^{\alpha-\beta}\zeta\|_{\infty}v_{n} \gamma^{2 - |\alpha|} 
 \kappa (1-\gamma+|x|)^{-1-\epsilon}. 
\end{align*}
$(iii)$ This is clear.  \\[1mm] 
$(iv)$ By part $(i)$, there exists $\gamma_{0}>0$ such that 
$\sup_{x\in\mathbb{R}^{n}}\left\Vert \Phi(x)-\Phi_{\gamma}(x)\right\Vert < c/2$ 
for all $0<\gamma<\gamma_{0}$.
Let $x\in\mathbb{R}^{n}$ such that $\left|x\right|\geq R$. From
$|\Phi(x)|\geq cI_d$ it follows that $\|\Phi(x)^{-1}\|\leq 1/c$.
Hence, $\left\Vert \Phi(x)^{-1}\left(\Phi(x)-\Phi_{\gamma}(x)\right)\right\Vert \leq 1/2$
and therefore,
\begin{align*}
& \sum_{k=0}^{\infty}\left(\Phi(x)^{-1}\left(\Phi(x)-\Phi_{\gamma}(x)\right)\right)^{k}\Phi(x)^{-1} 
=\left(1-\Phi(x)^{-1}\left(\Phi(x)-\Phi_{\gamma}(x)\right)\right)^{-1}\Phi(x)^{-1}\\
& \quad =\left(\Phi(x)-\left(\Phi(x)-\Phi_{\gamma}(x)\right)\right)^{-1}
= \Phi_{\gamma}(x)^{-1}.
\end{align*}
Using $\left\Vert \Phi_{\gamma}(x)^{-1}\right\Vert \leq
c^{-1} \sum_{k=0}^{\infty} 2^{-k}= 2/c$,
one deduces with the help of the spectral theorem that
$|\Phi_{\gamma}(x)|\geq (c/2) I_d$. 
\end{proof}

\begin{remark}\label{r:9.2} Let $\Phi$ be as in Theorem \ref{thm:Perturbation}. More precisely, 
let $\Phi\in C_b^2\big(\mathbb{R}^n;\mathbb{C}^{d\times d}\big)$ be pointwise self-adjoint, 
suppose that for some $c>0$ and $R>0$, $|\Phi(x)|\geq c I_d$ for all 
$x\in \mathbb{R}^n\backslash  B(0,R)$, 
and assume there exists $\epsilon> 1/2$, such that for all $\alpha\in \mathbb{N}_0^n$ there 
exists $\kappa>0$ such that
\[
   \|\partial^\alpha \Phi(x)\|\leq \kappa \begin{cases}
                                             (1+|x|)^{-1}, & |\alpha|=1, \\[1mm] 
                                             (1+|x|)^{-1 - \epsilon}, &|\alpha|=2.
                                          \end{cases}
\]
By Proposition \ref{prop:mollpot}, there exists $\gamma_0>0$ such that for all $\gamma\in (0,\gamma_0)$, $\Phi_\gamma$ (defined as in Proposition \ref{prop:mollpot}) satisfies the assumptions imposed on $\Phi$ in Theorem \ref{t:7.27}. Moreover, by Proposition \ref{prop:mollpot}\,$(i)$, for some $\tilde \kappa>0$, the estimate
\[
   \| \partial_j \Phi(x)-\partial_j\Phi_\gamma (x)\|\leq \tilde \kappa 
   (1+|x|)^{-1-\epsilon}
\]
holds for all $j\in\{1,\ldots,n\}$, $x\in \mathbb{R}^n$, $0<\gamma<1$. The latter observation will be used in the proof of the general case of Theorem \ref{thm:Perturbation}. \hfill $\diamond$
\end{remark}

For the sake of completeness, we shall also state the Fredholm property for $C^2$-potentials:

\begin{lemma}\label{lem:2L_Fred_Rge0} Let $n,d\in \mathbb{N}$, $L= \cQ+\Phi$ as in \eqref{eq:def_of_L(2)} with $\Phi\in C_b^2\big(\mathbb{R}^n; \mathbb{C}^{d\times d}\big)$ with 
$\Phi(x)=\Phi(x)^*$, $x\in \mathbb{R}^n$. Assume that there exist $c>0$ and $R>0$ such that 
$|\Phi(x)|\geq cI_{d}$ for all $x\in \mathbb{R}^n\backslash  B(0,R)$, and that $(\partial_j \Phi)(x)\to 0$ as $|x|\to\infty$, $j\in \{1,\ldots,n\}$. Then $L$ is a Fredholm operator, and there is $\gamma_0>0$ such that for all $\gamma\in (0,\gamma_0)$, $\ind(L)=\ind(L_\gamma)$, where $L_\gamma = \cQ+\Phi_\gamma$ with $\Phi_\gamma$ given as in Proposition \ref{prop:mollpot}.
\end{lemma}
\begin{proof}
  By Proposition \ref{prop:mollpot}, $L$ is a $\cQ$-compact perturbation of $L_\gamma$. Moreover, the latter is a Fredholm operator for all $\gamma\in (0,\gamma_0)$ for some $\gamma_0>0$ by Proposition \ref{prop:mollpot} \big(guaranteeing that for some $\tilde c>0$ and $\tilde R >0$, 
$|\Phi_\gamma(x)|\geq \tilde cI_{d}$ 
for all $x\in \mathbb{R}^n\backslash B\big(0,\tilde R\big)$\big) and Lemma \ref{lem:1L_Fred_Rge0}. Thus, the assertion follows from the invariance of the Fredholm index under relatively compact perurbations.
\end{proof}

We are prepared to conclude the proof of the main theorem also for $C^2$-potentials.

\begin{proof}[Proof of Theorem \ref{thm:Perturbation}, the nonsmooth case]
Let $\zeta\in C_0^{\infty}(\mathbb{R}^{n})$ be as in Proposition \ref{prop:mollpot}, and define
$\Phi_{\gamma}$ as in the latter proposition. Let $\gamma_{0}\in (0,1)$
be as in Proposition \ref{prop:mollpot}\,$(iv)$. As observed in Remark \ref{r:9.2}, $\Phi_{\gamma}$ satisfies the assumptions imposed on $\Phi$ in Theorem \ref{t:7.27}. Next, by 
Lemma \ref{lem:2L_Fred_Rge0}, $L_{\gamma}\coloneqq \cQ +\Phi_{\gamma}$
is Fredholm and $\ind (L) = \ind (L_{\gamma})$.

Hence, by the $C^\infty$-version of Theorem \ref{thm:Perturbation}, that is, by Threorem \ref{t:7.27}, one infers that
\begin{align*}
\ind\left(L_{\gamma}\right) 
& =\left(\frac{i}{8\pi}\right)^{(n-1)/2}\frac{1}{\left[(n-1)/2\right]!}
  \lim_{\Lambda \to\infty}\frac{1}{2 \Lambda }\sum_{j,i_{1},\ldots,i_{n-1} = 1}^n \epsilon_{ji_{1}\ldots i_{n-1}}  \\
& \qquad \times \int_{\Lambda S^{n-1}}\tr\big(\sgn(\Phi_{\gamma}(x)) 
(\partial_{i_{1}}\sgn(\Phi_{\gamma}))(x)   \no \\
& \qquad \qquad \qquad \quad \dots (\partial_{i_{n-1}}\sgn(\Phi_{\gamma}))(x)\big) 
x_{j}\, d^{n-1} \sigma(x). 
\end{align*}
It sufices to shown that the limit $\gamma\to0$ in
the latter expression exists and coincides with the formula asserted.
By differentiation, one observes that $\omega\colon\mathbb{R}_{>0}\ni x\mapsto\sqrt{x^{2}+c}-x$
is decreasing and, denoting $\|\Phi\|_\infty\coloneqq \sup_{x\in \mathbb{R}^n}\|\Phi(x)\|$, one 
gets $0< \omega (\|\Phi\|_{\infty}) \leq  \omega (\|\Phi(x)\|)$, $x\in\mathbb{R}^{n}.$ Let 
$0<\gamma_{1}< \omega (\|\Phi\|_{\infty})/(4\kappa) \land (1/2) \land\gamma_{0}$. 
For $0<\gamma<\gamma_{1}$ and all $x\in\mathbb{R}^{n}$ one deduces with the help of 
Proposition \ref{prop:mollpot}\,$(i)$ that 
\[
\|\Phi(x)-\Phi_{\gamma}(x)\|\leq \gamma_{1} 2\kappa (1-\gamma_{1}+ |x|)^{-1}  
\leq \omega (\|\Phi\|_{\infty})/2.
\] 
Hence, by Theorem \ref{thm:properties of sign function}, one gets for some $K>0$ and all 
$x\in \mathbb{R}^n$ with $|x|\geq R$,  
\begin{align*}
\left\Vert \sgn(\Phi(x))-\sgn(\Phi_{\gamma}(x))\right\Vert  & \leq \sup_{T\in\bigcup_{|x|\geq R}B(\Phi(x), 
\omega( (\|\Phi\|_{\infty})/2)}\| {\sgn}'(T)\|\|\Phi(x)-\Phi_{\gamma}(x)\|\\
 & \leq \gamma K (1+ |x|)^{-1}.  
\end{align*}
Similarly, Proposition \ref{prop:mollpot}\,$(i)$ implies for some $K'>0$, 
\begin{align*}
 & \max_{j \in \{1,\dots,n\}} \Vert {\sgn}'(\Phi(x)) (\partial_{j}\Phi)(x)- {\sgn}'(\Phi_{\gamma}(x)) 
 (\partial_{j}\Phi)(x)\Vert \\
 & \quad \, \leq \max_{j  \in \{1,\dots,n\}}\|(\partial_{j}\Phi)(x)\|\sup_{T\in\bigcup_{|x| 
 \geq R}B(\Phi(x), \omega(\|\Phi\|_{\infty})/2)}\|{\sgn}''(T)\|\|\Phi(x)-\Phi_{\gamma}(x)\|\\
 & \quad \, \leq \gamma K' (1+|x|)^{-2}. 
\end{align*}
For $i_{1},\ldots,i_{n-1}\in\{1,\ldots,n\}$ and $x\in\mathbb{R}^{n}$, and with the convention 
$\partial_{i_{0}}\coloneqq1$, one gets for some constants $K'',K'''>0$,  
\begin{align}
 & \Vert (\sgn(\Phi(x)) (\partial_{i_{1}}\sgn(\Phi))(x)\ldots (\partial_{i_{n-1}}
 \sgn(\Phi))(x))  \no \\
 & \; \,\, \quad - (\sgn(\Phi_{\gamma}(x)) (\partial_{i_{1}}\sgn(\Phi_{\gamma}))(x)\ldots (\partial_{i_{n-1}}\sgn(\Phi_{\gamma}))(x))\Vert\no \\
 & \quad \leq \sum_{j=0}^{n-1}\bigg\Vert \prod_{k=0}^{j-1} (\partial_{i_{k}}\sgn(\Phi))(x)((\partial_{i_{j}}\sgn(\Phi))(x)-(\partial_{i_{j}}\sgn(\Phi_{\gamma}))(x))\no \\ 
& \hspace*{1.5cm} 
\times \prod_{k=j+1}^{n-1} (\partial_{i_{k}}\sgn(\Phi_{\gamma}))(x)\bigg\Vert\no \\ 
& \quad \leq  K'' (1+|x|)^{2-n} \bigg(\sum_{j=1}^{n-1} \Vert (\partial_{i_{j}}\sgn(\Phi))(x) 
- (\partial_{i_{j}}\sgn(\Phi_{\gamma}))(x) \Vert  \no\\
 & \hspace*{4cm} + (1+ |x|)^{-1} \Vert \sgn(\Phi(x)) - \sgn(\Phi_{\gamma}(x)) \Vert \bigg) \no\\
 & \quad \leq  K'' (1+|x|)^{2-n} \bigg(\sum_{j=1}^{n-1} \Vert {\sgn}'(\Phi(x))(\partial_{i_{j}}\Phi)(x) - {\sgn}'(\Phi_{\gamma}(x))(\partial_{i_{j}}\Phi_{\gamma})(x) \Vert     \no\\
 & \hspace*{4cm} +\gamma K (1+|x|)^{-2}\bigg)\no\\
 & \quad \leq  K'' (1+|x|)^{2-n}\bigg(\sum_{j=1}^{n-1}\| {\sgn}'(\Phi(x))(\partial_{i_{j}}\Phi)(x) - {\sgn}'(\Phi_{\gamma}(x))(\partial_{i_{j}}\Phi)(x)\|      \no\\
 & \qquad +\sum_{j=1}^{n-1}\| {\sgn}'(\Phi_{\gamma}(x))(\partial_{i_{j}}\Phi)(x) 
 - {\sgn}'(\Phi_{\gamma})(x) (\partial_{i_{j}}\Phi_{\gamma})(x)\| 
 + \gamma K (1+|x|)^{-2}\bigg)\no \\
 & \quad \leq  K'' (1+|x|)^{2-n}\bigg(\sum_{j=1}^{n-1}\gamma K' (1+|x|)^{-2}   \no \\
 & \quad\quad +\sum_{j=1}^{n-1}\sup_{T\in\bigcup_{|x|\geq R}B(\Phi(x), \omega(\|\Phi\|_{\infty})/2)}\|{\sgn}'(T)\| \|(\partial_{i_{j}}\Phi)(x) - (\partial_{i_{j}}\Phi_{\gamma})(x)\|  \no \\ 
 & \qquad +\gamma K (1+|x|)^{-2}\bigg)\no\\
 & \quad \leq  K'''\gamma (1+|x|)^{1-n-\epsilon}.       \label{e:9.4}
\end{align}

Next, for $\Lambda >0$ we define
\begin{align*}
& \phi_{\Lambda }\coloneqq \frac{1}{\Lambda }\sum_{j,i_{1},\ldots,i_{n-1} = 1}^n  
\varepsilon_{ji_{1}\ldots i_{n-1}} \\
& \qquad \;\; \times \int_{\Lambda S^{n-1}}\tr \big(\sgn(\Phi(x)) (\partial_{i_{1}}\sgn(\Phi(x))) \ldots 
(\partial_{i_{n-1}}\sgn(\Phi (x)))\big)x_{j}\, d^{n-1} \sigma(x)
\end{align*}
and 
\begin{align*}
& \phi_{\Lambda }^\gamma\coloneqq \frac{1}{\Lambda }\sum_{j,i_{1},\ldots,i_{n-1} = 1}^n  
\epsilon_{ji_{1}\ldots i_{n-1}} 
 \int_{\Lambda S^{n-1}}\tr \big(\sgn(\Phi_\gamma(x))     \\
& \hspace*{9mm} \times (\partial_{i_{1}}\sgn(\Phi_\gamma(x))) \ldots 
(\partial_{i_{n-1}}\sgn(\Phi_\gamma (x)))\big)x_{j}\, d^{n-1} \sigma(x).
\end{align*}
It remains to prove that $\{\phi_\Lambda \}_\Lambda $ converges and that its limit coincides 
with $\ind (L)$. But, with the help of estimate \eqref{e:9.4} one gets 
\[
   \limsup_{\Lambda\to\infty} |\phi_\Lambda^\gamma-\phi_\Lambda|\leq \limsup_{\Lambda\to\infty} \int_{\Lambda S^{n-1}}  K'''\gamma (1+|x|)^{1-n-\epsilon} \frac{|x|}{\Lambda}d^{n-1}\sigma(x) =0,
\]
which implies the remaining assertion.
\end{proof}

\begin{remark}\label{rem:differentiability}
$(i)$ Of course a simple manner in which to invoke less regular potentials is the
perturbation with compactly supported potentials. Thus, the
above result should be read as $\Phi$ is $C^{2}$ ``in a neighborhood of infinity''.  \\[1mm] 
$(ii)$ The formula for the Fredholm index suggests that Theorem \ref{thm:Perturbation}
might be weakened in the sense that potentials that are only $C^{1}$
should lead to the same result. Our method of proof needs that for some $K>0$ and 
all $\gamma>0$ sufficiently small, 
$\left\Vert (\partial_{i_{j}}\Phi)(x)-(\partial_{i_{j}}\Phi_{\gamma})(x)\right\Vert \leq  K (1+|x|)^{-1-\epsilon}$. To prove the latter estimate we need information on the second derivative of $\Phi$. 
\hfill $\diamond$
\end{remark} 

We conclude with a nontrivial example of the Fredholm index. In view of the discussion in Example \ref{exa:standard_example_0} and the erroneous statement in \eqref{eq:wrong_cancellation} this could be the type of potentials Callias had in mind. 

\begin{example}\label{exa:standard_example}
Let $n=3$, $\gamma_{1,3},\gamma_{2,3},\gamma_{3,3}\in\mathbb{C}^{2\times2}$
be the corresponding matrices of the Euclidean Dirac Algebra $($see
Appendix \ref{sec:Appendix:-the-Construction}$)$. Consider $L= \cQ +\Phi$
as in \eqref{eq:def_of_L(2)} with 
$\Phi(x)\coloneqq\sum_{j=1}^{3}\gamma_{j,3} x_{j} \left|x\right|^{-1}$, $j\in\{1,2,3\}$. Then $\Phi(x)^{2}=I_2$ and Theorem \ref{thm:Perturbation} applies. Given formula 
\eqref{eq:Fredhom_index} for the Fredholm index, a straightforward computation yields  
\[
\tr_{2}\big(\Phi(x) (\partial_{i_{1}}\Phi)(x) (\partial_{i_{2}}\Phi)(x)\big) = 
\tr_{2} \big(\gamma_{j,3}\gamma_{{i_{1},3}}\gamma_{{i_{2},3}} x_{j}\left|x\right|^{-3}\big), \quad 
x\in\mathbb{R}^{n}\backslash \{0\}, 
\]
 for all $j,i_{1},i_{2}\in\{1,2,3\}$ pairwise distinct, and hence,
\begin{align*}
\ind (L) & =\frac{i}{16\pi}\lim_{\Lambda \to\infty}\sum_{j,i_{1},i_{2} = 1}^3 
\epsilon_{ji_{1}i_{2}}\frac{1}{\Lambda }\int_{\Lambda S^{2}}\tr\big(\Phi(x) (\partial_{i_{1}}\Phi)(x)
(\partial_{i_{2}}\Phi)(x)\big)x_{j}
\, d^2\sigma(x)\\
 & =\frac{i}{16\pi}\lim_{\Lambda \to\infty}\sum_{j=1}^{3}\frac{1}{\Lambda }\int_{\Lambda S^{2}}4i\frac{x_{j}}{\Lambda ^{3}}x_{j}
 \, d^2 \sigma(x)\\
 & =\frac{i}{16\pi}\lim_{\Lambda \to\infty}\frac{1}{\Lambda ^{2}}\int_{\Lambda S^{2}}4i \, d^2 \sigma(x)\\
 & =\frac{-1}{4\pi}\int_{S^{2}}d^2 \sigma(x)=-1. 
 \end{align*}
\end{example}

\newpage

\section{A Particular Class of Non-Fredholm Operators $L$ and Their Generalized 
Witten Index} \lb{s14} 

This section is devoted to a particular non-Fredhom situation and motivated in part by 
extensions of  index theory for a certain class of non-Fredholm operators initiated in 
\cite{BGGSS87}, \cite{CGPST14a}, \cite{GS88} (see also \cite{BMS88}, \cite{CP86}). 
Here we make the first steps in the direction of 
non-Fredholm operators closely related to the operator $L$ in \eqref{eq:def_of_L} 
studied by Callias \cite{Ca78} and introduce a generalized Witten index. 

We very briefly outline the idea presented in \cite{GS88}: Let $L$ be a densely defined, closed, linear operator in a Hilbert space $\mathcal{H}$.  Assume that 
\[ 
\big[(L^*L+z)^{-1}-(LL^*+z)^{-1}\big] \in \cB_1(\cH)
\]
for one (and hence for all) $z\in \rho(-L^*L)\cap \rho(-LL^*)$, and that the limit
\begin{equation} 
\ind_W (L) \coloneqq \lim_{x \downarrow 0_+} x \tr_\mathcal{H}\big((L^*L+x)^{-1}-(LL^*+x)^{-1}\big) 
\lb{14.1} 
\end{equation}
exists. Then $\ind_W (L)$ is called the \emph{Witten index} of $L$. In fact, for the 
special case of operators in space dimension $n=1$ (with appropriate potential), this limit is easily shown to exist and to assume values in $(1/2) \mathbb{Z}$, see \cite{CGPST14}. These examples, however, heavily rely on the fact that the underlying spatial dimension for the operator $L$ 
equals one.

While the Fredholm index is well-known to be invariant with respect to relatively compact additive perturbations, we emphasize that this cannot hold for the Witten index (cf.\ \cite{BGGSS87}, 
\cite{GS88}). In fact, it can be shown that the Witten index is invariant under additive perturbations that are relatively trace class (among additional conditions, see \cite{GS88} for details). 

We now provide a further generalization of the Witten index adapted to the non-Fredholm 
operators discussed in this section for odd dimensions $n\geq3$. The abstract set-up reads 
as follows:

\begin{definition}\label{d:gW} Let $L$ be a densely defined, closed linear operator in $\mathcal{H}^m$ for some $m\in \mathbb{N}$. Assume there exist sequences $\{T_\Lambda\}_{\Lambda \in \bbN}$, 
$\{S_\Lambda^*\}_{\Lambda \in \bbN}$ in $\mathcal{B}(\mathcal{H})$ converging to $I_\mathcal{H}$ in the strong operator topology, and denote $S_\Lambda\coloneqq S_\Lambda^{**}$, $\Lambda\in \mathbb{N}$. In addition, suppose that the map 
\[
  \Omega \ni z \mapsto T_\Lambda \tr_m \big((L^*L+z)^{-1}-(LL^*+z)^{-1}\big)S_\Lambda
\]
assumes values in $\mathcal{B}_1(\mathcal{H})$ for some open set $\Omega\subseteq \rho(-L^*L)\cap \rho(-LL^*)$ with $(0,\delta)\subseteq \Omega\cap\mathbb{R}$ for some $\delta>0$. 
Moreover, assume that $\{f_\Lambda\}_{\Lambda \in \bbN}$, where
\[
  f_\Lambda \colon \Omega\ni z \mapsto  z\tr_{\mathcal{H}}\big(T_\Lambda\tr_m \big((L^*L+z)^{-1}-(LL^*+z)^{-1}\big)S_\Lambda\big),\quad \Lambda\in \mathbb{N},
\]
converges in the compact open topology as $\Lambda\to\infty$ to some function $f\colon \Omega\to \mathbb{C}$ and that $f(0_+)$ exists. Then we call
\begin{equation} 
  \ind_{gW,T,S} (L) \coloneqq f(0_+).    \lb{14.2} 
\end{equation} 
the \emph{generalized Witten index} of $L$ $($with respect to $T$ and $S$$)$. If $L$ satisfies the assumptions needed for defining $\ind_{gW,T,S} (L)$, then we say that $L$ \emph{admits a generalized Witten index}.
\end{definition}

Whenever the sequences $\{T_{\Lambda}\}_{\Lambda \in \bbN}$ and 
$\{S_{\Lambda}\}_{\Lambda \in \bbN}$ in Definition \ref{d:gW} are clear from the context, we will just write $\ind_{gW} (\cdot)$ instead of $\ind_{gW,T,S} (\cdot)$.

\begin{remark} We briefly elaborate on some properties of the regularized index just defined. \\[1mm] 
$(i)$ It is easy to see that the generalized Witten index is independent of the chosen $\Omega$. Indeed, the main observation needed is that if $\Omega_1$ and $\Omega_2$ satisfy the requirements imposed on $\Omega$ in Definition \ref{d:gW}, then so does $\Omega_1\cap \Omega_2$. \\[1mm] 
$(ii)$ The generalized Witten index is invariant under unitary equivalence of $\mathcal{H}$. Indeed, let $L$ be a densely defined, closed linear operator in $\mathcal{H}^m$ for which the generalized Witten index exists with respect to $\{T_\Lambda\}_{\Lambda \in \bbN}$ and 
$\{S_\Lambda\}_{\Lambda \in \bbN}$. Let $\mathcal{H}_1$ be another Hilbert space and $U\colon \mathcal{H}_1\to \mathcal{H}$ an isometric isomorphism.
Then
\[
  \tilde L\coloneqq \begin{pmatrix} U^* & 0 & \cdots & 0 \\
                                    0   & U^* &      & \vdots \\
                                    \vdots &   & \ddots &   \\
                                    0   & \cdots & & U^*
                    \end{pmatrix} L \begin{pmatrix} U & 0 & \cdots & 0 \\
                                    0   & U &      & \vdots \\
                                    \vdots &   & \ddots &   \\
                                    0   & \cdots & & U
                    \end{pmatrix}
\]
admits a generalized Witten index, and  
\[
  \ind_{gW,T,S} (L) = \ind_{gW,U^*TU,U^*SU} \big(\tilde L\big),
\]
where, in obvious notation, we used $\ind_{gW,U^*TU,U^*SU} \big(\tilde L\big)$ to denote the generalized Witten index of $\tilde L$ with respect to $\{U^*T_\Lambda U\}_{\Lambda \in \bbN}$ 
and $\{U^* S_{\Lambda} U\}_{\Lambda \in \bbN}$.

For the proof of $\tilde L$ admitting a generalized Witten index, it suffices to observe that 
for $\Lambda\in \mathbb{N}$ and $z\in \Omega$, 
\begin{align*}
 & \tr_{\mathcal{H}}\big(T_\Lambda\tr_m \big((L^*L+z)^{-1}-(LL^*+z)^{-1}\big)S_\Lambda\big)\\
 & \quad = \tr_{\mathcal{H}}\big(T_\Lambda U U^*\tr_m \big((L^*L+z)^{-1}-(LL^*+z)^{-1}\big)UU^*S_\Lambda\big)
 \\ & \quad =\tr_{\mathcal{H}}\big(T_\Lambda U \tr_m \big((\tilde{L}^*\tilde{L}+z)^{-1}-(\tilde{L}\tilde{L}^*+z)^{-1}\big)U^*S_\Lambda\big)
 \\ & \quad =\tr_{\mathcal{H}_1}\big(U^*T_\Lambda U \tr_m \big((\tilde{L}^*\tilde{L}+z)^{-1}-(\tilde{L}\tilde{L}^*+z)^{-1}\big)U^*S_\Lambda U\big).   
\end{align*}
${}$\hfill $\diamond$
\end{remark}

\begin{remark} 
The definition of the Witten index in \eqref{14.1} suggests introducing the spectral shift 
function $\xi (\, \cdot \, ; LL^*, L^*L)$ for the pair of self-adjoint operators $(L L^*, L^* L)$ and hence to express the Witten index as 
\begin{equation} 
\ind_W (L) = \xi (0_+ ; LL^*, L^*L),    \lb{14.3}
\end{equation} 
employing the fact (see, e.g., \cite[Ch.~8]{Ya92}), 
\begin{equation} 
\tr_\mathcal{H}\big(f(L^*L) - f(LL^*)\big) = \int_{[0,\infty)} f'(\lambda) \, \xi (\lambda ; LL^*, L^*L) \, 
d\lambda,  
\end{equation} 
for a large class of functions $f$. The approach \eqref{14.3} in terms of spectral shift functions was 
introduced in \cite{BGGSS87}, \cite{GS88} (see also \cite{BMS88}, \cite{Br95}, 
\cite[Chs.~IX, X]{Mu87}, \cite{Mu88}) and independently in \cite{CP86}. 
This circle of ideas continues to generate much interest, see, for instance, \cite{CGPST14}, 
\cite{CGPST14a}, \cite{GLMST11}, and the extensive list of references therein. It remains to be seen if this can be applied to the generalized Witten index \eqref{14.2}. 
${}$\hfill $\diamond$
\end{remark}

Next, we will construct non-Fredholm Callias-type operators $L$, which meet the assumptions in the definition for the generalized Witten index, that is, operators $L$ which admit a generalized Witten index. In fact, the theory developed in the previous chapters provides a variety of such examples 
(cf.\ the end of this section in Theorem \ref{t:13.9}).

We start with an elementary observation:  

\begin{proposition}\label{p:13.1} Let $n=2\hatt n+1\in \mathbb{N}$ odd. Then $Q$ as in \eqref{eq:Def_of_Q} with $\dom(Q)=H^1(\mathbb{R}^n)^{2^{\hat n}}$ as an operator in $L^2(\mathbb{R}^n)^{2^{\hatt n}}$ is non-Fredholm.
\end{proposition}
\begin{proof}
It suffices to observe that the symbol of $Q$ is a continuous function vanishing at $0$.
\end{proof}
The fundamental result leading to Theorem \ref{t:13.9} is contained in the following lemma.

\begin{lemma}\label{l:13.3} Let $n=2\hatt n+1\in \mathbb{N}$ odd, $d\in \mathbb{N}$, $\mathcal{Q}$ as in \eqref{eq:def_of_Q2}, $\Phi\in C^\infty_b\big(\mathbb{R}^n;\mathbb{C}^{d\times d}\big)$ pointwise self-adjoint, and let $L=\mathcal{Q}+\Phi$ be as in \eqref{eq:def_of_L(2)}. Assume that there exists $P=P^*=P^2\in \mathbb{C}^{d\times d} \backslash \{0\}$ such that for all $x\in \mathbb{R}^n$, 
\[
  P\Phi(x)=\Phi(x)P=0.
\]
Define $\mathcal{P}\coloneqq I_{L^2(\mathbb{R}^n)^{2^{\hatt n}}}\otimes P$ and denote $\mathcal{H}_\mathcal{P}\coloneqq L^2(\mathbb{R}^n)^{2^{\hatt n}}\otimes\ran(P)$. Then $L$ and $L^*$ leave the space $\mathcal{H}_P$ invariant. Moreover, $L$ is unitarily equivalent to
\[
  \begin{pmatrix}
        \iota_P^*\mathcal{Q}\iota_P
      & 0 \\ 0 & \iota_{P_\bot}^*\mathcal{Q}\iota_{P_\bot} + \iota_{P_\bot}^*\Phi\iota_{P_\bot}
      \end{pmatrix}
\]
with $\iota_P$ and $\iota_{P_\bot}$ the canonical embeddings from $\mathcal{H}_\mathcal{P}$ 
and $\mathcal{H}_{\mathcal{P}}^\bot$ into $L^2(\mathbb{R}^n)^{2^{\hatt n}d}$, respectively.
\end{lemma}

\begin{remark}\label{r:13.7} Recalling $Q$ given as in \eqref{eq:Def_of_Q}, we claim that in the situation of Lemma \ref{l:13.3}, 
\[
   \iota_P^*\mathcal{Q}\iota_P = Q\otimes I_{\ran P}.
\]
Indeed, equality is plain when applied to $C_0^\infty$-functions and thus the general case follows by a closure argument. For the closedness of $\iota_P^*\mathcal{Q}\iota_P$ we use Lemma \ref{l:13.3}: $L$ is closed and, by unitary equivalence, so is
\begin{equation}\label{e:13.7}
  \begin{pmatrix}
        \iota_P^*\mathcal{Q}\iota_P
      & 0 \\ 0 & \iota_{P_\bot}^*\mathcal{Q}\iota_{P_\bot} + \iota_{P_\bot}^*\Phi\iota_{P_\bot}
      \end{pmatrix}.
\end{equation}
Hence, the diagonal entries of the closed block operator matrix in \eqref{e:13.7}, and thus 
$\iota_P^*\mathcal{Q}\iota_P$, are closed. \hfill $\diamond$
\end{remark}

In order to prove Lemma \ref{l:13.3}, we invoke some auxiliary results of a general nature. The 
first two (Lemmas \ref{l:13.4} and \ref{l:13.5}) are concerned with commutativity properties of the operator $\mathcal{Q}$.

\begin{lemma}\label{l:13.4}
 Let $n,m\in \mathbb{N}$, $P\in \mathbb{C}^{m\times m}$, $j\in\{1,\ldots,n\}$. Then
 \[
    (I_{L^2(\mathbb{R}^n)}\otimes P)\partial_j \subseteq \partial_j (I_{L^2(\mathbb{R}^n)}\otimes P),
 \]
 where $\partial_j\colon H_j^1(\mathbb{R}^n)^m\subseteq L^2(\mathbb{R}^n)^m\to L^2(\mathbb{R}^n)^m$ is the distributional derivative with respect to the $j$th variable and, $H_j^1(\mathbb{R}^n)$ is the space of $L^2$-functions whose derivative with respect to the $j$th variable can be represented as an $L^2$-function.
\end{lemma}
\begin{proof}
Clearly, 
 \[
    (I_{L^2(\mathbb{R}^n)}\otimes P)\partial_j \phi =\partial_j (I_{L^2(\mathbb{R}^n)}\otimes P) \phi, 
    \quad \phi\in C_0^\infty(\mathbb{R}^n)^m.
 \]
 Next, the operator $ \partial_j (I_{L^2(\mathbb{R}^n)}\otimes P)$ is closed, hence, 
 \[
     (I_{L^2(\mathbb{R}^n)}\otimes P)\partial_j \subseteq \overline{ (I_{L^2(\mathbb{R}^n)}\otimes P)\partial_j }\subseteq \partial_j (I_{L^2(\mathbb{R}^n)}\otimes P),
 \]
yields the assertion.
\end{proof}

\begin{lemma}\label{l:13.5}
 Let $n,d\in \mathbb{N}$, $n=2\hatt n$ or $n=2\hatt n+1$ for some $\hatt n\in \mathbb{N}$. Let $\mathcal{Q}$ as in \eqref{eq:def_of_Q2} $($defined in $L^2(\mathbb{R}^n)^{2^{\hatt n}d}$$)$. Let $P\in \mathbb{C}^{d\times d}$ and denote $\mathcal{P}\coloneqq I_{L^2(\mathbb{R}^n)^{2^{\hatt n}}}\otimes P$. Then, 
 \[
    \mathcal{P}\mathcal{Q}\subseteq \mathcal{Q}\mathcal{P}.
 \]
\end{lemma}
\begin{proof}
We note that for all $j\in\{1,\ldots,n\}$ and $\gamma_{j,n}$ as in Section \ref{sec:Appendix:-the-Construction}, $\gamma_{j,n}\mathcal{P}=\mathcal{P}\gamma_{j,n}$, where we viewed $\gamma_{j,n}\in \mathcal{B}(L^2(\mathbb{R}^n)^{2^{\hatt n}})$. Hence, using 
$\dom(\mathcal{Q})=\bigcap_{j=1}^n \dom(\partial_j)$, Lemma \ref{l:13.4} implies 
\begin{align*}
   \mathcal{P}\mathcal{Q}&=\mathcal{P}\sum_{j=1}^n \gamma_{j,n}\partial_j =\sum_{j=1}^n \mathcal{P}\gamma_{j,n}\partial_j
   \\&=\sum_{j=1}^n \gamma_{j,n}\mathcal{P}\partial_j \subseteq\sum_{j=1}^n \gamma_{j,n}\partial_j\mathcal{P}=\mathcal{Q}\mathcal{P}.\qedhere
\end{align*}
\end{proof}

Before turning to the proof of Lemma \ref{l:13.3}, we recall a general result on the representability of operators as block operator matrices (the same calculus has also been employed in \cite[Lemma 3.2]{PTW14}): 

\begin{lemma}\label{l:13.6} Let $P\in \mathcal{B}(\mathcal{H})$ be an orthogonal projection, $W\colon D(W)\subseteq \mathcal{H}\to \mathcal{H}$ closed and linear. Assume that
\[
  PW\subseteq WP \, \text{ and } \, (I_{\mathcal{H}}-P)W\subseteq W(I_\mathcal{H}-P).
\]
Denote by $\iota_P\colon \ran(P)\to \mathcal{H}$ and $\iota_{P_\bot} \colon\ker(P)\to \mathcal{H}$ the canonical embeddings, respectively.
Then $W$ is unitarily equivalent to a block operator matrix. More precisely,  
\begin{equation}\label{e:13.6}
  \begin{pmatrix}
   \iota_P^* \\ \iota_{P_\bot}^*
  \end{pmatrix}
W \begin{pmatrix}
   \iota_P & \iota_{P_\bot}
  \end{pmatrix}=\begin{pmatrix}
        \iota_P^* W \iota_{P} & 0 \\ 0 & \iota_{P_\bot}^* W \iota_{P_\bot}
     \end{pmatrix}
\end{equation}
with $\iota_P^* W \iota_{P}$ and $\iota_{P_\bot}^* W \iota_{P_\bot}$ closed linear operators.
\end{lemma}
\begin{proof}
  First, one observes that the operators $\begin{pmatrix}
   \iota_P^* \\ \iota_{P_\bot}^*
  \end{pmatrix}$ and $\begin{pmatrix}
   \iota_P & \iota_{P_\bot}
  \end{pmatrix}$ are unitary and inverses of each other. Moreover, it is plain that a block diagonal operator matrix is closed if and only if its diagonal entries are closed. Thus, as $W$ is closed by hypothesis, it suffices to prove equality \eqref{e:13.6}. One notes that $P=\iota_P\iota_P^*$ and  similarly, $(I_\mathcal{H}-P) = \iota_{P_\bot}\iota_{P_\bot}^*$, and hence computes, 
  \begin{align*}
     W & = \begin{pmatrix}
   \iota_P & \iota_{P_\bot}
  \end{pmatrix} \begin{pmatrix}
   \iota_P^* \\ \iota_{P_\bot}^*
  \end{pmatrix} W \begin{pmatrix}
   \iota_P & \iota_{P_\bot}
  \end{pmatrix} \begin{pmatrix}
   \iota_P^* \\ \iota_{P_\bot}^*
  \end{pmatrix}
  \\ & = \begin{pmatrix}
   \iota_P & \iota_{P_\bot}
  \end{pmatrix} \begin{pmatrix}
   \iota_P^*W \\ \iota_{P_\bot}^*W
  \end{pmatrix}  \begin{pmatrix}
   \iota_P & \iota_{P_\bot}
  \end{pmatrix} \begin{pmatrix}
   \iota_P^* \\ \iota_{P_\bot}^*
  \end{pmatrix}
  \\ & = \begin{pmatrix}
   \iota_P & \iota_{P_\bot}
  \end{pmatrix} \begin{pmatrix}
   \iota_P^*\iota_P\iota_P^* W \\ \iota_{P_\bot}^*\iota_{P_\bot} \iota_{P_\bot}^*W
  \end{pmatrix}  \begin{pmatrix}
   \iota_P & \iota_{P_\bot}
  \end{pmatrix} \begin{pmatrix}
   \iota_P^* \\ \iota_{P_\bot}^*
  \end{pmatrix}
  \\ & \subseteq \begin{pmatrix}
   \iota_P & \iota_{P_\bot}
  \end{pmatrix} \begin{pmatrix}
   \iota_P^* W \iota_P\iota_P^* \\ \iota_{P_\bot}^*W\iota_{P_\bot} \iota_{P_\bot}^*
  \end{pmatrix}  \begin{pmatrix}
   \iota_P & \iota_{P_\bot}
  \end{pmatrix} \begin{pmatrix}
   \iota_P^* \\ \iota_{P_\bot}^*
  \end{pmatrix}
  \\ & = \begin{pmatrix}
   \iota_P & \iota_{P_\bot}
  \end{pmatrix} \begin{pmatrix}
   \iota_P^* W \iota_P\iota_P^*\iota_P & 0 \\ 0& \iota_{P_\bot}^*W\iota_{P_\bot} \iota_{P_\bot}^*\iota_{P_\bot}
  \end{pmatrix} \begin{pmatrix}
   \iota_P^* \\ \iota_{P_\bot}^*
  \end{pmatrix}
  \\ & = \begin{pmatrix}
   \iota_P & \iota_{P_\bot}
  \end{pmatrix} \begin{pmatrix}
   \iota_P^* W \iota_P & 0 \\ 0& \iota_{P_\bot}^*W\iota_{P_\bot}
  \end{pmatrix} \begin{pmatrix}
   \iota_P^* \\ \iota_{P_\bot}^*
  \end{pmatrix}
  \\ & = \begin{pmatrix}
   \iota_P & \iota_{P_\bot}
  \end{pmatrix} \begin{pmatrix}
   \iota_P^* W \iota_P\iota_P^* \\ \iota_{P_\bot}^*W\iota_{P_\bot} \iota_{P_\bot}^*
  \end{pmatrix}
  \\ & = \begin{pmatrix}
   \iota_P & \iota_{P_\bot}
  \end{pmatrix} \begin{pmatrix}
   \iota_P^* W \iota_P\iota_P^*\big(\iota_P\iota_P^*+\iota_{P_\bot} \iota_{P_\bot}^*\big) \\ \iota_{P_\bot}^*W\iota_{P_\bot} \iota_{P_\bot}^*\big(\iota_P\iota_P^*+\iota_{P_\bot} \iota_{P_\bot}^*\big)
  \end{pmatrix}
  \\ & = \begin{pmatrix}
   \iota_P & \iota_{P_\bot}
  \end{pmatrix} \begin{pmatrix}
   \iota_P^* W \big(\iota_P\iota_P^*+\iota_{P_\bot} \iota_{P_\bot}^*\big) \\ \iota_{P_\bot}^*W\big(\iota_P\iota_P^*+\iota_{P_\bot} \iota_{P_\bot}^*\big)
  \end{pmatrix}
  \\ & \subseteq W,
  \end{align*}
concluding the proof.
\end{proof}

At this instant we are in a position to prove Lemma \ref{l:13.3}.

\begin{proof}[Proof of Lemma \ref{l:13.3}]
   By Lemma \ref{l:13.5}, 
\[
\mathcal{P}L\subseteq L\mathcal{P}\, \text{ and } \, 
(I_{L^2(\mathbb{R}^n)^{2^{\hatt n}d}}-\mathcal{P})L\subseteq 
L(I_{L^2(\mathbb{R}^n)^{2^{\hatt n}d}}-\mathcal{P}).
\] 
Hence, by Lemma \ref{l:13.6}, $L$ is unitarily equivalent to
\[
  \begin{pmatrix}
        \iota_P^*L\iota_P
      & 0 \\ 0 & \iota_{P_\bot}^*L\iota_{P_\bot}.
      \end{pmatrix}
\]
The assertion, thus, follows from $\iota_P^*\Phi\iota_P=0$ (valid by hypothesis).
\end{proof}

From Proposition \ref{p:13.1} and Lemma \ref{l:13.3} one infers the following result.

\begin{corollary}\label{c:13.2} Let $n=2\hatt n+1\in \mathbb{N}$ odd, $d\in \mathbb{N}$, 
assume $\mathcal{Q}$ is given by \eqref{eq:def_of_Q2}, 
$\Phi\in C^\infty_b\big(\mathbb{R}^n;\mathbb{C}^{d\times d}\big)$ pointwise self-adjoint, and let $L=\mathcal{Q}+\Phi$ be as in \eqref{eq:def_of_L(2)}. Assume that there exists $P=P^*=P^2\in \mathbb{C}^{d\times d} \backslash \{0\}$ such that 
\[
  P\Phi(x)=\Phi(x)P=0, \quad x\in \mathbb{R}^n.
\]
Then $L$ is non-Fredholm.
\end{corollary}
\begin{proof}
 Define $\mathcal{H}_P\coloneqq L^2(\mathbb{R}^n;\mathbb{C}^{2^{\hatt n}}\otimes \ran(P))$, denote the embedding from $\mathcal{H}_P$ into $L^2(\mathbb{R}^n)^{2^{\hatt n}d}$ by $\iota_P$, and denote the embedding from $\mathcal{H}_P^\bot$ into $L^2(\mathbb{R}^n)^{2^{\hatt n}d}$ by $\iota_{P_\bot}$. By Lemma \ref{l:13.3}, the operator $L$ is unitarily equivalent to
 \[
     \begin{pmatrix}
        \iota_P^*\mathcal{Q}\iota_P
      & 0 \\ 0 & \iota_{P_\bot}^*\mathcal{Q}\iota_{P_\bot} + \iota_{P_\bot}^*\Phi\iota_{P_\bot}
      \end{pmatrix},
 \]
 which by Remark \ref{r:13.7} may also be written as
 \[
      \begin{pmatrix}
        Q\otimes I_{\ran P}
      & 0 \\ 0 & Q \otimes I_{\ker P}+ \iota_{P_\bot}^*\Phi\iota_{P_\bot}
      \end{pmatrix}. 
 \]
In particular, 
 \[
    \sigma(L)=\sigma\big(Q\otimes I_{\ran P}\big)\cup \sigma\big(Q \otimes I_{\ker P}+ \iota_{P_\bot}^*\Phi\iota_{P_\bot}\big).
 \]
 Since $\ran (P)$ is at least one-dimensional, it follows from Proposition \ref{p:13.1} that $0\in \sigma_{\textrm{ess}}\big(Q\otimes I_{\ran P}\big)$. Hence, $0\in \sigma_{\textrm{ess}}(L)$, 
 implying that $L$ is non-Fredholm.
\end{proof}

We conclude this section with non-trivial examples illustrating the generalized Witten index introduced in this section:

\begin{theorem}\label{t:13.9} Assume Hypothesis \ref{hyp:parameters} and let 
$U \in C_b^\infty\big(\mathbb{R}^n;\mathbb{C}^{d\times d}\big)$ be a $\tau$-admissible potential. Let $\ell\in \mathbb{N}$ and define
\[
  \Phi\colon \mathbb{R}^n \to \mathbb{C}^{(d+\ell)\times (d+\ell)}, \quad x\mapsto \begin{pmatrix}
                                                                               0 & 0 \\
                                                                               0 & U(x)
                                                                            \end{pmatrix}.
\]
Let $L\coloneqq \mathcal{Q}+\Phi$, as in \eqref{Def:L}. Then the following assertions 
$(\alpha)$--$(\delta)$ hold:
\begin{align}
& \text{$(\alpha)$ For all $\Lambda>0$, the family} \no \\
&  \qquad 
\Sigma_{0,\theta} \ni z\mapsto z\chi_\Lambda \tr_{2^{\hatt n}(d+\ell)} ((L^*L+z)^{-1}-(LL^*+z)^{-1})\in \mathcal{B}_1(L^2(\mathbb{R}^n))    \label{i:13.9.1} \\
& \quad \;\;\; \text{is analytic.}  \no \\
& \text{$(\beta)$ The family $\{ f_\Lambda\}_{\Lambda>0}$ of analytic functions}  \no \\
& \qquad \quad    
 f_\Lambda \colon \Sigma_{0,\theta} \ni  z\mapsto \tr \big(z\chi_\Lambda \tr_{2^{\hatt n}(d+\ell)} ((L^*L+z)^{-1}-(LL^*+z)^{-1})\big)    \label{i:13.9.2} \\
& \quad \;\;\; \text{is locally bounded $($see \eqref{d:normal}$)$.}  \no \\
& \text{$(\gamma)$ The limit $f\coloneqq \lim_{\Lambda\to\infty} f_\Lambda$ exists in the compact open topology and satisfies for} \no \\
& \quad \;\;\;  \text{all $z\in \Sigma_{0,\theta}$,} \no \\
& \qquad \quad  f(z) = c_n (1+z)^{-n/2}  \lim_{\Lambda\to\infty}\frac{1}{\Lambda}\sum_{j,i_{1},\ldots,i_{n-1} = 1}^n \epsilon_{ji_{1}\ldots i_{n-1}}   \no \\ 
&\qquad \qquad \qquad \times
 \int_{\Lambda S^{n-1}}\tr (U(x) (\partial_{i_{1}} U)(x) \ldots 
 (\partial_{i_{n-1}} U)(x) )x_{j}\, d^{n-1} \sigma(x),    \label{i:13.9.3} \\
& \quad \;\;\;  \text{where}   \no \\
& \qquad \quad 
c_n\coloneqq \frac{1}{2} \left(\frac{i}{8\pi}\right)^{(n-1)/2}\frac{1}{\left[(n-1)/2\right]!}.  \no \\
& \text{$(\delta)$ $L$ is non-Fredholm, it admits a generalized Witten index, given by} 
\no \\
&  \qquad  \quad  \ind_{gW} (L) = f(0_+)     \no \\ 
 & \qquad \qquad \qquad \quad \;
 = c_n  \lim_{\Lambda\to\infty}\frac{1}{\Lambda}\sum_{j,i_{1},\ldots,i_{n-1} = 1}^n \epsilon_{ji_{1}\ldots i_{n-1}}      \label{i:13.9.4} \\ 
 &\qquad \qquad \qquad \qquad \; \times
 \int_{\Lambda S^{n-1}}\tr (U(x) (\partial_{i_{1}} U)(x) \ldots 
 (\partial_{i_{n-1}} U)(x))x_{j}\, d^{n-1} \sigma(x),   \no \\
& \quad \;\;\; \text{which is actually an integer as it coincides with the Fredholm index of} \no \\
& \quad \;\;\; \text{$\tilde L\coloneqq \mathcal{Q}+U$ in $L^2(\mathbb{R}^n)^{2^{\hatt n}d}$,  
that is,}  \no \\
& \hspace*{5cm} \ind_{gW} (L) = \ind \big(\tilde L\big). 
\end{align}
\end{theorem}
\begin{proof}
The proof rests on Theorem \ref{t:7.27}, Lemma \ref{l:13.3}, Remark \ref{r:13.7}, and specifically,  for the assertion that $L$ is non-Fredholm, on Corollary \ref{c:13.2}. Indeed, invoking Lemma \ref{l:13.3} and Remark \ref{r:13.7} with
\[
   P=\begin{pmatrix}
        I_\ell & 0\\ 0 & 0
     \end{pmatrix}\in \mathbb{C}^{(d+\ell)\times (d+\ell)},
\]
one computes, recalling $\tilde L=\mathcal{Q}+U$,
\begin{align*}
    L^*L&=  (-\mathcal{Q}+\Phi)(\mathcal{Q}+\Phi)
    \\  &= \begin{pmatrix}
            -\mathcal{Q} & 0\\
              0 & -\mathcal{Q}+U
           \end{pmatrix}\begin{pmatrix}
            \mathcal{Q} & 0\\
              0 & \mathcal{Q}+U
           \end{pmatrix}
           \\ & = \begin{pmatrix}
            -\Delta & 0\\
              0 & \tilde{L}^*\tilde{L}
           \end{pmatrix}.
\end{align*}
A similar computation applies to $LL^*$. One deduces for $z\in \mathbb{C}_{\Re>0}$, 
\begin{align*}
   & ((L^*L+z)^{-1}-(LL^*+z)^{-1})
   \\& \quad = \begin{pmatrix}
            (-\Delta+z)^{-1} & 0\\
              0 & (\tilde{L}^*\tilde{L}+z)^{-1}
           \end{pmatrix}-\begin{pmatrix}
            (-\Delta+z)^{-1} & 0\\
              0 & (\tilde{L}\tilde{L}^*+z)^{-1}
           \end{pmatrix}
           \\& \quad =\begin{pmatrix}
              0 & 0\\
              0 & (\tilde{L}^*\tilde{L}+z)^{-1}-(\tilde{L}\tilde{L}^*+z)^{-1}
           \end{pmatrix}.
\end{align*}
Thus,
\[
  \tr_{2^{\hatt n}(d+\ell)}\big((L^*L+z)^{-1}-(LL^*+z)^{-1})\big)= \tr_{2^{\hatt n}d} \big( (\tilde{L}^*\tilde{L}+z)^{-1}-(\tilde{L}\tilde{L}^*+z)^{-1}\big).
\]
Hence, the assertions \eqref{i:13.9.1}--\eqref{i:13.9.4} indeed follow from Theorem \ref{t:7.27} applied to $\tilde L$.
\end{proof}

\newpage

\appendix
\section{Construction of the Euclidean Dirac Algebra}\label{sec:Appendix:-the-Construction}
\renewcommand{\theequation}{A.\arabic{equation}}
\renewcommand{\thetheorem}{A.\arabic{theorem}}
\setcounter{theorem}{0} \setcounter{equation}{0}

For a concise presentation of the construction of the Euclidean Dirac
algebra as a specific case of Clifford algebras, see, for instance, 
\cite[Chapter 11]{Sn97}.

\begin{definition}
Given two matrices $A=(a_{ij})_{i,j\in\{1,\ldots,n\}}\in\mathbb{C}^{n\times n}$ and 
$B=(b_{ij})_{i,j\in\{1,\ldots,m\}}\in\mathbb{C}^{m\times m}$, one defines their 
\emph{Kronecker product} $A\circ B$ by
\begin{align*}
A\circ B & \coloneqq\left(\begin{array}{ccccc}
a_{11}B & a_{12}B & \cdots &  & a_{1n}B\\
a_{21}B & \ddots &  &  & \vdots\\
\vdots\\
\\
a_{n1}B & \cdots &  &  & a_{nn}B
\end{array}\right)\\
 & \, =\big(a_{\lceil\frac{p}{m}\rceil\lceil\frac{q}{m}\rceil}b_{(((p-1)\!\!\!\mod m)+1)(((q-1)\!\!\!\mod m)+1)}\big)_{p,q\in\{1,\ldots,mn\}}\in\mbox{\ensuremath{\mathbb{C}}}^{nm\times nm},
\end{align*}
where $\left\lceil x\right\rceil \coloneqq\min\{z\in\mathbb{Z}\,|\,z\geq x\}$
for all $x\in\mathbb{R}$ and $k\!\!\!\mod \ell$ denotes the nonnegative
integer $j\in\{0,\ldots,\ell-1\}$ such that $k-j$ is divisible by $\ell$, with 
$\ell,k\in\mathbb{Z}$.
\end{definition}

\begin{proposition}
\label{prop:rules Kronecker}Let $n,m,\ell,k\in\mathbb{N}$, $A\in\mathbb{C}^{n\times n},$
$B\in\mathbb{C}^{m\times m},$ $C\in\mathbb{C}^{\ell\times \ell}$, $D\in\mathbb{C}^{k\times k}$. 
Then one concludes that 
\begin{align*}
& A\circ(B\circ C)=(A\circ B)\circ C,   \\
& \left(A\circ B\right)^{*}=A^{*}\circ B^{*},  \\
& \tr(A\circ B)=\tr (A) \tr (B),  \\
& \, \text{if $n=m$ and $\ell=k$ then, } \, AB\circ CD=\left(A\circ C\right)(B\circ D).
\end{align*}
\end{proposition}
\begin{proof}
We only sketch a proof for the first assertion. It boils down to the
following equations, 
\begin{align*}
\left\lceil \frac{\left\lceil \frac{j}{k}\right\rceil }{m}\right\rceil  & =\left\lceil \frac{j}{m\,k}\right\rceil ,\\
\left(\left(\left\lceil \frac{j}{k}\right\rceil -1\right)\!\!\!\mod m\right)+1 & =\left\lceil \frac{\left(j-1\!\!\!\mod mk\right)+1}{k}\right\rceil ,\\
\left(j-1\!\!\!\mod mk\right)\!\!\!\mod k & = j-1\!\!\!\mod k, \quad j \in\{1,\ldots,mnk\}.
\end{align*} 
The expressions on the left-hand
side correspond to the indices of the entries of $A,B$ and $C$,
respectively, in $\left(A\circ B\right)\circ C$ and, similarly, the
expressions on the right-hand sides correspond to the respective indices
of the entries of $A$, $B$ and $C$ in $A\circ(B\circ C)$. 
\end{proof}

\begin{definition}
\label{def:Euc-D-A} Introduce the Pauli matrices 
$$
\sigma_{1}\coloneqq\left(\begin{array}{cc}
0 & 1\\
1 & 0
\end{array}\right),  \quad \sigma_{2}\coloneqq\left(\begin{array}{cc}
0 & -i\\
i & 0
\end{array}\right), \quad \sigma_{3}\coloneqq\left(\begin{array}{cc}
1 & 0\\
0 & -1
\end{array}\right), 
$$
in addition, define 
\[
\gamma_{1,2}\coloneqq\sigma_{1},\quad\gamma_{2,2}\coloneqq\sigma_{2}.
\]
Let ${\hatt n}\in\mathbb{N}$. Recursively, one sets
\begin{align*}
\gamma_{k,2{\hatt n}+1} & \coloneqq\gamma_{k,2{\hatt n}}, \quad k\in\{1,\ldots,2{\hatt n}\},\\
\gamma_{2{\hatt n}+1,2{\hatt n}+1} & \coloneqq\left(-i\right)^{{\hatt n}} \gamma_{1,2{\hatt n}}\cdots\gamma_{2{\hatt n},2{\hatt n}}, 
\end{align*}
and
\begin{align*}
\gamma_{k,2{\hatt n}+2} & \coloneqq\sigma_{1}\circ\gamma_{k,2{\hatt n}}, \quad k\in\{1,\ldots,2{\hatt n}\},\\
\gamma_{2{\hatt n}+1,2{\hatt n}+2} & \coloneqq i^{{\hatt n}}\sigma_{1}\circ\left(\gamma_{1,2{\hatt n}}\cdots\gamma_{2{\hatt n},2{\hatt n}}\right),\\
\gamma_{2{\hatt n}+2,2{\hatt n}+2} & \coloneqq\sigma_{2}\circ I_{2^{\hatt n}},   
\end{align*}
with $I_r$ the identity matrix in $\bbC^r$, $r \in \bbN$. 
\end{definition}

\begin{remark}
By induction, one obtains 
\begin{equation}
\gamma_{k,2{\hatt n}}, \gamma_{k,2{\hatt n}+1}, \gamma_{2{\hatt n}+1,2{\hatt n}+1} 
\in\mathbb{C}^{2^{{\hatt n}}\times2^{{\hatt n}}}, \quad k \in \{1,\dots,2{\hatt n}\}.
\end{equation} 
${}$ \hfill $\diamond$
\end{remark}

\begin{lemma}
\label{lem:anti-implies-minus}Let $\gamma_{1},\ldots,\gamma_{k}\in \cB(\cK)$
for some Hilbert space $\cK$ and such that for all $j,k \in\{1,\ldots,k\},j\neq k,$
one has $\gamma_{j}\gamma_{k}+\gamma_{k}\gamma_{j}=0$. 
Then 
\[
\gamma_{k}\gamma_{k-1}\cdots\gamma_{1}=(-1)^{k(k-1)/2}\gamma_{1}\gamma_{2}\cdots\gamma_{k}.
\]
\end{lemma}
\begin{proof}
The assertion being obvious for $k=1$, we assume that the assertion
of the lemma holds for some $k\in\mathbb{N}$. Then
\begin{align*}
\gamma_{k+1}\gamma_{k}\gamma_{k-1}\cdots\gamma_{1} 
& =(-1)^{k}\gamma_{k}\gamma_{k-1}\cdots\gamma_{1}\gamma_{k+1}\\
 & =(-1)^{[k(k-1)/2]+k}\gamma_{1}\gamma_{2}\cdots\gamma_{k}\gamma_{k+1}\\
 & =(-1)^{k(k+1)/2}\gamma_{1}\gamma_{2}\cdots\gamma_{k}\gamma_{k+1}.\tag*{{\qedhere}}
\end{align*}
\end{proof}

\begin{corollary}
\label{cor:anti-comm}For all $k,l\in\{1,\ldots,n\}$, $n\in\mathbb{N}_{\geq2}$, 
one has
\[
\gamma_{k,n}\gamma_{l,n}+\gamma_{l,n}\gamma_{k,n}=2\delta_{kl} I_{2^{\hatt n}},
\]
where $\gamma_{j,n}$ is given in Definition \ref{def:Euc-D-A}, $j\in\{1,\ldots,n\}$, and 
${\hatt n}\in \mathbb{N}$ is such that $n=2{\hatt n}$ or $n=2{\hatt n}+1$. 
\end{corollary}
\begin{proof}
The assertion holds for $n=2.$ Assume that the assertion is valid for $n=2{\hatt n}$ for some 
${\hatt n}\in\mathbb{N}$. Then Lemma \ref{lem:anti-implies-minus} implies 
\begin{align*}
\gamma_{2{\hatt n}+1,2{\hatt n}+1}^{2} & =(-i)^{2{\hatt n}}\left(\gamma_{1,2{\hatt n}}\cdots\gamma_{2{\hatt n},2{\hatt n}}\right)\left(\gamma_{1,2{\hatt n}}\cdots\gamma_{2{\hatt n},2{\hatt n}}\right)\\
 & =(-1)^{{\hatt n}}\left(\gamma_{1,2{\hatt n}}\cdots\gamma_{2{\hatt n},2{\hatt n}}\right)
 (-1)^{2{\hatt n}(2{\hatt n}-1)/2}\left(\gamma_{2{\hatt n},2{\hatt n}}\cdots\gamma_{1,2{\hatt n}}\right)\\
 & =\left(-1\right)^{{\hatt n}+2{\hatt n}^{2}-{\hatt n}}I_{2^{\hatt n}}=I_{2^{\hatt n}}.
\end{align*}
For $k\in\{1,\ldots,2{\hatt n}-1\}$ one computes 
\begin{align*}
\gamma_{k,2{\hatt n}+1}\gamma_{2{\hatt n}+1,2{\hatt n}+1} & =\gamma_{k,2{\hatt n}}(-i)^{{\hatt n}}\left(\gamma_{1,2{\hatt n}}\cdots\gamma_{2{\hatt n},2{\hatt n}}\right)\\
 & =(-1)^{2{\hatt n}-1}(-i)^{{\hatt n}}\left(\gamma_{1,2{\hatt n}}\cdots\gamma_{2{\hatt n},2{\hatt n}}\right)\gamma_{k,2{\hatt n}}   \\ 
 & =-\gamma_{2{\hatt n}+1,2{\hatt n}+1}\gamma_{k,2{\hatt n}+1}.
\end{align*}
Hence, the assertion is established for $\gamma_{k,2{\hatt n}+1}$, $k\in\{1,\ldots,2{\hatt n}+1\}$.  

For $k,l\in\{1,\ldots,2{\hatt n}\}$ one computes with the help of Proposition
\ref{prop:rules Kronecker},  
\begin{align*}
\gamma_{k,2{\hatt n}+2}\gamma_{l,2{\hatt n}+2}+\gamma_{l,2{\hatt n}+2}\gamma_{k,2{\hatt n}+2} & =\left(\sigma_{1}\circ\gamma_{k,2{\hatt n}}\right)\left(\sigma_{1}\circ\gamma_{l,2{\hatt n}}\right)   \\ 
& \quad +\left(\sigma_{1}\circ\gamma_{l,2{\hatt n}}\right)\left(\sigma_{1}\circ\gamma_{k,2{\hatt n}}\right)\\
 & =\sigma_{1}^{2}\circ\gamma_{k,2{\hatt n}}\gamma_{l,2{\hatt n}}+\sigma_{1}^{2}\circ\gamma_{l,2{\hatt n}}\gamma_{k,2{\hatt n}}\\
 & =I_2 \circ\left(\gamma_{k,2{\hatt n}}\gamma_{l,2{\hatt n}}+\gamma_{l,2{\hatt n}}\gamma_{k,2{\hatt n}}\right)\\
 & =I_2 \circ2\delta_{kl}I_{2^{\hatt n}}=2\delta_{kl}I_{2^{{\hatt n}+1}}.
\end{align*}
One observes that 
\begin{align*}
\gamma_{2{\hatt n}+1,2{\hatt n}+2}^{2} & =\left(i^{{\hatt n}}\right)^{2}\sigma_{1}^{2}\circ\left(\gamma_{1,2{\hatt n}}\cdots\gamma_{2{\hatt n},2{\hatt n}}\right)^{2} \\
 & =\left(i^{{\hatt n}}\right)^{2}\sigma_{1}^{2}\circ\left(-i\right)^{-2{\hatt n}}I_{2^{\hatt n}}=\left(i^{{\hatt n}}\right)^{2}\sigma_{1}^{2}\circ\left(-1\right)^{-2{\hatt n}}\left(i^{{\hatt n}}\right)^{-2}I_{2^{{\hatt n}}}=I_{2^{{\hatt n}+1}}, 
\end{align*}
using $\gamma_{2{\hatt n}+1,2{\hatt n}+1}^{2}=I_{2^{\hatt n}}$. Moreover, $\gamma_{2{\hatt n}+2,2{\hatt n}+2}^{2}=\sigma_{2}^{2}\circ I_{2^{\hatt n}}=I_{2^{{\hatt n}+1}}$. In addition, one notes that 
\begin{align*}
\gamma_{2{\hatt n}+2,2{\hatt n}+2}\gamma_{2{\hatt n}+1,2{\hatt n}+2} & =\sigma_{2}i^{{\hatt n}}\sigma_{1}\circ\gamma_{1,2{\hatt n}}\cdots\gamma_{2{\hatt n},2{\hatt n}}\\
 & =-i^{{\hatt n}}\sigma_{1}\sigma_{2}\circ\gamma_{1,2{\hatt n}}\cdots\gamma_{2{\hatt n},2{\hatt n}}=-\gamma_{2{\hatt n}+1,2{\hatt n}+2}\gamma_{2{\hatt n}+2,2{\hatt n}+2}, \\
\gamma_{2{\hatt n}+2,2{\hatt n}+2}\gamma_{k,2{\hatt n}+2}&=\sigma_{2}\sigma_{1}\circ\gamma_{k,{\hatt n}}=-\gamma_{k,2{\hatt n}+2}\gamma_{2{\hatt n}+2,2{\hatt n}+2},
\end{align*} 
and 
\begin{align*}
\gamma_{2{\hatt n}+1,2{\hatt n}+2}\gamma_{k,2{\hatt n}+2} & =i^{{\hatt n}}\sigma_{1}\sigma_{1}\circ\gamma_{1,2{\hatt n}}\cdots\gamma_{2{\hatt n},2{\hatt n}}\gamma_{k,2{\hatt n}}\\
 & =\sigma_{1}i^{{\hatt n}}\sigma_{1}\circ(-1)^{2{\hatt n}-1}\gamma_{k,2{\hatt n}}\gamma_{1,2{\hatt n}}\cdots\gamma_{2{\hatt n},2{\hatt n}}    \\
 & =-\gamma_{k,2{\hatt n}+2}\gamma_{2{\hatt n}+1,2{\hatt n}+2}
\end{align*}
for all $k\in\{1,\ldots,2{\hatt n}\}$, implying the assertion. 
\end{proof}

\begin{corollary}
For all $k \in\mbox{\ensuremath{\mathbb{N}}}$, $n\in\mathbb{N}_{\geq2}$, and
$k\leq  n$, one has
\[
\gamma_{k,n}^{*}=\gamma_{k,n},
\]
where $\gamma_{j,n}$ is given in Definition \ref{def:Euc-D-A}, $j\in\{1,\ldots,n\}$. 
\end{corollary}
\begin{proof}
We will proceed by induction. Before doing so, we note that due to Corollary
\ref{cor:anti-comm} and Lemma \ref{lem:anti-implies-minus}, 
for all $k\in\{1,\ldots,n\}$,  
\[
\gamma_{k,n}\gamma_{k-1,n}\cdots\gamma_{1,n}
=(-1)^{k(k-1)/2}\gamma_{1,n}\gamma_{2,n}\cdots\gamma_{k,n}.
\]
One observes that $\gamma_{1,2}$ and $\gamma_{2,2}$ are
self-adjoint. We assume that $\gamma_{k,2\hat n}$ is self-adjoint for all
$k\in\{1,\ldots,2{\hatt n}\}$ for some ${\hatt n}\in\mathbb{N}.$ The only matrices 
not obviously self-adjoint using the induction hypothesis
and Proposition \ref{prop:rules Kronecker} are $\gamma_{2{\hatt n}+1,2{\hatt n}+2}$
and $\gamma_{2{\hatt n}+1,2{\hatt n}+1}$. Since the proof for either case follows along similar lines, it 
suffices to prove the self-adjointness of $\gamma_{2{\hatt n}+1,2{\hatt n}+2}$. For this purpose 
one computes, 
\begin{align*}
\gamma_{2{\hatt n}+1,2{\hatt n}+2}^{*} & =\left(i^{{\hatt n}}\sigma_{1}\circ\left(\gamma_{1,2{\hatt n}}\cdots\gamma_{2{\hatt n},2{\hatt n}}\right)\right)^{*}\\
 & =i^{{\hatt n}}(-1)^{{\hatt n}}\sigma_{1}^{*}\circ\left(\gamma_{1,2{\hatt n}}\cdots\gamma_{2{\hatt n},2{\hatt n}}\right)^{*}\\
 & =i^{{\hatt n}}(-1)^{{\hatt n}}\sigma_{1}\circ\left(\gamma_{2{\hatt n},2{\hatt n}}\cdots\gamma_{1,2{\hatt n}}\right)\\
 & =i^{{\hatt n}}(-1)^{{\hatt n}+[2{\hatt n}(2{\hatt n}-1)/2]}\sigma_{1}\circ\left(\gamma_{1,2{\hatt n}}\cdots\gamma_{2{\hatt n},2{\hatt n}}\right)\\
 & =i^{{\hatt n}}(-1)^{{\hatt n}+2{\hatt n}^{2}-{\hatt n}}\sigma_{1}\circ\left(\gamma_{1,2{\hatt n}}\cdots\gamma_{2{\hatt n},2{\hatt n}}\right)\\
 & =\gamma_{2{\hatt n}+1,2{\hatt n}+2}.\tag*{{\qedhere}}
\end{align*}
\end{proof}

Next, we proceed to establish the following result on traces:

\begin{proposition}
\label{prop:comp_of_Dirac_Alge} 
Let ${\hatt n}\in\mathbb{N}$ and suppose that $\gamma_{j,2{\hatt n}}$, $\gamma_{j',2{\hatt n}+1}$, 
$j\in\{1,\ldots,2{\hatt n}\}$, $j'\in\{1,\ldots,2{\hatt n}+1\}$, are given as in Definition \ref{def:Euc-D-A}. 
Then, 
\begin{align*} 
& \text{$\tr\left(\gamma_{i_{1},2{\hatt n}+1}\cdots\gamma_{i_{2k+1},2{\hatt n}+1}\right)=0$,
if $i_{1},\ldots,i_{2k+1}\in\{1,\ldots,2{\hatt n}+1\}$ and $k<{\hatt n}$,} \\ 
& \text{$\tr\left(\gamma_{i_{1},2{\hatt n}}\cdots\gamma_{i_{2k+1},2{\hatt n}}\right)=0$,
if $i_{1},\ldots,i_{2k+1}\in\{1,\ldots,2{\hatt n}\}$ and $k\in\mathbb{N}$,} \\ 
& \text{$\tr\left(\gamma_{i_{1},2{\hatt n}+1}\cdots\gamma_{i_{2{\hatt n}+1},2{\hatt n}+1}\right)=(2i)^{{\hatt n}}\epsilon_{i_{1}\cdots i_{2{\hatt n}+1}}$,
if $i_{1},\ldots,i_{2{\hatt n}+1}\in\{1,\ldots,2{\hatt n}+1\}$,}
\end{align*} 
where $\epsilon_{i_{1}\cdots i_{2{\hatt n}+1}}$
is the fully anti-symmetric symbol in $2{\hatt n}+1$ dimensions, that is, 
$\epsilon_{i_{1}\ldots i_{2{\hatt n}+1}}=0$
whenever $\left|\{i_{1},\ldots,i_{2{\hatt n}+1}\}\right|<2{\hatt n}+1$ and if 
$\pi\colon\{1,\ldots,2{\hatt n}+1\}\to\{1,\ldots,2{\hatt n}+1\}$
is bijective, then $\epsilon_{\pi(1)\cdots\pi(2{\hatt n}+1)}=\sgn (\pi).$ 
\end{proposition}
\begin{proof}
The first formula can be seen as follows. Since $k< \hatt n$, there exists
$i\in\{1,\ldots,2{\hatt n}+1\}\backslash \{i_{1},\ldots,i_{2k+1}\}$, and one computes
\begin{align*}
& \tr\left(\gamma_{i_{1},2{\hatt n}+1}\cdots\gamma_{i_{2k+1},2{\hatt n}+1}\right) 
= \tr\left(\gamma_{i_{1},2{\hatt n}+1}\cdots\gamma_{i_{2k+1},2{\hatt n}+1}\gamma_{i,2{\hatt n}+1}^{2}\right)\\
 & \quad =\tr\left(\gamma_{i,2{\hatt n}+1}\gamma_{i_{1},2{\hatt n}+1}\cdots
 \gamma_{i_{2k+1},2{\hatt n}+1}\gamma_{i,2{\hatt n}+1}\right)\\
 & \quad =-\tr\left(\gamma_{i_{1},2{\hatt n}+1}\gamma_{i,2{\hatt n}+1}\cdots\gamma_{i_{2k+1},2{\hatt n}+1}\gamma_{i,2{\hatt n}+1}\right)\\
 & \quad =\ldots=(-1)^{2k+1}\tr\left(\gamma_{i_{1},2{\hatt n}+1}\cdots
 \gamma_{i_{2k+1},2{\hatt n}+1}\gamma_{i,2{\hatt n}+1}\gamma_{i,2{\hatt n}+1}\right)\\
 & \quad =-\tr\left(\gamma_{i_{1},2{\hatt n}+1}\cdots\gamma_{i_{2k+1},2{\hatt n}+1}\right).
\end{align*}
Hence, $\tr\left(\gamma_{i_{1},2{\hatt n}+1}\cdots\gamma_{i_{2k+1},2{\hatt n}+1}
\gamma_{i,2{\hatt n}+1}\gamma_{i,2{\hatt n}+1}\right)=0$.

The second assertion can be proved along the same lines.

The third assertion follows upon taking into account the cancellation and anti-commuting
properties of the algebra in conjunction with the first statement, once the following
equality has been established:
\[
\tr\left(\gamma_{1,2{\hatt n}+1}\cdots\gamma_{2{\hatt n}+1,2{\hatt n}+1}\right)=(2i)^{{\hatt n}}.
\]
To verify the latter identity one computes 
\begin{align*}
 & \tr\left(\gamma_{1,2{\hatt n}+3}\cdots\gamma_{2{\hatt n}+3,2{\hatt n}+3}\right)\\
 & \quad =\tr\left(\gamma_{1,2{\hatt n}+2}\cdots\gamma_{2{\hatt n}+2,2{\hatt n}+2}(-i)^{{\hatt n}+1}\gamma_{1,2{\hatt n}+2}\cdots\gamma_{2{\hatt n}+2,2{\hatt n}+2}\right)\\
 & \quad =(-i)^{{\hatt n}+1}\tr\left(\left(\sigma_{1}\circ\gamma_{1,2{\hatt n}}\right)\cdots\left(\sigma_{1}\circ\gamma_{2{\hatt n},2{\hatt n}}\right)i^{{\hatt n}}\left(\sigma_{1}\circ\gamma_{1,2{\hatt n}}\ldots\gamma_{2{\hatt n},2{\hatt n}}\right)\left(\sigma_{2}\circ I_{2^{{\hatt n}}}\right)\right.\\
 & \quad\quad\left. \times \left(\sigma_{1}\circ\gamma_{1,2{\hatt n}}\right)\cdots\left(\sigma_{1}\circ\gamma_{2{\hatt n},2{\hatt n}}\right)i^{{\hatt n}}\left(\sigma_{1}\circ\gamma_{1,2{\hatt n}}\ldots\gamma_{2{\hatt n},2{\hatt n}}\right)\left(\sigma_{2}\circ I_{2^{{\hatt n}}}\right)\right)\\
 & \quad =\left(i^{2}\right)^{{\hatt n}}\left(-i\right)^{{\hatt n}+1}\tr\left(\sigma_{1}^{2{\hatt n}+1}\sigma_{2}\sigma_{1}^{2{\hatt n}+1}\sigma_{2}\circ\left(\gamma_{1,2{\hatt n}}\ldots\gamma_{2{\hatt n},2{\hatt n}}\right)^{4}\right)\\
 & \quad =\left(-1\right)^{{\hatt n}+1}\left(-i\right)^{{\hatt n}+1}\tr\left(\sigma_{1}\sigma_{1}\sigma_{2}\sigma_{2}\circ I_{2^{\hatt n}}\right)=i^{{\hatt n}+1}2^{{\hatt n}+1}.\tag*{{\qedhere}} 
\end{align*}
\end{proof}

We conclude with the following result.

\begin{corollary}
\label{cor:comp_Dirac_alg_trace-1}Let $n\in\mathbb{N}_{\geq 2}$ be odd,
$V$ be a complex vector space, $k\in\mathbb{N}_{0}$, with $k+1<n$,
$i_{1},\ldots,i_{k}\in\{1,\ldots,n\}$. Let $\Phi\colon\{1,\ldots,n\}^{n}\to V$
be satisfying the property 
\begin{align*}
& \sum_{(i_{k+3},\ldots,i_{n})\in\{1,\ldots,n\}^{n-k-2}}\Phi(i_{1},\ldots,i_{k},i,j,i_{k+3},\ldots,i_{n}) \\
& \quad =\sum_{(i_{k+3},\ldots,i_{n})\in\{1,\ldots,n\}^{n-k-2}}\Phi(j_{1},\ldots,j_{k},j,i,i_{k+3},\ldots,i_{n}), \quad i,j\in\{1,\ldots,n\}. 
\end{align*} 
Then
\[
\sum_{(i_{k+1},i_{k+2},i_{k+3},\ldots,i_{n})\in\{1,\ldots,n\}^{n-k}}\tr\big(\gamma_{i_{1},n}\cdots\gamma_{i_{n},n}\big)\Phi(i_{1},\ldots,i_{n})=0,
\] 
where $\gamma_{j,n}$, $j\in\{1,\ldots,n\}$, are given by Definition \ref{def:Euc-D-A}. 
\end{corollary}
\begin{proof}
In the course of this proof we shall suppress the index $n$ in $\gamma_{i,n}$.  
\begin{align*}
 & \sum_{(i_{k+1},i_{k+2},i_{k+3},\ldots,i_{n})\in\{1,\ldots,n\}^{n-k}}\gamma_{i_{1}}\cdots\gamma_{i_{n}}\Phi(i_{1},\ldots,i_{k},i_{k+1},i_{k+2},i_{k+3},\ldots,i_{n})\\
 & \quad =\frac{1}{2}\sum_{(i_{k+1},i_{k+2},i_{k+3},\ldots,i_{n})\in\{1,\ldots,n\}^{n-k}}\gamma_{i_{1}}\cdots\gamma_{i_{n}}\Phi(i_{1},\ldots,i_{k},i_{k+1},i_{k+2},i_{k+3},\ldots,i_{n})\\
 & \qquad +\frac{1}{2}\sum_{(i_{k+1},i_{k+2},i_{k+3},\ldots,i_{n})\in\{1,\ldots,n\}^{n-k}}\gamma_{i_{1}}\cdots\gamma_{i_{n}}   \\ 
 & \hspace*{5.9cm} \times \Phi(i_{1},\ldots,i_{k},i_{k+2},i_{k+1},i_{k+3},\ldots,i_{n})    \\
 & \quad =\frac{1}{2}\sum_{(i_{k+1},i_{k+2},i_{k+3},\ldots,i_{n})\in\{1,\ldots,n\}^{n-k}}\\
 & \quad \qquad\left(\gamma_{i_{1}}\cdots\gamma_{i_{k}}\gamma_{i_{k+1}}\gamma_{i_{k+2}}\gamma_{i_{k+3}}\cdots\gamma_{i_{n}}+\gamma_{i_{1}}\cdots\gamma_{i_{k}}\gamma_{i_{k+2}}\gamma_{i_{k+1}}\gamma_{i_{k+3}}\cdots\gamma_{i_{n},n}\right)  \\
 & \qquad \quad \times \Phi(i_{1},\ldots,i_{n})    \\
 & \quad =\frac{1}{2}\sum_{(i_{k+1},i_{k+2},i_{k+3},\ldots,i_{n})\in\{1,\ldots,n\}^{n-k},i_{k+1}\neq i_{k+2}}\\
 & \quad \qquad\left(\gamma_{i_{1}}\cdots\gamma_{i_{k}}\gamma_{i_{k+1}}\gamma_{i_{k+2}}\gamma_{i_{k+3}}\cdots\gamma_{i_{n}}+\gamma_{i_{1}}\cdots\gamma_{i_{k}}\gamma_{i_{k+2}}\gamma_{i_{k+1}}\gamma_{i_{k+3}}\cdots\gamma_{i_{n}}\right)   \\
& \qquad \quad \times \Phi(i_{1},\ldots,i_{n})    \\
 & \qquad+\frac{1}{2}\sum_{(i_{k+1},i_{k+2},i_{k+3},\ldots,i_{n})\in\{1,\ldots,n\}^{n-k},i_{k+1}=i_{k+2}}\\
 & \quad \qquad\left(\gamma_{i_{1}}\cdots\gamma_{i_{k}}\gamma_{i_{k+1}}\gamma_{i_{k+2}}\gamma_{i_{k+3}}\cdots\gamma_{i_{n}}+\gamma_{i_{1}}\cdots\gamma_{i_{k}}\gamma_{i_{k+2}}\gamma_{i_{k+1}}\gamma_{i_{k+3}}\cdots\gamma_{i_{n}}\right)   \\ 
 & \qquad \quad \times \Phi(i_{1},\ldots,i_{n})    \\
 & \quad =\frac{1}{2}\sum_{(i_{k+1},i_{k+2},i_{k+3},\ldots,i_{n})\in\{1,\ldots,n\}^{n-k},i_{k+1}
 =i_{k+2}}\\
 & \quad \qquad\left(\gamma_{i_{1}}\cdots\gamma_{i_{k}}\gamma_{i_{k+3}}\cdots\gamma_{i_{n}}+\gamma_{i_{1}}\cdots\gamma_{i_{k}}\gamma_{i_{k+3}}\cdots\gamma_{i_{n}}\right)\Phi(i_{1},\ldots,i_{n})\\
 & \quad =\sum_{(i_{k+1},i_{k+2},i_{k+3},\ldots,i_{n})\in\{1,\ldots,n\}^{n-k},i_{k+1}=i_{k+2}}\gamma_{i_{1}}\cdots\gamma_{i_{k}}\gamma_{i_{k+3}}\cdots\gamma_{i_{n}}\Phi(i_{1},\ldots,i_{n}).
\end{align*} 
Applying the internal trace to the latter sum, one infers that each
term vanishes by Proposition \ref{prop:comp_of_Dirac_Alge}.
\end{proof}

\newpage

\section{A Counterexample to \cite[Lemma~5]{Ca78}} \label{sB}
\renewcommand{\theequation}{B.\arabic{equation}}
\renewcommand{\thetheorem}{B.\arabic{theorem}}
\setcounter{theorem}{0} \setcounter{equation}{0}

In this appendix we shall provide a counterexample for the trace class property asserted in \cite[Lemma~5]{Ca78}. The counterexample is constructed in dimension $n=3$ and recorded 
in Theorem \ref{thm:countex_is_successful}. 

Analogously to Example \ref{exa:standard_example_0}, we let $\Phi$ assume values in the $2\times2$ matrices and denote the Pauli matrices (see also Example \ref{exa:standard_example_0}) again by $\sigma_j$, $j\in \{1,2,3\}$. Before we give an explicit formula for $\Phi$, we need the following definitions.
Let $\phi_1\in C^\infty(\mathbb{R})$ be a function interpolating between $0$ and $1$ with 
\[
\phi_1 (x) = \begin{cases}  0,&x\leq  0,\\
1,&x\geq 1, \end{cases} \quad x\in \mathbb{R}, 
 \quad  \phi_2\coloneqq \phi_1(-(\cdot+1))
\]						    
and let
\[
 \phi_{1,r,t} \coloneqq \phi_1 \big(t^{-1}(\cdot)- t^{-1}r\big), \quad \phi_{2,r,t} \coloneqq \phi_2 \big(t^{-1}(\cdot)- t^{-1}r\big),\quad r,t>0. 
\]
For $r_1,r_2,t_1,t_2 \in (0,\infty)$ with  $r_1+t_1<r_2-t_2$, this yields the following variant of a smooth ``cut-off'' function
\begin{equation} \label{eq:def_of_smooth_indi}
  \psi_{r_1,r_2,t_1,t_2} \coloneqq \phi_{1,r_1,t_1}\phi_{2,r_2-t_2,t_2}. 
\end{equation}
One notes that $\psi_{r_1,r_2,t_1,t_2}\in C^\infty(\bbR)$. We will use the following properties of $\psi_{r_1,r_2,t_1,t_2}$ (all of them are easily checked): 
\begin{align}
 & 0\leq  \psi_{r_1,r_2,t_1,t_2} \leq  1,   \label{psi_bdd}\\
 & \psi_{r_1,r_2,t_1,t_2}|_{[r_1+t_1,r_2-t_2]}=1,    \label{psi_1_on_IV}\\
 & \psi_{r_1,r_2,t_1,t_2}|_{\mathbb{R}\backslash [r_1,r_2]}=0,   \label{psi_0_on_compIV}\\
 & |\psi'_{r_1,r_2,t_1,t_2}|\leq  d_1\left(\frac{1}{t_1}\lor \frac{1}{t_2}\right) \text{ on }[r_1,r_1+t_1]\cup[r_2-t_2,r_2],    \label{psi_prime_bounded}\\
 & |\psi_{r_1,r_2,t_1,t_2}^{(\ell)}|\leq  d_\ell\left(\frac{1}{t_1^\ell}\lor \frac{1}{t_2^\ell}\right),    
 \quad \ell\in \mathbb{N}_{\geq2},  \label{psi_double_prime}
\end{align}
with $d_1 \coloneqq \|\phi_1'\|_\infty \coloneqq \sup_{x\in \mathbb{R}} |\phi_1'(x)|$ 
and $d_\ell \coloneqq \|\phi_1^{(\ell)}\|_\infty$, $\ell\in \mathbb{N}_{\geq2}$. 
For $k\in \mathbb{N}_{>1}$ define
\begin{align*} 
& r_k \coloneqq \sum_{j=1}^{k-1} 2^j=2^k-2,   \\
& \psi_{1,k}\coloneqq \psi_{r_k,r_{k+1},\frac{1}{2}2^k,\frac{1}{20}2^k}, 
\quad \psi_{2,k}\coloneqq \psi_{r_k,r_{k+1},\frac{1}{36}2^k,\frac{17}{18}2^k}. 
\end{align*} 
One observes that 
\[
r_k+\frac{1}{2}2^k=r_{k+1}-\frac{1}{2}2^k<r_{k+1}-\frac{1}{20}2^k, \quad 
r_k+\frac{1}{36}2^k=r_{k+1}-\frac{35}{36}2^k< r_{k+1}-\frac{17}{18}2^k, 
\]
so that $\psi_{1,k}$ and $\psi_{2,k}$ are well-defined.
For $x=(x_1,x_2,x_3)\in \mathbb{R}^3$ we let $\Phi\colon\mathbb{R}^3\to\mathbb{C}^{2\times 2}$ be defined as follows, 
\begin{equation}\label{eq:count_phi}
    \Phi(x)\coloneqq \sum_{j=1}^3\sigma_j+\sum_{k=2}^\infty \frac{1}{k^{1/3}} 
    \sum_{j=1}^3 \sigma_j \, \xi_{k,j}(x),    \quad x \in \bbR^3, 
\end{equation}
where 
\begin{equation}\label{eq:pkj}
   \xi_{k,j}(x) \coloneqq \frac{1}{r_{k+1}}\psi_{1,k} (|x|)(x_j-r_k)\psi_{2,k}(x_j), 
   \quad x \in \bbR^3.
\end{equation}
One observes that $\Phi \in C^{\infty}(\bbR)$. Next, we introduce the sets
\begin{equation}\label{eq:def_B_k}
 B_k \coloneqq \big\{x\in \mathbb{R}^3 \, \big| \, r_k \leq  |x|\leq  r_{k+1}\big\} 
 \cap \bigcup_{j\in \{1,2,3\}} \big\{x\in \mathbb{R}^3 \, \big| \, r_k\leq  x_j\leq  r_{k+1}\big\}, 
 \quad k \in \bbN, 
\end{equation}
and 
\begin{align}\label{eq:Bktilde}
 \no
\tilde B_k & \coloneqq \left\{ x\in \mathbb{R}^3 \, \bigg| \, r_k+\frac{1}{2}2^k\leq  |x|\leq  r_{k+1}-\frac{1}{20}2^k \right\}    \\ 
& \qquad \bigcap \left\{ x\in \mathbb{R}^3 \, \bigg| \, r_k+\frac{1}{36}2^k\leq  x_1,x_2,x_3 \leq  r_{k+1}-\frac{17}{18}2^k\right\},  \quad k \in \bbN. 
\end{align}

Before turning to the properties of $\Phi$, we study $\xi_{k,j}$ first.

\begin{lemma}\label{lem:b1} Let $j\in \{1,2,3\}$, $\ell\in \{1,2,3\}$, $\xi_{k,j}$ as in \eqref{eq:pkj}, $B_k$, $\tilde B_k$ as in \eqref{eq:def_B_k} and \eqref{eq:Bktilde}, respectively. Then the following assertions $(\alpha)$--$(\gamma)$ hold: 
\begin{align} 
& \text{$(\alpha)$ For all $k\in \mathbb{N}$, $x\in \mathbb{R}^3$, 
$\xi_{k,j}(x)\neq 0$ implies $x\in B_k$.}    \label{i:b11} \\
& \text{$(\beta)$ For all $\alpha\in \mathbb{N}^{3}_0$, there exists $\kappa>0$ such that for all $k\in \mathbb{N}$,}   \no \\
& \hspace*{3.2cm}
 |\partial^\alpha\xi_{k,j}^{}(x)|\leq \kappa (1+|x|)^{- |\alpha|}, 
    \quad x\in B_k.  \label{i:b12} \\
& \text{$(\gamma)$ For all $\ell\in \{1,2,3\}$, and all $k \in \bbN$, 
$\partial_\ell \xi_{k,j} (x)=\delta_{\ell j}$, $x\in \tilde B_k$.}   \label{i:b13}
\end{align} 
\end{lemma}
\begin{proof}
\eqref{i:b11}: The assertion follows from \eqref{psi_0_on_compIV} and the definition of $B_k$. 
\\[1mm] 
  \eqref{i:b12}: One observes that $\psi_{2,k}\neq 0$ on $(r_k,r_{k+1})$ and that $0\leq  \psi_{2,k}\leq  1$ by  \eqref{psi_0_on_compIV} and \eqref{psi_bdd}; hence, 
\[|(x_j-r_k)\psi_{2,k}(x_j)|\leq  2^k, \quad j\in\{1,2,3\}, \; k\in \mathbb{N}_{\geq 2}. 
\]
One recalls, 
\[
   r_k =\sum_{j=1}^{k-1} 2^j =2^k-2 <r_{k+1} = 2^{k+1}-2=2(2^k-1), 
\]
in particular, $(1/r_{k+1}) \leq  \kappa_02^{-k}$ for some $\kappa_0>0$. Hence, 
\[
  \bigg\|\frac{1}{r_{k+1}}\psi_{1,k} (|x|)\sum_{j=1}^3 \sigma_j (x_j-r_k)\psi_{2,k}(x_j)\bigg\| 
  \leq  \chi_{B_k}(x) \kappa_0, \quad x\in \mathbb{R}^3, 
\]
with $B_k$ introduced in \eqref{eq:def_B_k}. Thus, \eqref{i:b12} holds for $\ell=0$. Next, for the first partial derivatives in item \eqref{i:b12} one computes for  $\ell\neq j$,  
\[
(\partial_\ell \xi_{k,j})(x)=  \frac{1}{r_{k+1}}\psi_{1,k}'(|x|)\frac{x_\ell}{|x|} (x_j-r_k)\psi_{2,k}(x_j)
\]
and for $\ell=j$, 
\begin{align*}
 (\partial_j \xi_{k,j})(x) &= \frac{1}{r_{k+1}}\psi_{1,k}'(|x|)\frac{x_j}{|x|} (x_j-r_k)\psi_{2,k}(x_j) + \frac{1}{r_{k+1}}\psi_{1,k}(|x|) \sigma_j\psi_{2,k}(x_j)\\ 
 & \quad + \frac{1}{r_{k+1}}\psi_{1,k}(|x|) \sigma_j(x_j-r_k)\psi_{2,k}'(x_j), \quad j\in \{1,2,3\}.
\end{align*}
 For $x\in B_k$, one has 
$|(x_j-r_k)\psi_k'(x_j)|\leq  c$ by \eqref{psi_prime_bounded},
$|\psi_k'(|x|)(x_\ell-r_k)\psi_k(x_\ell)|\leq  c$ by \eqref{psi_bdd} and \eqref{psi_prime_bounded} 
and for some $\kappa,c>0$ and all $k\in \mathbb{N}_{\geq 2}$, 
\[
\bigg\|\frac{1}{r_{k+1}}\psi_k(|x|) \sigma_j\psi_k(x_j)\bigg\| 
\leq \bigg\|\frac{1}{r_{k+1}}\psi_k(|x|)\bigg\| \leq \kappa (1+|x|)^{-1} 
\]
since for all $x\in B_k$ one has $r_{k+1} \geq |x|$. 
Higher-order derivatives can be treated similarly, using \eqref{psi_double_prime}, proving 
assertion \eqref{i:b12}. \\[1mm]  
\eqref{i:b13}: This is obvious.
\end{proof}

The next lemma gives an account of the asymptotic properties of $\Phi$ and its derivatives.

\begin{lemma}\label{lem:asy of phi_count} Let $\Phi$ be given by \eqref{eq:count_phi}. Then the following assertions $(\alpha)$--$(\gamma)$ hold: \\[1mm] 
$(\alpha)$ $\Phi$ is bounded, pointwise self-adjoint, 
$\Phi \in C^{\infty}\big(\bbR^3;\bbC^{2\times 2}\big)$, $\Phi(x)^{-1}$ exists for all $x\in \mathbb{R}^3$, and $\Phi(x)^2 \underset{|x|\to\infty}{\longrightarrow} I_2$. 
\\[1mm] 
$(\beta)$ There exists $\kappa>0$ such that 
$$
|(\partial_j\Phi)(x)|\leq  \kappa (1+|x|)^{-1}, \quad x\in \mathbb{R}^3, \; j\in\{1,2,3\}, 
$$
and the formula
\begin{align*}
(\partial_j \Phi)(x)&= k^{-1/3}\sigma_j \quad x\in \tilde B_k, \; j\in\{1,2,3\}, \; k\in \mathbb{N}, 
\end{align*}
holds, where $\tilde B_k$ is given by \eqref{eq:Bktilde}. \\[1mm] 
$(\gamma)$ For all $\alpha\in \mathbb{N}_0^n$ with $|\alpha|\geq 2$, there exists $\kappa'>0$, 
such that
$$
|(\partial^\alpha\Phi)(x)|\leq  \kappa'\left(1+|x|\right)^{-|\alpha|}, \quad
x\in \mathbb{R}^3.
$$ 
\end{lemma}
\begin{proof}
 For item $(\alpha)$, we use Lemma \ref{lem:b1} \eqref{i:b12} with $\ell=0$ together with the fact that $B_k\cap B_{k'}=\emptyset$ for $k' > k + 1$, so  $\Phi$ is bounded. $\Phi$ is easily verified to be pointwise self-adjoint. For showing invertibility 
of $\Phi$, one computes for $x\in B_k$, 
\begin{align*}
  \Phi(x)\Phi(x)& = \bigg(\sum_{j=1}^3\left(\sigma_j+\frac{1}{k^{1/3}}\xi_{k,j} (x)\sigma_j \right)\bigg)^2 \\
                & = \sum_{j=1}^3 \left(1+\frac{1}{k^{1/3}}\xi_{k,j} (x)\right)^2I_2\\
                & = \sum_{j=1}^3 \left(1+2\frac{1}{k^{1/3}}\xi_{k,j} (x)+\left(\frac{1}{k^{1/3}}\xi_{k,j} (x)\right)^2\right)I_2 \geq I_2,
\end{align*}
implying $(\alpha)$. 
Item $(\beta)$ follows from Lemma \ref{lem:b1}, \eqref{i:b12}, and \eqref{i:b13}, whereas 
item $(\gamma)$ follows from \eqref{lem:b1}, \eqref{i:b12}.
\end{proof}

In order to prove that $\tr_4\left((L^*L+z)^{-1}-(LL^*+z)^{-1}\right)$ for $L= \cQ +\Phi$  
(with $\cQ$ as in \eqref{eq:def_of_Q2}) is \emph{not} trace class for $z$ in a neighborhood of $0$, we need to invoke the following general statement:

\begin{theorem}[{\cite[Theorem 3.1]{Br88}}]\label{thm:potential_is_in_L1} Let $K\in \mathcal{B}(L^2(\mathbb{R}^n))$ be an operator induced by a continuous integral kernel 
$k\colon \mathbb{R}^{n} \times \bbR^n \to\mathbb{C}$. Assume that $K\in \mathcal{B}_1\big(L^2(\mathbb{R}^n)\big)$. Then the 
function $x\mapsto k(x,x)$ defines an element of $L^1(\mathbb{R}^n)$. 
\end{theorem}

Before we state and prove the main result of this section, we need to study the volume 
of $\tilde B_k$:

\begin{lemma}\label{lem:b2} Let $\tilde B_k$, $k\in \mathbb{N}$, be as in \eqref{eq:Bktilde}. Then there exists $k_0\in \mathbb{N}$, such that for all $k\in \mathbb{N}_{\geq k_0}$,  
\[
   \vol \big(\tilde B_k\big) = 2^{3k}/(36)^3.
\] 
\end{lemma}
\begin{proof}
Let $k\in \mathbb{N}$. One observes that if 
\[
x\in \left\{ x\in \mathbb{R}^3 \, \bigg| \, r_k+\frac{1}{36}2^k\leq  x_1,x_2,x_3 \leq  r_{k+1}-\frac{17}{18}2^k\right\}, 
\] 
then  
\[
 \sqrt{3}\left(r_k+\frac{1}{36}2^k\right)\leq  |x|\leq  \sqrt{3}\left(r_{k+1}-\frac{17}{18}2^k\right).
\]
Since $16/10 \leq  \sqrt{3}\leq 18/10$, for sufficiently large $k\in \mathbb{N}$, the estimates
\begin{equation*}
   \sqrt{3}\left(r_k+\frac{1}{36}2^k\right)\geq \left(\frac{16}{10}+\frac{16}{10}\frac{1}{36}\right)2^k - \frac{16}{10}2 \geq \frac{3}{2}2^k-2= r_k+\frac{1}{2}2^k, 
\end{equation*}
and 
\begin{equation*}
    \sqrt{3}\left(r_{k+1}-\frac{17}{18}2^k\right)\leq  \frac{18}{10}\left(2-\frac{17}{18}\right)2^{k}-2\frac{18}{10}\leq  \frac{19}{10}2^k-2 = r_{k+1}-\frac{1}{20}2^k, 
\end{equation*}
hold. Consequently, for sufficiently large $k\in \mathbb{N}$, 
\[
 \left\{ x\in \mathbb{R}^3 \, \bigg| \, r_k+\frac{1}{36}2^k\leq  x_1,x_2,x_3 \leq  r_{k+1}-\frac{17}{18}2^k\right\}\subseteq \tilde B_k. 
\]
Hence, there exists $k_0 \in \bbN$, such that for all $k\in\mathbb{N}_{\geq k_0}$, 
\[
    \vol \big(\tilde B_k\big) = \left(r_{k+1}-\frac{17}{18}2^k-(r_k+\frac{1}{36}2^k)\right)^3.\qedhere
\]
\end{proof}

\begin{theorem}\label{thm:countex_is_successful} Let $n=3$ and $\cQ$ and $\Phi$ be given by \eqref{eq:def_of_Q2} and \eqref{eq:count_phi}, respectively. Then there exists $\delta>0$ such that for $L= \cQ +\Phi$, and for any real $z\in B(0,\delta)\backslash \{0\}$,  
\[
   \tr_4\left((L^*L+z)^{-1}-(LL^*+z)^{-1}\right) \notin \cB_1\big(L^2(\bbR^3)\big). 
\] 
\end{theorem}
\begin{proof} 
In view of Remark \ref{rem:remark on op norm of Rpsiz} and Lemma \ref{lem:almost_Neumann_series} it suffices to check whether or not 
\[
 \tilde T\coloneqq \tr_{4} \big((R_{1+z}C)^3R_{1+z}\big)
\]
 is a trace class operator, where $C=[ \cQ,\Phi]$, and $R_{1+z}$ are given by \eqref{eq:def_commutator} and \eqref{eq:resolvent_of_laplace}, respectively. 

Arguing by contradiction, we shall assume that 
$\tilde T\in \mathcal{B}_1\big(L^2(\mathbb{R}^3)\big)$. One observes, 
\begin{align}\notag
& (R_{1+z}C)^3R_{1+z} = R_{1+z}CR_{1+z}CR_{1+z}CR_{1+z} \\
 & \quad = [R_{1+z},C] R_{1+z}C R_{1+z}CR_{1+z}+ C R_{1+z} R_{1+z}CR_{1+z}CR_{1+z} 
 \notag\\
 & \quad = [R_{1+z},C] R_{1+z}C R_{1+z}CR_{1+z}+ C R_{1+z} [R_{1+z},C] R_{1+z}CR_{1+z}\notag\\  
 & \qquad+C R_{1+z}C R_{1+z}R_{1+z}CR_{1+z}\notag\\
 & \quad = [R_{1+z},C] R_{1+z}C R_{1+z}CR_{1+z}+ C R_{1+z} [R_{1+z},C] R_{1+z}CR_{1+z}\notag\\ 
 & \qquad +C R_{1+z}C R_{1+z}[R_{1+z},C]R_{1+z}+C R_{1+z}C R_{1+z}CR_{1+z}R_{1+z}.\label{eq:com_count_1st_round}
\end{align}
By Lemmas \ref{lem:Schatten-class-1-operator} and \ref{lem:asy of phi_count}, one gets  $CR_{1+z},R_{1+z}C\in \mathcal{B}_4\big(L^2(\mathbb{R}^3)\big)$ and 
$[R_{1+z},C]\in \mathcal{B}_2\big(L^2(\mathbb{R}^3)\big)$. Hence, by Theorem \ref{thm:trace-class-crit}, one infers that despite the last term in \eqref{eq:com_count_1st_round}, all operators are trace class. In addition, one computes 
\begin{align}\notag
& C R_{1+z}C R_{1+z}CR_{1+z}R_{1+z} = C [R_{1+z},C] R_{1+z}CR_{1+z}R_{1+z}   \\
& \hspace*{4.3cm} + C^2 R_{1+z}R_{1+z}CR_{1+z}R_{1+z} \\
& \quad = C [R_{1+z},C] R_{1+z}CR_{1+z}R_{1+z} + C^2 R_{1+z}[R_{1+z},C]R_{1+z}R_{1+z} \notag \\ 
& \qquad +  C^2 R_{1+z}CR_{1+z}^3 \notag\\
& \quad = C [R_{1+z},C] R_{1+z}CR_{1+z}R_{1+z} + C^2 R_{1+z}[R_{1+z},C]R_{1+z}R_{1+z} \notag \\ 
& \qquad +  C^2 [R_{1+z},C]R_{1+z}^3+C^3R_{1+z}^4\label{eq:com_count_2nd_round}.
\end{align}
Next, one notes that Lemma \ref{lem:commutator} implies the relation 
\[
[R_{1+z},C]=R_{1+z}\left(\Delta C\right)R_{1+z}+2R_{1+z}\left( \cQ C\right) \cQ R_{1+z}.
\] 
With the help of Lemma \ref{lem:asy of phi_count}, there exists $\kappa>0$ such that
\[
\max\{ \|C(x)^2\|, \|(\Delta C)(x)\|, \|( \cQ C)(x)\|\}\leq  \kappa (1+|x|)^{-2}, 
\quad x\in \mathbb{R}^3. 
\] 
Therefore, Lemma \ref{lem:Schatten-class-1-operator} and 
Theorem \ref{thm:trace-class-crit} imply 
\begin{align*}
& C [R_{1+z},C] R_{1+z}CR_{1+z} = C(R_{1+z} (\Delta C)R_{1+z}  \\
& \hspace*{3.9cm} +2R_{1+z}( \cQ C) \cQ R_{1+z})R_{1+z}CR_{1+z} \\
& \quad = CR_{1+z}(\Delta C)R_{1+z}CR_{1+z}+ 2 CR_{1+z} \left( \cQ C\right) 
R_{1+z} \cQ R_{1+z}CR_{1+z} \\
& \qquad \in \mathcal{B}_4\cdot \mathcal{B}_2\cdot \mathcal{B}_4 + \mathcal{B}_4\cdot \mathcal{B}_2\cdot \mathcal{B}\cdot \mathcal{B}_4 \subseteq \mathcal{B}_{1},
\end{align*}
and,
\[
   C^2 R_{1+z}[R_{1+z},C]R_{1+z} \in \mathcal{B}_2\cdot 
   \mathcal{B}_2\cdot \mathcal{B} \subseteq \mathcal{B}_1,
\]
as well as,
\begin{align*}
 C^2 [R_{1+z},C]R_{1+z}^3 &= C^2 R_{1+z}\left(\left(\Delta C\right) 
 R_{1+z}+2R_{1+z}\left( \cQ C\right) \cQ R_{1+z}\right)R_{1+z}^3 \\
                          &= C^2 R_{1+z}\left(\Delta C\right)R_{1+z}R_{1+z}^3 +2C^2 R_{1+z}
                          R_{1+z}\left( \cQ C\right)QR_{1+z}R_{1+z}^3 \\
                          & \quad \in \mathcal{B}_2\cdot \mathcal{B}_2\cdot \mathcal{B}+\mathcal{B}_2\cdot \mathcal{B}_2\cdot \mathcal{B} \subseteq \mathcal{B}_1.
\end{align*}
Noting that the inner trace maps trace class operators to trace class operators (cf.\ Remark \ref{rem-tracem}), and combining \eqref{eq:com_count_1st_round} and \eqref{eq:com_count_2nd_round} together with our assumption that $\tilde T$ is trace class, one concludes that 
\[
 T\coloneqq \tr_{4} \big(C^3R^4_{1+z}\big)
 = \tr_{4}\big(C^3\big) R^4_{1+z} \in \mathcal{B}_1\big(L^2(\mathbb{R}^3)\big).
\]

Next, one observes that $T$ is an integral operator induced by the following integral kernel
\begin{align*} 
& t\colon (x,y)\mapsto \int_{\left(\mathbb{R}^3\right)^3} 
\tr_{4}\big(C^3\big)(x) r_{1+z}(x-x_1)r_{1+z}(x_1-x_2)r_{1+z}(x_2-x_3)r_{1+z}(x_3-y)    \\
& \hspace*{2.6cm} \times d^3x_1 d^3x_2 d^3x_3,
\end{align*} 
where $r_{1+z}$ is the Helmholtz Green's function, see \eqref{C.1} associated with 
$(- \Delta + (1 + z))^{-1}$. By Theorem \ref{thm:continuity of integral kernel} (and 
Proposition \ref{prop:concrete case}), $t$ is continuous. As $T$ is trace class, Theorem \ref{thm:potential_is_in_L1} implies that the map $x\mapsto t(x,x)$ generates an 
$L^1(\mathbb{R}^3)$-function. Hence, 
\begin{align*}
   & \int_{\mathbb{R}^3} |t(x,x)| \, d^3x\\ 
   & \quad = \int_{\mathbb{R}^3} \bigg|\int_{\left(\mathbb{R}^3\right)^3} \tr_{4}\big(C^3\big)(x)r_{1+z}(x-x_1)r_{1+z}(x_1-x_2)r_{1+z}(x_2-x_3)r_{1+z}(x_3-x)   \\
& \hspace*{2.1cm} \times d^3x_1 d^3x_2 d^3x_3\bigg| \, d^3x \\ 
   & \quad = \int_{\mathbb{R}^3} \bigg|\int_{\left(\mathbb{R}^3\right)^3} 
   \tr_{4}\big(C^3\big)(x)r_{1+z}(x_1)r_{1+z}(x_1-x_2)r_{1+z}(x_2-x_3)r_{1+z}(x_3)    \\
& \quad \hspace*{2.1cm} \times d^3x_1 d^3x_2 d^3x_3\bigg| \, d^3x \\
& \quad = \int_{\mathbb{R}^3} 
\big|\tr_{4}\big(C^3\big)(x)\big| \, d^3x \, 
\big\langle \delta_{\{0\}}, R_{1+z}^4\delta_{\{0\}}\big\rangle < \infty.
\end{align*} 
In other words, 
\begin{equation}\label{eq:mult_is_in_L1}
    \tr_{4}\big(C^3\big)\in L^1(\mathbb{R}^3).
\end{equation}
The rest of the proof aims at showing that the statement \eqref{eq:mult_is_in_L1} is false. For this purpose we need to compute $\tr_{4} \big([ \cQ,\Phi]^3\big)$ on 
$\bigcup_{k\in \mathbb{N}_{\geq 2}} \tilde B_k$, with $\tilde B_k$ given in \eqref{eq:Bktilde}.
We recall from Lemma \ref{lem:asy of phi_count}\,$(ii)$, 
\[
(\partial_j \Phi)(x)= \frac{1}{k^{1/3}}\frac{1}{r_{k+1}}\sigma_j, \quad x\in \tilde B_k, \; j\in\{1,2,3\}.
\]
Hence,
\begin{align*}
\tr_{4}\big([ \cQ,\Phi]^3\big)(x) &= \sum_{j,m,\ell=1}^3 2i\epsilon_{jm\ell} \tr_2\big( 
(\partial_j\Phi)(x) (\partial_m\Phi)(x) (\partial_\ell\Phi)(x)\big) \\
Ê&= \sum_{j,m,\ell=1}^3 2i\epsilon_{jm\ell} \frac{1}{k}\frac{1}{r^3_{k+1}}\tr_2\left( ÊÊ
\sigma_j Ê\sigma_m \sigma_\ell\right) \\
Ê&= - \sum_{j,m,\ell=1}^3 4\epsilon^2_{jm\ell} \frac{1}{k}\frac{1}{r^3_{k+1}} \\ 
Ê&= -24\frac{1}{k}\frac{1}{r^3_{k+1}}, 
\end{align*}
implying,
\begin{equation}\label{eq:esti_on_B_k_tilde}
 \big|\tr_{4}\big([ \cQ,\Phi]^3\big)(x)\big|\geq 24\frac{1}{k}\frac{1}{r^3_{k+1}}, 
 \quad x\in \tilde B_k, \; k\in \mathbb{N}_{\geq 2}.
\end{equation}
However, employing Lemma \ref{lem:b2} one infers with the help \eqref{eq:esti_on_B_k_tilde} 
that for some $k_0\in \mathbb{N}$, 
\begin{align*}
\tr_4\big(C^3\big) &= \big\|\tr_{4}\big([ \cQ,\Phi]^3\big)\big\|_{L^1(\bbR^3)} 
\geq \sum_{k=k_0}^\infty \frac{1}{k}\frac{1}{r^3_{k+1}} 
   \vol \big(\tilde B_k\big) \\
 & = \frac{1}{(36)^3} \sum_{k=k_0}^\infty \frac{1}{k}\frac{1}{r^3_{k+1}}2^{3k} 
 = \frac{1}{(36)^3} \sum_{k=k_0}^\infty \frac{1}{k}\frac{1}{\left(2^k-2\right)^3}2^{3k}=\infty,
\end{align*}
contradicting \eqref{eq:mult_is_in_L1}. 
\end{proof}

\begin{remark} It might be of interest to compute the index of $\mathcal{Q}+\Phi$, with the potential $\Phi$ constructed in this section: One notes that $\Phi$ is a $\mathcal{Q}$-compact perturbation of the operator
  \[
      \mathcal{Q}+ U\text{ in }L^2(\mathbb{R}^n), \, \text{ where } \, 
    U\coloneqq \sum_{j=1}^3 \sigma_j.
  \] 
Since $U^2=I_2$ and $\partial_j U=0$, $j\in \{1,2,3\}$, one infers that $U$ is admissible. The index formula in Theorem \ref{thm:Fredholm-index} leads to $\ind(\mathcal{Q}+U)=0$, 
and hence to $\ind(\mathcal{Q}+\Phi)=0$.
\end{remark}

\newpage

 
\end{document}